\documentclass[12pt, leqno]{amsart} 
\usepackage{amssymb,amscd,amsfonts,amsbsy}
\usepackage{latexsym}
\usepackage{exscale}
\usepackage{amsmath,amsthm,amsfonts}
\usepackage{mathrsfs}
\usepackage{xcolor} 
\usepackage[colorlinks=true, linkcolor=blue, citecolor=red, urlcolor=red, 
]{hyperref} 
\usepackage{esint} 
\usepackage{stmaryrd}
\usepackage{pifont}
\usepackage{bm}
\usepackage{enumerate} 
\usepackage{enumitem} 

\usepackage[utf8]{inputenc}

\calclayout

\newcommand{\black}{\color{black}}

\usepackage{autobreak}
\allowdisplaybreaks

\makeatletter
\@namedef{subjclassname@2020}{%
  \textup{2020} Mathematics Subject Classification}
\makeatother

\theoremstyle{plain}
\newtheorem{theorem}[equation]{Theorem}

\newtheorem{lemma}[equation]{Lemma}

\theoremstyle{definition}
\newtheorem{definition}[equation]{Definition}

\numberwithin{equation}{section}

\def\C{\mathbb{C}}
\def\N{\mathbb{N}}
\def\D{\mathcal{D}}

\def\E{\mathbb{E}}
\def\F{\mathscr{F}}
\def\G{\mathscr{G}}
\def\Q{\mathcal{Q}}

\def\S{\mathbf{S}}
\def\K{\mathcal{K}}

\def\I{\mathbb{I}}

\def\R{\mathbb{R}}
\def\Rn{\mathbb{R}^n}
\def\Rnn{\R^{n_1} \times \R^{n_2}} 

\def\Z{\mathbb{Z}}

\def\a{\bm{\alpha}}
\def\b{\bm{b}}
\def\w{\omega}

\def\loc{\operatorname{loc}}

\def\supp{\operatorname{supp}}
\def\bmo{\operatorname{bmo}}
\def\cmo{\operatorname{cmo}}
\def\BMO{\operatorname{BMO}}
\def\VMO{\operatorname{VMO}}
\def\XMO{\operatorname{XMO}}
\def\CMO{\operatorname{CMO}}

\def\d{\operatorname{d}}
\def\rd{\operatorname{rd}}

\def\ch{\operatorname{ch}}

\DeclareMathOperator*{\esssup}{ess\,sup}
\DeclareMathOperator*{\essinf}{ess\,inf}

\renewcommand{\emptyset}{\text{\textup{\O}}}

\begin{document}

\author{Mingming Cao}
\address{Mingming Cao\\
Instituto de Ciencias Matem\'aticas CSIC-UAM-UC3M-UCM\\
Con\-se\-jo Superior de Investigaciones Cient{\'\i}ficas\\
C/ Nicol\'as Cabrera, 13-15\\
E-28049 Ma\-drid, Spain} \email{mingming.cao@icmat.es}

\author{K\^{o}z\^{o} Yabuta}
\address{K\^{o}z\^{o} Yabuta\\
Research Center for Mathematics and Data Science\\
Kwansei Gakuin University\\
Gakuen 2-1, Sanda 669-1337\\
Japan} \email{kyabuta3@kwansei.ac.jp}

\thanks{The first author acknowledges financial support from Spanish Ministry of Science and Innovation through the Ram\'{o}n y Cajal  2021 (RYC2021-032600-I), through the ``Severo Ochoa Programme for Centres of Excellence in R\&D'' (CEX2023-001347-S), and through PID2022-141354NB-I00.}

\year=2025 \month=03 \day=19

\date{\today}

\subjclass[2020]{42B20, 42B35}


\keywords{
Bi-parameter singular integrals,
$T1$ theorem, 
Dyadic analysis,
Multilinear theory, 
Weighted compactness, 
Mean continuity, 
Commutators}

\begin{abstract}
We develop the compactness theory of multilinear singular integrals on product spaces using a modern point of view. The first main result is a compact $T1$ theorem for multilinear Calder\'{o}n--Zygmund operators on product spaces. More specifically, we prove that a multilinear singular integral operator $T$ on product spaces can be extended to a compact multilinear operator from $L^{p_1}(w_1^{p_1}) \times \cdots \times L^{p_m}(w_m^{p_m})$ to $L^p(w^p)$ for all exponents $\frac1p = \sum_{j=1}^m \frac{1}{p_j}>0$ with $p_1, \ldots, p_m \in (1, \infty]$ and for all weights $\vec{w} \in A_{\vec{p}}(\mathbb{R}^{n_1} \times \mathbb{R}^{n_2})$ if the following hypotheses are satisfied: (H1) $T$ admits a compact full kernel representation, (H2) $T$ admits a compact partial kernel representation, (H3) $T$ satisfies the weak compactness property, (H4) $T$ satisfies the diagonal $\mathrm{CMO}$ condition, and (H5) $T$ satisfies the product $\mathrm{CMO}$ condition. This is a multilinear compact extension of Journ\'{e}'s $T1$ theorem on product spaces. The second main result establishes the mean continuity of commutators $[\bm{b}, T]_{\bm{\alpha}}$ on weighted Lebesgue spaces as above, which can be viewed as a substitution of compactness  because the compactness of $[\bm{b}, T]_{\bm{\alpha}}$ is equivalent to $\bm{b} \equiv \text{constant}$ when $T$ is a non-degenerate bi-parameter singular integral. Our main tools include multilinear bi-parameter dyadic representation, multilinear extrapolation, multilinear interpolation, and Kolmogorov--Riesz compactness criterion.
\end{abstract}

\title{Dyadic analysis of compactness on product spaces}

\maketitle
\tableofcontents

\section{Introduction}

\subsection{Motivation and the main results} 
The current investigation belongs to an extensive research program, which aims to advance our understanding of the compactness of singular integral operators on product spaces. The main reason why we study it comes from three aspects: 
\begin{itemize}
\item There exists very limited literature concerning the compactness of singular integrals, although the compactness of commutators has attracted a lot of interest recently and many new techniques have been developed, for example, \cite{CIRXY, COY, HL22, HL23, HLTY, HOS, TYY}.  

\item There is no a mature theory to research the compactness of bi-parameter singular integrals, except a compactness result by means of projection \cite{CYY}. In terms of boundedness on product spaces, plenty of remarkable works have been done by \cite{HLMV, LMV20, LMV21, Mar, Ou} for singular integrals and by \cite{HPW, OPS} for commutators. Extending boundedness to compactness arises naturally. 

\item What is the precise dyadic structure behind compact bi-parameter Calder\'{o}n--Zygmund operators? The rapid development of dyadic analysis has leaded to a much better understanding of one-parameter and bi-parameter Calder\'{o}n--Zygmund operators. A modern thinking of singular integrals was given by Hyt\"{o}nen \cite{Hyt}, who represented singular integral operators as an average of some natural dyadic model operators. 
The bi-parameter extension was obtained in \cite{LMV20, Mar}, while sparse domination \cite{Ler, LOR}  provides a brand new perspective. All these make many problems in the weighted or/and bi-parameter setting more attainable. Inspired by the above, we explore the dyadic structure of compact bi-parameter Calder\'{o}n--Zygmund operators. 
\end{itemize} 

Before entering into a detailed discussion of the multi-parameter theory, we present the first main result of this paper.

\begin{theorem}\label{thm:cpt}
Let $T$ be an $m$-linear bi-parameter singular integral operator (cf. Definition {\rm \ref{def:SIO}}). Assume that $T$ satisfies the following hypotheses: 
\begin{list}{\rm (\theenumi)}{\usecounter{enumi}\leftmargin=1.3cm \labelwidth=1cm \itemsep=0.1cm \topsep=0.2cm \renewcommand{\theenumi}{H\arabic{enumi}}}

\item\label{H1} $T$ admits the compact full kernel representation (cf. Definition {\rm \ref{def:full}}),

\item\label{H2} $T$ admits the compact partial kernel representation (cf. Definition {\rm \ref{def:partial}}),

\item\label{H3} $T$ satisfies the weak compactness property (cf. Definition {\rm \ref{def:WCP}}),

\item\label{H4} $T$ satisfies the diagonal $\CMO$ condition (cf. Definition {\rm \ref{def:diag-CMO}}),

\item\label{H5} $T$ satisfies the product $\CMO$ condition (cf. Definition {\rm \ref{def:prod-BMO}}).
\end{list}
Then $T$ is compact from $L^{p_1}(w_1^{p_1}) \times \cdots \times L^{p_m}(w_m^{p_m})$ to $L^p(w^p)$ for all $\vec{p} = (p_1, \ldots, p_m) \in (1, \infty]^m$ and $\vec{w} = (w_1, \ldots, w_m) \in A_{\vec{p}}(\Rnn)$, where $\frac1p = \sum_{j=1}^m \frac{1}{p_j} > 0$ and $w = \prod_{j=1}^m w_j$. 
\end{theorem}

To be explicit, let us recall the compactness of multilinear operators. Given quasi-normed spaces $\mathscr{X}_1, \ldots, \mathscr{X}_m, \mathscr{X}$, an $m$-linear operator $T: \mathscr{X}_1 \times \cdots \times \mathscr{X}_m \to \mathscr{X}$ is said to be \emph{compact} if $T(B_1 \times \cdots \times B_m)$ is \emph{precompact} in $\mathscr{X}$ for all bounded sets $B_j \subset \mathscr{X}_j$, $j=1, \ldots, m$, i.e., $\overline{T(B_1 \times \cdots \times B_m)}$ is a compact subset of $\mathscr{X}$.

On the technical level there is no existing approach to our result. First, the compactness of singular integrals originated in \cite{Vil}, where wavelet analysis was used. Even in the bilinear one-parameter situation, the complicated geometric relationship among cubes $I, J, K$ makes it impossible to obtain good localization and then readily parametrize the sums according to eccentricities and relative distances of cubes. Second, although the authors \cite{CYY} successfully established weighted compactness of linear bi-parameter Calder\'{o}n--Zygmund operators, the method of projection cannot be applied to the bilinear case. A main reason is that the projection operator $P_{\D(N)}^{\perp}$ in \eqref{def:PN} depends on the dyadic grid $\D$ so that the goodness of dyadic cubes in $\D$ cannot be added into the summation, which leads to the disappearance of satisfactory local estimates involving three cubes $I, J, K \in \D$. Third, the theory of sparse domination is missing in the multi-parameter scenario, even for the boundedness. In the one-parameter setting, the authors in \cite[Theorem 3.7]{SVW} obtained a pointwise sparse domination to show weighted compactness. But it requires the weak (1, 1) boundedness, which is false in the multi-parameter case.

Our result should be compared with \cite[Theorem 1.2]{LMV21}, which established the weighted boundedness of    multilinear bi-parameter singular integral operator under less restrictive assumptions. Unfortunately, those assumptions are not sufficient to obtain compactness. Generally speaking, multilinear bi-parameter Calder\'{o}n--Zygmund operators (cf. Definition \ref{def:CZO}) are not compact. To clarify this, recall the bilinear Riesz transform $\mathcal{R}_j^i$ on $\R^{n_i}$ defined by 
\begin{align*}
\mathcal{R}_j^i (f_1, f_2)(x) 
:= \mathrm{p.v. } \int_{\R^{n_i}} \int_{\R^{n_i}} 
\frac{(x - y)_j + (x - z)_j}{(|x-y|^2 + |x-z|^2)^{\frac{2n_i+1}{2}}} 
f_1(y) f_2(z)\, dy \, dz,  
\end{align*}
where $(x-y)_j$ is the $j$-th coordinate of $x-y \in \R^{n_i}$, $j=1, \ldots, n_i$. Let $\frac1p = \frac{1}{p_1} + \frac{1}{p_2}$ with $p_1, p_2 \in (1, \infty)$. Then one can check that 
\begin{enumerate}
\item\label{R1} $\mathcal{R}_{j_1}^1 \otimes \mathcal{R}_{j_2}^2$ satisfies the full kernel representation (cf. Definition \ref{def:SIO});

\item\label{R2} $\mathcal{R}_{j_1}^1 \otimes \mathcal{R}_{j_2}^2$ satisfies the partial kernel representation (cf. Definition \ref{def:SIO}); 

\item\label{R3} $\mathcal{R}_{j_1}^1 \otimes \mathcal{R}_{j_2}^2$ satisfies \eqref{H3}, \eqref{H4}, and \eqref{H5}; 

\item\label{R4} $\mathcal{R}_{j_1}^1 \otimes \mathcal{R}_{j_2}^2$ is bounded from $L^{p_1}(\Rnn) \times L^{p_2}(\Rnn)$ to $L^p(\Rnn)$; 

\item\label{R5} $\mathcal{R}_{j_1}^1 \otimes \mathcal{R}_{j_2}^2$ is not compact from $L^{p_1}(\Rnn) \times L^{p_2}(\Rnn)$ to $L^p(\Rnn)$. 
\end{enumerate}
Indeed, properties \eqref{R1} and \eqref{R2} are trivial, and property \eqref{R3} is a consequence of that 
\begin{align*}
\mathcal{R}_{j_i}^i(1, 1) = \mathcal{R}_{j_i}^{i, 1*}(1, 1) = \mathcal{R}_{j_i}^{i, 2*}(1, 1) = 0 
= \langle \mathcal{R}_{j_i}^i(\mathbf{1}_{I^i}, \mathbf{1}_{I^i}), \mathbf{1}_{I^i} \rangle,  
\end{align*}
for all cubes $I^i \subset \R^{n_i}$, and that $(\mathcal{R}_{j_1}^1 \otimes \mathcal{R}_{j_2}^2)_{1, 2}^{k_1*, k_2*} = \mathcal{R}_{j_1}^{1, k_1*} \otimes \mathcal{R}_{j_2}^{2, k_2*}$ for all $k_1, k_2 \in \{0, 1, 2\}$, where the adjoints of a bilinear bi-parameter operator are defined in Section \ref{sec:SIO} and the adjoints of a bilinear one-parameter operator $T$ are given by  
\begin{align*}
\langle T(f_1, f_2), f_3 \rangle 
= \langle T^{1*}(f_3, f_2), f_1 \rangle 
= \langle T^{2*}(f_1, f_3), f_2 \rangle. 
\end{align*} 
Then property \eqref{R4} follows from properties \eqref{R1}--\eqref{R3} and \cite[Theorem 1.2]{LMV21}. Additionally, property \eqref{R5} is a consequence of Lemma \ref{lem:ncpt} and \cite[Lemma B.13]{CLSY}, where the latter implies that bilinear one-parameter Riesz transforms are not compact. Therefore, the above indicates that one has to strengthen the assumptions in order to achieve  compactness of bilinear bi-parameter Calder\'{o}n--Zygmund operators.

Next, let us turn to commutators. Let $T$ be an operator from $\mathscr{X}_1  \times \cdots \times \mathscr{X}_m$ into $\mathscr{X}$, where $\mathscr{X}_1, \ldots, \mathscr{X}_m$ are some normed spaces and and $\mathscr{X}$ is a quasi-normed space. Given $\vec{f} := (f_1, \ldots, f_m) \in \mathscr{X}_1 \times \cdots \times \mathscr{X}_m$, $\b=(b_1, \ldots, b_m)$ of measurable functions, and $j \in \{1, \ldots, m\}$, we define, whenever it makes sense, the first order commutator as 
\begin{align*}
[\b, T]_{e_j} (\vec{f})
:= b_j T(f_1,\ldots, f_j, \ldots, f_m) - T(f_1,\ldots, b_j f_j, \ldots, f_m), 
\end{align*}
where $e_j$ is the basis of $\Rn$ with the $j$-th component being $1$ and other components being $0$. Then for any $k \in \N_+$, we define the $k$-th order commutator of $T$ in the $j$-th entry of $T$ as 
\begin{align*}
[\b, T]_{k e_j} 
:= [\b, \cdots [\b, [\b, T]_{e_j}]_{e_j} \cdots]_{e_j},
\end{align*}
where the commutator is performed $k$ times. Finally, for a multi-index $\a = (\alpha_1, \ldots, \alpha_m) \in \N^m$, we define 
\begin{align*}
[\b, T]_{\a}:= [\b, \cdots [\b, [\b, T]_{\alpha_1 e_1}]_{\alpha_2 e_2} \cdots]_{\alpha_m e_m}.
\end{align*}

It is natural to wonder whether the commutator $[\b, T]_{\a}$ is compact for a bi-parameter singular integral operator $T$, for which the boundedness of $[\b, T]_{\a}$ was established in \cite{HPW, LMV21}. A recent work \cite{LM} gave an interesting result, which asserts that the compactness of $[b, T]_1$ is equivalent to $b \equiv \text{constant}$ when $T$ is a non-degenerate bi-parameter singular integral. This also indicates that for multi-parameter commutators $[\b, T]_{\a}$ it is not meaningful to study finer properties of compactness, for example, the Schatten property. In view of these facts above, we would like to seek a property of $[\b, T]_{\a}$ which is weaker than the compactness but stronger than the boundedness. More precisely, we introduce the following definition.

\begin{definition}
Given exponents $p_1, \ldots, p_m \in [1, \infty]$ and weights $w_1, \ldots, w_m$ on $\Rnn$, a bounded $m$-linear operator $T: L^{p_1}(w_1^{p_1}) \times \cdots \times L^{p_m}(w_m^{p_m}) \to L^p(w^p)$ is said to be \emph{mean continuous} if 
\begin{align*}
\lim_{r \to 0}
\sup_{\substack{\|f_j\|_{L^{p_j}(w_j^{p_j})} \le 1 \\ j=1, \ldots, m}} 
\bigg\|\bigg[\fint_{B_{\vec{n}}(0, r)}
\big|(\tau_y - \tau_{y_1} - \tau_{y_2} + I) T(\vec{f}) \big|^a \, dy 
\bigg]^{\frac1a}\bigg\|_{L^p(w^p)} 
= 0,
\end{align*}
for some $a \in (0, 1]$, where $\frac1p = \sum_{j=1}^m \frac{1}{p_j}$, $w = \prod_{j=1}^m w_j$,
\begin{align*}
B_{\vec{n}}(x, r) := \big\{y \in \Rnn: |y-x| = |y_1 - x_1| + |y_2 - x_2| < r\big\},
\end{align*} 
and
\begin{align*}
& (\tau_y - \tau_{y_1} - \tau_{y_2} + I) f(x)  
:= \tau_y f(x) - \tau_{y_1} f(x) - \tau_{y_2} f(x) + f(x) 
\\
& = f(x_1 - y_1, x_2 - y_2) - f(x_1 - y_1, x_2) - f(x_1, x_2 - y_2) + f(x_1, x_2).
\end{align*}
\end{definition}

The second main result of this paper is formulated as follows.

\begin{theorem}\label{thm:bT}
Let $T$ be an $m$-linear bi-parameter Calder\'{o}n--Zygmund operator (cf. Definition {\rm \ref{def:CZO}}). Then for any $\a \in \N^m \setminus \{0\}^m$ and $\b = (b_1, \ldots, b_m) \in \cmo(\Rnn)^m$, the commutator $[\b, T]_{\a}$ is mean continuous from $L^{p_1}(w_1^{p_1}) \times \cdots \times L^{p_m}(w_m^{p_m})$ to $L^p(w^p)$ for all $\vec{p} = (p_1, \ldots, p_m) \in (1, \infty]^m$ and $\vec{w} = (w_1, \ldots, w_m) \in A_{\vec{p}}(\Rnn)$, where $\frac1p = \sum_{j=1}^m \frac{1}{p_j} > 0$ and $w = \prod_{j=1}^m w_j$. 
\end{theorem}

The mean continuity is extracted from Kolmogorov--Riesz theorems, which give necessary and sufficient conditions for a subset of a Lebesgue space to be compact (see Section \ref{sec:hist} below). In the bi-parameter setting, we also establish similar compactness criterion, see Theorems \ref{thm:KRWA}--\ref{thm:KRttt}, where the last one implies the mean continuity is indeed weaker than the compactness. It is worth mentioning that the form of mean continuity is suitable to weighted Lebesgue spaces $L^p(v)$ with general weights $v$ which are not necessarily invariant under translations.

\subsection{Outline of the proof of Theorem \ref{thm:cpt}} 
Let us briefly outline how to prove Theorem \ref{thm:cpt}. Definitions and notation needed are postponed to Section \ref{sec:SIO}. The core of the proof is a compact $m$-linear bi-parameter dyadic representation as follows.

\begin{theorem}\label{thm:repre}
Let $T$ be an $m$-linear bi-parameter singular integral operator (cf. Definition {\rm \ref{def:SIO}}). Assume that $T$ satisfies the hypotheses \eqref{H1}--\eqref{H5}. Then $T$ admits a compact $m$-linear bi-parameter dyadic representation (cf. Definition {\rm \ref{def:repre}}). 
\end{theorem}

It is the first time to introduce a compact dyadic representation in the bi-parameter setting in order to study (weighted) compactness of singular integrals on product spaces. It definitely states the dyadic structure of multilinear bi-parameter Calder\'{o}n--Zygmund operators. In contrast to pointwise or integral form domination, such representation has a great advantage to transfer smoothly initial problems to corresponding dyadic problems, for example, commutators estimates and compactness. Theorem \ref{thm:repre} is based on our previous research  in the bilinear case \cite{CLSY} and in the bi-parameter setting \cite{CYY} (although no dyadic representation is established), and the general framework of multilinear bi-parameter dyadic operators in \cite{LMV21}. Besides, our compact dyadic operators (cf. Definitions \ref{def:shift}--\ref{def:paraproduct}) are much more complicated than those in \cite{CLSY}, and carry much more subtle information than those in \cite{LMV21}.

The usefulness of dyadic representation has been shown in other aspects. It can trace back to sharp weighted norm inequalities \cite{Pet}, where a representation was given for the Hilbert transform. Then it was refined in \cite{Hyt} to solve the well-known $A_2$ conjecture for general Calder\'{o}n--Zygmund operators. After that, the bi-parameter extension was given by \cite{Mar} and the bilinear case was due to \cite{LMOV}, while the multilinear multi-parameter representation was established in \cite{LMV20, LMV21}, which can be applied to obtain genuinely multilinear weighted estimates for any multi-parameter Calder\'{o}n--Zygmund operators  \cite{LMV21} and their commutators. Moreover, the bi-parameter representation in \cite{Mar} has been used to two-weight commutators estimates \cite{HPW} and higher order commutators \cite{OPS}. Recently, the one-parameter representation in \cite{Hyt} has found its application to Schatten class property of noncommutative martingale paraproducts and operator-valued commutators \cite{WZ}.

Having presented a dyadic representation, we have to show weighted compactness of bi-parameter dyadic operators. For this purpose, we establish weighted compactness of the average of the family of dyadic operators, instead of treating one by one.

\begin{theorem}\label{thm:dyadic-cpt}
Suppose that the family $\{\mathbf{T}_{\w}\}_{\w \in \Omega}$ is of the same type for any $\mathbf{T}_{\w} \in \big\{\S_{\D_{\w}}^k, \mathbf{P}_{\D_{\w}}^{1, k}, \mathbf{P}_{\D_{\w}}^{2, k}, \mathbf{F}_{\mathbf{a}_{\w}} \big\}$ (cf. Definitions {\rm \ref{def:shift}}--{\rm \ref{def:paraproduct}}). Then $\mathbb{E}_{\w} \mathbf{T}_{\w}$ is compact from $L^{p_1}(w_1^{p_1}) \times \cdots \times L^{p_m}(w_m^{p_m})$ to $L^p(w^p)$ for all $\vec{p} = (p_1, \ldots, p_m) \in (1, \infty]^m$ and $\vec{w} = (w_1, \ldots, w_m) \in A_{\vec{p}}(\Rnn)$, where $\frac1p = \sum_{j=1}^m \frac{1}{p_j} > 0$ and $w = \prod_{j=1}^m w_j$. 
\end{theorem}

In addition, we develop Rubio de Francia extrapolation of compactness on product spaces below. 

\begin{theorem}\label{thm:RdF-cpt}
Assume that $T$ is an $m$-linear operator such that  
\begin{list}{\rm (\theenumi)}{\usecounter{enumi}\leftmargin=1.2cm \labelwidth=1cm \itemsep=0.2cm \topsep=0.2cm \renewcommand{\theenumi}{\alph{enumi}}}

\item\label{RdFcpt-1} $T$ is compact from $L^{p_1}(u_1^{p_1}) \times \cdots \times L^{p_m}(u_m^{p_m})$ to $L^p(u^p)$ for some $\vec{p} = (p_1, \ldots, p_m) \in [1, \infty]^m$ and for some $\vec{u} = (u_1, \ldots, u_m) \in A_{\vec{p}}(\Rnn)$, where $\frac1p = \sum_{j=1}^m \frac{1}{p_j} > 0$ and $u = \prod_{j=1}^m u_j$;

\item\label{RdFcpt-2} $T$ is bounded from $L^{q_1}(v_1^{q_1}) \times \cdots \times L^{q_m}(v_m^{q_m})$ to $L^q(v^q)$ for some $\vec{q} = (q_1, \ldots, q_m) \in [1, \infty]^m$ and for all $\vec{v} = (v_1, \ldots, v_m) \in A_{\vec{q}}(\Rnn)$, where $\frac1q = \sum_{j=1}^m \frac{1}{q_j}$ and $v = \prod_{j=1}^m v_j$. 
\end{list} 
Then $T$ is compact from $L^{r_1}(w_1^{r_1}) \times \cdots \times L^{r_m}(w_m^{r_m})$ to $L^r(w^r)$ for all $\vec{r} = (r_1, \ldots, r_m) \in (1, \infty]^m$ and for all $\vec{w} = (w_1, \ldots, w_m) \in A_{\vec{r}}(\Rnn)$, where $\frac1r = \sum_{j=1}^m \frac{1}{r_j} > 0$ and $w = \prod_{j=1}^m w_j$.  
\end{theorem}

Theorem \ref{thm:RdF-cpt} is indispensable because quasi-Banach compactness can not be obtained directly from the dyadic representation above. Furthermore, extrapolation reduces what we want to unweighted compactness in Banach range, which is quite useful in practice. We should mention that Theorem \ref{thm:RdF-cpt} also holds for any $m$-parameter, which can be justified from its proof in Section \ref{sec:RdF}. Its feasibility benefits from our significant work \cite{COY}, where interpolation and extrapolation for multilinear compact operators were established systematacially. Beyond that, unlike to extrapolation of boundedness \cite[Theorem 3.12]{LMV21}, the case $p = p_1 = \cdots = p_m = \infty$ cannot be used as the starting point of extrapolation of compactness. For example, even in one-parameter setting, it is very hard to show compactness in the case $p = p_1 = \cdots = p_m = \infty$ because it needs $L^1 \times \cdots \times L^1 \to L^{\frac1m, \infty}$ compactness (cf. \cite[Section 9]{CLSY}), which is proved by means of $L^{q_1} \times \cdots \times L^{q_m} \to L^{q_m}$ compactness for some exponents $q, q_1, \ldots q_m \in (1, \infty)$ (cf. \cite[Section 8]{CLSY}).

\subsection{Outline of the proof of Theorem \ref{thm:bT}} 
Our basic proof strategy of Theorem \ref{thm:bT} is to use bi-parameter dyadic representation in \cite{LMV21}. This allows us to reduce matters to showing mean continuity of bi-parameter dyadic operators. 

\begin{theorem}\label{thm:dyadic-bT}
Suppose that the family $\{\mathbb{T}_{\w}\}_{\w \in \Omega}$ is of the same type for any $\mathbb{T}_{\w} \in \big\{\mathbb{S}_{\D_{\w}}^k, \mathbb{P}_{\D_{\w}}^{1, k}, \mathbb{P}_{\D_{\w}}^{2, k}, \mathbb{F}_{\mathbf{a}_{\w}} \big\}$ (cf. Definitions {\rm \ref{def:shift}}--{\rm \ref{def:paraproduct}}). Then for any $\a \in \N^m \setminus \{0\}^m$ and $\b = (b_1, \ldots, b_m) \in \cmo(\Rnn)^m$, $[\b, \mathbb{E}_{\w} \mathbb{T}_{\w}]_{\a}$ is mean continuous from $L^{p_1}(w_1^{p_1}) \times \cdots \times L^{p_m}(w_m^{p_m})$ to $L^p(w^p)$ for all $\vec{p} = (p_1, \ldots, p_m) \in (1, \infty]^m$ and $\vec{w} = (w_1, \ldots, w_m) \in A_{\vec{p}}(\Rnn)$, where $\frac1p = \sum_{j=1}^m \frac{1}{p_j} > 0$ and $w = \prod_{j=1}^m w_j$. 
\end{theorem}

As mentioned above, the biggest challenge in the multilinear setting is to obtain desired estimates in quasi-Banach range. Interestingly and unexpectedly, the mean continuity also enjoys the property of extrapolation as follows, which provides us with great convenience in practice.

\begin{theorem}\label{thm:RdF-bT}
Let $\a \in \N^m \setminus \{0\}^m$ and $\b = (b_1, \ldots, b_m) \in \bmo(\Rnn)^m$. Assume that $T$ is an $m$-linear operator such that  
\begin{list}{\rm (\theenumi)}{\usecounter{enumi}\leftmargin=1.2cm \labelwidth=1cm \itemsep=0.2cm \topsep=0.2cm \renewcommand{\theenumi}{\alph{enumi}}}

\item\label{RdFbT-1} $[\b, T]_{\a}$ is mean continuous from $L^{p_1}(u_1^{p_1}) \times \cdots \times L^{p_m}(u_m^{p_m})$ to $L^p(u^p)$ for some $\vec{p} = (p_1, \ldots, p_m) \in [1, \infty]^m$ and for some $\vec{u} = (u_1, \ldots, u_m) \in A_{\vec{p}}(\Rnn)$, where $\frac1p = \sum_{j=1}^m \frac{1}{p_j} > 0$ and $u = \prod_{j=1}^m u_j$;

\item\label{RdFbT-2} $T$ is bounded from $L^{q_1}(v_1^{q_1}) \times \cdots \times L^{q_m}(v_m^{q_m})$ to $L^q(v^q)$ for some $\vec{q} = (q_1, \ldots, q_m) \in [1, \infty]^m$ and for all $\vec{v} = (v_1, \ldots, v_m) \in A_{\vec{q}}(\Rnn)$, where $\frac1q = \sum_{j=1}^m \frac{1}{q_j}$ and $v = \prod_{j=1}^m v_j$. 
\end{list} 
Then $[\b, T]_{\a}$ is mean continuous from $L^{r_1}(w_1^{r_1}) \times \cdots \times L^{r_m}(w_m^{r_m})$ to $L^r(w^r)$ for all $\vec{r} = (r_1, \ldots, r_m) \in (1, \infty]^m$ and for all $\vec{w} = (w_1, \ldots, w_m) \in A_{\vec{r}}(\Rnn)$, where $\frac1r = \sum_{j=1}^m \frac{1}{r_j} > 0$ and $w = \prod_{j=1}^m w_j$.  
\end{theorem}

The formulation of Theorem \ref{thm:RdF-bT} is inspired by Theorem \ref{thm:RdF-cpt}. Theorem \ref{thm:RdF-bT} enables us to extrapolate the mean continuity of $[\b, T]_{\a}$ from just one unweighted space to the full range of weighted spaces, whenever an $m$-linear operator $T$ is bounded on weighted Lebesgue spaces. The proof of Theorem \ref{thm:RdF-bT} contains three main ingredients: (i) interpolation of multiple weights (cf. Lemma \ref{lem:AA}); (ii) extrapolation of boundedness for commutators (cf. Theorem \ref{thm:TTb}); (iii) interpolation on mixed-norm Lebesgue spaces (cf. \cite[Theorem 3.5]{COY}).

\subsection{Historical overview}\label{sec:hist}
The multi-parameter theory of singular integrals was pioneered by Fefferman and Stein \cite{FS}, who dealt with bi-parameter singular integral operators of convolution type. In the language of vector-valued Calder\'{o}n--Zygmund theory, Journ\'{e} \cite{Jou} formulated a $T1$ theorem for bi-parameter singular integrals with general kernels. From a fresh viewpoint on kernels (full and partial kernel representations), Martikainen \cite{Mar}  introduced a class of bi-parameter singular integrals and proved a bi-parameter dyadic representation, and then obtained a bi-parameter $T1$ theorem as a pleasant byproduct. Unexpectedly, full and partial kernel representations are actually equivalent to Journ\'{e}'s vector-valued formulation, see \cite{Grau}. Furthermore,  the bi-parameter $T1$ theorem was extended to non-homogeneous spaces in \cite{HM} using dyadic probabilistic methods and non-homogeneous analysis. See \cite{C79, CF80, CF85, Fef87, Fef88, FL} for more classical multi-parameter results and applications. 

On the other hand, there are a large number of literatures devoted to studying multilinear multi-parameter theory, which is more challenging than the multilinear one-parameter and linear multi-parameter theory of singular integrals. A typical example is bilinear bi-parameter Coifman--Meyer multipliers by Muscalu et al. \cite{MPTT1}, which appeared naturally in the Leibniz rule for fractional derivatives. They were proved in \cite{LMV20} to be bilinear bi-parameter Calder\'{o}n--Zygmund operators, and the usage of modern tools in \cite{LMV20, LMV21} greatly enriches the multi-parameter theory of multipliers by establishing novel commutators, weighted, and square function estimates. Moreover, some interesting investigations were carried out for variants of bilinear bi-parameter Coifman--Meyer multipliers, such as multilinear multi-parameter Hilbert transform \cite{DL1}, multilinear multi-parameter pseudo-differential operators \cite{DL2}, multilinear multi-parameter H\"{o}rmander multipliers \cite{CHLPZ}, and so on.

In term of the compactness of commutators, Uchiyama \cite{Uch} first established the equivalence between $b \in \CMO(\Rn)$ and $L^p$ compactness of commutators $[b, T]$ for any $p \in (1, \infty)$ and Calder\'{o}n--Zygmund operator $T$ with smooth kernel, where $\CMO(\Rn)$ denotes the closure in $\BMO(\Rn)$ of infinitely differentiable functions with compact support. A weighted analogy was proved for Riesz transforms and fractional integrals in \cite{WY}. In addition, the unweighted compactness was extended by B\'{e}nyi and Torres \cite{BT} to the bilinear setting $L^{p_1}(\Rn) \times L^{p_2}(\Rn) \to L^p(\Rn)$ for $\frac1p = \frac{1}{p_1} + \frac{1}{p_2}$ with $p, p_1, p_2 \in (1, \infty)$, while a complete characterization of compactness without the restriction $p \in (1, \infty)$ was shown in \cite{CCHTW} for certain homogeneous bilinear Calder\'{o}n--Zygmund operators. All results aforementioned were given for $b \in \CMO(\Rn)$. However, some recent works have shown that a weaker assumption is sufficient to obtain weighted compactness of commutators. In fact, introducing a new space $\XMO(\Rn)$ between $\CMO(\Rn)$ and $\BMO(\Rn)$, Torres and Xue \cite{TX} obtained unweighted compactness of commutators of bilinear Fourier multipliers and bilinear pseudo-differential operators. It was improved to the weighted case in \cite{TXYY}, where an equivalent characterization of $\XMO(\Rn)$ and the fact $\CMO(\Rn) \subsetneq \XMO(\Rn) \subsetneq \VMO(\Rn) \subsetneq \BMO(\Rn)$ were demonstrated in the spirit of Uchiyama's proof, although the proof needs some essential novel techniques on dyadic cubes and delicate geometrical observations. The more general  fractional variants were established in \cite{TYY} recently.

As a classical compactness criterion, the celebrated Kolmogorov--Riesz theorem was first discovered by Kolmogorov in $L^p([0, 1])$ for $p \in (1, \infty)$, which is a one-parameter version of Theorem \ref{thm:KRWA} with $w \equiv 1$ (the condition (b) holds automatically in this case). Soon after, it was expanded by Tamarkin \cite{Tam} to the general situation $\Rn$, which was further extended to the case $p=1$ by Tulajkov \cite{Tul}. Simultaneously, Riesz \cite{Riesz} independently proved a result in the form of Theorem \ref{thm:RKB} for $w \equiv 1$ and $p \in [1, \infty)$, which was improved to the setting $p \in (0, 1)$ by Tsuji \cite{Tsu}. Recently, these results aforementioned were refined on weighted Lebesgue spaces $L^p(w)$ for general weights $w$ and exponents $p \in (0, \infty)$ in \cite{COY}. Moreover, this theorem has been generalized to various function spaces, such as Banach function spaces \cite{GW, GRa}, metric measure spaces \cite{GM14, HH, Kro}, variable Lebesgue spaces \cite{GM15}, Musielak--Orlicz spaces \cite{Mus}, and weighted Lorentz spaces \cite{CLSY}. Beyond that, it has found many applications to Harmonic Analysis such as compactness of commutators \cite{BT, CIRXY, CLSY, CCHTW, TYY, TX, Uch} and extrapolation of compactness \cite{CIRXY, COY}, and to PDEs such as Poiseuille flow of nematic liquid crystals \cite{CHL}, Beltrami equations \cite{CC}, time fractional PDEs \cite{LLiu}, motion by mean curvature \cite{SY}, and so on.

\subsection{Structure of the paper} 
This paper is organized as follows. Section \ref{sec:pre} gives preliminaries including notation, dyadic grids, Haar functions, and some useful estimates. In Section \ref{sec:SIO}, we formulate the general framework of multilinear singular integrals on product spaces. After establishing some auxiliary results in Section \ref{sec:aux}, we prove the $m$-linear bi-parameter dyadic representation (cf. Theorem \ref{thm:repre}) in Section \ref{sec:bbd}. Then Section \ref{sec:wcpt} is devoted to showing weighted compactness of dyadic operators (cf. Theorem \ref{thm:dyadic-cpt}), while Section \ref{sec:wcc} aims to demonstrate the mean continuity of commutators (cf. Theorem \ref{thm:dyadic-bT}). Finally, Section \ref{sec:RdF} contains the proof of extrapolation of compactness and mean continuity (cf. Theorems \ref{thm:RdF-cpt} and \ref{thm:RdF-bT}), and then shows Theorems \ref{thm:cpt} and \ref{thm:bT}. For completeness, we present all details of the proof of Uchiyama's characterization of $\CMO(\Rn)$ in Section \ref{sec:Uch}.

\section{Preliminaries}\label{sec:pre}

\subsection{Notation}
Let us introduce some useful notation in this article. 
\begin{itemize}

\item For convenience, we use $\ell^{\infty}$ metrics on $\R^{n_1}$ and $\R^{n_2}$. 

\item Let $\Z :=\{0, \pm 1, \pm 2, \ldots\}$ be the set of all integers.  

\item Let $\N :=\{0, 1, 2, \ldots\}$ be the set of natural numbers. 

\item Let $\N_+ :=\{1, 2, \ldots\}$ be the set of positive integers.  

\item Let $\vec{f} = (f_1, \ldots, f_m)$ denote the vector of functions $f_1, \ldots, f_m$. 

\item For any $p \in (0, \infty]$ and a weight $w$ on $\Rn$, write $\|f\|_{L^p(w^p)} := \|fw\|_{L^p}$. 

\item Given $p \in (1, \infty)$, let $p' = \frac{p}{p-1}$ be the H\"{o}lder conjugate exponent of $p$.

\item For a measurable set $A \subset \Rn$ with $0<|A|<\infty$, write $\langle f \rangle_A = \fint_{A} f\,dx := \frac{1}{|A|} \int_A f\,dx$. 

\item A cube $I$ in $\Rn$ means $I := \prod_{i=1}^n [a_i, a_i+\ell)$, where $a_i \in \R$ and $\ell>0$.  

\item Given a cube $I \subset \Rn$, let $c_I$ and $\ell(I)$ be its center and sidelength respectively. For any $\lambda>0$, let $\lambda I$ be the cube with center $c_I$ and sidelength $\lambda \ell(I)$.  

\item Write $\I^i := [-\frac12, \frac12)^{n_i} \subset \R^{n_i}$ and $\lambda \I^i := [-\frac{\lambda}{2}, \frac{\lambda}{2})^{n_i}$ for any $\lambda>0$. 

\item Let $\Q^i$ denote the family of all cubes in $\R^{n_i}$. Then set $\mathcal{R} := \Q^1 \times \Q^2$.

\item Let $\D^i$ denote a generic dyadic grid on $\R^{n_i}$ (cf. Section \ref{sec:dyadic}). 

\item For any $I \in \D^i$, set $\ch(I) := \{I' \in \D^i: I' \subset I, \ell(I') = \ell(I)/2\}$.

\item Given $k \in \N$ and $I \in \D^i$, write $\D_k^i(I) := \{I' \in \D^i: I' \subset I, \ell(I') = 2^{-k} \ell(I)\}$, and let $I^{(k)}$ denote the unique dyadic cube $J \in \D^i$ such that $I \subset J$ and $\ell(J) = 2^k \ell(I)$. 

\item The distance between sets $E$ and $F$ is given by $\d(E, F) := \inf \{|x-y|: x \in E, y \in F\}$.  

\item The relative distance between cubes $I$ and $J$ is defined by $\rd(I, J) := \frac{\d(I, J)}{\max\{\ell(I), \ell(J)\}}$. 

\item For every $N \in \N$, set $\D^i(N) :=\{I \in \D^i: 2^{-N} \leq  \ell(I) \leq 2^N, \rd(I, 2^N \I^i) \leq N\}$.  

\item We shall use $A \lesssim B$ and $A \simeq B$ to mean, respectively, that $A \leq C B$ and $0<c \leq A/B\leq C$, where the constants $c$ and $C$ are harmless positive constants, not necessarily the same at each occurrence, which depend only on dimension and the constants appearing in the hypotheses of theorems.  
\end{itemize}

\subsection{Dyadic grids}\label{sec:dyadic}
For each $i=1, 2$, let $\D_0^{n_i}$ be the standard dyadic grid on $\R^{n_i}$:
\begin{align*}
\D_0^{n_i} := \big\{2^{-k}([0, 1)^{n_i} + m): k \in \Z, \, m \in \Z^{n_i} \big\}.
\end{align*}
Let $\Omega_i := (\{0, 1\}^{n_i})^{\Z}$ and let $\mathbb{P}_{\w_i}$ be the natural probability measure on $\Omega_i$: each component $\w_i^j$ has an equal probability $2^{-n_i}$ of taking any of the $2^{n_i}$ values in $\{0, 1\}^{n_i}$, and all components are independent of each other. Let $\mathbb{E}_{\w_i}$ denote the expectation over the random variables $\w_i^j$, $j \in \Z$. 

Given $\w_i = (\w_i^j)_{j \in \Z} \in \Omega_i$, the \emph{random dyadic grid} $\D_{\w_i}^{n_i}$ on $\R^{n_i}$ is defined by
\begin{align*}
\D_{\w_i}^{n_i} 
:= \bigg\{I + \w_i := I + \sum_{j: 2^{-j}< \ell(Q)} 2^{-j} \w_i^j : I \in \D_0^{n_i}\bigg\}.
\end{align*}
A \emph{dyadic grid} $\D^{n_i}$ on $\R^{n_i}$ means $\D^{n_i} = \D_{\w_i}^{n_i}$ for some $\w_i \in \Omega_i$.

A cube $I \in \D_{\w_i}^{n_i}$ is called \emph{bad} if there exists a cube $J \in \D_{\w_i}^{n_i}$ such that 
\begin{align*}
\ell(J) \ge 2^{\vartheta} \ell(I)
\quad\text{ and }\quad 
\d(I, \partial J) \le 2 \ell(I)^{\gamma_i} \ell(J)^{1-\gamma_i},  
\end{align*}
where $\vartheta \in \N_+$, $\gamma_i = \frac{\delta_i}{2(2n_i + \delta_i)}$, and $\delta_i \in (0, 1]$ appears in the kernel estimates (cf. Section \ref{sec:SIO}).  Otherwise a cube is called \emph{good}. Let $\D_{\w_i, \rm{good}}^{n_i}$ denote the family of good cubes in $\D_{\w_i}^{n_i}$. Note that $\pi_{\rm{good}}^{n_i} := \mathbb{P}_{\w_i}(I + \w_i \text{ is good}) = \mathbb{E}_{\w_i} \mathbf{1}_{\text{good}}(I + \w_i)$ is independent of the choice of $I \in \D_0^{n_i}$. The appearing parameter $\vartheta \in \N_+$ is a large enough fixed integer so that $\pi_{\rm{good}}^{n_1} > 0$ and $\pi_{\rm{good}}^{n_2} > 0$. Moreover, given $I \in \D_0^{n_i}$, the set $I + \w_i$ depends on $\w_i^j$ with $2^{-j} < \ell(I)$, while the goodness of $I + \w_i$ depends on $\w_i^j$ with $2^{-j} \ge \ell(I)$.

The following lemma (see \cite{LMOV}) gives the existence of the common dyadic ancestor in two different situations.

\begin{lemma}\label{lem:cda}
Let $I, J \in \D^{n_i}$ and $K \in \D^{n_i}_{\rm{good}}$ be such that $\ell(K) \le \ell(I) = 2\ell(J)$. 
\begin{enumerate}[label = {\rm (\alph*)}, itemsep = 2pt, topsep = 0pt, itemindent = 1pt]
\item\label{cda-1} If $\max\{\d(K, I), \d(K, J)\} > 2 \ell(K)^{\gamma_i} \ell(J)^{1 - \gamma_i}$, then there exists a cube $Q \in \D^{n_i}$ such that 
\begin{align*}
I \cup J \cup K \subset Q
\quad\text{ and }\quad 
\max\{\d(K, I), \, \d(K, J)\} \gtrsim \ell(K)^{\gamma_i} \ell(Q)^{1-\gamma_i}.
\end{align*} 

\item\label{cda-2} If $\max\{\d(K, I), \d(K, J)\} \le 2\ell(K)^{\gamma_i} \ell(J)^{1 - \gamma_i}$ and either $K \cap I = \emptyset$ or $K = I$ or $K \cap J = \emptyset$, then there exists a cube $Q \in \D^{n_i}$ such that 
\begin{align*}
I \cup J \cup K \subset Q
\quad\text{ and }\quad \ell(Q) \le 2^{\vartheta} \ell(K).
\end{align*} 
\end{enumerate}
\end{lemma}

\subsection{Haar functions}\label{sec:Haar}
Let $\D$ be a dyadic grid on $\Rn$. Given a cube $I = I_1 \times \cdots \times I_n \in \D$, define the Haar function $h_I^{\eta}$, $\eta=(\eta_1, \ldots, \eta_n) \in \{0, 1\}^n$, by 
\begin{align*}
h_I^{\eta} := h_{I_1}^{\eta_1} \otimes \cdots \otimes h_{I_n}^{\eta_n},
\end{align*}
where $h_{I_i}^0 = |I_i|^{-\frac12} \mathbf{1}_{I_i}$ and $h_{I_i}^1 = |I_i|^{-\frac12}(\mathbf{1}_{I_i^-}-\mathbf{1}_{I_i^+})$ for every $i = 1, \ldots, n$. Here $I_i^-$ and $I_i^+$ are the left and right halves of the interval $I_i$ respectively. If $\eta \neq 0$, the Haar function is cancellative : $\int_{\Rn} h_I^{\eta} \, dx= 0$; in this case we usually suppress the presence of $\eta$ and simply write $h_I = h_I^{\eta}$.

Given $k \in \N$ and $I \in \D$, let $\D_k(I) := \{I' \in \D: I' \subset I, \ell(I') = 2^{-k} \ell(I)\}$. Define 
\begin{align*}
E_I f := \langle f \rangle_I \, \mathbf{1}_I, \quad 
\Delta_I f := \sum_{I' \in \ch(I)} 
\big[\langle f \rangle_{I'} - \langle f \rangle_I \big] \mathbf{1}_{I'},
\quad\text{and }\quad 
\Delta_I^k f := \sum_{I \in \D_k(I)} \Delta_I f. 
\end{align*}
Then for all $p \in (1, \infty)$ and $f \in L^p(\Rn)$, 
\begin{align}\label{f-mar}
f = \sum_{I \in \D} \Delta_I f 
= \sum_{I \in \D} \langle f, h_I \rangle \, h_I, 
\end{align}
where the convergence takes place unconditionally in $L^p(\Rn)$.

The dyadic square function $S_{\D}$ is defined by 
\begin{align*}
S_{\D} f 
:= \bigg( \sum_{I \in \D} |\langle f, h_I \rangle|^2 
\frac{\mathbf{1}_I}{|I|} \bigg)^{\frac12}.
\end{align*}
It was shown in \cite[Theorem 2.1]{Wil} that 
\begin{align}\label{sdff}
\|S_{\D} f\|_{L^r} \simeq \|f\|_{L^r}, \quad\text{for all } r  \in (1, \infty). 
\end{align}

For each $k \in \Z$, setting 
\begin{align}\label{def:Ek}
E_{2^k} f := \sum_{Q \in \D : \, \ell(Q)=2^k} E_Q f 
\quad\text{ and }\quad 
D_{2^k} f:= E_{2^{k-1}} f - E_{2^k} f, 
\end{align}
one can verify 
\begin{align}\label{ddf-2}
E_{2^k} f = \sum_{Q \in \D: \, \ell(Q)>2^k} \Delta_Q f, \qquad 
D_{2^k} f = \sum_{Q \in \D: \, \ell(Q)=2^k} \Delta_Q f, 
\end{align}
and 
\begin{align}\label{ddf-3}
\bigg\|\bigg(\sum_{k \in \Z} |D_{2^k} f|^2 \bigg)^{\frac12} \bigg\|_{L^r} 
\simeq \|f\|_{L^r}, \quad\text{ for all } r \in (1, \infty). 
\end{align}

Given $N \in \N$, we define the \emph{projection operator} and its \emph{orthogonal operator} by 
\begin{align}\label{def:PN}
P_{\D(N)} f  := \sum_{I \in \D(N)} \langle f, h_I \rangle \, h_I 
\quad\text{ and }\quad 
P_{\D(N)}^{\perp} f := f - P_{\D(N)} f, 
\end{align}
with convergence interpreted pointwise almost everywhere. Then there holds 
\begin{align}\label{ddf-4}
\sup_{N \in \N} \|P_N^{\perp} f\|_{L^r} 
\lesssim \|f\|_{L^r},\quad\text{ for all } r \in (1, \infty). 
\end{align}

\subsection{Muckenhoupt weights}
A measurable function $w$ on $\Rnn$ is called a \emph{weight} if $0<w(x)<\infty$ a.e. $x \in \Rnn$. Given $p \in (1, \infty)$, we define the bi-parameter class $A_p(\Rnn)$ as the collection of all weights $w$ on $\Rnn$ satisfying 
\begin{equation*}
[w]_{A_p(\Rnn)} 
:= \sup_{R \in \mathcal{R}} \langle w \rangle_R \langle w^{1-p'} \rangle_R^{p-1}
< \infty. 
\end{equation*} 
In the endpoint case $p=1$, we say that $w \in A_1(\Rnn)$ if  
\begin{equation*}
[w]_{A_1(\Rnn)} 
:= \sup_{R \in \mathcal{R}} \langle w \rangle_R \big(\esssup_R w^{-1} \big) 
< \infty.
\end{equation*}
Then, we define $A_{\infty}(\Rnn) := \bigcup_{p \geq 1} A_p(\Rnn)$.

Let us record square function estimates (cf. \cite[p. 1705]{HPW}): for any $p \in (1, \infty)$ and $w \in A_p(\Rnn)$, 
\begin{align}
\label{ssf-1}
\|f\|_{L^p(w)} 
&\simeq \bigg\|\bigg[\sum_{\substack{I^1 \in \D^1 \\ I^2 \in \D^2}} 
|\Delta_{I^1} \Delta_{I^2} f|^2 \bigg]^{\frac12}\bigg\|_{L^p(w)}
= \bigg\|\bigg[\sum_{\substack{I^1 \in \D^1 \\ I^2 \in \D^2}} 
|\Delta_{I^1}^{k_1} \Delta_{I^2}^{k_2} f|^2 \bigg]^{\frac12}\bigg\|_{L^p(w)}
\\
\label{ssf-2}
&\simeq \bigg\|\bigg[ \sum_{I^1 \in \D^1} 
|\Delta_{I^1} f|^2 \bigg]^{\frac12}\bigg\|_{L^p(w)}
= \bigg\|\bigg[ \sum_{I^1 \in \D^1} 
|\Delta_{I^1}^{k_1} f|^2 \bigg]^{\frac12}\bigg\|_{L^p(w)}
\\
\label{ssf-3}
&\simeq \bigg\|\bigg[ \sum_{I^2 \in \D^2} 
|\Delta_{I^2} f|^2 \bigg]^{\frac12}\bigg\|_{L^p(w)}
= \bigg\|\bigg[ \sum_{I^2 \in \D^2} 
|\Delta_{I^2}^{k_2} f|^2 \bigg]^{\frac12}\bigg\|_{L^p(w)}, 
\end{align}
for all $k_1, k_2 \in \N$, where the implicit constants are independent of $k_1, k_2, \D^1, \D^2$.

Next, let us recall multiple weights on product spaces introduced in \cite{GLPT}. 

\begin{definition}\label{def:Ap}
Given $\vec{p} = (p_1, \ldots, p_m)$ with $1 \le p_1, \ldots, p_m \le \infty$ and $\vec{w} = (w_1, \ldots, w_m)$ with $0<w_1, \ldots, w_m<\infty$ a.e. on $\Rnn$, we say that $\vec{w} \in A_{\vec{p}}(\Rnn)$ if 
\begin{align*}
[\vec{w}]_{A_{\vec{p}}(\Rnn)} 
:= \sup_{R \in \mathcal{R}} \langle w^p \rangle^{\frac1p} 
\prod_{j=1}^m \big\langle w_j^{-p'_j} \big\rangle_R^{\frac{1}{p'_j}} < \infty,
\end{align*}
where $\frac1p = \sum_{j=1}^m \frac{1}{p_j}$ and $w= \prod_{j=1}^m w_j$. If $p_j = 1$, $\big\langle w_j^{-p'_j} \big\rangle_R^{\frac{1}{p'_j}}$ is understood as $(\essinf_R w_j)^{-1}$; and if $p = \infty$, $\langle w^p \rangle^{\frac1p}$ is understood as $\esssup_R w$. 
\end{definition}

Let us define the multilinear strong maximal operator
\begin{equation*}
\mathcal{M}_{\mathcal{R}}(f_1, \ldots, f_m)(x)
:= \sup_{x \in R \in \mathcal{R}} \prod_{j=1}^m \langle |f_j| \rangle_R, \quad x \in \Rnn. 
\end{equation*}
In the case $m=1$, simply write $M_{\mathcal{R}} = \mathcal{M}_{\mathcal{R}}$.

Analogously to \cite{LOPTT}, we have the following characterizations of the class $A_{\vec{p}}(\Rnn)$, see  \cite[Lemma 3.6]{LMV21} and \cite[Corollary 2.4]{GLPT}. Although the latter does not contain the case $p_i=\infty$ for some/all $i \in \{1, \ldots, m\}$, it can be shown by Proposition 4.1 and Remark 8.1 in \cite{LMV21}.

\begin{theorem}\label{thm:ww}
Let $\vec{p} = (p_1, \ldots, p_m)$ with $1 \leq p_1, \ldots, p_m \le \infty$. Then the following hold: 
\begin{enumerate}
\item[{\rm (1)}] $\vec{w} \in A_{\vec{p}}(\Rnn)$ if and only if $w^p \in A_{mp}(\Rnn)$   and $w_j^{-p'_j} \in A_{mp'_j}(\Rnn)$, $j=1, \ldots, m$, where $\frac1p = \sum_{j=1}^m \frac{1}{p_j}$ and $w= \prod_{j=1}^m w_j$. If $p_j = 1$, $w_j^{-p'_j} \in A_{mp'_j}$ is understood as $w_j^{1/m} \in A_1$; and if $p = \infty$, $w^p \in A_{mp}$ is understood as $w^{-1/m} \in A_1$.

\item[{\rm (2)}] Let $p_1, \ldots, p_m >1$. Then $\vec{w} \in A_{\vec{p}}(\Rnn)$ if and only if $\mathcal{M}_{\mathcal{R}}$ is bounded from $L^{p_1}(w_1^{p_1}) \times \cdots \times L^{p_m}(w_m^{p_m})$ to $L^p(w^p)$, where $\frac1p = \sum_{j=1}^m \frac{1}{p_j}$ and $w = \prod_{j=1}^m w_j$. 

\end{enumerate} 
\end{theorem}

\subsection{BMO and CMO spaces} 
\begin{definition}\label{def:BMO}
A locally integrable function $b: \Rn \to \mathbb{C}$ belongs to $\BMO(\Rn)$ if
\begin{align*}
\|b\|_{\BMO(\Rn)} := \sup_{Q \in \Q} \fint_{Q} |b(x) - \langle b \rangle_Q| \, dx < \infty, 
\end{align*}
where $\Q$ is the collection of all cubes on $\Rn$. A dyadic analogue $\BMO(\D)$ is defined on the dyadic grid $\D$ instead of $\Q$. Let $\CMO(\Rn)$ denote the closure of $\mathscr{C}_c^{\infty}(\Rn)$ in $\BMO(\Rn)$. Additionally, the space $\CMO(\Rn)$ is endowed with the norm of $\BMO(\Rn)$. 
\end{definition}

It follows from \cite{LTW, Uch} and the John--Neirenberg inequality that for any $p \in [1, \infty)$, 
\begin{align*}
b \in \CMO(\Rn) 
& \iff b \in \BMO(\Rn) 
\, \, \text{and} \, 
\lim_{N \to \infty} \sup_{Q \notin \Q(N)} 
\bigg[\fint_Q |b - \langle b \rangle_Q|^p \bigg]^{\frac1p} = 0
\\
& \iff b \in \BMO(\Rn) 
\, \, \text{and} \, 
\lim_{N \to \infty} \sup_{\D} \|P_{\D(N)}^{\perp} b\|_{\BMO(\D)} = 0,
\end{align*}
where the first equivalence is shown in Theorem \ref{thm:Uch} and the second equivalence is given in \cite[Theorem A.1]{CLSY}. Moreover, for any sequence $\{a_I\}_{I \in \D}$ and $\{b_I\}_{I \in \D}$, 
\begin{align}\label{H1BMO}
\sum_{I \in \D} |a_I| |b_I| 
\lesssim \sup_{Q \in \D} \bigg(\frac{1}{|Q|} \sum_{I \in \D: I \subset Q} |a_I|^2 \bigg)^{\frac12} 
\bigg\| \bigg( \sum_{I \in \D} |b_I|^2 \frac{\mathbf{1}_{I}}{|I|} \bigg)^{\frac12} \bigg\|_{L^1}.
\end{align}

Next, let us turn to BMO and CMO spaces in the product setting $\Rnn$. Let $\D = \D^1 \times \D^2$, where $\D^i$ is a dyadic grid on $\R^{n_i}$, $i=1, 2$. We say that a locally integrable function $b: \Rnn \to \C$ belongs to the \emph{dyadic little $\BMO$ space} $\bmo(\D)$ if
\begin{align*}
\|b\|_{\bmo(\D)} := \sup_{I = I^1 \times I^2 \in \D} \fint_I |b(x) - \langle b \rangle_I| \, dx < \infty. 
\end{align*}
The \emph{little $\BMO$ space} $\bmo(\Rnn)$ is defined as the set of all functions $b$ satisfying 
\begin{align*}
\|b\|_{\bmo(\Rnn)} 
:= \sup_{\D = \D^1 \times \D^2}  \|b\|_{\bmo(\D)} < \infty. 
\end{align*}
Then \emph{little $\CMO$ space} $\cmo(\Rnn)$ is defined as the closure of $\mathscr{C}_c^{\infty}(\Rnn)$ in $\bmo(\Rnn)$ with norm of $\bmo(\Rnn)$. 

We say that a locally integrable function $b: \Rnn \to \C$ belongs to the \emph{dyadic product $\BMO$ space} $\BMO(\D)$ if
\begin{align*}
\|b\|_{\BMO(\D)}
:= \sup_U \bigg(\frac{1}{|U|} 
\sum_{\substack{I^1 \times I^2 \in \D \\ I^1 \times I^2 \subset U}}
|\langle b, h_{I^1} \otimes h_{I^2} \rangle|^2\bigg)^{\frac12}
< \infty,
\end{align*}
where the supremum is taken over all open sets $U \subset \Rnn$ with $0<|U|<\infty$. Then the \emph{product $\BMO$ space} $\BMO(\Rnn)$ is defined as the collection of all locally integrable functions $b$ satisfying 
\begin{align*}
\|b\|_{\BMO(\Rnn)} 
:= \sup_{\D = \D^1 \times \D^2}  \|b\|_{\BMO(\D)} < \infty.
\end{align*}

We say that $b$ belongs to the \emph{dyadic product $\CMO$ space $\CMO(\D)$} if $b \in \BMO(\D)$ and satisfies 
\begin{align*}
\lim_{N \to \infty} \|P_{\D(N)}^{\perp} b\|_{\BMO(\D)} = 0,  
\end{align*}
where $\D(N) := \D^1(N) \times \D^2(N)$,  $P_{\D(N)}^{\perp} b := b - P_{\D(N)} b$, and 
\begin{align*}
P_{\D(N)} b
:= \sum_{J^1 \times J^2 \in \D(N)} 
\langle b, h_{J^1} \otimes h_{J^2} \rangle \, h_{J^1} \otimes h_{J^2}.
\end{align*}
The \emph{product $\CMO$ space $\CMO(\Rnn)$} is defined as the set of all functions $b \in \BMO(\Rnn)$ satisfying 
\begin{align*}
\lim_{N \to \infty} \sup_{\D} \|P_{\D(N)}^{\perp} b\|_{\BMO(\D)} = 0. 
\end{align*}

For later use, we present a bi-parameter Carleson embedding theorem, which is a particular result of \cite{Han}. 

\begin{lemma}\label{lem:Car}
Let $\D = \D^1 \times \D^2$, where $\D^i$ is a dyadic grid on $\R^{n_i}$, $i=1, 2$. Then the following are equivalent: 
\begin{enumerate} 
\item[{\rm (i)}] A family of non-negative numbers $\{\lambda_R\}_{R \in \D}$ is Carleson: 
\begin{align*}
\sum_{R \in \D: R \subset U} \lambda_R  
\le C_1 |U|, 
\quad\text{for all open sets } U \subset \Rnn. 
\end{align*}

\item[{\rm (ii)}] For every family $\{a_R \}_{R \in \D}$ of nonnegative numbers, there holds
\begin{align*}
\sum_{R \in \D} \lambda_R \, a_R 
\le C_2 \int_{\Rnn} \sup_{x \in R \in \D} a_R \, dx.
\end{align*}
\end{enumerate}
\end{lemma}

Let us end up this section with the relationship of compactness on two separated spaces and on product spaces. 

\begin{lemma}\label{lem:ncpt}
Let $p, p_1, \ldots, p_m \in (0, \infty)$. For each $i=1, 2$, let $T_i : L^{p_1}(\R^{n_i}) \times \cdots \times L^{p_m}(\R^{n_i}) \to L^p(\R^{n_i})$ boundedly with operator norm being non-zero. If $T_1 \otimes T_2$ is compact from $L^{p_1}(\Rnn) \times \cdots \times L^{p_m}(\Rnn)$ to $L^p(\Rnn)$, then $T_i$ is compact from $L^{p_1}(\R^{n_i}) \times \cdots \times L^{p_m}(\R^{n_i}) \to L^p(\R^{n_i})$ for each $i=1, 2$. 
\end{lemma}

\begin{proof}
Assume that $T_1$ is not compact from $L^{p_1}(\R^{n_1}) \times \cdots \times L^{p_m}(\R^{n_1}) \to L^p(\R^{n_1})$. Then there exists a bounded sequence $\{(f_k^1, \ldots, f_k^m)\}_{k=1}^{\infty} \subset L^{p_1}(\R^{n_1}) \times \cdots \times L^{p_m}(\R^{n_1})$ so that for every subsequence $\{(f_{k_j}^1, \ldots, f_{k_j}^m)\}_{j=1}^{\infty}$ of $\{(f_k^1, \ldots, f_k^m)\}_{k=1}^{\infty}$, $\{T_1(f_{k_j}^1, \ldots, f_{k_j}^m)\}_{j=1}^{\infty}$ is not Cauchy in $L^p(\R^{n_1})$. By the fact that $T_2$ is non-zero, one can find some $g^i \in L^{p_i}(\R^{n_2})$, $i=1, \ldots, m$, satisfying $\|T_2(g^1, \ldots, g^m)\|_{L^p(\R^{n_2})} =: A_0 \in (0, \infty)$. 

Pick $\phi_k^i := f_k^i \otimes g^i$ for each $k \in \N_+$ and $i=1, \ldots, m$. By definition, we have 
\begin{align*}
&\|T_1 \otimes T_2 (\phi_{k_i}^1, \ldots, \phi_{k_i}^m) 
- T_1 \otimes T_2 (\phi_{k_j}^1, \ldots, \phi_{k_j}^m)\|_{L^p(\Rnn)} 
\\
&= A_0 \|T_1 (\phi_{k_i}^1, \ldots, \phi_{k_i}^m) 
- T_1 (\phi_{k_j}^1, \ldots, \phi_{k_j}^m)\|_{L^p(\R^{n_1})}, 
\end{align*}
for all $i, j \in \N_+$, which implies that $\{T_1 \otimes T_2 (\phi_{k_j}^1, \ldots, \phi_{k_j}^m)\}_{j=1}^{\infty}$ is not Cauchy in $L^p(\Rnn)$. Hence, $T_1 \otimes T_2$ is not compact from $L^{p_1}(\Rnn) \times \cdots \times L^{p_m}(\Rnn)$ to $L^p(\Rnn)$. 
\end{proof}

\section{Multilinear singular integrals on product spaces}\label{sec:SIO}
Let us define multilinear singular integrals on product spaces. Let $T$ be an $m$-linear operator on $\Rnn$. Let $f_j = f_j^1 \otimes f_j^2$, $j=1, \ldots, m+1$. Denote $T_{1, 2}^{0*, 0*} := T$, and for each $j \in \{1, \ldots, m\}$, we define the full adjoint $T_{1, 2}^{j*, j*} := T^{j*}$ by
\begin{align*}
\langle T(f_1, \ldots, f_m), f_{m+1} \rangle 
= \langle T^{j*}(f_1, \ldots, f_{j-1}, f_{m+1}, f_{j+1}, \ldots, f_m), f_j \rangle, 
\end{align*}
and define the partial adjoints $T_{1, 2}^{j*, 0*} := T_1^{j*}$ and $T_{1, 2}^{0*, j*} := T_2^{j*}$ by 
\begin{align*}
\langle T(f_1, \ldots, f_m), f_{m+1} \rangle 
&= \langle T_1^{j*}(f_1, \ldots, f_{j-1}, f_{m+1}^1 \otimes f_j^2, f_{j+1}, \ldots, f_m), f_j^1 \otimes f_{m+1}^2 \rangle, 
\\
\langle T(f_1, \ldots, f_m), f_{m+1} \rangle 
&= \langle T_2^{j*}(f_1, \ldots, f_{j-1}, f_j^1 \otimes f_{m+1}^2, f_{j+1}, \ldots, f_m), f_{m+1}^1 \otimes f_j^2 \rangle. 
\end{align*}
Then for any $j_1, j_2 \in \{1, \ldots, m\}$, we define $T_{1, 2}^{j_1*, j_2*}$ by taking partial adjoints of $T_1^{j_1*}$ with respect to the second parameter, i.e., $T_{1, 2}^{j_1*, j_2*} = (T_1^{j_1*})_2^{j_2*}$.

\begin{definition}\label{def:FF}
Let $\F$ consist of all triples $(F_1, F_2, F_3)$ of bounded functions $F_1, F_2, F_3: [0, \infty) \to [0, \infty)$ satisfying
\begin{align*}
\lim_{t \to 0} F_1(t)
=\lim_{t \to \infty} F_2(t)
= \lim_{t \to \infty} F_3(t)
=0.
\end{align*}
For $i=1, 2$, let $\mathscr{F}^i$ be the family of all bounded functions $F^i: \mathcal{Q}^i \to [0, \infty)$ satisfying
\begin{align*}
\lim_{\ell(I^i) \to 0} F^i(I^i)
= \lim_{\ell(I^i) \to \infty} F^i(I^i)
= \lim_{|c_{I^i}| \to \infty} F^i(I^i)
=0, 
\end{align*}
where $\mathcal{Q}^i$ is the collection of all cubes in $\R^{n_i}$, $i=1, 2$.
\end{definition}

Let $\mathbb{F}^i$ consist of all finite linear combinations of indicators of cubes in $\R^{n_i}$, $i=1, 2$. Set $\mathbb{F} = \mathbb{F}^1 \times \mathbb{F}^2$. Let $\mathbb{L}$ denote the set of locally integrable functions on $\Rnn$. Let $\delta_1, \delta_2 \in (0, 1]$. For convenience, we use $\ell^{\infty}$ metrics on $\R^{n_1}$ and $\R^{n_2}$ throughout this paper. 

First, we define the kernel estimates. 

\begin{definition}\label{def:full}
An $m$-linear operator $T: \mathbb{F} \times \cdots \times \mathbb{F} \to \mathbb{L}$ admits a \emph{compact full kernel representation} if the following hold.
If $f_j = f_j^1 \otimes f_j^2 \in \mathbb{F}$, $j=1, \ldots, m+1$, and $\supp(f_1^i) \cap \cdots \cap \supp(f_{m+1}^i) = \emptyset$, $i=1, 2$, then
\begin{align*}
\langle T(f_1, \ldots, f_m), f_{m+1} \rangle
= \int_{(\Rnn)^{(m+1)}} K(x_{m+1}, x_1, \ldots, x_m) \prod_{j=1}^{m+1} f_j(x_j) \, dx,
\end{align*}
where 
\begin{align*}
K: \R^{(n_1+n_2)(m+1)} \setminus 
\big\{x_1^1 = \cdots = x_{m+1}^1 \in \R^{n_1}
\text{ or } x_1^2 = \cdots = x_{m+1}^2 \in \R^{n_2} \big\} \rightarrow \C
\end{align*}
is the kernel satisfying  
\begin{list}{\rm (\theenumi)}{\usecounter{enumi}\leftmargin=1.2cm \labelwidth=1cm \itemsep=0.2cm \topsep=.2cm \renewcommand{\theenumi}{\arabic{enumi}}}

\item\label{full-1} the size condition
\begin{align*}
|K(x_{m+1}, x_1, \ldots, x_m)| 
\leq \prod_{i=1}^2 \frac{F^i(x_{m+1}^i, x_1^i, \ldots, x_m^i)}{\big(\sum_{j=1}^m |x_{m+1}^i - x_j^i|\big)^{m n_i}}.
\end{align*}

\item\label{full-2} the H\"{o}lder condition
\begin{align*}
&|K(x_{m+1}, x_1, \ldots, x_m) 
- K((x_{m+1}^1, \widetilde{x}_{m+1}^2), x_1, \ldots, x_m) 
\\
&\quad- K((\widetilde{x}_{m+1}^1, x_{m+1}^2), x_1, \ldots, x_m) 
+ K((\widetilde{x}_{m+1}^1, \widetilde{x}_{m+1}^2), x_1, \ldots, x_m)| 
\\
&\leq \prod_{i=1}^2 \bigg(\frac{|x_{m+1}^i - \widetilde{x}_{m+1}^i|}{\sum_{j=1}^m |x_{m+1}^i - x_j^i|} \bigg)^{\delta_i} 
\frac{F^i(x_{m+1}^i, x_1^i, \ldots, x_m^i)}{\big(\sum_{j=1}^m |x_{m+1}^i - x_j^i|\big)^{m n_i}}
\end{align*}
whenever $|x_{m+1}^i - \widetilde{x}_{m+1}^i| \leq \frac12 \max\{|x_{m+1}^i - x_j^i|: 1 \le j \le m\}$ for all $i=1, 2$. 

\item\label{full-3} the mixed size-H\"{o}lder condition
\begin{align*}
&|K(x_{m+1}, x_1, \ldots, x_m) - K((\widetilde{x}_{m+1}^1, x_{m+1}^2), x_1, \ldots, x_m)|
\\
&\leq \bigg(\frac{|x_{m+1}^1 - \widetilde{x}_{m+1}^1|}{\sum_{j=1}^m |x_{m+1}^1 - x_j^1|} \bigg)^{\delta_1} 
\prod_{i=1}^2 \frac{F^i(x_{m+1}^i, x_1^i, \ldots, x_m^i)}{\big(\sum_{j=1}^m |x_{m+1}^i - x_j^i|\big)^{m n_i}}
\end{align*}
whenever $|x_{m+1}^1 - \widetilde{x}_{m+1}^1| \leq \frac12 \max\{|x_{m+1}^1 - x_j^1|: 1 \le j \le m\}$. 

\item\label{full-4} the function $F^i$ in \eqref{full-1}--\eqref{full-3} is given by
\begin{align*}
&F^i(x_{m+1}^i, x_1^i, \ldots, x_m^i) 
\\
&:= F^i_1 \bigg(\sum_{j=1}^m |x_{m+1}^i - x_j^i| \bigg)
F^i_2 \bigg(\sum_{j=1}^m |x_{m+1}^i - x_j^i| \bigg) 
F^i_3 \bigg(\sum_{j=1}^m |x_{m+1}^i + x_j^i| \bigg),
\end{align*}
where $(F^i_1, F^i_2, F^i_3) \in \mathscr{F}$, $i=1, 2$.

\item\label{full-5} The kernels $K_{1, 2}^{j_1*, j_2*}$ of $T_{1, 2}^{j_1*, j_2*}$, $j_1, j_2 \in \{0, 1, 2\}$, satisfy properties \eqref{full-1}--\eqref{full-4} as well.  
\end{list}
\end{definition}

\begin{definition}\label{def:partial}
An $m$-linear operator $T: \mathbb{F} \times \cdots \times \mathbb{F} \to \mathbb{L}$ admits a \emph{compact partial kernel representation on the first parameter} if the following hold. If $f_j = f_j^1 \otimes f_j^2$, $j=1, \ldots, m+1$, and $\supp(f_1^1) \cap \cdots \cap \supp(f_{m+1}^1) = \emptyset$, then
\begin{align*}
\langle T(f_1, \ldots, f_m), f_{m+1} \rangle
= \int_{\R^{n_1 (m+1)}} K_{(f_j^2)}(x_{m+1}^1, x_1^1, \ldots, x_m^1) \prod_{j=1}^{m+1} f_j^1(x_j^1) \, dx^1,
\end{align*}
where  
\begin{align*}
K_{(f_j^2)}: \R^{n_1 (m+1)} \setminus 
\big\{x_1^1 = \cdots = x_{m+1}^1 \in \R^{n_1} \big\} \rightarrow \C
\end{align*}
is the the kernel satisfying 
\begin{list}{\rm (\theenumi)}{\usecounter{enumi}\leftmargin=1.2cm \labelwidth=1cm \itemsep=0.2cm \topsep=.2cm \renewcommand{\theenumi}{\arabic{enumi}}}

\item\label{partial-1} the size condition
\begin{align*}
|K_{(f_j^2)}(x_{m+1}^1, x_1^1, \ldots, x_m^1)|
\leq C(f_1^2, \ldots, f_{m+1}^2) 
\frac{F^1(x_{m+1}^1, x_1^1, \ldots, x_m^1)}{\big(\sum_{j=1}^m |x_{m+1}^1 - x_j^1|\big)^{m n_1}}.
\end{align*}

\item\label{partial-2} the H\"{o}lder condition
\begin{align*}
&|K_{(f_j^2)}(x_{m+1}^1, x_1^1, \ldots, x_m^1) 
- |K_{(f_j^2)}(\widetilde{x}_{m+1}^1, x_1^1, \ldots, x_m^1)|
\\
&\leq C(f_1^2, \ldots, f_{m+1}^2) 
\bigg(\frac{|x_{m+1}^1 - \widetilde{x}_{m+1}^1|}{\sum_{j=1}^m |x_{m+1}^1 - x_j^1|} \bigg)^{\delta_1} 
\frac{F^1(x_{m+1}^1, x_1^1, \ldots, x_m^1)}{\big(\sum_{j=1}^m |x_{m+1}^1 - x_j^1|\big)^{m n_1}} 
\end{align*}
whenever $|x_{m+1}^1 - \widetilde{x}_{m+1}^1| \leq \frac12 \max\{|x_{m+1}^1 - x_j^1|: 1 \le j \le m\}$, and for each $j \in \{1, \ldots, m\}$, 
\begin{align*}
&|K_{(f_j^2)}(x_{m+1}^1, x_1^1, \ldots, x_j^1, \ldots, x_m^1) 
- |K_{(f_j^2)}(x_{m+1}^1, x_1^1, \ldots, \widetilde{x}_j^1, \ldots, x_m^1)|
\\
&\leq C(f_1^2, \ldots, f_{m+1}^2) 
\bigg(\frac{|x_j^1 - \widetilde{x}_j^1|}{\sum_{j=1}^m |x_{m+1}^1 - x_j^1|} \bigg)^{\delta_1} 
\frac{F^1(x_{m+1}^1, x_1^1, \ldots, x_m^1)}{\big(\sum_{j=1}^m |x_{m+1}^1 - x_j^1|\big)^{m n_1}} 
\end{align*}
whenever $|x_j^1 - \widetilde{x}_j^1| \leq \frac12 \max\{|x_{m+1}^1 - x_j^1|: 1 \le j \le m\}$. 

\item\label{partial-3} the function $F^1$ in \eqref{partial-1}--\eqref{partial-2} is given by
\begin{align*}
&F^1(x_{m+1}^1, x_1^1, \ldots, x_m^1) 
\\
&:= F^1_1 \bigg(\sum_{j=1}^m |x_{m+1}^1 - x_j^1| \bigg)
F^1_2 \bigg(\sum_{j=1}^m |x_{m+1}^1 - x_j^1| \bigg) 
F^1_3 \bigg(\sum_{j=1}^m |x_{m+1}^1 + x_j^1| \bigg),
\end{align*}
where $(F^1_1, F^1_2, F^1_3) \in \mathscr{F}$.

\item\label{partial-4} the minimal bound $C(f_1^2, \ldots, f_{m+1}^2)$ above verifies
\begin{align*}
C(\mathbf{1}_{I^2}, \mathbf{1}_{I^2}, \ldots, \mathbf{1}_{I^2})
+ C(a_{I^2}, \mathbf{1}_{I^2}, \ldots, \mathbf{1}_{I^2}) 
&\le F^2(I^2) \, |I^2|, 
\\
C(\mathbf{1}_{I^2}, a_{I^2}, \ldots, \mathbf{1}_{I^2}) + \cdots 
+ C(\mathbf{1}_{I^2}, \ldots, \mathbf{1}_{I^2}, a_{I^2}) 
&\le F^2(I^2) \, |I^2|,
\end{align*}
for all cubes $I^2 \subset \R^{n_2}$ and all functions $a_{I^2}$ satisfying $\supp(a_{I^2}) \subset I^2$, $|a_{I^2}| \le 1$, and $\int a_{I^2} =0$, where $F^2 \in \mathscr{F}^2$.
\end{list}

Analogously, one can define a \emph{compact partial kernel representation on the second parameter} when $\supp(f_1^2) \cap \cdots \cap \supp(f_{m+1}^2) = \emptyset$.
\end{definition}

\begin{definition}\label{def:SIO}
Let $T$ be an $m$-linear operator. 
\begin{itemize}
\item We say that $T$ admits the \emph{compact partial kernel representation} if it admits a compact partial kernel representation on both the first and second parameters.

\item We say that $T$ admits the \emph{full kernel representation} if both $F^1$ and $F^2$ in the compact full kernel representation is replaced by a uniform constant $C \ge 1$.

\item We say that $T$ admits the \emph{partial kernel representation} if both $F^1$ and $F^2$ in compact partial kernel representations are replaced by a uniform constant $C \ge 1$.

\item If $T$ admits the full and partial kernel representations, we call $T$ an \emph{$m$-linear bi-parameter singular integral operator}.
\end{itemize} 
\end{definition}

Next, let us give compactness assumptions.

\begin{definition}\label{def:WCP}
We say that $T$ satisfies the \emph{weak compactness property} if
\begin{align*}
|\langle T(\mathbf{1}_{I^1} \otimes \mathbf{1}_{I^2}, \ldots, 
\mathbf{1}_{I^1} \otimes \mathbf{1}_{I^2}), \mathbf{1}_{I^1} \otimes \mathbf{1}_{I^2} \rangle|
\leq F^1(I^1) |I^1| \, F^2(I^2) |I^2|,
\end{align*}
for all cubes $I^1 \subset \R^{n_1}$ and $I^2 \subset \R^{n_2}$, where $F^1 \in \mathscr{F}^1$ and $F^2 \in \mathscr{F}^2$. We say that $T$ satisfies the \emph{weak boundedness property} if both $F^1$ and $F^2$ above are replaced by a uniform constant $C \ge 1$.
\end{definition}

Moreover, let us introduce cancellation assumptions.

\begin{definition}\label{def:diag-CMO}
We say that $T$ satisfies the \emph{diagonal $\CMO$ condition} if for each $i=1, 2$, for any cube $I^i \subset \R^{n_i}$, and for any function $a_{I^i}$ such that $\supp(a_{I^i}) \subset I^i$, $|a_{I^i}| \le 1$, and $\int_{\R^{n_i}} a_{I^i} \, dx_i =0$, there holds
\begin{align*}
& |\langle T (a_{I^1} \otimes \mathbf{1}_{I^2}, \mathbf{1}_{I^1} \otimes \mathbf{1}_{I^2}, 
\ldots, \mathbf{1}_{I^1} \otimes \mathbf{1}_{I^2} ), \mathbf{1}_{I^1} \otimes \mathbf{1}_{I^2} \rangle| 
\leq F^1(I^1) |I^1| \, F^2(I^2) |I^2|, 
\\
& \qquad \cdots \cdots 
\\
& |\langle T (\mathbf{1}_{I^1} \otimes \mathbf{1}_{I^2}, \mathbf{1}_{I^1} \otimes \mathbf{1}_{I^2}, 
\ldots, \mathbf{1}_{I^1} \otimes \mathbf{1}_{I^2} ), a_{I^1} \otimes \mathbf{1}_{I^2} \rangle| 
\leq F^1(I^1) |I^1| \, F^2(I^2) |I^2|, 
\\ 
& |\langle T (\mathbf{1}_{I^1} \otimes a_{I^2}, \mathbf{1}_{I^1} \otimes \mathbf{1}_{I^2}, 
\ldots, \mathbf{1}_{I^1} \otimes \mathbf{1}_{I^2} ), \mathbf{1}_{I^1} \otimes \mathbf{1}_{I^2} \rangle| 
\leq F^1(I^1) |I^1| \, F^2(I^2) |I^2|, 
\\
& \qquad \cdots \cdots 
\\
& |\langle T (\mathbf{1}_{I^1} \otimes \mathbf{1}_{I^2}, \mathbf{1}_{I^1} \otimes \mathbf{1}_{I^2}, 
\ldots, \mathbf{1}_{I^1} \otimes \mathbf{1}_{I^2} ), \mathbf{1}_{I^1} \otimes a_{I^2} \rangle| 
\leq F^1(I^1) |I^1| \, F^2(I^2) |I^2|,
\end{align*}
where $F^1 \in \mathscr{F}^1$ and $F^2 \in \mathscr{F}^2$. We say that $T$ satisfies the \emph{diagonal $\BMO$ condition} if both $F^1$ and $F^2$ above are replaced by a uniform constant $C \ge 1$.
\end{definition}

\begin{definition}\label{def:prod-BMO}
An $m$-linear operator $T$ satisfies the \emph{product $\BMO$ condition} if
\begin{align*}
T_{1, 2}^{j_1*, j_2*}(1, \ldots, 1) \in \BMO(\Rnn) \quad\text{ for all } \, j_1, j_2 \in \{0, 1, \ldots, m\}.
\end{align*}
We say that $T$ satisfies the \emph{product $\CMO$ condition} if
\begin{align*}
T_{1, 2}^{j_1*, j_2*}(1, \ldots, 1) \in \CMO(\Rnn) \quad\text{ for all } \, j_1, j_2 \in \{0, 1, \ldots, m\}.
\end{align*}
\end{definition}

\begin{definition}\label{def:CZO}
An operator $T$ is called an \emph{$m$-linear bi-parameter Calder\'{o}n--Zygmund operator} if it admits full and partial kernel representations, and satisfies the weak boundedness property and the diagonal and product $\BMO$ conditions. 
\end{definition}

Beyond that, we would like to define $m$-linear bi-parameter dyadic operators, whose simple dyadic structure plays a significant role in our analysis. 

\begin{definition}\label{def:shift}
Given a dyadic grid $\D=\D^1 \times \D^2$ and $k=(k_1, \ldots, k_{m+1})$ with $k_j=(k_j^1, k_j^2) \in \N^2$, a \emph{compact $m$-linear bi-parameter shift $\mathbf{S}_{\D}^k$ of complexity $k$} is defined by  
\begin{align*}
\mathbf{S}_{\D}^k (\vec{f})
:= \sum_{Q = Q^1 \times Q^2 \in \D} \sum_{\substack{I_j \in \D_{k_j}(Q) \\ j=1, \ldots, m+1}}   
a_{(I_j), Q} \prod_{j=1}^m \langle f_j, \widetilde{h}_{I_j^1} \otimes \widetilde{h}_{I_j^2} \rangle \, 
\widetilde{h}_{I_{m+1}^1} \otimes \widetilde{h}_{I_{m+1}^2}, 
\end{align*}
where $I_j = I_j^1 \times I_j^2$, 
$\D_{k_j}(Q) := \D_{k_j^1}^1(Q^1) \times \D_{k_j^2}^2(Q^2)$, 
$\D_{k_j^i}^i(Q^i) := \{I^i \in \D^i: I^i \subset Q^i, \ell(I^i) = 2^{-k_j^i} \ell(Q^i)\}$, and 
$\widetilde{h}_{I_j} := \widetilde{h}_{I_j^1} \otimes \widetilde{h}_{I_j^2}$. 
Here we assume that for each $i=1, 2$, there exist two different indices $j_0^i, j_1^i \in \{1, \ldots, m+1\}$ so that $\widetilde{h}_{I_{j_0^i}^i} = h_{I_{j_0^i}^i}$, $\widetilde{h}_{I_{j_1^i}^i} = h_{I_{j_1^i}^i}$, and  $\widetilde{h}_{I_j^i} \in \{h_{I_j^i}^0, h_{I_j^i}\}$ for every $j \neq j_0^i, j_1^i$. Moreover, the coefficients $a_{(I_j), Q}$ satisfy  
\begin{align*}
|a_{(I_j), Q}| 
\le \mathcal{F}(Q)
\frac{\prod_{j=1}^{m+1} |I_j|^{\frac12}}{|Q|^m} 
\end{align*}
with  
\begin{align*}
\mathcal{F}(Q) \le 1 
\quad\text{and}\quad 
\lim_{N \to \infty} \mathcal{F}_N
:= \lim_{N \to \infty} \sup_{\D} \sup_{Q \not\in \D(N)} 
\mathcal{F}(Q) = 0,  
\end{align*}
where $\D(N) := \D^1(N) \times \D^2(N)$. Similarly, we define an \emph{$m$-linear bi-parameter shift $\mathbb{S}_{\D}^k$ of complexity $k$} if $\mathcal{F}(Q) \equiv 1$.
\end{definition}

\begin{definition}\label{def:pp}
Given a dyadic grid $\D=\D^1 \times \D^2$ and $k=(k_1, \ldots, k_{m+1}) \in \N^{m+1}$, a \emph{compact $m$-linear bi-parameter partial paraproduct $\mathbf{P}_{\D}^{1, k}$ of complexity $k$} is an operator of the form   
\begin{align*}
\mathbf{P}_{\D}^{1, k} (\vec{f})
:= \sum_{Q = Q^1 \times Q^2 \in \D} 
\sum_{\substack{I_j^1 \in \D_{k_j}^1(Q^1) \\ j=1, \ldots, m+1}}
a_{(I_j^1), Q} 
\prod_{j=1}^m \langle f_j, \widetilde{h}_{I_j^1} \otimes \overline{h}_{j, Q^2} \rangle \, 
\widetilde{h}_{I_{m+1}^1} \otimes \overline{h}_{m+1, Q^2}, 
\end{align*}
where the functions $\widetilde{h}_{I_j^1}$ and $\overline{h}_{j, Q^2}$ satisfy the following: there exist two different indices $j_0, j_1 \in \{1, \ldots, m+1\}$ so that $\widetilde{h}_{I_{j_0}^1} = h_{I_{j_0}^1}$, $\widetilde{h}_{I_{j_1}^1} = h_{I_{j_1}^1}$, and  $\widetilde{h}_{I_j^1} \in \{h_{I_j^1}^0, h_{I_j^1}\}$ for every $j \neq j_0, j_1$; there exists $j_2 \in \{1, \ldots, m+1\}$ so that $\overline{h}_{j_2, Q^2} = h_{Q^2}$ and $\overline{h}_{j, Q^2} = \frac{\mathbf{1}_{Q^2}}{|Q^2|}$ for every $j \ne j_2$. Moreover, the coefficients $a_{(I_j^1), Q}$ satisfy  
\begin{align*}
\sup_{\D^2} \sup_{Q_0^2 \in \D^2} \Bigg(\frac{1}{|Q_0^2|} 
\sum_{\substack{Q^2 \in \D^2 \\ Q^2 \subset Q_0^2}} 
|a_{(I_j^1), Q}|^2 \Bigg)^{\frac12}
\le \mathcal{F}^1(Q^1) \frac{\prod_{j=1}^{m+1} |I_j^1|^{\frac12}}{|Q^1|^m},  
\end{align*}
and 
\begin{align*}
\sup_{\D^2} \sup_{Q_0^2 \in \D^2} \Bigg(\frac{1}{|Q_0^2|} 
\sum_{\substack{Q^2 \notin \D^2(N) \\ Q^2 \subset Q_0^2}} 
|a_{(I_j^1), Q}|^2 \Bigg)^{\frac12}
\le \mathcal{F}^1_N \frac{\prod_{j=1}^{m+1} |I_j^1|^{\frac12}}{|Q^1|^m},  
\end{align*}
with 
\begin{align*}
\mathcal{F}^1(Q^1) \le1, \quad 
\mathcal{F}^1_N \le 1, \quad \text{and}\quad 
\lim_{N \to \infty} \Big(\sup_{\D^1} \sup_{Q^1 \notin \D^1(N)} \mathcal{F}^1(Q^1) 
+ \mathcal{F}_N^1 \Big) =0. 
\end{align*}
In the same way, we define an \emph{$m$-linear bi-parameter partial paraproduct $\mathbb{P}_{\D}^{1, k}$ of complexity $k$} if the coefficients $a_{(I_j^1), Q}$ just satisfy the first estimate with $\mathcal{F}^1(Q^1) \equiv 1$. Symmtrically, $\mathbf{P}_{\D}^{2, k}$ and $\mathbb{P}_{\D}^{2, k}$ are defined by interchanging parameters 1 and 2. 
\end{definition}

\begin{definition}\label{def:paraproduct}
Given a dyadic grid $\D=\D^1 \times \D^2$, a \emph{compact $m$-linear bi-parameter full paraproduct} takes the form   
\begin{align*}
\mathbf{F}_{\mathbf{a}}(\vec{f}) 
&:= \sum_{I = I^1 \times I^2 \in \D} a_I \, 
\prod_{j=1}^m \langle f_j, \overline{h}_{j, I^1} \otimes \overline{h}_{j, I^2} \rangle 
\, \overline{h}_{m+1, I^1} \otimes \overline{h}_{m+1, I^2}, 
\end{align*}
where there exist $j_0^1, j_0^2 \in \{1, \ldots, m+1\}$ so that $\overline{h}_{j_0^1, I^1} = h_{I^1}$, $\overline{h}_{j_0^2, I^2} = h_{I^2}$, $\overline{h}_{j, I^1} = \frac{\mathbf{1}_{I^1}}{|I^1|}$ for every $j \ne j_0^1$, and $\overline{h}_{j, I^2} = \frac{\mathbf{1}_{I^2}}{|I^2|}$ for every $j \ne j_0^2$. Moreover, the coefficients $a_I$ satisfy 
\begin{align*}
\sup_{\D} \sup_U \frac{1}{|U|} 
\sum_{I \in \D: \, I \subset U} |a_I|^2 
\le 1,  
\end{align*}
and 
\begin{align*}
\lim_{N \to \infty} \sup_{\D} \sup_U \frac{1}{|U|} 
\sum_{I \notin \D(N): \, I \subset U} |a_I|^2 
= 0, 
\end{align*}
where the supremum $\sup_U$ is taken over all open sets $U \subset \Rnn$ with $0<|U|<\infty$. Likewise,  we define an \emph{$m$-linear bi-parameter full paraproduct $\mathbb{F}_{\mathbf{a}}$} if the coefficients $a_I$ just satisfy the first inequality. 
\end{definition}

Finally, having prepared conceptions above, we introduce the compact $m$-linear bi-parameter dyadic representation of singular integrals. 

\begin{definition}\label{def:repre}
Given an $m$-linear operator $T$, we say that $T$ admits a \emph{compact $m$-linear bi-parameter dyadic representation} if there exists a constant $C_0 = C_0(T) \in (0,  \infty)$ so that for all compactly supported and bounded functions $f_1, \ldots, f_m, f_{m+1}$ on $\Rnn$, 
\begin{align*}
\big\langle T(\vec{f}), f_{m+1} \big\rangle 
&= C_0 \, \mathbb{E}_{\w} 
\sum_{k_1=0}^{\infty}  \sum_{k_2=0}^{\infty} 
2^{-k_1 \frac{\delta_1}{2}} 2^{-k_2 \frac{\delta_2}{2}} 
\big\langle \mathbf{S}_{\D_{\w}}^{k_1, k_2} (\vec{f}), f_{m+1} \big\rangle   
\end{align*}
with 
\begin{align*}
\mathbf{S}_{\D_{\w}}^{k_1, k_2} 
= \sum_{i_1 = 0}^{k_1} \sum_{i_2 = 0}^{k_2} 
\mathbf{S}_{\D_{\w}}^{k_1, k_2, i_1, i_2},  
\end{align*} 
where $\mathbf{S}_{\D_{\w}}^{k_1, k_2, i_1, i_2}$ is a finite sum of compact $m$-linear bi-parameter shifts on $\D_{\w}$ with complexity $(u_1, \ldots, u_{m+1})$, compact $m$-linear bi-parameter partial paraproducts on $\D_{\w}$ with complexity $(v_1^1, \ldots, v_{m+1}^1)$ and $(v_1^2, \ldots, v_{m+1}^2)$, and compact $m$-linear bi-parameter full paraproducts on $\D_{\w}$. In addition, 
\begin{align*}
u_j = (u_j^1, u_j^2), \quad 
u_j^1, v_j^1 \le k_1+1, 
\quad \text{ and } \quad 
u_j^2, v_j^2 \le k_2+1, \quad \forall j=1, \ldots, m+1. 
\end{align*} 

Analogously, we say that $T$ admits an \emph{$m$-linear bi-parameter dyadic representation} if compact $m$-linear bi-parameter shifts, compact $m$-linear bi-parameter partial paraproducts, and compact $m$-linear bi-parameter full paraproducts are replaced by $m$-linear bi-parameter shifts, $m$-linear bi-parameter partial paraproducts, and $m$-linear bi-parameter full paraproducts, respectively. 
\end{definition}

\section{Auxiliary results}\label{sec:aux}
In this section, we would like to give some useful estimates in order to show Theorem \ref{thm:repre}. For each $i=1, 2$, let $(F_1^i, F_2^i, F_3^i) \in \F$ so that $F_1^i$ is increasing, $F_2^i$ and $F_3^i$ are decreasing. Define 
\begin{align}\label{def:FKQ}
F^i(K^i, Q^i) 
:= F_1^i(\ell(K^i)) \widetilde{F}_2^i(\ell(K^i)) \widetilde{F}_3^i(Q^i), 
\end{align}
for all cubes $K^i, Q^i \subset \R^{n_i}$, where 
\begin{align}\label{def:F23}
\widetilde{F}_2^i(t)
:= \sum_{k=0}^{\infty} 2^{-k \theta} F_2^i(2^{-k} t)
\quad \text{and} \quad 
\widetilde{F}_3^i(Q^i)
:= \sum_{k=0}^{\infty} 2^{-k \theta} F_3^i(\rd(2^k Q^i, \I^i)).
\end{align}
The auxiliary parameter $\theta \in (0, 1)$ is harmless and small enough. Then it is not hard to verify 
\begin{align}\label{F3Q}
\widetilde{F}_3^i(Q^i)
\le \widetilde{F}_3^i(\lambda Q^i)
\le C_{\lambda, \theta} \, \widetilde{F}_3^i(Q^i), 
\quad \forall \, \lambda \ge 1. 
\end{align}

Since any dilation of functions in $\F$, $\F^1$, and $\F^2$ still belongs to the original space, we will often omit all universal constants appearing in the argument involving these functions.

\subsection{Some integral estimates}
The following three lemmas will be used frequently to establish the multilinear bi-parameter dyadic representation in Section \ref{sec:bbd}. 

\begin{lemma}\label{lem:PP}
For each $i=1, 2$, let $(F_1^i, F_2^i, F_3^i) \in \mathscr{F}$ satisfy that $F_1^i$ is increasing, $F_2^i$ and $F_3^i$ are decreasing. Let $I^i, J^i, K^i, Q^i \subset \R^{n_i}$ be cubes such that $\ell(K^i) \le 2\ell(J^i)$ and $I^i \cup J^i \cup K^i \subset Q^i$. If $\max\{\d(K^i, I^i), \, \d(K^i, J^i)\} > 2\ell(K^i)^{\gamma_i} \ell(J^i)^{1-\gamma_i}$ and $\max\{\d(K^i, I^i), \, \d(K^i, J^i)\} \gtrsim 2\ell(K^i)^{\gamma_i} \ell(Q^i)^{1-\gamma_i}$, then 
\begin{align*}
\mathscr{P}_i(I^i, J^i, K^i)
&:= \int_{J^i} \int_{I^i} \int_{K^i} 
\frac{F^i(x_i, y_i, z_i) \, \ell(K^i)^{\delta_i}}{(|x_i-y_i| + |x_i-z_i|)^{2n_i+\delta_i}}  
\, dx_i \, dy_i \, dz_i 
\\ 
&\lesssim \bigg[ \frac{\ell(K^i)}{\ell(Q^i)} \bigg]^{\frac{\delta_i}{2}} 
F^i(K^i, Q^i) \frac{|I^i| |J^i| |K^i|}{|Q^i|^2},  
\end{align*}
where 
\begin{align*}
F^i(x_i, y_i, z_i)
:= F_1^i(|x_i-c_{K^i}|) F_2^i(|x_i - y_i| + |x_i - z_i|) 
F_3^i \bigg(1 + \frac{|x_i + y_i| + |x_i + z_i|}{1+ |x_i - y_i| + |x_i - z_i|}\bigg),  
\end{align*}
and $F^i(K^i, Q^i)$ is given in \eqref{def:FKQ}.
\end{lemma}

\begin{proof}
For all $x_i \in K^i$, $y_i \in I^i$, and $z_i \in J^i$, we have 
\begin{align*}
\max\{|x_i - y_i|, |x_i - z_i|\} 
&\ge \max\{\d(K^i, I^i), \, \d(K^i, J^i)\} 
\\  
> 2 \ell(K^i)^{\gamma_i} \ell(J^i)^{1-\gamma_i}
&\ge 2^{\gamma_i} \ell(K^i)
\ge 2^{1+\gamma_i} |x_i - c_{K^i}|
\end{align*}
and 
\begin{align*}
1 + \frac{|x_i + y_i| + |x_i + z_i|}{1+ |x_i - y_i| + |x_i - z_i|}
\ge \frac{2 (|y_i| + |z_i|)}{1+ |x_i - y_i| + |x_i - z_i|}
\\ 
\ge \frac{4 \d(Q^i, \I^i)}{1 + 2 \ell(Q^i)}
\ge \frac{\d(Q^i, \I^i)}{\max\{\ell(Q^i), 1\}}
= \rd(Q^i, \I^i).
\end{align*}
By the monotoncity of $F_1^i$, $F_2^i$, and $F_3^i$, the estimates above imply  
\begin{align*}
F^i(x_i, y_i, z_i)
\le F_1^i(\ell(K_i)) F_2^i(\ell(K_i)) F_3^i(\rd(Q_i, \I_i)) 
\le F^i(K^i, Q^i),  
\end{align*}
which along with $\gamma_i = \frac{\delta_i}{2(2n_i + \delta_i)}$ gives 
\begin{align*}
\mathscr{P}_i(I^i, J^i, K^i) 
&\lesssim \frac{F^i(K^i, Q^i) \, \ell(K^i)^{\delta_i}}{[\ell(K_i)^{\gamma_i} \ell(Q^i)^{1-\gamma_i}]^{2n_i + \delta_i}} 
|I^i| |J^i| |K^i|
\\ \nonumber 
&= \bigg[ \frac{\ell(K^i)}{\ell(Q^i)} \bigg]^{\frac{\delta_i}{2}} 
F^i(K^i, Q^i) \frac{|I^i| |J^i| |K^i|}{|Q^i|^2}.  
\end{align*}
This completes the proof. 
\end{proof}

\begin{lemma}\label{lem:QQ}
For each $i=1, 2$, let $(F_1^i, F_2^i, F_3^i) \in \mathscr{F}$ satisfy that $F_1^i$ is increasing, $F_2^i$ and $F_3^i$ are decreasing. Let $I^i, J^i, K^i, Q^i \subset \R^{n_i}$ be cubes such that $\ell(I^i) \simeq \ell(J^i) \simeq \ell(K^i) \simeq \ell(Q^i)$ and $I^i \cup J^i \cup K^i \subset Q^i$. If either $K^i \cap I^i = \emptyset$ or $K^i \cap J^i = \emptyset$, then 
\begin{align*}
\mathscr{Q}_i(I^i, J^i, K^i)
&:= \int_{J^i} \int_{I^i} \int_{K^i} 
\frac{F^i(x_i, y_i, z_i)}{(|x_i-y_i| + |x_i-z_i|)^{2n_i}}  
\, dx_i \, dy_i \, dz_i 
\\ 
&\lesssim \bigg[ \frac{\ell(K^i)}{\ell(Q^i)} \bigg]^{\frac{\delta_i}{2}} 
F^i(K^i, Q^i) \frac{|I^i| |J^i| |K^i|}{|Q^i|^2},  
\end{align*}
where $F^i(K^i, Q^i)$ is given in \eqref{def:FKQ}, and 
\begin{align*}
F^i(x_i, y_i, z_i)
&:= F_1^i(|x_i-y_i| + |x_i-z_i|) F_2^i(|x_i - y_i| + |x_i - z_i|) 
F_3^i \bigg(1 + \frac{|x_i + y_i| + |x_i + z_i|}{1+ |x_i - y_i| + |x_i - z_i|}\bigg).
\end{align*}
\end{lemma}

\begin{proof}
By symmetry, we may assume that $K^i \cap I^i = \emptyset$. Then $K^i \subset \kappa I^i \setminus I^i$ for some universal constant $\kappa >1$. Let $\alpha_i \in (0, 1)$ be an auxiliary parameter. For any $x_i \in K^i$, we have $J^i-x_i \subset \{|z_i| \le \ell(Q^i)\}$ and 
\begin{align*}
\int_{J^i} \frac{dz_i}{|x_i - z_i|^{n_i - \alpha_i}} 
= \int_{J^i-x_i} \frac{dz_i}{|z_i|^{n_i - \alpha_i}} 
\lesssim \int_0^{\ell(Q^i)} t^{\alpha_i-1} \, dt 
\lesssim \ell(Q^i)^{\alpha_i}. 
\end{align*}
Pick $q \in (1, \frac{n_i+1}{n_i+\alpha_i})$. Then \cite[Lemma 2.12]{CYY} yields 
\begin{align*}
&\int_{I^i} \int_{\kappa I^i \setminus I^i} 
\frac{F_2^i(|x_i - y_i|)}{|x_i - y_i|^{n_i + \alpha_i}} dx_i \, dy_i 
\\ 
&\le \bigg[\int_{I^i} \int_{\kappa I^i \setminus I^i} 
\frac{dx_i \, dy_i}{|x_i - y_i|^{q(n_i + \alpha_i)}} \bigg]^{\frac1q}
\bigg[\int_{I^i} \int_{\kappa I^i \setminus I^i} 
F_2^i(|x_i - y_i|)^{q'} dx_i \, dy_i \bigg]^{\frac{1}{q'}} 
\\
&\lesssim |I^i|^{\frac{2}{q}-1-\frac{\alpha_i}{n_i}}  
\widetilde{F}_2^i(\ell(I^i)) |I^i|^{\frac{2}{q'}}
\le \widetilde{F}_2^i(\ell(K^i)) |I^i| \, \ell(I^i)^{-\alpha_i}.   
\end{align*}
In addition, for all $(x_i, y_i, z_i) \in K^i \times I^i \times J^i$, 
\begin{align*}
F_1^i (|x_i-y_i| + |x_i-z_i|) 
\le F_1^i(2 \ell(Q^i)) 
\lesssim F_1^i(\ell(K^i))
\end{align*}
and 
\begin{align*}
1 &+ \frac{|x_i + y_i| + |x_i + z_i|}{1+ |x_i - y_i| + |x_i - z_i|}
\ge \frac{2 (|y_i| + |z_i|)}{1+ |x_i - y_i| + |x_i - z_i|}
\\ 
&\ge \frac{4 \d(Q^i, \I^i)}{1 + 2 \ell(Q^i)}
\ge \frac{\d(Q^i, \I^i)}{\max\{\ell(Q^i), 1\}}
= \rd(Q^i, \I^i).
\end{align*}
Thus, the desired inequality follows from the estimates above.
\end{proof}

\begin{lemma}\label{lem:RR}
For each $i=1, 2$, let $(F_1^i, F_2^i, F_3^i) \in \mathscr{F}$ satisfy that $F_1^i$ is increasing, $F_2^i$ and $F_3^i$ are decreasing. Let $J^i, K^i \subset \R^{n_i}$ be cubes such that $K^i \subset J^i$. If $\d(K^i, (J^i)^c) \ge \ell_i \ge  \ell(K^i)$, then  
\begin{align}\label{RR1}
\mathscr{R}_i^1(J^i, K^i) 
&:= \int_{(J^i)^c} \int_{3 K^i} \int_{K^i} 
\frac{F^i(x_i, y_i, z_i) \, \ell(K^i)^{\delta_i}}{(|x_i-y_i| + |x_i-z_i|)^{2n_i+\delta_i}} 
dx_i \, dy_i \, dz_i  
\\ \nonumber 
&\lesssim F_1^i(\ell(K^i)) F_2^i(\ell(K^i)) \widetilde{F}_3^i(J^i) 
|K^i| [\ell(K^i)/\ell_i]^{\delta_i}, 
\end{align}
and 
\begin{align}\label{RR2}
\mathscr{R}_i^2(J^i, K^i) 
&:= \int_{(J^i)^c} \int_{(3 K^i)^c} \int_{K^i} 
\frac{F^i(x_i, y_i, z_i) \, \ell(K^i)^{\delta_i}}{(|x_i-y_i| + |x_i-z_i|)^{2n_i+\delta_i}} 
dx_i \, dy_i \, dz_i  
\\ \nonumber 
&\lesssim F_1^i(\ell(K^i)) F_2^i(\ell(K^i)) \widetilde{F}_3^i(J^i) 
|K^i| [\ell(K^i)/\ell_i]^{\delta_i}, 
\end{align}
where 
\begin{align*}
F^i(x_i, y_i, z_i)
:= F_1^i(|x_i - c_{K^i}|) F_2^i(|x_i-y_i| + |x_i-z_i|) 
F_3^i \bigg(1 + \frac{|x_i+y_i| + |x_i+z_i|}{1+ |x_i-y_i| + |x_i-z_i|}\bigg). 
\end{align*}
\end{lemma}

\begin{proof} 
For all $x_i \in K^i$ and $z_i \in (J^i)^c$, there holds 
\begin{align*}
|x_i - z_i| 
\ge \d(K^i, (J^i)^c) 
\ge \ell_i  
\ge \ell(K^i)  
\ge 2 |x_i - c_{K^i}|.
\end{align*}
Observe that for any $\alpha_i \in (0, 1)$, 
\begin{align*}
\int_{3 K^i} \frac{dy_i}{|x_i - y_i|^{n_i - \alpha_i}} 
\lesssim \ell(K^i)^{\alpha_i}, \qquad x_i \in K^i, 
\end{align*}
and for any $\beta_i>0$,  
\begin{align*}
\psi(\beta_i) 
& := \int_{K^i} \int_{|x_i - z_i| \ge \ell_i} 
F_3^i\bigg(\frac{4|x_i|}{1+|x_i - z_i|}\bigg) 
\frac{1}{|x_i - z_i|^{n_i + \beta_i}} dz_i \, dx_i 
\\ 
&\le \sum_{k \ge 0} \int_{K^i} \int_{2^k \ell_i \le |x_i - z_i| < 2^{k+1} \ell_i} 
(2^k \ell_i)^{-n_i - \beta_i} F_3^i(\rd(2^k J^i, \I^i)) \, dz_i \, dx_i 
\\ 
&\lesssim |K^i| \ell_i^{-\beta_i} 
\sum_{k \ge 0} 2^{-k \beta_i} F_3^i(\rd(2^k J^i, \I^i))
\end{align*}
provided that for all $x_i \in K^i \subset J^i$ and $2^k \ell_i \le |x_i - z_i| < 2^{k+1} \ell_i$, 
\begin{align*}
\frac{4|x_i|}{1+|x_i - z_i|}
\ge \frac{4 \d(2^k J^i, 0)}{1+2^{k+1} \ell(J^i)} 
\ge \frac{\d(2^k J^i, 0)}{\max\{2^k \ell(J^i), 1\}}
= \rd(2^k J^i, \I^i). 
\end{align*}
Thus, by the monotonicity of $F_1^i, F_2^i, F_3^i$, 
\begin{align*}
\mathscr{R}_i^1(J^i, K^i) 
&\lesssim  F_1^i(\ell(K^i)) F_2^i(\ell(K^i)) 
\ell(K^i)^{\delta_i + \alpha_i} \psi(\delta_i + \alpha_i)
\\ 
&\lesssim F_1^i(\ell(K^i)) F_2^i(\ell(K^i)) \widetilde{F}_3^i(J^i) 
|K^i| [\ell(K^i)/\ell_i]^{\delta_i},  
\end{align*}
which shows \eqref{RR1}.

Additionally, given $k \ge 0$ and $x_i \in K^i  \subset J^i$, letting 
\begin{align*}
R_k(x_i) := \{(y_i, z_i) \in \R^{2n_i}: 2^k \ell_i \le |x_i - y_i| + |x_i - z_i| < 2^{k+1} \ell_i\},
\end{align*} 
we obtain for all $(y_i, z_i) \in R_k(x_i)$, 
\begin{align*}
1 + \frac{|x_i + y_i| + |x_i + z_i|}{1+ |x_i - y_i| + |x_i - z_i|}
\ge \frac{4|x_i|}{1+ 2^{k+1} \ell(J^i)}
\ge \frac{\d(2^k J^i, 0)}{\max\{2^k \ell(J^i), 1\}}
= \rd(2^k J^i, \I^i),   
\end{align*}
and hence, 
\begin{align*}
\mathscr{R}_i^2(J^i, K^i) 
&\lesssim  F_1^i(\ell(K^i)) F_2^i(\ell(K^i)) \ell(K^i)^{\delta_i} 
\\ 
&\quad \times \int_{K^i} \bigg(\sum_{k \ge 0} \iint_{R_k(x_i)} 
\frac{F_3^i(\rd(2^k J^i, \I^i))}{(2^k \ell_i)^{2n_i + \delta_i}} dy_i\, dz_i\bigg) \, dx_i
\\ 
&\lesssim F_1^i(\ell(K^i)) F_2^i(\ell(K^i)) \ell(K^i)^{\delta_i} 
|K^i| \ell_i^{-\delta_i} \sum_{k \ge 0} 2^{-k \delta_i} F_3^i(\rd(2^k J^i, \I^i)) 
\\  
&\le F_1^i(\ell(K^i)) F_2^i(\ell(K^i)) \widetilde{F}_3^i(J^i) 
|K^i| [\ell(K^i)/\ell_i]^{\delta_i}. 
\end{align*}
This coincides with \eqref{RR2}.
\end{proof}

\subsection{$\BMO$ and $\CMO$ estimates} 
The estimates in this section will be applied to obtain bi-parameter partial paraproducts in Sections \ref{sec:SN}, \ref{sec:AN}, and \ref{sec:NN}. 

\begin{lemma}\label{lem:PPSN}
Let $I^1, J^1 \in \D^1$ and $K^1 \in \D_{\rm{good}}^1$ satisfy $\ell(K^1) \le \ell(I^1) = 2\ell(J^1)$ and $\max\{\d(K^1, I^1), \, \d(K^1, J^1)\} > 2\ell(K^1)^{\gamma_1} \ell(J^1)^{1-\gamma_1}$. Set 
\begin{align*}
b_{I^1, J^1, K^1} 
:= \langle T(h_{I^1} \otimes 1, h_{J^1}^0 \otimes 1), h_{K^1} \rangle.
\end{align*} 
Then 
\begin{align}\label{BK-1}
\big\|b_{I^1, J^1, K^1}\big\|_{\BMO(\D^2)} 
\lesssim \bigg[ \frac{\ell(K^1)}{\ell(Q^1)} \bigg]^{\frac{\delta_1}{2}} 
F^1(K^1, Q^1) \frac{|I^1|^{\frac12} |J^1|^{\frac12} |K^1|^{\frac12}}{|Q^1|^2},  
\end{align}
and 
\begin{align}\label{BK-2}
\big\|P_{\D^2(N)}^{\perp} b_{I^1, J^1, K^1}\big\|_{\BMO(\D^2)} 
\lesssim \mathcal{F}^2_N \, \bigg[ \frac{\ell(K^1)}{\ell(Q^1)} \bigg]^{\frac{\delta_1}{2}} 
F^1(K^1, Q^1) \frac{|I^1|^{\frac12} |J^1|^{\frac12} |K^1|^{\frac12}}{|Q^1|^2},  
\end{align}
where 
\begin{align}\label{def:F2N}
\mathcal{F}^2_N 
:= \sup_{\D^2} \sup_{K^2 \notin \D^2(N)} \mathcal{F}^2(K^2) 
:= \sup_{\D^2} \sup_{K^2 \notin \D^2(N)} \big[F^2(K^2) + F^2(K^2, K^2)\big], 
\end{align}
$F^2 \in \F^2$, $F^i(\cdot, \cdot)$ is given in \eqref{def:FKQ}, and $Q^1$ is defined in Lemma $\ref{lem:cda}$  \ref{cda-1}. 
\end{lemma}

\begin{proof}
Given $K^2 \in \D^2$, let $a'_{K^2}$ be an arbitrary $\infty$-atom of $\mathrm{H}_{\D^2}^1$, namely, $\supp (a'_{K^2}) \subset K^2$, $\|a'_{K^2}\|_{L^{\infty}} \le |K^2|^{-1}$, and $\int_{\R^{n_2}} a'_{K^2} \, dx_2 =0$. Write $a_{K^2} = a'_{K^2} |K^2|$. By a similar argument in \cite[p. 6272]{CYY}, the inequalities \eqref{BK-1} and \eqref{BK-2} are reduced to showing 
\begin{align}\label{SPP}
|\langle b_{I^1, J^1, K^1}, a_{K^2} \rangle| 
\lesssim \bigg[ \frac{\ell(K^1)}{\ell(Q^1)} \bigg]^{\frac{\delta_1}{2}} 
F^1(K^1, Q^1) \frac{|I^1|^{\frac12} |J^1|^{\frac12} |K^1|^{\frac12}}{|Q^1|^2} 
\mathcal{F}^2(K^2) \, |K^2|. 
\end{align}

By the fact $1 = \mathbf{1}_{K^2} + \mathbf{1}_{3K^2 \setminus K^2} + \mathbf{1}_{(3K^2)^c}$, $b_{I^1, J^1, K}$ can be split into nine terms. By symmetry, it suffices to bound the following:
\begin{align*}
\mathscr{I}_1 
&:= \langle T(h_{I^1} \otimes \mathbf{1}_{K^2}, 
h_{J^1}^0 \otimes \mathbf{1}_{K^2}), 
h_{K^1} \otimes a_{K^2} \rangle, 
\\
\mathscr{I}_2 
&:= \langle T(h_{I^1} \otimes \mathbf{1}_{K^2}, 
h_{J^1}^0 \otimes \mathbf{1}_{3K^2 \setminus K^2}), 
h_{K^1} \otimes a_{K^2} \rangle, 
\\
\mathscr{I}_3
&:= \langle T(h_{I^1} \otimes \mathbf{1}_{K^2}, 
h_{J^1}^0 \otimes \mathbf{1}_{(3K^2)^c}), 
h_{K^1} \otimes a_{K^2} \rangle, 
\\
\mathscr{I}_4
&:= \langle T(h_{I^1} \otimes \mathbf{1}_{3K^2 \setminus K^2}, 
h_{J^1}^0 \otimes \mathbf{1}_{3K^2 \setminus K^2}), 
h_{K^1} \otimes a_{K^2} \rangle, 
\\
\mathscr{I}_5
&:= \langle T(h_{I^1} \otimes \mathbf{1}_{3K^2 \setminus K^2}, 
h_{J^1}^0 \otimes \mathbf{1}_{(3K^2)^c}), 
h_{K^1} \otimes a_{K^2} \rangle, 
\\
\mathscr{I}_6
&:= \langle T(h_{I^1} \otimes \mathbf{1}_{(3K^2)^c}, h_{J^1}^0 \otimes \mathbf{1}_{(3K^2)^c}), 
h_{K^1} \otimes a_{K^2} \rangle. 
\end{align*}
By the cancellation of $h_{K^1}$, the compact partial kernel representation (cf. \eqref{H2}), the H\"{o}lder condition, and Lemma \ref{lem:PP}, there holds  
\begin{align}\label{SPP-1}
|\mathscr{I}_1| 
&\le \mathscr{P}_1(I^1, J^1, K^1) C(\mathbf{1}_{K^2}, \mathbf{1}_{K^2}, a_{K^2}) 
|I^1|^{-\frac12} |J^1|^{-\frac12} |K^1|^{-\frac12} 
\\ \nonumber 
&\lesssim \bigg[ \frac{\ell(K^1)}{\ell(Q^1)} \bigg]^{\frac{\delta_1}{2}} 
F^1(K^1, Q^1) \frac{|I^1|^{\frac12} |J^1|^{\frac12} |K^1|^{\frac12}}{|Q^1|^2} 
F^2(K^2) |K^2|. 
\end{align}
Note that for any cube $K^i \subset \R^{n_i}$, there exists a family of disjoint cubes $\{K_j^i\}_{j=0}^{3n_i-1}$ in $\R^{n_i}$ such that 
\begin{align}\label{3KK}
3K^i = \bigcup_{j=0}^{3n_i -1} K_j^i, \quad \text{where} \quad 
K_0^i = K^i, \quad 
\ell(K_j^i) = \ell(K^i), \quad 
\d(K_j^i, K^i) = 0.
\end{align}
Then using the cancellation of $h_{K^1}$, the compact full kernel representation (cf. \eqref{H1}), the mixed size-H\"{o}lder condition, and Lemmas \ref{lem:PP} and \ref{lem:QQ}, we obtain 
\begin{align}\label{SPP-2}
|\mathscr{I}_2| 
&\le \mathscr{P}_1(I^1, J^1, K^1) \sum_{j=1}^{3n_2 -1} \mathscr{Q}_2(K^2, K_j^2, K^2) 
|I^1|^{-\frac12} |J^1|^{-\frac12} |K^1|^{-\frac12}
\\ \nonumber 
&\lesssim \bigg[ \frac{\ell(K^1)}{\ell(Q^1)} \bigg]^{\frac{\delta_1}{2}} 
F^1(K^1, Q^1) \frac{|I^1|^{\frac12} |J^1|^{\frac12} |K^1|^{\frac12}}{|Q^1|^2} 
F^2(K^2, 3K^2) |K^2|, 
\end{align}
and $\mathscr{I}_4$ has the same bound replacing $\mathscr{Q}_2(K^2, K_j^2, K^2)$ by $\sum_{k=1}^{3n_2-1} \mathscr{Q}_2(K^2_k, K^2_j, K^2)$. Additionally, in view of the cancellation of $h_{K^1}$ and $a_{K^2}$, the compact full kernel representation (cf. \eqref{H1}), and the H\"{o}lder condition, Lemmas \ref{lem:PP} and \ref{lem:RR} applied to $\ell_2 = \ell(K^2)$ imply 
\begin{align}\label{SPP-3}
|\mathscr{I}_3| + |\mathscr{I}_5|
&\le \mathscr{P}_1(I^1, J^1, K^1) \mathscr{R}_2^1(3K^2, K^2) 
|I^1|^{-\frac12} |J^1|^{-\frac12} |K^1|^{-\frac12} 
\\ \nonumber 
&\lesssim \bigg[ \frac{\ell(K^1)}{\ell(Q^1)} \bigg]^{\frac{\delta_1}{2}} 
F^1(K^1, Q^1) \frac{|I^1|^{\frac12} |J^1|^{\frac12} |K^1|^{\frac12}}{|Q^1|^2} 
F^2(K^2, 3K^2) |K^2|. 
\end{align}
Besides, $\mathscr{I}_6$ can be estimated instead of $\mathscr{R}_2^1$ above by $\mathscr{R}_2^2$. Consequently, \eqref{SPP} follows from \eqref{SPP-1}, \eqref{SPP-2}, \eqref{SPP-3}, and the fact that $F^2(K^2, 3K^2) \simeq F^2(K^2, K^2)$ by \eqref{F3Q}.
\end{proof}

\begin{lemma}\label{lem:PPAN}
Let $I^1, J^1 \in \D^1$, $K^1 \in \D_{\rm{good}}^1$, and $K^2 \in \D_{\rm{good}}^2$ satisfy $\ell(K^1) \le \ell(I^1) = 2\ell(J^1)$, $\max\{\d(K^1, I^1), \, \d(K^1, J^1)\} \le 2\ell(K^1)^{\gamma_1} \ell(J^1)^{1-\gamma_1}$, and either $K^1 \cap I^1 = \emptyset$ or $K^1 \cap J^1 = \emptyset$ or $K^1 = I^1$. Set 
\begin{align*}
b_{I^1, J^1, K^1} := \langle T(h_{I^1} \otimes 1, h_{J^1}^0 \otimes 1), h_{K^1} \rangle.
\end{align*} 
Then 
\begin{align*}
\|b_{I^1, J^1, K^1}\|_{\BMO(\D^2)} 
\lesssim \bigg[ \frac{\ell(K^1)}{\ell(Q^1)} \bigg]^{\frac{\delta_1}{2}} 
\bigg[F^1(K^1, Q^1) + \sum_{J^1 \in \ch(K^1)} F^1(J^1) \bigg]
\frac{|I^1|^{\frac12} |J^1|^{\frac12} |K^1|^{\frac12}}{|Q^1|^2}, 
\end{align*}
and 
\begin{align*}
\big\|P_{\D^2(N)}^{\perp} b_{I^1, J^1, K^1}\big\|_{\BMO(\D^2)} 
\lesssim \mathcal{F}^2_N \bigg[ \frac{\ell(K^1)}{\ell(Q^1)} \bigg]^{\frac{\delta_1}{2}} 
\bigg[F^1(K^1, Q^1) + \sum_{J^1 \in \ch(K^1)} F^1(J^1) \bigg] 
\frac{|I^1|^{\frac12} |J^1|^{\frac12} |K^1|^{\frac12}}{|Q^1|^2}, 
\end{align*}
where $\mathcal{F}^2_N$ is given in \eqref{def:F2N} and $Q^1$ is defined in Lemma $\ref{lem:cda}$  \ref{cda-2}.
\end{lemma}

\begin{proof}
In the case $K^1 \cap I^1 = \emptyset$, the proof is much as that of Lemma \ref{lem:PPSN}. The only difference is that $\mathscr{P}_1(I^1, J^1, K^1)$ is replaced by $\mathscr{Q}_1(I^1, J^1, K^1)$. Then Lemma \ref{lem:QQ} gives the same estimates. The case $K^1 \cap J^1 = \emptyset$ is similar. 

Let $K^1 \cap J^1 \neq \emptyset$ and $K^1 = I^1$. Then, $J^1 \in \ch(K^1)$ and  
\begin{align*}
\langle b_{I^1, J^1, K^1}, a_{K^2} \rangle  
= \sum_{K_1^1, K_2^1 \in \ch(K^1)} 
\langle h_{I^1} \rangle_{K_1^1} \langle h_{K^1} \rangle_{K_2^1} 
\langle T(\mathbf{1}_{K_1^1} \otimes 1, h_{J^1}^0 \otimes 1), 
\mathbf{1}_{K_2^1} \otimes a_{K^2} \rangle. 
\end{align*}
If $K_1^1 = J^1 = K_2^1$, then we use the same decomposition as in the proof of Lemma \ref{lem:PPSN}. For the first term, the diagonal CMO condition (cf. \eqref{H4}) gives  
\begin{align*}
|\langle T(\mathbf{1}_{K_1^1} \otimes \mathbf{1}_{K^2}, 
h_{J^1}^0 \otimes \mathbf{1}_{K^2}), 
\mathbf{1}_{K_2^1} \otimes a_{K^2} \rangle| 
\le |J^1|^{-\frac12} F^1(J^1) |J^1| F^2(K^2) |K^2|. 
\end{align*}
Moreover, replacing $\mathscr{P}_1(I^1, J^1, K^1)$ by $C(\mathbf{1}_{J^1}, \mathbf{1}_{J^1}, \mathbf{1}_{J^1})$ in the proof of Lemma \ref{lem:PPSN}, the remaining terms are dominated by 
\begin{align*}
|J^1|^{-\frac12} F^1(J^1) |J^1| \big[F^2(K^2) + F^2(K^2, 3K^2) \big] |K^2|.  
\end{align*}
If $K_1^1 \neq J^1$ or $K_2^1 \neq J^1$, then it is similar to the case $K^1 \cap I^1 = \emptyset$. Hence, 
\begin{align*}
|\langle b_{I^1, J^1, K^1}, a_{K^2} \rangle|
&\lesssim \sup_{K_2^1 \in \ch(K^1)} 
\bigg[ \frac{\ell(K_2^1)}{\ell(Q^1)} \bigg]^{\frac{\delta_1}{2}} 
F^1(K_2^1, Q^1) \frac{|I^1|^{\frac12} |J^1|^{\frac12} |K_2^1|^{\frac12}}{|Q^1|^2} 
F^2(K^2, 3K^2) |K^2| 
\\
&\lesssim \bigg[ \frac{\ell(K^1)}{\ell(Q^1)} \bigg]^{\frac{\delta_1}{2}} 
F^1(K^1, Q^1) \frac{|I^1|^{\frac12} |J^1|^{\frac12} |K^1|^{\frac12}}{|Q^1|^2} 
F^2(K^2, 3K^2) |K^2|.  
\end{align*}
This completes the proof. 
\end{proof}

\begin{lemma}\label{lem:PPNN-1}
Let $K^2 \in \D_{\rm{good}}^2$ and $Q^2 \in \D^2$ satisfy $K^2 \subset Q^2$. Set 
\begin{align*}
b_{K^2, Q^2} 
:= \langle h_{(Q^2)^{(1)}} \rangle_{Q^2} 
\langle T(1, 1 \otimes \mathbf{1}_{(Q^2)^c}), h_{K^2} \rangle. 
\end{align*}
Then 
\begin{align}
\|b_{K^2, Q^2}\|_{\BMO(\D^1)} 
\lesssim \bigg[ \frac{\ell(K^2)}{\ell(Q^2)} \bigg]^{\frac{\delta_2}{2}} 
F^2(K^2, Q^2) \frac{|K^2|^{\frac12}}{|Q^2|^{\frac12}}, 
\end{align}
and 
\begin{align}
\|P_{\D^1(N)}^{\perp} b_{K^2, Q^2}\|_{\BMO(\D^1)} 
\lesssim \mathcal{F}^1_N \, \bigg[ \frac{\ell(K^2)}{\ell(Q^2)} \bigg]^{\frac{\delta_2}{2}} 
F^2(K^2, Q^2) \frac{|K^2|^{\frac12}}{|Q^2|^{\frac12}}, 
\end{align}
where \begin{align}\label{def:F1N}
\mathcal{F}^1_N 
:= \sup_{\D^1} \sup_{Q^1 \notin \D^1(N)} \mathcal{F}^1(Q^1) 
:= \sup_{\D^1} \sup_{Q^1 \notin \D^1(N)} \big[F^1(Q^1) + F^1(Q^1, Q^1)\big], 
\end{align}
$F^1 \in \F^1$, and $F^1(\cdot, \cdot)$ is given in \eqref{def:FKQ}.
\end{lemma}

\begin{proof}
Given $Q^1 \in \D^1$, let $a'_{Q^1}$ be an arbitrary $\infty$-atom of $\mathrm{H}_{\D^1}^1$, namely, $\supp (a'_{Q^1}) \subset Q^1$, $\|a'_{Q^1}\|_{L^{\infty}} \le |Q^1|^{-1}$, and $\int_{\R^{n_1}} a'_{Q^1} \, dx_1 =0$. Write $a_{Q^1} = a'_{Q^1} |Q^1|$. As argued in \cite[p. 6272]{CYY}, it is enough to prove  
\begin{align}\label{NPP}
|\langle b_{K^2, Q^2}, a_{Q^1} \rangle| 
\lesssim \mathcal{F}^1(Q^1) |Q^1| 
\bigg[ \frac{\ell(K^2)}{\ell(Q^2)} \bigg]^{\frac{\delta_2}{2}} 
F^2(K^2, Q^2) \frac{|K^2|^{\frac12}}{|Q^2|^{\frac12}}. 
\end{align}

First, we deal with the case $\ell(Q^2) > 2^{\vartheta} \ell(K^2)$. The goodness of $K^2$ gives 
\begin{align*}
\d(K^2, (Q^2)^c) 
> 2 \ell(K^2)^{\gamma_2} \ell(Q^2)^{1-\gamma_2} 
\ge \ell(K^2)^{\frac12} \ell(Q^2)^{\frac12} 
\ge \ell(K^2),  
\end{align*}
which allows us to write $\langle b_{K^2, Q^2}, a_{Q^1} \rangle = \mathscr{J}_1 + \mathscr{J}_2$, where 
\begin{align*}
\mathscr{J}_1 
&:= \langle h_{(Q^2)^{(1)}} \rangle_{Q^2} 
\langle T(1 \otimes \mathbf{1}_{3K^2}, 1 \otimes \mathbf{1}_{(Q^2)^c}), 
a_{Q^1} \otimes h_{K^2} \rangle, 
\\
\mathscr{J}_2 
&:= \langle h_{(Q^2)^{(1)}} \rangle_{Q^2} 
\langle T(1 \otimes \mathbf{1}_{(3K^2)^c}, 1 \otimes \mathbf{1}_{(Q^2)^c}), 
a_{Q^1} \otimes h_{K^2} \rangle. 
\end{align*}
Using the fact $1 = \mathbf{1}_{Q^1} + \mathbf{1}_{3Q^1 \setminus Q^1} + \mathbf{1}_{(3Q^1)^c}$, we split $\mathscr{I}_1$ into nine terms. By symmetry, it suffices to treat the following terms: 
\begin{align*}
\mathscr{J}_{1, 1} 
&:= \langle h_{(Q^2)^{(1)}} \rangle_{Q^2} 
\langle T(\mathbf{1}_{Q^1} \otimes \mathbf{1}_{3K^2}, 
\mathbf{1}_{Q^1} \otimes \mathbf{1}_{(Q^2)^c}), 
a_{Q^1} \otimes h_{K^2} \rangle, 
\\
\mathscr{J}_{1, 2} 
&:= \langle h_{(Q^2)^{(1)}} \rangle_{Q^2} 
\langle T(\mathbf{1}_{Q^1} \otimes \mathbf{1}_{3K^2}, 
\mathbf{1}_{3Q^1 \setminus Q^1} \otimes \mathbf{1}_{(Q^2)^c}), 
a_{Q^1} \otimes h_{K^2} \rangle, 
\\
\mathscr{J}_{1, 3} 
&:= \langle h_{(Q^2)^{(1)}} \rangle_{Q^2} 
\langle T(\mathbf{1}_{Q^1} \otimes \mathbf{1}_{3K^2}, 
\mathbf{1}_{(3Q^1)^c} \otimes \mathbf{1}_{(Q^2)^c}), 
a_{Q^1} \otimes h_{K^2} \rangle, 
\\
\mathscr{J}_{1, 4} 
&:= \langle h_{(Q^2)^{(1)}} \rangle_{Q^2} 
\langle T(\mathbf{1}_{3Q^1 \setminus Q^1} \otimes \mathbf{1}_{3K^2}, 
\mathbf{1}_{3Q^1 \setminus Q^1} \otimes \mathbf{1}_{(Q^2)^c}), 
a_{Q^1} \otimes h_{K^2} \rangle, 
\\
\mathscr{J}_{1, 5} 
&:= \langle h_{(Q^2)^{(1)}} \rangle_{Q^2} 
\langle T(\mathbf{1}_{3Q^1 \setminus Q^1} \otimes \mathbf{1}_{3K^2}, 
\mathbf{1}_{(3Q^1)^c} \otimes \mathbf{1}_{(Q^2)^c}), 
a_{Q^1} \otimes h_{K^2} \rangle, 
\\
\mathscr{J}_{1, 6} 
&:= \langle h_{(Q^2)^{(1)}} \rangle_{Q^2} 
\langle T(\mathbf{1}_{(3Q^1)^c} \otimes \mathbf{1}_{3K^2}, 
\mathbf{1}_{(3Q^1)^c} \otimes \mathbf{1}_{(Q^2)^c}), 
a_{Q^1} \otimes h_{K^2} \rangle. 
\end{align*}
By the cancellation of $h_{K^2}$, the compact partial kernel representation (cf. \eqref{H2}), and the H\"{o}lder condition, Lemma \ref{lem:RR} applied to $\ell_2 = \ell(K^2)^{\frac12} \ell(Q^2)^{\frac12}$ yields 
\begin{align}\label{NPP-1}
|\mathscr{J}_{1, 1}| 
&\lesssim C(\mathbf{1}_{Q^1}, \mathbf{1}_{Q^1}, a_{Q^1})  
\mathscr{R}_2^1(Q^2, K^2) |Q^2|^{-\frac12} |K^2|^{-\frac12} 
\\ \nonumber 
&\lesssim F^1(Q^1) |Q^1| 
\bigg[ \frac{\ell(K^2)}{\ell(Q^2)} \bigg]^{\frac{\delta_2}{2}} 
F^2(K^2, Q^2) \frac{|K^2|^{\frac12}}{|Q^2|^{\frac12}}.
\end{align}
It follows from the cancellation of $h_{K^2}$, the compact full kernel representation (cf. \eqref{H1}), the mixed size-H\"{o}lder condition, and Lemmas \ref{lem:QQ} and \ref{lem:RR} applied to $\ell_2 = \ell(K^2)^{\frac12} \ell(Q^2)^{\frac12}$ that 
\begin{align}\label{NPP-2}
|\mathscr{J}_{1, 2}| 
&\lesssim \sum_{j=1}^{3^{n_1}-1} \mathscr{Q}_1(Q^1, Q^1_j, Q^1)  
\mathscr{R}_2^1(Q^2, K^2) |Q^2|^{-\frac12} |K^2|^{-\frac12} 
\\ \nonumber 
&\lesssim F^1(Q^1, 3Q^1) |Q^1| 
\bigg[ \frac{\ell(K^2)}{\ell(Q^2)} \bigg]^{\frac{\delta_2}{2}} 
F^2(K^2, Q^2) \frac{|K^2|^{\frac12}}{|Q^2|^{\frac12}}. 
\end{align}
Likewise, instead of $\mathscr{Q}_1(Q^1, Q_j^1, Q^1)$ by $\sum_{k=1}^{3^{n_1}-1} \mathscr{Q}_1(Q_k^1, Q_j^1, Q^1)$, the same estimate holds for $\mathscr{J}_{1, 4}$. Moreover, using the cancellation of $a_{Q^1}$ and $h_{K^2}$, the compact full kernel representation (cf. \eqref{H1}), the H\"{o}lder condition, and Lemma \ref{lem:RR} applied to $\ell_1 = \ell(Q^1)$ and $\ell_2 = \ell(K^2)^{\frac12} \ell(Q^2)^{\frac12}$, we arrive at 
\begin{align}\label{NPP-3}
|\mathscr{J}_{1, 3}| + |\mathscr{J}_{1, 5}| 
&\lesssim \mathscr{R}_1^1(3Q^1, Q^1)  
\mathscr{R}_2^1(Q^2, K^2) |Q^2|^{-\frac12} |K^2|^{-\frac12} 
\\ \nonumber 
&\lesssim F^1(Q^1, 3Q^1) |Q^1| 
\bigg[ \frac{\ell(K^2)}{\ell(Q^2)} \bigg]^{\frac{\delta_2}{2}} 
F^2(K^2, Q^2) \frac{|K^2|^{\frac12}}{|Q^2|^{\frac12}}. 
\end{align}
Similarly, replacing $\mathscr{R}_1^1$ by $\mathscr{R}_1^2$, we obtain the same bound for $\mathscr{J}_{1, 6}$. 

Much as above, instead of $\mathscr{R}_2^1(Q^2, K^2)$ by $\mathscr{R}_2^2(Q^2, K^2)$, we have the same estimates for $\mathscr{J}_2$. 

Next, in the case $\ell(Q^2) \le 2^{\vartheta} \ell(K^2)$, using $\mathbf{1}_{(Q^2)^c} = \mathbf{1}_{3Q^2 \setminus Q^2} + \mathbf{1}_{(3Q^2)^c}$, we split $\mathscr{J}_{1, k}$ into two terms $\mathscr{J}_{1, k}^1$ and $\mathscr{J}_{1, k}^2$, $k=1, \ldots, 6$. By the compact partial kernel representation (cf. \eqref{H2}) and the size condition, Lemma \ref{lem:QQ} implies  
\begin{align}\label{NPP-4}
|\mathscr{J}_{1, 1}^1| 
&\lesssim C(\mathbf{1}_{Q^1}, \mathbf{1}_{Q^1}, a_{Q^1})  
\sum_{j=1}^{3n_2-1} \mathscr{Q}_2(3K^2, Q_j^2, K^2) 
|Q^2|^{-\frac12} |K^2|^{-\frac12} 
\\ \nonumber 
&\lesssim F^1(Q^1) |Q^1| 
\bigg[ \frac{\ell(K^2)}{\ell(Q^2)} \bigg]^{\frac{\delta_2}{2}} 
F^2(K^2, Q^2) \frac{|K^2|^{\frac12}}{|Q^2|^{\frac12}}, 
\end{align}
provided $\ell(K^2) \simeq \ell(Q^2)$. Moreover, using the cancellation of $h_{K_2}$, the compact partial kernel representation (cf. \eqref{H2}), the H\"{o}lder condition, Lemma \ref{lem:RR} applied to $\ell_2 = \ell(K^2)$, we conclude 
\begin{align}\label{NPP-5}
|\mathscr{J}_{1, 1}^2| 
&\lesssim C(\mathbf{1}_{Q^1}, \mathbf{1}_{Q^1}, a_{Q^1})   
\mathscr{R}_2^1(3Q^2, K^2) |Q^2|^{-\frac12} |K^2|^{-\frac12} 
\\ \nonumber 
&\lesssim F^1(Q^1) |Q^1|  
\bigg[ \frac{\ell(K^2)}{\ell(Q^2)} \bigg]^{\frac{\delta_2}{2}} 
F^2(K^2, 3Q^2) \frac{|K^2|^{\frac12}}{|Q^2|^{\frac12}}. 
\end{align}
Combining the techniques above, we can handle $\mathscr{J}_{1, k}^j$ for all $k=2, \ldots, 6$ and $j=1, 2$. 

Finally, we conclude \eqref{NPP} from the argument above, \eqref{NPP-1}--\eqref{NPP-5}, and the fact that $F^1(Q^1, 3Q^1) \simeq F^1(Q^1, Q^1)$ by \eqref{F3Q}. The proof is complete. 
\end{proof}

\begin{lemma}\label{lem:PPNN-2}
Let $K^2 \in \D_{\rm{good}}^2$ and $Q^2 \in \D^2$ satisfy $K^2 \subset Q^2$. Set 
\begin{align*}
\phi_{Q^2} := \mathbf{1}_{(Q^2)^c} \big(h_{(Q^2)^{(1)}} - \langle h_{(Q^2)^{(1)}} \rangle_{Q^2}\big)
\quad \text{and} \quad
b_{K^2, Q^2} 
:= \langle T(1 \otimes \phi_{Q^2}, 1 \otimes \mathbf{1}_{Q^2}), h_{K^2} \rangle. 
\end{align*}
Then 
\begin{align*}
\|b_{K^2, Q^2}\|_{\BMO(\D^1)} 
\lesssim \bigg[ \frac{\ell(K^2)}{\ell(Q^2)} \bigg]^{\frac{\delta_2}{2}} 
F^2(K^2, Q^2) \frac{|K^2|^{\frac12}}{|Q^2|^{\frac12}}, 
\end{align*}
and 
\begin{align*}
\|P_{\D^1(N)}^{\perp} b_{K^2, Q^2}\|_{\BMO(\D^1)} 
\lesssim \mathcal{F}_N^1 \, \bigg[ \frac{\ell(K^2)}{\ell(Q^2)} \bigg]^{\frac{\delta_2}{2}} 
F^2(K^2, Q^2) \frac{|K^2|^{\frac12}}{|Q^2|^{\frac12}}, 
\end{align*}
where $\mathcal{F}_N^1$ is defined in \eqref{def:F1N}.
\end{lemma}

\begin{proof}
Since the proof is almost the same as that of Lemma \ref{lem:PPNN-1}, we omit the details. 
\end{proof}

\section{A compact bilinear bi-parameter dyadic representation}\label{sec:bbd}
We are now ready to prove Theorem \ref{thm:repre}. To avoid cumbersome notation, we treat the most essential case $m=2$. For this purpose, we extend our method in \cite{CLSY} to the bi-parameter setting.

\subsection{Dyadic reductions}  
Let $\mathbb{E}_{\w} = \mathbb{E}_{\w_1} \mathbb{E}_{\w_2}$ for any $\w = (\w_1, \w_2) \in \Omega_1 \times \Omega_2$. It follows from \eqref{f-mar} that 
\begin{align*}
\langle T(f_1, f_2), f_3 \rangle 
& = \mathbb{E}_{\w} 
\sum_{\substack{I^1, J^1, K^1 \in \D_{\w_1}^1 \\ I^2, J^2, K^2 \in \D_{\w_2}^2}} 
\langle T(\Delta_{I^1} \Delta_{I^2} f_1, \Delta_{J^1} \Delta_{J^2} f_2), 
\Delta_{K^1} \Delta_{K^2} f_3 \rangle. 
\end{align*}
For each $i=1, 2$, we split 
\begin{align*}
\sum_{I^i, J^i, K^i \in \D_{\w_i}^i} 
= \sum_{K^i \in \D_{\w_i}^i} 
\sum_{\substack{I^i, J^i \in \D_{\w_i}^i \\ \ell(K^i) \le \ell(I^i) \\ \ell(K^i) \le \ell(J^i)}}  
+ \sum_{J^i \in \D_{\w_i}^i} 
\sum_{\substack{I^i \in \D_{\w_i}^i \\ \ell(J^i) \le \ell(I^i) \\ \ell(J^i) < \ell(K^i)}} 
+ \sum_{I^i \in \D_{\w_i}^i} 
\sum_{\substack{J^i \in \D_{\w_i}^i \\ \ell(I^i) < \ell(J^i) \\ \ell(I^i) < \ell(K^i)}}. 
\end{align*}
Then $\langle T(f_1, f_2), f_3 \rangle $ can be decomposed into nine symmetrical terms. By symmetry, it suffices to treat the following term: 
\begin{align*}
\mathscr{S} 
:= \mathbb{E}_{\w} 
\sum_{\substack{I^1, J^1, K^1 \in \D_{\w_1}^1 \\ \ell(K^1) \le \ell(I^1) \\ \ell(K^1) \le \ell(J^1)}}
\sum_{\substack{I^2, J^2, K^2 \in \D_{\w_2}^2 \\ \ell(K^2) \le \ell(I^2) \\ \ell(K^2) \le \ell(J^2)}} 
\langle T(\Delta_{I^1} \Delta_{I^2} f_1, \Delta_{J^1} \Delta_{J^2} f_2), 
\Delta_{K^1} \Delta_{K^2} f_3 \rangle. 
\end{align*}  
Setting 
\begin{align*}
E_{2^k}^{\w_i} f 
:= \sum_{I^i \in \D_{\w_i}^i: \, \ell(I^i) = 2^k} 
\langle f \rangle_{I^i} \mathbf{1}_{I^i}, \quad i=1, 2, 
\end{align*}
we invoke \eqref{ddf-2} to arrive at 
\begin{align*}
\mathscr{S} 
&= \mathbb{E}_{\w} 
\sum_{\substack{K^1 \in \D_{\w_1}^1 \\ K^2 \in \D_{\w_2}^2}}  
\big\langle T(E_{\ell(K^1)/2}^{\w_1} E_{\ell(K^2)/2}^{\w_2} f_1, 
E_{\ell(K^1)/2}^{\w_1} E_{\ell(K^2)/2}^{\w_2} f_2), 
\Delta_{K^1} \Delta_{K^2} f_3 \big\rangle 
\\
&= \mathbb{E}_{\w} 
\sum_{\substack{K^1 \in \D_0^1 \\ K^2 \in \D_0^2}} 
\big\langle T(E_{\ell(K^1)/2}^{\w_1} E_{\ell(K^2)/2}^{\w_2} f_1, 
E_{\ell(K^1)/2}^{\w_1} E_{\ell(K^2)/2}^{\w_2} f_2), 
\Delta_{(K^1 +\w_1)} \Delta_{(K^2+\w_2)} f_3 \big\rangle,  
\end{align*}
where the inner trilinear form is denoted by $\Lambda_K(\w)$. 
By definition, $\mathbf{1}_{\rm{good}}(K^1 + \w_1)$ depends on $\w_1^i$ for $2^{-i} \ge \ell(K^1)$, and $\mathbf{1}_{\rm{good}}(K^2 + \w_2)$ depends on $\w_2^j$ for $2^{-j} \ge \ell(K^2)$, while both $E_{\ell(K^1)/2}^{\w_1} E_{\ell(K^2)/2}^{\w_2} f_1$ and $E_{\ell(K^1)/2}^{\w_1} E_{\ell(K^2)/2}^{\w_2} f_2$ depend on $\w_1^i$ and $\w_2^j$ for $2^{-i} <  \ell(K^1)/2 < \ell(K^1)$ and $2^{-j} <  \ell(K^2)/2 < \ell(K^2)$, and $\Delta_{K^1 + \w_1} \Delta_{K^2 + \w_2} f_3$ depends on $\w_1^i$ and $\w_2^j$ for $2^{-i} < \ell(K^1)$ and $2^{-j} < \ell(K^2)$. Thus, by independence, we rewrite  
\begin{align*}
\mathscr{S} 
&= \frac{1}{\pi_{\rm{good}}^{n_1}} \frac{1}{\pi_{\rm{good}}^{n_2}} 
\sum_{\substack{K^1 \in \D_0^1 \\ K^2 \in \D_0^2}}  
\mathbb{E}_{\w} \big[ \mathbf{1}_{\rm{good}}(K^1 + \w_1) 
\mathbf{1}_{\rm{good}}(K^2 + \w_2)\big] 
\mathbb{E}_{\w} \Lambda_K(\w)
\\
&= \frac{1}{\pi_{\rm{good}}^{n_1}} \frac{1}{\pi_{\rm{good}}^{n_2}} 
\sum_{\substack{K^1 \in \D_0^1 \\ K^2 \in \D_0^2}}  
\mathbb{E}_{\w} \big[ \mathbf{1}_{\rm{good}}(K^1 + \w_1) 
\mathbf{1}_{\rm{good}}(K^2 + \w_2) \Lambda_K(\w) \big]
\\
&= \frac{1}{\pi_{\rm{good}}^{n_1}} \frac{1}{\pi_{\rm{good}}^{n_2}} 
\mathbb{E}_{\w} \sum_{\substack{K^1 \in \D_{\w_1, \rm{good}}^1 \\ 
K^2 \in \D_{\w_2, \rm{good}}^2}} \Lambda_K(\w) 
= \frac{1}{\pi_{\rm{good}}^{n_1}} \frac{1}{\pi_{\rm{good}}^{n_2}} 
\mathbb{E}_{\w} \mathscr{S}(\w),  
\end{align*}
where 
\begin{align*}
\mathscr{S}(\w) 
:= \sum_{\substack{K^1 \in \D_{\w_1, \rm{good}}^1 \\ 
K^2 \in \D_{\w_2, \rm{good}}^2}} 
\sum_{\substack{I^1, J^1 \in \D_{\w_1}^1 \\ \ell(K^1) \le \ell(I^1) \\ \ell(K^1) \le \ell(J^1)}}
\sum_{\substack{I^2, J^2 \in \D_{\w_2}^2 \\ \ell(K^2) \le \ell(I^2) \\ \ell(K^2) \le \ell(J^2)}} 
\big\langle T(\Delta_{I^1} \Delta_{I^2} f_1, \Delta_{J^1} \Delta_{J^2} f_2), 
\Delta_{K^1} \Delta_{K^2} f_3 \big\rangle. 
\end{align*}
In what follows, fix an arbitrary $\w = (\w_1, \w_2)$ and denote $\mathscr{S} = \mathscr{S}(\w)$. Simply write $\D^i = \D_{\w_i}^i$ and $\D^i_{\rm{good}} = \D_{\w_i, \rm{good}}^i$ for each $i=1, 2$. Then using \eqref{ddf-2} and that 
\begin{align*}
\sum_{\substack{I^i, J^i \in \D^i \\ \ell(K^i) \le \ell(I^i) \\ \ell(K^i) \le \ell(J^i)}}
= \sum_{\substack{I^i, J^i \in \D^i \\ \ell(K^i) \le \ell(I^i) \\ \ell(I^i) \le \ell(J^i)}}
+ \sum_{\substack{I^i, J^i \in \D^i \\ \ell(K^i) \le \ell(I^i) \\ \ell(J^i) < \ell(I^i)}}, 
\qquad i=1, 2, 
\end{align*} 
we split    
\begin{align*}
\mathscr{S}
= \sum_{j=1}^4 \mathscr{S}^j, 
\end{align*}
where 
\begin{align*}
\mathscr{S}^1 
&:= 
\sum_{\substack{I^1, J^1 \in \D^1, K^1 \in \D_{\rm{good}}^1 \\ \ell(K^1) \le \ell(I^1) = 2 \ell(J^1)}}
\sum_{\substack{I^2, J^2 \in \D^2, K^2 \in \D_{\rm{good}}^2 \\ \ell(K^2) \le \ell(I^2) = 2 \ell(J^2)}} 
\langle T(\Delta_{I^1} \Delta_{I^2} f_1, E_{J^1} E_{J^2} f_2), 
\Delta_{K^1} \Delta_{K^2} f_3 \rangle, 
\\
\mathscr{S}^2
&:= 
\sum_{\substack{I^1, J^1 \in \D^1, K^1 \in \D_{\rm{good}}^1 \\ \ell(K^1) \le \ell(I^1) = 2 \ell(J^1)}}
\sum_{\substack{I^2, J^2 \in \D^2, K^2 \in \D_{\rm{good}}^2 \\ \ell(K^2) \le \ell(I^2) = \ell(J^2)}} 
\langle T(\Delta_{I^1} E_{I^2} f_1, E_{J^1} \Delta_{J^2} f_2), 
\Delta_{K^1} \Delta_{K^2} f_3 \rangle,  
\\
\mathscr{S}^3
&:= 
\sum_{\substack{I^1, J^1 \in \D^1, K^1 \in \D_{\rm{good}}^1 \\ \ell(K^1) \le \ell(I^1) = \ell(J^1)}}
\sum_{\substack{I^2, J^2 \in \D^2, K^2 \in \D_{\rm{good}}^2 \\ \ell(K^2) \le \ell(I^2) = 2\ell(J^2)}} 
\langle T(E_{I^1} \Delta_{I^2} f_1, \Delta_{J^1} E_{J^2} f_2), 
\Delta_{K^1} \Delta_{K^2} f_3 \rangle,  
\\
\mathscr{S}^4
&:= 
\sum_{\substack{I^1, J^1 \in \D^1, K^1 \in \D_{\rm{good}}^1 \\ \ell(K^1) \le \ell(I^1) = \ell(J^1)}} 
\sum_{\substack{I^2, J^2 \in \D^2, K^2 \in \D_{\rm{good}}^2 \\ \ell(K^2) \le \ell(I^2) = \ell(J^2)}} 
\langle T(E_{I^1} E_{I^2} f_1, \Delta_{J^1} \Delta_{J^2} f_2), 
\Delta_{K^1} \Delta_{K^2} f_3 \rangle. 
\end{align*}
Considering symmetry, to show Theorem \ref{thm:repre}, we mainly focus on the term $\mathscr{S}^1$:  
\begin{align*}
\mathscr{S}^1 = 
\sum_{\substack{I^1, J^1 \in \D^1, K^1 \in \D_{\rm{good}}^1 \\ \ell(K^1) \le \ell(I^1) = 2 \ell(J^1)}}
\sum_{\substack{I^2, J^2 \in \D^2, K^2 \in \D_{\rm{good}}^2 \\ \ell(K^2) \le \ell(I^2) = 2 \ell(J^2)}} 
\G_{I, J, K} \, f_1^I \, f_2^J \, f_3^K,  
\end{align*}
where 
\begin{align*}
\G_{I, J, K} &:= \langle T(h_{I^1} \otimes h_{I^2}, h_{J^1}^0 \otimes h_{J^2}^0), 
h_{K^1} \otimes h_{K^2}\rangle, 
\quad 
f_1^I := \langle f_1, h_{I^1} \otimes h_{I^2} \rangle,  
\\
f_2^J &:= \langle f_2, h_{J^1}^0 \otimes h_{J^2}^0 \rangle, 
\quad \text{ and } \quad  
f_3^K := \langle f_3, h_{K^1} \otimes h_{K^2} \rangle.
\end{align*} 
Observe that for each $i=1, 2$, 
\begin{align*} 
\sum_{\substack{I^i, J^i \in \D^i, K^i \in \D_{\rm{good}}^i \\ \ell(K^i) \le \ell(I^i) = 2 \ell(J^i)}} 
&:= \sum_{\substack{I^i, J^i \in \D^i, \, K^i \in \D_{\rm{good}}^i \\ \ell(K^i) \le \ell(I^i) = 2\ell(J^i) \\ 
\max\{\d(K^i, I^i), \, \d(K^i, J^i)\} > 2\ell(K^i)^{\gamma_i} \ell(J^i)^{1-\gamma_i}}} 
\\
&\quad+ \sum_{\substack{I^i, J^i \in \D^i, \, K^i \in \D_{\rm{good}}^i \\ \ell(K^i) \le \ell(I^i) = 2 \ell(J^i) \\ 
\max\{\d(K^i, I^i), \, \d(K^i, J^i)\} \le 2\ell(K^i)^{\gamma_i} \ell(J^i)^{1-\gamma_i} \\ 
K^i \cap I^i = \emptyset \text{ or } K^i \cap J^i = \emptyset \text{ or } K^i = I^i}} 
+ \sum_{\substack{I^i, J^i \in \D^i, \, K^i \in \D_{\rm{good}}^i \\ 
\ell(I^i) = 2\ell(J^i), \, K^i \subset J^i \subset I^i}}. 
\end{align*}
These three parts are called \emph{separated, adjacent}, and \emph{nested} respectively. By symmetry, it is enough to analyze six cases: 
\begin{align*}
&\text{separated/separated, separated/adjacent, separated/nested}, 
\\
&\text{adjacent/adjacent, adjacent/nested, and nested/nested},
\end{align*} 
for which the corresponding summations are successively denoted by $\mathscr{S}^1_1, \ldots, \mathscr{S}^1_6$.

\subsection{Refined functions} 
Before starting the proof, we would like to refine some properties of functions in hypotheses in order to greatly  simplify our estimates below. Following the strategy in \cite{CLSY, CYY}, we may assume that

\begin{list}{\rm (\theenumi)}{\usecounter{enumi}\leftmargin=1.2cm \labelwidth=1cm \itemsep=0.2cm \topsep=.2cm \renewcommand{\theenumi}{\arabic{enumi}}}

\item\label{list:P1} Given $i=1, 2$, the function $(F_1^i, F_2^i, F_3^i) \in \F$ in the compact full kernel representation is the same as that in the compact partial kernel representation. Moreover, $F_1^i$ is monotone increasing while $F_2^i$ and $F_3^i$ are monotone decreasing. 

\item\label{list:P2} Given $i=1, 2$, the function $F^i \in \F^i$ in the compact partial kernel representation is the same as those in the weak compactness property and the diagonal CMO condition. 

\item\label{list:P3} For any harmless constant $\lambda \in (0, \infty)$, $F_j^i(\lambda t)$ and $F^i(\lambda I^i)$ are denoted by $F_j^i(t)$ and $F^i(I^i)$ respectively, since any dilation of them still belongs to the original space. 

\item\label{list:P4} In the compact full and partial kernel representations, when the size estimate happens on $\R^{n_i}$, the bound $F^i(x_{m+1}^i, x_1^i, \ldots, x_m^i)$ will be replaced by 
\begin{align*}
&F^i(x_{m+1}^i, x_1^i, \ldots, x_m^i) 
\\
&= F^i_1 \bigg(\sum_{j=1}^m |x_{m+1}^i - x_j^i| \bigg)
F^i_2 \bigg(\sum_{j=1}^m |x_{m+1}^i - x_j^i| \bigg) 
F^i_3 \Bigg(1 + \frac{\sum_{j=1}^m |x_{m+1}^i + x_j^i| }{\sum_{j=1}^m |x_{m+1}^i - x_j^i|}\Bigg); 
\end{align*} 
when the H\"{o}lder estimate happens to the variable $x_{m+1}^i \in \R^{n_i}$ (namely, $|x_{m+1}^i - \widetilde{x}_{m+1}^i| \leq \frac12 \max\{|x_{m+1}^i - x_j^i|: 1 \le j \le m\}$), the bound $F^i(x_{m+1}^i, x_1^i, \ldots, x_m^i)$ is replaced by 
\begin{align*}
&F^i(x_{m+1}^i, x_1^i, \ldots, x_m^i) 
\\
&= F^i_1 (|x_{m+1}^i - \widetilde{x}_{m+1}^i|)
F^i_2 \bigg(\sum_{j=1}^m |x_{m+1}^i - x_j^i| \bigg) 
F^i_3 \Bigg(1 + \frac{\sum_{j=1}^m |x_{m+1}^i + x_j^i| }{\sum_{j=1}^m |x_{m+1}^i - x_j^i|}\Bigg). 
\end{align*}
Other alternative estimates can be formulated in a similar way. 
\end{list}

Henceforth, for each $i=1, 2$, given $F^i \in \F^i$ and $(F_1^i, F_2^i, F_3^i) \in \F$, we define $\widetilde{F}_2^i$ and $\widetilde{F}_3^i$ as in \eqref{def:F23}. By homogeneity, we may assume that 
\begin{align*}
\|F^i\|_{L^{\infty}} \le 2^{-(r+1)n_i}, \quad 
\|F_k^i\|_{L^{\infty}} \le 1, 
\quad\text{ and }\quad 
\|\widetilde{F}_j^i\|_{L^{\infty}} \le 1, 
\end{align*}
for all $k=1, 2, 3$ and $j=2, 3$.

\subsection{Separated/Separated}\label{sec:SS}
By Lemma \ref{lem:cda} part \ref{cda-1}, for each $i=1, 2$, let $Q^i := I^i \vee J^i \vee K^i$ be the smallest common dyadic ancestor of $I^i$, $J^i$, and $K^i$ such that 
\begin{align}\label{SS-common}
I^i \cup J^i \cup K^i \subset Q^i 
\quad\text{ and }\quad 
\max\{\d(K^i, I^i), \, \d(K^i, J^i)\} \gtrsim \ell(K^i)^{\gamma_i} \ell(Q)^{1-\gamma_i}.
\end{align} 
By the cancellation of $h_{K^1}$ and $h_{K^2}$, the compact full kernel representation (cf. \eqref{H1}), the H\"{o}lder condition, and Lemma \ref{lem:PP}, we obtain 
\begin{align}\label{GSS}
|\G_{I, J, K}|
&\le \prod_{i=1}^2 \mathscr{P}_i(I^i, J^i, K^i) 
|I^i|^{-\frac12} |J^i|^{-\frac12} |K^i|^{-\frac12} 
\\ \nonumber 
&\lesssim \prod_{i=1}^2 \bigg[ \frac{\ell(K^i)}{\ell(Q^i)} \bigg]^{\frac{\delta_i}{2}} 
F^i(K^i, Q^i) \frac{|I^i|^{\frac12} |J^i|^{\frac12} |K^i|^{\frac12}}{|Q^i|^2}.  
\end{align}

Let $C_0 \in (0, \infty)$ be a universal constant chosen later. Set 
\begin{align}\label{SS-aIJKQ}
a_{I, J, K, Q} 
= \frac{\G_{I, J, K}}{C_0 \prod_{i=1}^2 [\ell(K^i)/\ell(Q^i)]^{\delta_i/2}} 
\end{align}
if $I^i, J^i \in \D^i$ and $K^i \in \D_{\rm{good}}^i$ satisfy $\ell(K^i) \le \ell(I^i) = 2\ell(J^i)$ and $\max\{\d(K^i, I^i), \, \d(K^i, J^i)\} > 2\ell(K^i)^{\gamma_i} \ell(J^i)^{1-\gamma_i}$ for each $i=1, 2$, and otherwise set $a_{I, J, K, Q} = 0$. Then it follows from \eqref{SS-common} and \eqref{SS-aIJKQ} that  
\begin{align*}
\mathscr{S}_1^1  
&= \sum_{\substack{0 \le i_1 \le k_1 \\ 0 \le i_2 \le k_2}} 
\sum_{\substack{Q^1 \in \D^1 \\ Q^2 \in \D^2}} 
\sum_{\substack{I^1, J^1 \in \D^1, K^1 \in \D_{\rm{good}}^1 \\ 
\max\{\d(K^1, I^1), \, \d(K^1, J^1)\} > 2 \ell(K^1)^{\gamma_1} \ell(J^1)^{1-\gamma_1} \\ 
2 \ell(J^1) = \ell(I^1) = 2^{-i_1} \ell(Q^1) \\ \ell(K^1) = 2^{-k_1} \ell(Q^1), \, 
I^1 \vee J^1 \vee K^1 = Q^1}} 
\\
&\quad\times \sum_{\substack{I^2, J^2 \in \D^2, K^2 \in \D_{\rm{good}}^2 \\ 
\max\{\d(K^2, I^2), \, \d(K^2, J^2)\} > 2 \ell(K^2)^{\gamma_2} \ell(J^2)^{1-\gamma_2} \\ 
2 \ell(J^2) = \ell(I^2) = 2^{-i_2} \ell(Q^2) \\ \ell(K^2) = 2^{-k_2} \ell(Q^2), \, 
I^2 \vee J^2 \vee K^2 = Q^2}}
\G_{I, J, K} \, f_1^I \, f_2^J \, f_3^K
\\
&= C_0 \sum_{\substack{0 \le i_1 \le k_1 \\ 0 \le i_2 \le k_2}} 
2^{-k_1 \frac{\delta_1}{2}} 2^{-k_2 \frac{\delta_2}{2}}
\sum_{\substack{Q^1 \in \D^1 \\ Q^2 \in \D^2}} 
\sum_{\substack{I^1 \in \D_{i_1}^1(Q^1) \\ 
J^1 \in \D_{i_1+1}^1(Q^1) \\ K^1 \in \D_{k_1}^1(Q^1)}}
\sum_{\substack{I^2 \in \D_{i_2}^2(Q^2) \\ 
J^2 \in \D_{i_2+1}^2(Q^2) \\ K^2 \in \D_{k_2}^2(Q^2)}} 
a_{I, J, K, Q} \, f_1^I \, f_2^J \, f_3^K. 
\end{align*}
For any $k=(k_1, k_2) \in \N^2$, if we define  
\begin{align}\label{SS-FQ}
\mathcal{F}(Q) 
:= \prod_{i=1}^2 F_1^i(\ell(Q^i)) \widetilde{F}_2^i(2^{-k_i} \ell(Q^i)) \widetilde{F}_3^i(Q^i),  
\end{align}
for all $Q = Q^1 \times Q^2 \in \D^1 \times \D^2 := \D$, then \eqref{GSS} gives 
\begin{align}\label{FQ-1}
|a_{I, J, K, Q}| 
\le \mathcal{F}(Q) 
\prod_{i=1}^2 \frac{|I^i|^{\frac12} |J^i|^{\frac12} |K^i|^{\frac12}}{|Q^i|^2}
\end{align} 
with 
\begin{align}\label{FQ-2}
\mathcal{F}(Q) \le 1 
\quad\text{ and }\quad 
\lim_{N \to \infty} \mathcal{F}_N 
:= \lim_{N \to \infty} \sup_{\D} \sup_{Q \notin \D(N)} \mathcal{F}(Q) = 0,  
\end{align}
provided that $C_0$ is large enough and $(F^i_1, F^i_2, F^i_3) \in \F$. As a consequence of \eqref{FQ-1} and \eqref{FQ-2}, we conclude  
\begin{align*}
\mathscr{S}_1^1  
= C_0 \sum_{k_1=0}^{\infty} \sum_{k_2=0}^{\infty} 
2^{-k_1 \frac{\delta_1}{2}} 2^{-k_2 \frac{\delta_2}{2}}
\sum_{i_1=0}^{k_1} \sum_{i_2=0}^{k_2} 
\big\langle \mathbf{S}_{\D}^{\sigma(i, k)}(f_1, f_2), f_3 \big\rangle,  
\end{align*}
where $\sigma(i, k) := \big((i_1, i_2), (i_1+1, i_2+1), (k_1, k_2) \big)$.

\subsection{Separated/Adjacent}\label{sec:SA}
By Lemma \ref{lem:cda} part \ref{cda-2}, there exists a minimal cube $Q^2 := I^2 \vee J^2 \vee K^2 \in \D^2$ such that 
\begin{align}\label{SA-common}
I^2 \cup J^2 \cup K^2 \subset Q^2
\quad\text{ and }\quad \ell(Q^2) \le 2^{\vartheta} \ell(K^2).
\end{align}

\begin{lemma}\label{lem:GSA} 
There holds 
\begin{align}\label{GSA}
|\G_{I, J, K}| 
\lesssim \prod_{i=1}^2 \bigg[ \frac{\ell(K^i)}{\ell(Q^i)} \bigg]^{\frac{\delta_i}{2}} 
\widetilde{F}^i(K^i, Q^i) \frac{|I^i|^{\frac12} |J^i|^{\frac12} |K^i|^{\frac12}}{|Q^i|^2}, 
\end{align}
where $\widetilde{F}^1(K^1, Q^1) := F^1(K^1, Q^1)$ defined in \eqref{def:FKQ} and 
\begin{align*}
\widetilde{F}^2(K^2, Q^2) 
:= \widetilde{F}^2(Q^2) 
= F^2(Q^2, Q^2) 
+ \sum_{\substack{0 \le k_2 \le \vartheta \\ J^2 \in \D_{k_2+1}(Q^2)}} F^2(J^2).
\end{align*}
\end{lemma}

\begin{proof}
First, consider the case $K^2 \cap I^2 = \emptyset$. By the cancellation of $h_{K^1}$, the compact full kernel representation (cf. \eqref{H1}), the mixed size-H\"{o}lder condition, and Lemmas \ref{lem:PP} and \ref{lem:QQ}, we deduce  
\begin{align}\label{GSA-1}
|\G_{I, J, K}| 
&\le \mathscr{P}_1(I^1, J^1, K^1) 
\mathscr{Q}_2(I^2, J^2, K^2) 
\prod_{i=1}^2 |I^i|^{-\frac12} |J^i|^{-\frac12} |K^i|^{-\frac12} 
\\ \nonumber 
&\lesssim \prod_{i=1}^2 \bigg[ \frac{\ell(K^i)}{\ell(Q^i)} \bigg]^{\frac{\delta_i}{2}} 
F^i(K^i, Q^i) \frac{|I^i|^{\frac12} |J^i|^{\frac12} |K^i|^{\frac12}}{|Q^i|^2}. 
\end{align}
Likewise, the same estimate holds in the case $K^2 \cap J^2 = \emptyset$. 

Next, we treat the case $K^2 \cap J^2 \neq \emptyset$ and $K^2 = I^2$. Obviously, $J^2 \in \ch(K^2)$. Then we rewrite 
\begin{align}\label{GSA-2}
\G_{I, J, K}
= \sum_{K_1^2, K_2^2 \in \ch(K^2)} \G_{I, J, K}^{K_1^2, K_2^2}, 
\end{align}
where 
\begin{align*}
\G_{I, J, K}^{K_1^2, K_2^2} 
:= \langle h_{K^2} \rangle_{K_1^2} \langle h_{K^2} \rangle_{K_2^2} |J^2|^{-\frac12} 
\big\langle T(h_{I^1} \otimes \mathbf{1}_{K_1^2}, h_{J^1}^0 \otimes \mathbf{1}_{J^2}), 
h_{K^1} \otimes \mathbf{1}_{K_2^2} \big\rangle. 
\end{align*}
If $K_1^2 = J^2 = K_2^2$, then the H\"{o}lder condition, the compact partial kernel representation (cf. \eqref{H2}), and Lemma \ref{lem:PP} imply   
\begin{align}\label{GSA-3}
|\G_{I, J, K}^{K_1^2, K_2^2}|
& \le \mathscr{P}_1(I^1, J^1, K^1) \, 
C(\mathbf{1}_{J^2}, \mathbf{1}_{J^2}, \mathbf{1}_{J^2})
\prod_{i=1}^2 |I^i|^{-\frac12} |J^i|^{-\frac12} |K^i|^{-\frac12}
\\ \nonumber 
&\lesssim \bigg[ \frac{\ell(K^1)}{\ell(Q^1)} \bigg]^{\frac{\delta_1}{2}} 
F^1(K^1, Q^1) \frac{|I^1|^{\frac12} |J^1|^{\frac12} |K^1|^{\frac12}}{|Q^1|^2}
\\ \nonumber 
&\quad\times \bigg[ \frac{\ell(K^2)}{\ell(Q^2)} \bigg]^{\frac{\delta_2}{2}} 
F^2(J^2) \frac{|I^2|^{\frac12} |J^2|^{\frac12} |K^2|^{\frac12}}{|Q^2|^2}, 
\end{align}
provided that $\ell(I^2) \simeq \ell(J^2) \simeq \ell(K^2) \simeq \ell(Q^2)$. If $K_2^2 \neq J^2$, then $K_2^2 \subset 3J^2 \setminus J^2$, which together with Lemmas \ref{lem:PP} and \ref{lem:QQ} gives 
\begin{align}\label{GSA-4}
|\G_{I, J, K}^{K_1^2, K_2^2}|
& \le \mathscr{P}_1(I^1, J^1, K^1) \, 
\mathscr{Q}_2(K_1^2, J^2, K_2^2) 
\prod_{i=1}^2 |I^i|^{-\frac12} |J^i|^{-\frac12} |K^i|^{-\frac12}
\\ \nonumber 
&\lesssim \bigg[ \frac{\ell(K^1)}{\ell(Q^1)} \bigg]^{\frac{\delta_1}{2}} 
F^1(K^1, Q^1) \frac{|I^1|^{\frac12} |J^1|^{\frac12} |K^1|^{\frac12}}{|Q^1|^2}
\\ \nonumber 
&\quad \times \bigg[ \frac{\ell(K_2^2)}{\ell(Q^2)} \bigg]^{\frac{\delta_2}{2}} 
F^2(K_2^2, Q^2) \frac{|K_1^2|^{\frac12} |J^2|^{\frac12} |K_2^2|^{\frac12}}{|Q^2|^2}
\\ \nonumber 
&\lesssim \prod_{i=1}^2 \bigg[ \frac{\ell(K^i)}{\ell(Q^i)} \bigg]^{\frac{\delta_i}{2}} 
F^i(K^i, Q^i) \frac{|I^i|^{\frac12} |J^i|^{\frac12} |K^i|^{\frac12}}{|Q^i|^2}.  
\end{align}
By symmetry, the same bound also holds in the case $K_1^2 \neq J^2$. Accordingly, \eqref{GSA} is a consequence of \eqref{GSA-1}--\eqref{GSA-4} and that $F^2(K^2, Q^2) \simeq F^2(Q^2, Q^2)$ because of 
$\ell(K^2) \simeq \ell(Q^2)$.
\end{proof}

Set 
\begin{align*}
a_{I, J, K, Q} 
= \frac{\G_{I, J, K}}{C_0 \prod_{i=1}^2 [\ell(K^i)/\ell(Q^i)]^{\delta_i/2}} 
\end{align*}
if $I^i, J^i \in \D^i$ and $K^i \in \D_{\rm{good}}^i$ with $\ell(K^i) \le \ell(I^i) = 2\ell(J^i)$, $i=1, 2$, satisfy 
$\max\{\d(K^1, I^1), \, \d(K^1, J^1)\} > 2 \ell(K^1)^{\gamma_1} \ell(J^1)^{1-\gamma_1}$,  
$\max\{\d(K^2, I^2), \, \d(K^2, J^2)\} \le 2 \ell(K^2)^{\gamma_2} \ell(J^2)^{1-\gamma_2}$, and either $K^2 \cap I^2 = \emptyset$ or $K^2 \cap J^2 = \emptyset$ or $K^2 = I^2$, and otherwise set $a_{I, J, K, Q} = 0$. In view of \eqref{SS-common} and \eqref{SA-common}, this allows us to write  
\begin{align*}
\mathscr{S}_2^1   
&= \sum_{\substack{0 \le i_1 \le k_1 \\ 0 \le i_2 \le k_2 \le \vartheta}} 
\sum_{\substack{Q^1 \in \D^1 \\ Q^2 \in \D^2}} 
\sum_{\substack{I^1, J^1 \in \D^1, K^1 \in \D_{\rm{good}}^1 \\ 
\max\{\d(K^1, I^1), \, \d(K^1, J^1)\} > 2 \ell(K^1)^{\gamma_1} \ell(J^1)^{1-\gamma_1} \\ 
2 \ell(J^1) = \ell(I^1) = 2^{-i_1} \ell(Q^1) \\ \ell(K^1) = 2^{-k_1} \ell(Q^1), \, 
I^1 \vee J^1 \vee K^1 = Q^1}} 
\\
&\quad\times \sum_{\substack{I^2, J^2 \in \D^2, \, K^2 \in \D_{\rm{good}}^2 \\ 
\max\{\d(K^2, I^2), \, \d(K^2, J^2)\} \le 2 \ell(K^2)^{\gamma_2} \ell(J^2)^{1-\gamma_2} \\ 
K^2 \cap I^2 = \emptyset \text{ or } K^2 \cap J^2 = \emptyset \text{ or } K^2 = I^2\\ 
2 \ell(J^2) = \ell(I^2) = 2^{-i_2} \ell(Q^2) \\ \ell(K^2) = 2^{-k_2} \ell(Q^2), \, 
I^2 \vee J^2 \vee K^2 = Q^2}} 
\G_{I, J, K} \, f_1^I \, f_2^J \, f_3^K
\\
&= C_0 \sum_{\substack{0 \le i_1 \le k_1 \\ 0 \le i_2 \le k_2 \le \vartheta}}   
2^{-k_1 \frac{\delta_1}{2}} 2^{-k_2 \frac{\delta_2}{2}}
\sum_{\substack{Q^1 \in \D^1 \\ Q^2 \in \D^2}} 
\sum_{\substack{I^1 \in \D_{i_1}^1(Q^1) \\ 
J^1 \in \D_{i_1+1}^1(Q^1) \\ K^1 \in \D_{k_1}^1(Q^1)}}
\sum_{\substack{I^2 \in \D_{i_2}^2(Q^2) \\ 
J^2 \in \D_{i_2+1}^2(Q^2) \\ K^2 \in \D_{k_2}^2(Q^2)}} 
a_{I, J, K, Q} \, f_1^I \, f_2^J \, f_3^K. 
\end{align*}
For any $k_1 \in \N^2$ and $Q=(Q^1, Q^2) \in \D^1 \times \D^2 := \D$, denote 
\begin{align*}
\mathcal{F}(Q) 
:= F_1^1(\ell(Q^1)) \widetilde{F}_2^1(2^{-k_1} \ell(Q^1)) \widetilde{F}_3^1(Q^1) 
\widetilde{F}^2(Q^2). 
\end{align*}
Then by Lemma \ref{lem:GSA} and that $F^i \in \F^i$ and $(F_1^i, F_2^i, F_3^i) \in \F$, we see that \eqref{FQ-1} and \eqref{FQ-2} hold whenever $C_0$ is large. Hence, we have 
\begin{align*}
\mathscr{S}_2^1 
= C_0 \sum_{k_1=0}^{\infty} \sum_{k_2=0}^{\vartheta}
2^{-k_1 \frac{\delta_1}{2}} 2^{-k_2 \frac{\delta_2}{2}}
\sum_{i_1=0}^{k_1} \sum_{i_2=0}^{k_2} 
\big\langle \mathbf{S}_{\D}^{\sigma(i, k)}(f_1, f_2), f_3 \big\rangle, 
\end{align*}
where $\sigma(i, k) := \big((i_1, i_2), (i_1+1, i_2+1), (k_1, k_2) \big)$.

\subsection{Separated/Nested}\label{sec:SN}
In this case, denote $I^2 := (J^2)^{(1)}$. Noting that $h_{I^2} \mathbf{1}_{J^2} = \langle h_{I^2} \rangle_{J^2} \mathbf{1}_{J^2}$, we perform the decomposition 
\begin{align}\label{SNG}
\G_{I, J, K}
= \G_{I, J, K}^1 + \G_{I, J, K}^2 + \G_{I, J, K}^3, 
\end{align}
where 
\begin{align*}
\G_{I, J, K}^1 
&:= \langle h_{I^2} \rangle_{J^2} |J^2|^{-\frac12} 
\langle T(h_{I^1} \otimes 1, h_{J^1}^0 \otimes 1), 
h_{K^1} \otimes h_{K^2} \rangle, 
\\
\G_{I, J, K}^2 
&:= \langle h_{I^2} \rangle_{J^2} |J^2|^{-\frac12} 
\langle T(h_{I^1} \otimes 1, h_{J^1}^0 \otimes \mathbf{1}_{(J^2)^c}), 
h_{K^1} \otimes h_{K^2} \rangle, 
\\
\G_{I, J, K}^3 
&:= |J^2|^{-\frac12}  
\langle T(h_{I^1} \otimes \phi_{J^2}, h_{J^1}^0 \otimes \mathbf{1}_{J^2}), 
h_{K^1} \otimes h_{K^2} \rangle,  
\end{align*}
with $\phi_{J^2} := \mathbf{1}_{(J^2)^c} \big(h_{I^2} - \langle h_{I^2} \rangle_{J^2}\big)$. We then define the  summation $\mathscr{S}_{3, k}^1$ corresponding to $\G_{I, J, K}^k$, $k=1, 2, 3$.

\medskip 
\noindent{\bf $\bullet$ Partial paraproducts.} Rewrite 
\begin{align*}
\mathscr{S}_{3, 1}^1   
&= \sum_{\substack{I^1, J^1 \in \D^1, K^1 \in \D_{\rm{good}}^1 \\ 
\max\{\d(K^1, I^1), \, \d(K^1, J^1)\} > 2 \ell(K^1)^{\gamma_1} \ell(J^1)^{1-\gamma_1} \\ 
\ell(K^1) \le \ell(I^1) = 2 \ell(J^1)}} 
\sum_{K^2 \in \D_{\rm{good}}^2}  
\sum_{\substack{J^2 \in \D^2 \\ J^2 \supset K^2}}  
\langle f_3, h_{K^1} \otimes h_{K^2} \rangle
\\
&\quad\times \langle T(h_{I^1} \otimes 1, h_{J^1}^0 \otimes 1), 
h_{K^1} \otimes h_{K^2} \rangle
\big\langle \langle f_1, h_{I^1} \rangle, h_{I^2} \big\rangle 
\big\langle \langle f_2, h_{J^1}^0 \rangle \big\rangle_{J^2} 
\langle h_{I^2} \rangle_{J^2}. 
\end{align*}
Much as above, for $\mathscr{S}^2$, the separated/nested term is given by 
\begin{align*}
\mathscr{S}_3^2   
&:= \sum_{\substack{I^1, J^1 \in \D^1, K^1 \in \D_{\rm{good}}^1 \\ 
\max\{\d(K^1, I^1), \, \d(K^1, J^1)\} > 2 \ell(K^1)^{\gamma_1} \ell(J^1)^{1-\gamma_1} \\ 
\ell(K^1) \le \ell(I^1) = 2 \ell(J^1)}} 
\sum_{K^2 \in \D_{\rm{good}}^2}  
\sum_{\substack{J^2 \in \D^2 \\ J^2 \supset K^2}}   
\langle f_3, h_{K^1} \otimes h_{K^2} \rangle
\\
&\quad\times \langle T(h_{I^1} \otimes \mathbf{1}_{I^2}, h_{J^1}^0 \otimes h_{I^2}), 
h_{K^1} \otimes h_{K^2} \rangle
\big\langle \langle f_1, h_{I^1} \rangle \big\rangle_{I^2} 
\big\langle \langle f_2, h_{J^1}^0 \rangle, h_{I^2} \big\rangle.  
\end{align*}
As done in \eqref{SNG}, we obtain the corresponding term 
\begin{align*}
\mathscr{S}_{3, 1}^2   
&= \sum_{\substack{I^1, J^1 \in \D^1, K^1 \in \D_{\rm{good}}^1 \\ 
\max\{\d(K^1, I^1), \, \d(K^1, J^1)\} > 2 \ell(K^1)^{\gamma_1} \ell(J^1)^{1-\gamma_1} \\ 
\ell(K^1) \le \ell(I^1) = 2 \ell(J^1)}} 
\sum_{K^2 \in \D_{\rm{good}}^2}  
\sum_{\substack{J^2 \in \D^2 \\ J^2 \supset K^2}}   
\langle f_3, h_{K^1} \otimes h_{K^2} \rangle
\\
&\quad\times \langle T(h_{I^1} \otimes 1, h_{J^1}^0 \otimes 1), 
h_{K^1} \otimes h_{K^2} \rangle
\big\langle \langle f_1, h_{I^1} \rangle \big\rangle_{I^2} 
\big\langle \langle f_2, h_{J^1}^0 \rangle, h_{I^2} \big\rangle 
\langle h_{I^2} \rangle_{J^2}.  
\end{align*}
Note that 
\begin{align*}
\langle g, h_{I^{(1)}} \rangle \langle h_{I^{(1)}} \rangle_I
= \langle g \rangle_I - \langle g \rangle_{I^{(1)}}, 
\end{align*}
which gives 
\begin{align}\label{gghh}
\big[\langle g_1, h_{I^{(1)}} \rangle \langle g_2 \rangle_I 
+ \langle g_1 \rangle_{I^{(1)}} \langle g_2, h_{I^{(1)}} \rangle \big] 
\langle h_{I^{(1)}} \rangle_I
= \langle g_1 \rangle_{I} \langle g_2 \rangle_I 
- \langle g_1 \rangle_{I^{(1)}} \langle g_2 \rangle_{I^{(1)}}. 
\end{align}
Hence, 
\begin{align*}
\mathscr{S}_{3, 1}^1 + \mathscr{S}_{3, 1}^2  
&= \sum_{\substack{I^1, J^1 \in \D^1, K^1 \in \D_{\rm{good}}^1 \\ 
\max\{\d(K^1, I^1), \, \d(K^1, J^1)\} > 2 \ell(K^1)^{\gamma_1} \ell(J^1)^{1-\gamma_1} \\ 
\ell(K^1) \le \ell(I^1) = 2 \ell(J^1)}} 
\sum_{K^2 \in \D_{\rm{good}}^2}  
\langle f_3, h_{K^1} \otimes h_{K^2} \rangle
\\
&\quad\times \langle T(h_{I^1} \otimes 1, h_{J^1}^0 \otimes 1), 
h_{K^1} \otimes h_{K^2} \rangle
\big\langle \langle f_1, h_{I^1} \rangle \big\rangle_{K^2} 
\big\langle \langle f_2, h_{J^1}^0 \rangle \big\rangle_{K^2} 
\\
&= \sum_{\substack{0 \le i_1 \le k_1 \\ Q^1 \in \D^1}} 
\sum_{\substack{I^1, J^1 \in \D^1, K^1 \in \D_{\rm{good}}^1 \\ 
\max\{\d(K^1, I^1), \, \d(K^1, J^1)\} > 2 \ell(K^1)^{\gamma_1} \ell(J^1)^{1-\gamma_1} \\ 
2 \ell(J^1) = \ell(I^1) = 2^{-i_1} \ell(Q^1) \\ \ell(K^1) = 2^{-k_1} \ell(Q^1), \, I^1 \vee J^1 \vee K^1 = Q^1}} 
\sum_{K^2 \in \D_{\rm{good}}^2}  
\langle f_3, h_{K^1} \otimes h_{K^2} \rangle
\\
&\quad\times \langle T(h_{I^1} \otimes 1, h_{J^1}^0 \otimes 1), 
h_{K^1} \otimes h_{K^2} \rangle
\langle f_1, h_{I^1} \otimes \overline{h}_{K^2} \rangle
\langle f_2, h_{J^1}^0 \otimes \overline{h}_{K^2} \rangle. 
\end{align*}
Let $Q=(Q^1, Q^2) \in \D^1 \times \D^2 =: \D$. Set 
\begin{align*}
a_{I^1, J^1, K^1, Q} 
= \frac{\langle T(h_{I^1} \otimes 1, h_{J^1}^0 \otimes 1), h_{K^1} \otimes h_{Q^2} \rangle}{C_0 [\ell(K^1)/\ell(Q^1)]^{\delta_1/2}} 
\end{align*}
if $I^1, J^1 \in \D^1$, $K^1 \in \D_{\rm{good}}^1$, and $Q^2 \in \D_{\rm{good}}^2$ satisfy $\ell(K^1) \le \ell(I^1) = 2\ell(J^1)$ and $\max\{\d(K^1, I^1), \, \d(K^1, J^1)\} > 2\ell(K^1)^{\gamma_1} \ell(J^1)^{1-\gamma_1}$, and otherwise set $a_{I^1, J^1, K^1, Q} = 0$. Then we rewrite 
\begin{align}\label{S312}
\mathscr{S}_{3, 1}^1 + \mathscr{S}_{3, 1}^2  
&= C_0 \sum_{0 \le i_1 \le k_1}  2^{-k_1 \frac{\delta_1}{2}} 
\sum_{\substack{Q^1 \in \D^1 \\ Q^2 \in \D^2}} 
\sum_{\substack{I^1 \in \D_{i_1}^1(Q^1) \\ 
J^1 \in \D_{i_1+1}^1(Q^1) \\ K^1 \in \D_{k_1}^1(Q^1)}} 
a_{I^1, J^1, K^1, Q} 
\\ \nonumber 
&\quad\times \langle f_1, h_{I^1} \otimes \overline{h}_{Q^2} \rangle 
\langle f_2, h_{J^1}^0 \otimes \overline{h}_{Q^2} \rangle
\langle f_3, h_{K^1} \otimes h_{Q^2} \rangle 
\\ \nonumber 
&= C_0 \sum_{k_1=0}^{\infty} 2^{-k_1 \frac{\delta_1}{2}} 
\sum_{i_1=0}^{k_1} \langle \mathbf{P}_{\D}^{1, i_1, i_1+1, k_1}(f_1, f_2), f_3 \rangle, 
\end{align}
where $\overline{h}_{Q^2} := \frac{\mathbf{1}_{Q^2}}{|Q^2|}$. By Lemma \ref{lem:PPSN}, $\mathbf{P}_{\D}^{1, i_1, i_1+1, k_1}$ is a partial paraproduct (cf. Definition \ref{def:pp}), where 
\begin{align*}
\mathcal{F}^1(Q^1) 
:= F_1^1(\ell(Q^1)) \widetilde{F}_2^1(2^{-k_1} \ell(Q^1)) \widetilde{F}_3^1(Q^1). 
\end{align*}

\medskip 
\noindent{\bf $\bullet$ Bilinear dyadic shifts.} 
Let us turn to the estimates for $\mathscr{S}_{3, 2}^1$ and $\mathscr{S}_{3, 3}^1$. 

\begin{lemma}\label{lem:GSN23}
For each $k=2, 3$, there holds 
\begin{align*}
|\G_{I, J, K}^k| 
\lesssim \bigg[ \frac{\ell(K^1)}{\ell(Q^1)} \bigg]^{\frac{\delta_1}{2}} 
F^1(K^1, Q^1) \frac{|I^1|^{\frac12} |J^1|^{\frac12} |K^1|^{\frac12}}{|Q^1|^2}
\bigg[ \frac{\ell(K^2)}{\ell(J^2)} \bigg]^{\frac{\delta_2}{2}} 
F^2(K^2, J^2) \frac{|K^2|^{\frac12}}{|J^2|}, 
\end{align*}
where $F^i$ is defined in \eqref{def:FKQ}. 
\end{lemma}

\begin{proof}
First, let us estimate $\G_{I, J, K}^2$. In the case $\ell(J^2) > 2^{\vartheta} \ell(K^2)$, we have $\G_{I, J, K}^2 = \G_{I, J, K}^{2, 1} + \G_{I, J, K}^{2, 2}$, where 
\begin{align*}
\G_{I, J, K}^{2, 1}
&:= \langle h_{I^2} \rangle_{J^2} |J^2|^{-\frac12}
\langle T(h_{I^1} \otimes \mathbf{1}_{3K^2}, h_{J^1}^0 \otimes \mathbf{1}_{(J^2)^c}), 
h_{K^1} \otimes h_{K^2} \rangle, 
\\
\G_{I, J, K}^{2, 2}
&:= \langle h_{I^2} \rangle_{J^2} |J^2|^{-\frac12} 
\langle T(h_{I^1} \otimes \mathbf{1}_{(3K^2)^c}, h_{J^1}^0 \otimes \mathbf{1}_{(J^2)^c}), 
h_{K^1} \otimes h_{K^2} \rangle.
\end{align*}
For all $x_2 \in K^2$ and $z_2 \in (J^2)^c$, the goodness of $K^2$ gives 
\begin{align}\label{dxz}
|x_2 - z_2| 
&\ge \d(K^2, (J^2)^c) 
> 2 \ell(K^2)^{\gamma_2} \ell(J^2)^{1-\gamma_2} 
\\ \nonumber 
&\ge \ell(K^2)^{\frac12} \ell(J^2)^{\frac12} 
\ge \ell(K^2) \ge 2|x_2 - c_{K^2}|, 
\end{align}
which, along with the cancellation of $h_{K^1}$ and $h_{K^2}$, the compact full kernel representation (cf. \eqref{H1}), the H\"{o}lder condition, Lemma \ref{lem:PP} and Lemma \ref{lem:RR} applied to $\ell_2 = \ell(K^2)^{\frac12} \ell(J^2)^{\frac12}$, implies  
\begin{align}\label{GG21}
&|\G_{I, J, K}^{2, j}| 
\le \mathscr{P}_1(I^1, J^1, K^1) \mathscr{R}_2^j(J^2, K^2) 
\prod_{i=1}^2 |I^i|^{-\frac12} |J^i|^{-\frac12} |K^i|^{-\frac12} 
\\ \nonumber 
&\lesssim \bigg[ \frac{\ell(K^1)}{\ell(Q^1)} \bigg]^{\frac{\delta_1}{2}} 
F_1(K^1, Q^1) \frac{|I^1|^{\frac12} |J^1|^{\frac12} |K^1|^{\frac12}}{|Q^1|^2}
\bigg[ \frac{\ell(K^2)}{\ell(J^2)} \bigg]^{\frac{\delta_2}{2}} 
F^2(K^2, J^2) \frac{|K^2|^{\frac12}}{|J^2|},  
\end{align}
for every $j=1, 2$.

To treat the case $\ell(J^2) \le 2^{\vartheta} \ell(K^2)$, we split 
\begin{align}\label{GLE}
\G_{I, J, K}^2 
&= \langle h_{I^2} \rangle_{J^2} |J^2|^{-\frac12} 
\langle T(h_{I^1} \otimes \mathbf{1}_{3K^2}, 
h_{J^1}^0 \otimes \mathbf{1}_{3J^2 \setminus J^2}), 
h_{K^1} \otimes h_{K^2} \rangle 
\\ \nonumber
&\quad+ \langle h_{I^2} \rangle_{J^2} |J^2|^{-\frac12} 
\langle T(h_{I^1} \otimes \mathbf{1}_{3K^2}, 
h_{J^1}^0 \otimes \mathbf{1}_{(3J^2)^c}), 
h_{K^1} \otimes h_{K^2} \rangle 
\\ \nonumber
&\quad+ \langle h_{I^2} \rangle_{J^2} |J^2|^{-\frac12} 
\langle T(h_{I^1} \otimes \mathbf{1}_{(3K^2)^c}, 
h_{J^1}^0 \otimes \mathbf{1}_{3J^2 \setminus J^2}), 
h_{K^1} \otimes h_{K^2} \rangle 
\\ \nonumber 
&\quad+ \langle h_{I^2} \rangle_{J^2} |J^2|^{-\frac12} 
\langle T(h_{I^1} \otimes \mathbf{1}_{(3K^2)^c}, 
h_{J^1}^0 \otimes \mathbf{1}_{(3J^2)^c}), 
h_{K^1} \otimes h_{K^2} \rangle. 
\end{align}
The first term is similar to $\G_{I, J, K}$ in the case $K^2 \cap J^2 = \emptyset$ in Section \ref{sec:SA}, which along with \eqref{GSA-1} yields the desired bound. The second and last terms are analogous to $\G_{I, J, K}^{2, 1}$ and $\G_{I, J, K}^{2, 2}$ respectively, but now $\ell_2 = \ell(K^2)$. The third term is symmetric to the second one. Eventually, the last three terms are dominated by 
\begin{align*}
\bigg[ \frac{\ell(K^1)}{\ell(Q^1)} \bigg]^{\frac{\delta_1}{2}} 
F_1(K^1, Q^1) \frac{|I^1|^{\frac12} |J^1|^{\frac12} |K^1|^{\frac12}}{|Q^1|^2}
F^2(K^2, J^2) \frac{|K^2|^{\frac12}}{|J^2|}, 
\end{align*}
which together with $\ell(J^2) \simeq \ell(K^2)$ gives the estimate as desired.

Next, we bound $\G_{I, J, K}^3$. If $\ell(J^2) > 2^{\vartheta} \ell(K^2)$, then by \eqref{dxz}, $\G_{I, J, K}^3$ is similar to $\G_{I, J, K}^{2, 2}$. If $\ell(J^2) \le 2^{\vartheta} \ell(K^2)$, we split 
\begin{align*}
\G_{I, J, K}^3 
&= |J^2|^{-\frac12}  
\langle T(h_{I^1} \otimes (\phi_{J^2} \mathbf{1}_{3J^2 \setminus J^2}), 
h_{J^1}^0 \otimes \mathbf{1}_{J^2}), 
h_{K^1} \otimes h_{K^2} \rangle
\\
&\quad + |J^2|^{-\frac12}  
\langle T(h_{I^1} \otimes (\phi_{J^2} \mathbf{1}_{(3J^2)^c}), 
h_{J^1}^0 \otimes \mathbf{1}_{J^2}), 
h_{K^1} \otimes h_{K^2} \rangle,  
\end{align*}
which are similar to the first two terms in \eqref{GLE} respectively. 
\end{proof}

Recall that $I^2 = (J^2)^{(1)}$. Now set  $Q^2 = I^2$ and 
\begin{align*}
a_{I, J, K, Q} 
= \frac{\G_{I, J, K}^2 + \G_{I, J, K}^3}{C_0 \prod_{i=1}^2 [\ell(K^i)/\ell(Q^i)]^{\delta_i/2}} 
\end{align*}
if $I^1, J^1 \in \D^1$, $K^1 \in \D_{\rm{good}}^1$, $J^2 \in \D^2$, and $K^2 \in \D_{\rm{good}}^2$ satisfy 
$\max\{\d(K^1, I^1), \, \d(K^1, J^1)\} > 2 \ell(K^1)^{\gamma_1} \ell(J^1)^{1-\gamma_1}$,  $\ell(K^1) \le \ell(I^1) = 2 \ell(J^1)$, and $K^2 \subset J^2$, and otherwise set $a_{I, J, K, Q} = 0$. 
For any $k=(k_1, k_2) \in \N^2$ and $Q=Q^1 \times Q^2 \in \D^1 \times \D^2$, denote 
\begin{align*}
\mathcal{F}(Q) 
:= \prod_{i=1}^2 F_1^i(\ell(Q^i)) \widetilde{F}_2^i(2^{-k_i} \ell(Q^i)) \widetilde{F}_3^i(Q^i). 
\end{align*}
Then Lemma \ref{lem:GSN23} and that $(F_1^i, F_2^i, F_3^i) \in \F$ imply both \eqref{FQ-1} and \eqref{FQ-2} hold whenever $C_0$ is large. Consequently, by \eqref{SS-common} and $K^2 \subset J^2 \in \ch(Q^2)$, we deduce 
\begin{align*}
\mathscr{S}_{3, 2}^1 + \mathscr{S}_{3, 3}^1 
&= C_0 \sum_{\substack{0 \le i_1 \le k_1 \\ k_2 \ge 1}}   
2^{-k_1 \frac{\delta_1}{2}} 2^{-k_2 \frac{\delta_2}{2}}
\sum_{\substack{Q^1 \in \D^1 \\ Q^2 \in \D^2}} 
\sum_{\substack{I^1 \in \D_{i_1}^1(Q^1) \\ 
J^1 \in \D_{i_1+1}^1(Q^1) \\ K^1 \in \D_{k_1}^1(Q^1)}}
\\
&\qquad\times \sum_{\substack{K^2 \in \D_{k_2}^2(Q^2)}} 
a_{I, J, K, Q} \, f_1^I \, f_2^J \, f_3^K
\\
&= C_0 \sum_{k_1=0}^{\infty} \sum_{k_2=1}^{\infty} 
2^{-k_1 \frac{\delta_1}{2}} 2^{-k_2 \frac{\delta_2}{2}}
\sum_{i_1=0}^{k_1} 
\big\langle \mathbf{S}_{\D}^{\sigma(i, k)}(f_1, f_2), f_3 \big\rangle,  
\end{align*}
where $\sigma(i, k) := \big((i_1, 0), (i_1+1, 1), (k_1, k_2) \big)$.

\subsection{Adjacent/Adjacent}\label{sec:AA} 
By Lemma \ref{lem:cda} part \ref{cda-2}, for each $i=1, 2$, there exists a minimal cube $Q^i := I^i \vee J^i \vee K^i \in \D^i$ such that 
\begin{align}\label{AA-common}
I^i \cup J^i \cup K^i \subset Q^i
\quad\text{ and }\quad \ell(Q^i) \le 2^{\vartheta} \ell(K^i).
\end{align}

In the current scenario, we first prove the following estimate. 

\begin{lemma}\label{lem:GAA}
For each $i=1, 2$, let $I^i, J^i \in \D^i$ and $K^i \in \D_{\rm{good}}^i$ satisfy $\ell(K^i) \le \ell(I^i) = 2\ell(J^i)$ and either $K^i \cap I^i = \emptyset$ or $K^i \cap J^i = \emptyset$ or $K^i = I^i$. Then 
\begin{align*}
|\G_{I, J, K}| 
\lesssim \prod_{i=1}^2 \bigg[ \frac{\ell(K^i)}{\ell(Q^i)} \bigg]^{\frac{\delta_i}{2}} 
\mathcal{F}^i(Q^i) \frac{|I^i|^{\frac12} |J^i|^{\frac12} |K^i|^{\frac12}}{|Q^i|^2}, 
\end{align*}
where 
\begin{align}\label{def:FQ}
\mathcal{F}^i(Q^i) := F^i(Q^i, Q^i) 
+ \sum_{\substack{0 \le k_i \le \vartheta \\ J^i \in \D_{k_i+1}(Q^i)}} F^i(J^i) 
\end{align}
with $F^i(\cdot, \cdot)$ defined in \eqref{def:FKQ}. 
\end{lemma}

\begin{proof} 
By similarity and symmetry, it suffices to treat the following cases: 
\begin{align*}
\text{(i) } \, & K^1 \cap I^1 = \emptyset \text{ and } K^2 \cap I^2 = \emptyset; 
\\ 
\text{(ii) } \, & K^1 \cap I^1 = \emptyset, \, K^2 \cap J^2 \neq \emptyset \text{ and } K^2 = I^2; 
\\ 
\text{(iii) } \, & K^i \cap J^i \neq \emptyset \text{ and } K^i = I^i, \quad i=1, 2. 
\end{align*} 
In case (i), by the cancellation of $h_{K^1}$ and $h_{K^2}$, the compact full kernel representation (cf. \eqref{H1}), the size condition, and Lemma \ref{lem:QQ}, we deduce  
\begin{align}\label{GAA-1}
|\G_{I, J, K}| 
&\le \prod_{i=1}^2 \mathscr{Q}_i(I^i, J^i, K^i) 
|I^i|^{-\frac12} |J^i|^{-\frac12} |K^i|^{-\frac12} 
\\ \nonumber 
&\lesssim \prod_{i=1}^2 \bigg[ \frac{\ell(K^i)}{\ell(Q^i)} \bigg]^{\frac{\delta_i}{2}} 
F^i(K^i, Q^i) \frac{|I^i|^{\frac12} |J^i|^{\frac12} |K^i|^{\frac12}}{|Q^i|^2}. 
\end{align}
In case (ii), we have $J^2 \in \ch(K^2)$ and 
\begin{align*}
\G_{I, J, K}
= \sum_{K_1^2, K_2^2 \in \ch(K^2)} \G_{I, J, K}^{K_1^2, K_2^2}, 
\end{align*}
where 
\begin{align*}
\G_{I, J, K}^{K_1^2, K_2^2} 
:= |J^1|^{-\frac12} |J^2|^{-\frac12} 
\langle h_{K^2} \rangle_{K_1^2} \langle h_{K^2} \rangle_{K_2^2} 
\big\langle T(h_{I^1} \otimes \mathbf{1}_{K_1^2}, \mathbf{1}_{J^1} \otimes \mathbf{1}_{J^2}), 
h_{K^1} \otimes \mathbf{1}_{K_2^2} \big\rangle. 
\end{align*}
If $K_1^2 = K_2^2 = J^2$, then the compact partial kernel representation (cf. \eqref{H2}), the size condition, and Lemma \ref{lem:QQ} imply   
\begin{align}\label{GAA-3}
|\G_{I, J, K}^{K_1^2, K_2^2}|
& \le \mathscr{Q}_1(I^1, J^1, K^1) \, 
C(\mathbf{1}_{J^2}, \mathbf{1}_{J^2}, \mathbf{1}_{J^2})
\prod_{i=1}^2 |I^i|^{-\frac12} |J^i|^{-\frac12} |K^i|^{-\frac12}
\\ \nonumber 
&\lesssim \bigg[ \frac{\ell(K^1)}{\ell(Q^1)} \bigg]^{\frac{\delta_1}{2}} 
F^1(K^1, Q^1) \frac{|I^1|^{\frac12} |J^1|^{\frac12} |K^1|^{\frac12}}{|Q^1|^2}
\\ \nonumber 
&\quad\times \bigg[ \frac{\ell(K^2)}{\ell(Q^2)} \bigg]^{\frac{\delta_2}{2}} 
F^2(J^2) \frac{|I^2|^{\frac12} |J^2|^{\frac12} |K^2|^{\frac12}}{|Q^2|^2}, 
\end{align}
provided that $\ell(I^2) \simeq \ell(J^2) \simeq \ell(K^2) \simeq \ell(Q^2)$. If $K_1^2 \neq J^2$, then $K_1^2 \subset 3J^2 \setminus J^2$, which is similar to case (i). Thus, \eqref{GAA-1} gives 
\begin{align}\label{GAA-4}
|\G_{I, J, K}^{K_1^2, K_2^2}|
&\lesssim \bigg[ \frac{\ell(K^1)}{\ell(Q^1)} \bigg]^{\frac{\delta_1}{2}} 
F^1(K^1, Q^1) \frac{|I^1|^{\frac12} |J^1|^{\frac12} |K^1|^{\frac12}}{|Q^1|^2}
\\ \nonumber 
&\quad\times \bigg[ \frac{\ell(K_2^2)}{\ell(Q^2)} \bigg]^{\frac{\delta_2}{2}} 
F^2(K_2^2, Q^2) \frac{|K_1^2|^{\frac12} |J^2|^{\frac12} |K_2^2|^{\frac12}}{|Q^2|^2}
\\ \nonumber 
&\lesssim \prod_{i=1}^2 \bigg[ \frac{\ell(K^i)}{\ell(Q^i)} \bigg]^{\frac{\delta_i}{2}} 
F^i(K^i, Q^i) \frac{|I^i|^{\frac12} |J^i|^{\frac12} |K^i|^{\frac12}}{|Q^i|^2}. 
\end{align}
By symmetry, \eqref{GAA-4} also holds in the case $K_2^2 \neq J^2$.

In case (iii), we see that $J^1 \in \ch(K^1)$, $J^2 \in \ch(K^2)$, and 
\begin{align*}
\G_{I, J, K}
= \sum_{\substack{K_1^1, K_2^1 \in \ch(K^1) \\ K_1^2, K_2^2 \in \ch(K^2)}} 
\G_{I, J, K}^{\mathbf{K}}, 
\end{align*}
where $\mathbf{K} = (K_1^1, K_1^2, K_2^1, K_2^2)$ and 
\begin{align*}
\G_{I, J, K}^{\mathbf{K}} 
&:= |J^1|^{-\frac12} |J^2|^{-\frac12}  
\langle h_{K^1} \rangle_{K_1^1} \langle h_{K^1} \rangle_{K_2^1} 
\langle h_{K^2} \rangle_{K_1^2} \langle h_{K^2} \rangle_{K_2^2} 
\\
&\qquad\times \big\langle T(\mathbf{1}_{K_1^1} \otimes \mathbf{1}_{K_1^2}, 
\mathbf{1}_{J^1} \otimes \mathbf{1}_{J^2}), 
\mathbf{1}_{K_2^1} \otimes \mathbf{1}_{K_2^2} \big\rangle. 
\end{align*}
If $K_1^1 = K_2^1 = J^1$ and $K_1^2 = K_2^2 = J^2$, then the hypothesis \eqref{H3}, namely the weak compactness property, gives 
\begin{align*}
\G_{I, J, K}^{\mathbf{K}} 
\le \prod_{i=1}^2 F_i(J^i) |J^i|^{\frac12} |K^i|^{-1} 
\lesssim \prod_{i=1}^2 F_i(J^i) |J^i|^{-\frac12}. 
\end{align*}
If $K_1^1 = K_2^1 = J^1$ and either $K_1^2 \neq J^2$ or $K_2^2 \neq J^2$, then by the compact partial kernel representation (cf. \eqref{H2}), the size condition, and Lemma \ref{lem:QQ}, 
\begin{align*}
\G_{I, J, K}^{\mathbf{K}} 
&\le C(\mathbf{1}_{J^1}, \mathbf{1}_{J^1}, \mathbf{1}_{J^1}) 
\mathscr{Q}_2(K_1^2, J^2, K_2^2)  
\prod_{i=1}^2 |J^i|^{-\frac12} |K^i|^{-1} 
\\
&\lesssim F_1(J^1) |J^1|^{\frac12} |K^1|^{-1} 
\bigg[ \frac{\ell(K^2)}{\ell(Q^2)} \bigg]^{\frac{\delta_2}{2}} 
F^2(K^2, Q^2) \frac{|I^2|^{\frac12} |J^2|^{\frac12} |K^2|^{\frac12}}{|Q^2|^2}. 
\end{align*}
If either $K_1^1 \neq J^1$ or $K_1^2 \neq J^1$ and $K_1^2 = K_2^2 = J^2$, then it is similar to case (ii) with $K_1^2 = J^2 = K_2^2$. Thus, \eqref{GAA-3} leads to 
\begin{align*}
\G_{I, J, K}^{\mathbf{K}} 
&\lesssim \bigg[ \frac{\ell(K_2^1)}{\ell(Q^1)} \bigg]^{\frac{\delta_1}{2}} 
F_1(K_2^1, Q^1) \frac{|K_1^1|^{\frac12} |J^1|^{\frac12} |K_2^1|^{\frac12}}{|Q^1|^2}
\\  
&\quad\times F_2(J^2) |I^2|^{-\frac12} |J^2|^{\frac12} |K^2|^{-\frac12}
\\ 
&\lesssim \bigg[ \frac{\ell(K^1)}{\ell(Q^1)} \bigg]^{\frac{\delta_1}{2}} 
F_1(K^1, Q^1) \frac{|I^1|^{\frac12} |J^1|^{\frac12} |K^1|^{\frac12}}{|Q^1|^2}
\\ \nonumber 
&\quad\times F_2(J^2) |I^2|^{-\frac12} |J^2|^{\frac12} |K^2|^{-\frac12}
\end{align*}
Moreover, if either $K_1^1 \neq J^1$ or $K_1^2 \neq J^1$, and either $K_1^2 \neq J^2$ or $K_2^2 \neq J^2$, then it is similar to case (i). Hence, by \eqref{GAA-1}, 
\begin{align*}
|\G_{I, J, K}^{\bf K}| 
&\le \prod_{i=1}^2 \mathscr{Q}_i(K_1^i, J^i, K_2^i) 
|K_1^i|^{-\frac12} |J^i|^{-\frac12} |K_2^i|^{-\frac12} 
\\
&\lesssim \prod_{i=1}^2 \bigg[ \frac{\ell(K_2^i)}{\ell(Q^i)} \bigg]^{\frac{\delta_i}{2}} 
F_i(K_2^i, Q^i) \frac{|K_1^i|^{\frac12} |J^i|^{\frac12} |K_2^i|^{\frac12}}{|Q^i|^2}
\\
&\lesssim \prod_{i=1}^2 \bigg[ \frac{\ell(K^i)}{\ell(Q^i)} \bigg]^{\frac{\delta_i}{2}} 
F_i(K^i, Q^i) \frac{|I^i|^{\frac12} |J^i|^{\frac12} |K^i|^{\frac12}}{|Q^i|^2}. 
\end{align*}
Accordingly, the desired inequality follows from these estimates above.
\end{proof}

Set 
\begin{align*}
a_{I, J, K, Q} 
= \frac{\G_{I, J, K}}{C_0 \prod_{i=1}^2 [\ell(K^i)/\ell(Q^i)]^{\delta_i/2}} 
\end{align*}
if for each $i=1, 2$, $I^i, J^i \in \D^i$ and $K^i \in \D_{\rm{good}}^i$ with $\ell(K^i) \le \ell(I^i) = 2\ell(J^i)$, satisfy $\max\{\d(K^i, I^i), \, \d(K^i, J^i)\} \le 2 \ell(K^i)^{\gamma_i} \ell(J^i)^{1-\gamma_i}$, and either $K^i \cap I^i = \emptyset$ or $K^i \cap J^i = \emptyset$ or $K^i = I^i$, and otherwise set $a_{I, J, K, Q} = 0$. Define 
\begin{align*}
\mathcal{F}(Q) := \mathcal{F}^1(Q^1) \mathcal{F}^2(Q^2), 
\quad \forall \, Q = Q^1 \times Q^2 \in \D^1 \times \D^2 =: \D. 
\end{align*}
By Lemma \ref{lem:GAA} and that $F^i \in \F^i$ and $(F_1^i, F_2^i, F_3^i) \in \F$, we see that \eqref{FQ-1} and \eqref{FQ-2} hold whenever $C_0$ is large enough. Thus, it follows from \eqref{AA-common} that 
\begin{align*}
\mathscr{S}_4^1   
&= \sum_{\substack{0 \le i_1 \le k_1 \le \vartheta \\ 0 \le i_2 \le k_2 \le \vartheta}} 
\sum_{\substack{I^1, J^1 \in \D^1, \, K^1 \in \D_{\rm{good}}^1 \\ 
\max\{\d(K^1, I^1), \, \d(K^1, J^1)\} \le 2 \ell(K^1)^{\gamma_1} \ell(J^1)^{1-\gamma_1} \\ 
K^1 \cap I^1 = \emptyset \text{ or } K^1 \cap J^1 = \emptyset \text{ or } K^1 = I^1 \\ 
2 \ell(J^1) = \ell(I^1) = 2^{-i_1} \ell(Q^1) \\ \ell(K^1) = 2^{-k_1} \ell(Q^1), \, 
I^1 \vee J^1 \vee K^1 = Q^1}}
\\
&\quad\times \sum_{\substack{I^2, J^2 \in \D^2, \, K^2 \in \D_{\rm{good}}^2 \\ 
\max\{\d(K^2, I^2), \, \d(K^2, J^2)\} \le 2 \ell(K^2)^{\gamma_2} \ell(J^2)^{1-\gamma_2} \\ 
K^2 \cap I^2 = \emptyset \text{ or } K^2 \cap J^2 = \emptyset \text{ or } K^2 = I^2\\ 
2 \ell(J^2) = \ell(I^2) = 2^{-i_2} \ell(Q^2) \\ \ell(K^2) = 2^{-k_2} \ell(Q^2), \, 
I^2 \vee J^2 \vee K^2 = Q^2}} 
\G_{I, J, K} \, f_1^I \, f_2^J \, f_3^K
\\
&= C_0 \sum_{\substack{0 \le i_1 \le k_1 \le \vartheta \\ 0 \le i_2 \le k_2 \le \vartheta}}   
2^{-k_1 \frac{\delta_1}{2}} 2^{-k_2 \frac{\delta_2}{2}}
\sum_{\substack{Q^1 \in \D^1 \\ Q^2 \in \D^2}} 
\sum_{\substack{I^1 \in \D_{i_1}^1(Q^1) \\ 
J^1 \in \D_{i_1+1}^1(Q^1) \\ K^1 \in \D_{k_1}^1(Q^1)}}
\sum_{\substack{I^2 \in \D_{i_2}^2(Q^2) \\ 
J^2 \in \D_{i_2+1}^2(Q^2) \\ K^2 \in \D_{k_2}^2(Q^2)}} 
a_{I, J, K, Q} \, f_1^I \, f_2^J \, f_3^K
\\   
&= C_0 \sum_{k_1=0}^{\vartheta} \sum_{k_2=0}^{\vartheta} 
2^{-k_1 \frac{\delta_1}{2}} 2^{-k_2 \frac{\delta_2}{2}}
\sum_{i_1=0}^{k_1} \sum_{i_2=0}^{k_2} 
\big\langle \mathbf{S}_{\D}^{\sigma(i, k)}(f_1, f_2), f_3 \big\rangle, 
\end{align*}
where $\sigma(i, k) := \big((i_1, i_2), (i_1+1, i_2+1), (k_1, k_2) \big)$.

\subsection{Adjacent/Nested}\label{sec:AN}

In this case, denote $I^2 := (J^2)^{(1)}$. We perform the decomposition 
\begin{align*}
\G_{I, J, K}
= \G_{I, J, K}^1 + \G_{I, J, K}^2 + \G_{I, J, K}^3, 
\end{align*}
where 
\begin{align*}
\G_{I, J, K}^1 
&:= \langle h_{I^2} \rangle_{J^2} |J^2|^{-\frac12} 
\langle T(h_{I^1} \otimes 1, h_{J^1}^0 \otimes 1), 
h_{K^1} \otimes h_{K^2} \rangle, 
\\
\G_{I, J, K}^2 
&:= \langle h_{I^2} \rangle_{J^2} |J^2|^{-\frac12} 
\langle T(h_{I^1} \otimes 1, h_{J^1}^0 \otimes \mathbf{1}_{(J^2)^c}), 
h_{K^1} \otimes h_{K^2} \rangle, 
\\
\G_{I, J, K}^3 
&:= |J^2|^{-\frac12}  
\langle T(h_{I^1} \otimes \phi_{J^2}, h_{J^1}^0 \otimes \mathbf{1}_{J^2}), 
h_{K^1} \otimes h_{K^2} \rangle,  
\end{align*}
where $\phi_{J^2} := \mathbf{1}_{(J^2)^c} \big(h_{I^2} - \langle h_{I^2} \rangle_{J^2}\big)$. Define the  summation $\mathscr{S}_{5, k}^1$ corresponding to $\G_{I, J, K}^k$, $k=1, 2, 3$.

\medskip
\noindent{\bf $\bullet$ Partial paraproducts.} We rewrite 
\begin{align*}
\mathscr{S}_{5, 1}^1   
&= \sum_{\substack{I^1, J^1 \in \D^1, \, K^1 \in \D_{\rm{good}}^1 \\ \ell(K^1) \le \ell(I^1) = 2 \ell(J^1) \\ 
\max\{\d(K^1, I^1), \, \d(K^1, J^1)\} \le 2\ell(K^1)^{\gamma_1} \ell(J^1)^{1-\gamma_1} \\ 
K^1 \cap I^1 = \emptyset \text{ or } K^1 \cap J^1 = \emptyset \text{ or } K^1 = I^1}}  
\sum_{K^2 \in \D_{\rm{good}}^2}  
\sum_{\substack{J^2 \in \D^2 \\ J^2 \supset K^2}}  
\langle f_3, h_{K^1} \otimes h_{K^2} \rangle
\\
&\quad\times \langle T(h_{I^1} \otimes 1, h_{J^1}^0 \otimes 1), 
h_{K^1} \otimes h_{K^2} \rangle
\big\langle \langle f_1, h_{I^1} \rangle, h_{I^2} \big\rangle 
\big\langle \langle f_2, h_{J^1}^0 \rangle \big\rangle_{J^2}
\langle h_{I^2} \rangle_{J^2}. 
\end{align*}
Much as above, from $\mathscr{S}^2$, we obtain the corresponding term 
\begin{align*}
\mathscr{S}_{5, 1}^2   
&= \sum_{\substack{I^1, J^1 \in \D^1, \, K^1 \in \D_{\rm{good}}^1 \\ \ell(K^1) \le \ell(I^1) = 2 \ell(J^1) \\ 
\max\{\d(K^1, I^1), \, \d(K^1, J^1)\} \le 2\ell(K^1)^{\gamma_1} \ell(J^1)^{1-\gamma_1} \\ 
K^1 \cap I^1 = \emptyset \text{ or } K^1 \cap J^1 = \emptyset \text{ or } K^1 = I^1}} 
\sum_{K^2 \in \D_{\rm{good}}^2}  
\sum_{\substack{J^2 \in \D^2 \\ J^2 \supset K^2}}   
\langle f_3, h_{K^1} \otimes h_{K^2} \rangle
\\
&\quad\times \langle T(h_{I^1} \otimes 1, h_{J^1}^0 \otimes 1), 
h_{K^1} \otimes h_{K^2} \rangle
\big\langle \langle f_1, h_{I^1} \rangle \big\rangle_{I^2} 
\big\langle \langle f_2, h_{J^1}^0 \rangle, h_{I^2} \big\rangle 
\langle h_{I^2} \rangle_{J^2}.  
\end{align*} 
Write $Q = Q^1 \times Q^2 \in \D^1 \times \D^2 =: \D$. Set 
\begin{align*}
a_{I^1, J^1, K^1, Q} 
= \frac{\langle T(h_{I^1} \otimes 1, h_{J^1}^0 \otimes 1), h_{K^1} \otimes h_{Q^2} \rangle}{C_0 [\ell(K^1)/\ell(Q^1)]^{\delta_1/2}} 
\end{align*}
if $I^1, J^1 \in \D^1$, $K^1 \in \D_{\rm{good}}^1$, and $K^2 \in \D_{\rm{good}}^2$ satisfy $\ell(K^1) \le \ell(I^1) = 2\ell(J^1)$, $\max\{\d(K^1, I^1), \, \d(K^1, J^1)\} \le 2\ell(K^1)^{\gamma_1} \ell(J^1)^{1-\gamma_1}$, and either $K^1 \cap I^1 = \emptyset$ or $K^1 \cap J^1 = \emptyset$ or $K^1 = I^1$, and otherwise set $a_{I^1, J^1, K^1, Q} = 0$. As done in \eqref{S312} with $\overline{h}_{Q^2} := \frac{\mathbf{1}_{Q^2}}{|Q^2|}$, we have 
\begin{align*}
\mathscr{S}_{5, 1}^1 + \mathscr{S}_{5, 1}^2  
&= C_0 \sum_{0 \le i_1 \le k_1 \le \vartheta}  2^{-k_1 \frac{\delta_1}{2}} 
\sum_{\substack{Q^1 \in \D^1 \\ Q^2 \in \D^2}} 
\sum_{\substack{I^1 \in \D_{i_1}^1(Q^1) \\ 
J^1 \in \D_{i_1+1}^1(Q^1) \\ K^1 \in \D_{k_1}^1(Q^1)}} 
a_{I^1, J^1, K^1, Q} 
\\
&\quad\times \langle f_1, h_{I^1} \otimes \overline{h}_{Q^2} \rangle 
\langle f_2, h_{J^1}^0 \otimes \overline{h}_{Q^2} \rangle
\langle f_3, h_{K^1} \otimes h_{Q^2} \rangle 
\\ 
\\
&= C_0 \sum_{k_1=0}^{\vartheta} 2^{-k_1 \frac{\delta_1}{2}} 
\sum_{i_1=0}^{k_1} \langle \mathbf{P}_{\D}^{1, i_1, i_1+1, k_1}(f_1, f_2), f_3 \rangle. 
\end{align*}
Due to Lemma \ref{lem:PPAN}, $\mathbf{P}_{\D}^{1, i_1, i_1+1, k_1}$ is a partial paraproduct, where $\mathcal{F}^1(Q^1)$ is given in \eqref{def:FQ}.

\medskip
\noindent{\bf $\bullet$ Bilinear dyadic shifts. }
We then treat $\mathscr{S}_{5, 2}^1$ and $\mathscr{S}_{5, 3}^1$.

\begin{lemma}\label{lem:GAN23}
For each $k=2, 3$, there holds 
\begin{align*}
|\G_{I, J, K}^k| 
&\lesssim \bigg[ \frac{\ell(K^1)}{\ell(Q^1)} \bigg]^{\frac{\delta_1}{2}} 
\widetilde{F}^1(Q^1) \frac{|I^1|^{\frac12} |J^1|^{\frac12} |K^1|^{\frac12}}{|Q^1|^2}
\bigg[ \frac{\ell(K^2)}{\ell(J^2)} \bigg]^{\frac{\delta_2}{2}} 
F^2(K^2, J^2) \frac{|K^2|^{\frac12}}{|J^2|}, 
\end{align*}
where $\widetilde{F}^1$ and $F^2(\cdot, \cdot)$ are defined in \eqref{def:FQ} and \eqref{def:FKQ} respectively.  
\end{lemma}

\begin{proof}
We only focus on $\G_{I, J, K}^2$ because the estimate for $\G_{I, J, K}^3$ is similar.  
First, suppose $K^1 \cap I^1 = \emptyset$. In the case $\ell(J^2) > 2^{\vartheta} \ell(K^2)$, splitting $1 = \mathbf{1}_{3K^2} + \mathbf{1}_{(3K^2)^c}$, we have 
\begin{align*}
\G_{I, J, K}^2 = \G_{I, J, K}^{2, 1} + \G_{I, J, K}^{2, 2}.
\end{align*}
In view of \eqref{dxz} and Lemmas \ref{lem:QQ} and \ref{lem:RR} applied to $\ell_2 = \ell(K^2)^{\frac12} \ell(J^2)^{\frac12}$, we use the cancellation of $h_{K^1}$ and $h_{K^2}$, the compact full kernel representation (cf. \eqref{H1}), and the mixed size-H\"{o}lder condition to arrive at 
\begin{align}\label{ANG21}
&|\G_{I, J, K}^{2, j}| 
\le \mathscr{Q}_1(I^1, J^1, K^1) \mathscr{R}_2^j(J^2, K^2) 
\prod_{i=1}^2 |I^i|^{-\frac12} |J^i|^{-\frac12} |K^i|^{-\frac12} 
\\ \nonumber 
&\lesssim \bigg[ \frac{\ell(K^1)}{\ell(Q^1)} \bigg]^{\frac{\delta_1}{2}} 
F^1(K^1, Q^1) \frac{|I^1|^{\frac12} |J^1|^{\frac12} |K^1|^{\frac12}}{|Q^1|^2}
\bigg[ \frac{\ell(K^2)}{\ell(J^2)} \bigg]^{\frac{\delta_2}{2}} 
F^2(J^2, K^2) \frac{|K^2|^{\frac12}}{|J^2|},  
\end{align}
for each $j=1, 2$.

To treat the case $\ell(J^2) \le 2^{\vartheta} \ell(K^2)$, we split 
\begin{align}\label{ANG2}
\G_{I, J, K}^2 
&= \langle h_{I^2} \rangle_{J^2} |J^2|^{-\frac12} 
\langle T(h_{I^1} \otimes \mathbf{1}_{3K^2}, 
h_{J^1}^0 \otimes \mathbf{1}_{3J^2 \setminus J^2}), 
h_{K^1} \otimes h_{K^2} \rangle 
\\ \nonumber
&\quad+ \langle h_{I^2} \rangle_{J^2} |J^2|^{-\frac12} 
\langle T(h_{I^1} \otimes \mathbf{1}_{3K^2}, 
h_{J^1}^0 \otimes \mathbf{1}_{(3J^2)^c}), 
h_{K^1} \otimes h_{K^2} \rangle 
\\ \nonumber
&\quad+ \langle h_{I^2} \rangle_{J^2} |J^2|^{-\frac12} 
\langle T(h_{I^1} \otimes \mathbf{1}_{(3K^2)^c}, 
h_{J^1}^0 \otimes \mathbf{1}_{3J^2 \setminus J^2}), 
h_{K^1} \otimes h_{K^2} \rangle 
\\ \nonumber 
&\quad+ \langle h_{I^2} \rangle_{J^2} |J^2|^{-\frac12} 
\langle T(h_{I^1} \otimes \mathbf{1}_{(3K^2)^c}, 
h_{J^1}^0 \otimes \mathbf{1}_{(3J^2)^c}), 
h_{K^1} \otimes h_{K^2} \rangle. 
\end{align}
The first term is similar to $\G_{I, J, K}$ in the case $K^2 \cap J^2 = \emptyset$ in Section \ref{sec:AA}, which along with Lemma \ref{lem:GAA} yields the desired bound. The second and last terms are analogous to $\G_{I, J, K}^{2, 1}$ and $\G_{I, J, K}^{2, 2}$ respectively, but now $\ell_2 = \ell(K^2)$. The third term is symmetric to the second one. Eventually, the last three terms are dominated by 
\begin{align*}
\bigg[ \frac{\ell(K^1)}{\ell(Q^1)} \bigg]^{\frac{\delta_1}{2}} 
F_1(K^1, Q^1) \frac{|I^1|^{\frac12} |J^1|^{\frac12} |K^1|^{\frac12}}{|Q^1|^2}
F^2(J^2, K^2) \frac{|K^2|^{\frac12}}{|J^2|}, 
\end{align*}
which together with $\ell(J^2) \simeq \ell(K^2)$ gives the estimate as desired. 

Similarly, one can handle the case $K^1 \cap J^1 = \emptyset$. 

Next, suppose $K^1 \cap J^1 \neq \emptyset$ and $K^1 = I^1$. Obviously, $J^1 \in \ch(K^1)$. Then we rewrite 
\begin{align*}
\G_{I, J, K}^2
= \sum_{K_1^1, K_2^1 \in \ch(K^1)} \G_{I, J, K}^{K_1^1, K_2^1}, 
\end{align*}
where 
\begin{align*}
\G_{I, J, K}^{K_1^1, K_2^1} 
:= \langle h_{I^1} \rangle_{K_1^1} \langle h_{I^2} \rangle_{J^2}
|J^1|^{-\frac12} |J^2|^{-\frac12} 
\langle h_{K^1} \rangle_{K_2^1} 
\big\langle T(\mathbf{1}_{K_1^1} \otimes 1, \mathbf{1}_{J^1} \otimes \mathbf{1}_{(J^2)^c}), 
\mathbf{1}_{K_2^1} \otimes h_{K^2} \big\rangle. 
\end{align*}
In the case $\ell(J^2) > 2^{\vartheta} \ell(K^2)$, we split $1 = \mathbf{1}_{3K^2} + \mathbf{1}_{(3K^2)^c}$. If $K_1^1 = J^1 = K_2^1$, then the compact partial kernel representation (cf. \eqref{H2}) and Lemma \ref{lem:RR} applied to $\ell_2 = \ell(K^2)^{\frac12} \ell(J^2)^{\frac12}$ give 
\begin{align}\label{ANGK-1}
|\G_{I, J, K}^{K_1^1, K_2^1}| 
\le  C(\mathbf{1}_{J^1}, \mathbf{1}_{J^1}, \mathbf{1}_{J^1}) 
\big[(\mathscr{R}_2^1 + \mathscr{R}_2^2)(J^2, K^2) \big]
\prod_{i=1}^2 |I^i|^{-\frac12} |J^i|^{-\frac12} |K^i|^{-\frac12} 
\\ \nonumber 
\lesssim \bigg[ \frac{\ell(K^1)}{\ell(Q^1)} \bigg]^{\frac{\delta_1}{2}} F_1(J^1) 
\frac{|I^1|^{\frac12} |J^1|^{\frac12} |K^1|^{\frac12}}{|Q^1|^2}
\bigg[ \frac{\ell(K^2)}{\ell(J^2)} \bigg]^{\frac{\delta_2}{2}} 
F^2(K^2, J^2) \frac{|K^2|^{\frac12}}{|J^2|}, 
\end{align}
provided that $\ell(I^1) \simeq \ell(J^1) \simeq \ell(K^1) \simeq \ell(Q^1)$. If $K_1^1 \neq J^1$ or $K_2^1 \neq J^1$, then the compact full kernel representation (cf. \eqref{H1}) and Lemmas \ref{lem:QQ} and \ref{lem:RR} applied to $\ell_2 = \ell(K^2)^{\frac12} \ell(J^2)^{\frac12}$ give 
\begin{align}\label{ANGK-2}
|\G_{I, J, K}^{K_1^1, K_2^1}| 
\le \mathscr{Q}_1(K_1^1, J^1, K_2^1) 
\big[(\mathscr{R}_2^1 + \mathscr{R}_2^2)(J^2, K^2) \big]
\prod_{i=1}^2 |I^i|^{-\frac12} |J^i|^{-\frac12} |K^i|^{-\frac12}
\\ \nonumber 
\lesssim \bigg[ \frac{\ell(K_2^1)}{\ell(Q^1)} \bigg]^{\frac{\delta_1}{2}} 
F^1(K_2^1, Q^1) \frac{|K_1^1|^{\frac12} |J^1|^{\frac12} |K_2^1|^{\frac12}}{|Q^1|^2}  
\bigg[ \frac{\ell(K^2)}{\ell(J^2)} \bigg]^{\frac{\delta_2}{2}} 
F^2(K^2, J^2) \frac{|K^2|^{\frac12}}{|J^2|}
\\ \nonumber 
\lesssim \bigg[ \frac{\ell(K^1)}{\ell(Q^1)} \bigg]^{\frac{\delta_1}{2}} 
F^1(K^1, Q^1) \frac{|I^1|^{\frac12} |J^1|^{\frac12} |K^1|^{\frac12}}{|Q^1|^2}  
\bigg[ \frac{\ell(K^2)}{\ell(J^2)} \bigg]^{\frac{\delta_2}{2}} 
F^2(K^2, J^2) \frac{|K^2|^{\frac12}}{|J^2|}. 
\end{align}
In the case $\ell(J^2) \le 2^{\vartheta} \ell(K^2)$, using the decomposition as in \eqref{ANG2} and the argument as in \eqref{ANGK-1}--\eqref{ANGK-2} with $\ell_2 = \ell(K^2) \simeq \ell(J^2)$, we obtain the desired estimate. 

Finally, note that $F^1(K^1, Q^1) \simeq F^1(Q^1, Q^1)$ due to $\ell(K^1) \simeq \ell(Q^1)$. Gathering all estimates above, we conclude the proof. 
\end{proof}

Note that $I^2 = (J^2)^{(1)}$. Set $Q^2 = I^2$ and 
\begin{align*}
a_{I, J, K, Q} 
= \frac{\G_{I, J, K}^2 + \G_{I, J, K}^3}{C_0 \prod_{i=1}^2 [\ell(K^i)/\ell(Q^i)]^{\delta_i/2}} 
\end{align*}
if $I^1, J^1 \in \D^1$, $K^1 \in \D_{\rm{good}}^1$, $J^2 \in \D^2$, and $K^2 \in \D_{\rm{good}}^2$ satisfy $\ell(K^1) \le \ell(I^1) = 2\ell(J^1)$, $\max\{\d(K^1, I^1), \, \d(K^1, J^1)\} \le 2 \ell(K^1)^{\gamma_1} \ell(J^1)^{1-\gamma_1}$, either $K^1 \cap I^1 = \emptyset$ or $K^1 \cap J^1 = \emptyset$ or $K^1 = I^1$, and $K^2 \subset J^2 \in \ch(Q^2)$, and otherwise set $a_{I, J, K, Q} = 0$. Thus, \eqref{AA-common} applied to $i=1$ gives 
\begin{align*}
\mathscr{S}_{5, 2}^1 + \mathscr{S}_{5, 3}^1   
= C_0 \sum_{\substack{0 \le i_1 \le k_1 \le \vartheta \\ k_2 \ge 1}} 
2^{-k_1 \frac{\delta_1}{2}} 2^{-k_2 \frac{\delta_2}{2}}
\sum_{\substack{Q^1 \in \D^1 \\ Q^2 \in \D^2 \\ K^2 \in \D_{k_2}^2(Q^2)}}
\sum_{\substack{I^1 \in \D^1_{i_1}(Q^1) \\ J^1 \in \D^1_{i_1+1}(Q^1) \\ K^1 \in \D^1_{k_1}(Q^1)}}
a_{I, J, K, Q} \, f_1^I \, f_2^J \, f_3^K. 
\end{align*}
For each $k_2 \in \N$, if we define 
\begin{align*}
\mathcal{F}(Q) := \widetilde{F}^1(Q^1) 
F_1^2(\ell(Q^2)) \widetilde{F}_2^2(2^{-k_2} \ell(Q^2)) \widetilde{F}_3^2(Q^2),  
\end{align*}
for all $Q = Q^1 \times Q^2 \in \D^1 \times \D^2 =: \D$, then Lemma \ref{lem:GAN23} and that $F^i \in \F^i$ and $(F_1^i, F_2^i, F_3^i) \in \F$ impliy \eqref{FQ-1} and \eqref{FQ-2} hold whenever $C_0$ is large enough. This leads to 
\begin{align*}
\mathscr{S}_{5, 2}^1 + \mathscr{S}_{5, 3}^1      
= C_0 \sum_{k_1=0}^{\vartheta} \sum_{k_2=1}^{\infty} 
2^{-k_1 \frac{\delta_1}{2}} 2^{-k_2 \frac{\delta_2}{2}}
\sum_{i_1=0}^{k_1} 
\big\langle \mathbf{S}_{\D}^{i, i+1, k}(f_1, f_2), f_3 \big\rangle,  
\end{align*}
where $\sigma(i, k) := \big((i_1, 0), (i_1+1, 1), (k_1, k_2) \big)$.

\subsection{Nested/Nested}\label{sec:NN} 
Denote $I^1 := (J^1)^{(1)}$ and $I^2 := (J^2)^{(1)}$. Set $\phi_{J^i} := \mathbf{1}_{(J^i)^c} \big(h_{I^i} - \langle h_{I^i} \rangle_{J^i}\big)$, $i=1, 2$. Splitting $h_{I^i} = \langle h_{I^i} \rangle_{J^i} + \phi_{J^i}$, $i=1, 2$, we obtain four terms, then for the terms containing $\langle h_{I^i} \rangle_{J^i}$, we split $\mathbf{1}_{J^i} = 1 - \mathbf{1}_{(J^i)^c}$. Thus, 

\begin{align*}
\G_{I, J, K} = \sum_{j=1}^9 \G_{I, J, K}^j, 
\end{align*}
where 
\begin{align*}
\G_{I, J, K}^1 
&:= |J^1|^{-\frac12} |J^2|^{-\frac12} 
\langle h_{I^1} \rangle_{J^1}  \langle h_{I^2} \rangle_{J^2}  
\langle T(1, 1), h_{K^1} \otimes h_{K^2} \rangle, 
\\ 
\G_{I, J, K}^2 
&:= - |J^1|^{-\frac12} |J^2|^{-\frac12} 
\langle h_{I^1} \rangle_{J^1}  \langle h_{I^2} \rangle_{J^2}  
\langle T(1, 1 \otimes \mathbf{1}_{(J^2)^c}), h_{K^1} \otimes h_{K^2} \rangle, 
\\
\G_{I, J, K}^3 
&:= - |J^1|^{-\frac12} |J^2|^{-\frac12} 
\langle h_{I^1} \rangle_{J^1}  \langle h_{I^2} \rangle_{J^2}  
\langle T(1, \mathbf{1}_{(J^1)^c} \otimes 1), h_{K^1} \otimes h_{K^2} \rangle, 
\\
\G_{I, J, K}^4 
&:= |J^1|^{-\frac12} |J^2|^{-\frac12} 
\langle h_{I^1} \rangle_{J^1}  
\langle T(1 \otimes \phi_{J^2}, 1 \otimes \mathbf{1}_{J^2}), h_{K^1} \otimes h_{K^2} \rangle,  
\\
\G_{I, J, K}^5 
&:= |J^1|^{-\frac12} |J^2|^{-\frac12} 
\langle h_{I^2} \rangle_{J^2}  
\langle T(\phi_{J^1} \otimes 1, \mathbf{1}_{J^1} \otimes 1), h_{K^1} \otimes h_{K^2} \rangle,  
\\
\G_{I, J, K}^6 
&:= |J^1|^{-\frac12} |J^2|^{-\frac12} 
\langle h_{I^1} \rangle_{J^1}  \langle h_{I^2} \rangle_{J^2}  
\langle T(1, \mathbf{1}_{(J^1)^c} \otimes \mathbf{1}_{(J^2)^c}), h_{K^1} \otimes h_{K^2} \rangle, 
\\
\G_{I, J, K}^7 
&:= - |J^1|^{-\frac12} |J^2|^{-\frac12} 
\langle h_{I^1} \rangle_{J^1}  
\langle T(1 \otimes \phi_{J^2}, \mathbf{1}_{(J^1)^c} \otimes \mathbf{1}_{J^2}), 
h_{K^1} \otimes h_{K^2} \rangle,  
\\
\G_{I, J, K}^8 
&:= - |J^1|^{-\frac12} |J^2|^{-\frac12} 
\langle h_{I^2} \rangle_{J^2}  
\langle T(\phi_{J^1} \otimes 1, \mathbf{1}_{J^1} \otimes \mathbf{1}_{(J^2)^c}), 
h_{K^1} \otimes h_{K^2} \rangle,  
\\
\G_{I, J, K}^9 
&:= |J^1|^{-\frac12} |J^2|^{-\frac12} 
\langle T(\phi_{J^1} \otimes \phi_{J^2}, \mathbf{1}_{J^1} \otimes \mathbf{1}_{J^2}), 
h_{K^1} \otimes h_{K^2} \rangle. 
\end{align*}
We define the summation $\mathscr{S}_{6, j}^1$ corresponding to $\G_{I, J, K}^j$, $j=1, \ldots, 9$. We will see that $\mathscr{S}_{6, 1}^1$ produces a full paraproduct; $\mathscr{S}_{6, j}^1$, $j=2, 3, 4, 5$, give partial paraproducts; and $\mathscr{S}_{6, j}^1$, $j=6, 7, 8, 9$, produce bilinear shifts.

\noindent{\bf $\bullet$ A full paraproduct.}
The same argument yields $\mathscr{S}_{6, 1}^k$ from $\mathscr{S}^k$, $k=2, 3, 4$. More precisely, 
\begin{align*}
\mathscr{S}_{6, 1}^1 
&= \sum_{\substack{K^1 \in \D_{\rm{good}}^1 \\ K^2 \in \D_{\rm{good}}^2}}
\sum_{\substack{J^1 \in \D^1 \\ J^1 \supset K^1}}
\sum_{\substack{J^2 \in \D^2 \\ J^2 \supset K^2}}
\langle T(1, 1), h_{K^1} \otimes h_{K^2} \rangle
\\
&\quad\times 
\langle h_{I^1} \rangle_{J^1}  \langle h_{I^2} \rangle_{J^2}  
\langle f_1, h_{I^1} \otimes h_{I^2} \rangle
\langle f_2 \rangle_{J^1 \times J^2}
\langle f_3, h_{K^1} \otimes h_{K^2} \rangle, 
\\ 
\mathscr{S}_{6, 1}^2 
&= \sum_{\substack{K^1 \in \D_{\rm{good}}^1 \\ K^2 \in \D_{\rm{good}}^2}}
\sum_{\substack{J^1 \in \D^1 \\ J^1 \supset K^1}}
\sum_{\substack{J^2 \in \D^2 \\ J^2 \supset K^2}}
\langle T(1, 1), h_{K^1} \otimes h_{K^2} \rangle
\\
&\quad\times 
\langle h_{I^1} \rangle_{J^1}  \langle h_{I^2} \rangle_{J^2}  
\big\langle \langle f_1, h_{I^1} \rangle \big\rangle_{I^2}
\big\langle \langle f_2 \rangle_{J^1}, h_{I^2} \big\rangle 
\langle f_3, h_{K^1} \otimes h_{K^2} \rangle, 
\\ 
\mathscr{S}_{6, 1}^3 
&= \sum_{\substack{K^1 \in \D_{\rm{good}}^1 \\ K^2 \in \D_{\rm{good}}^2}}
\sum_{\substack{J^1 \in \D^1 \\ J^1 \supset K^1}}
\sum_{\substack{J^2 \in \D^2 \\ J^2 \supset K^2}}
\langle T(1, 1), h_{K^1} \otimes h_{K^2} \rangle
\\
&\quad\times 
\langle h_{I^1} \rangle_{J^1}  \langle h_{I^2} \rangle_{J^2}  
\big\langle \langle f_1 \rangle_{I^1}, h_{I^2} \big\rangle
\big\langle \langle f_2, h_{I^1} \rangle \big\rangle_{J^2}
\langle f_3, h_{K^1} \otimes h_{K^2} \rangle, 
\\ 
\mathscr{S}_{6, 1}^4 
&= \sum_{\substack{K^1 \in \D_{\rm{good}}^1 \\ K^2 \in \D_{\rm{good}}^2}}
\sum_{\substack{J^1 \in \D^1 \\ J^1 \supset K^1}}
\sum_{\substack{J^2 \in \D^2 \\ J^2 \supset K^2}}
\langle T(1, 1), h_{K^1} \otimes h_{K^2} \rangle
\\
&\quad\times 
\langle h_{I^1} \rangle_{J^1}  \langle h_{I^2} \rangle_{J^2}  
\langle f_1 \rangle_{I^1 \times I^2}
\langle f_2, h_{I^1} \otimes h_{I^2} \rangle
\langle f_3, h_{K^1} \otimes h_{K^2} \rangle. 
\end{align*}
It follows from \eqref{gghh} that 
\begin{align*}
\sum_{k=1}^4 \mathscr{S}_{6, 1}^k 
&= (\mathscr{S}_{6, 1}^1 + \mathscr{S}_{6, 1}^2) 
+ (\mathscr{S}_{6, 1}^3 + \mathscr{S}_{6, 1}^4) 
\\ 
&= \sum_{\substack{K^1 \in \D_{\rm{good}}^1 \\ K^2 \in \D_{\rm{good}}^2}}
\sum_{\substack{J^1 \in \D^1 \\ J^1 \supset K^1}}
\sum_{\substack{J^2 \in \D^2 \\ J^2 \supset K^2}}
\langle T(1, 1), h_{K^1} \otimes h_{K^2} \rangle 
\langle f_3, h_{K^1} \otimes h_{K^2} \rangle
\\
&\quad \big[ \langle h_{I^1} \rangle_{J^1} 
\big(\big\langle \langle f_1, h_{I^1} \rangle \big\rangle_{J^2}
\big\langle \langle f_2 \rangle_{J^1} \big\rangle_{J^2} 
- \big\langle \langle f_1, h_{I^1} \rangle \big\rangle_{I^2}
\big\langle \langle f_2 \rangle_{J^1} \big\rangle_{I^2} \big)
\\
&\quad+ \langle h_{I^1} \rangle_{J^1} 
\big(\big\langle \langle f_1 \rangle_{I^1} \big\rangle_{J^2}
\big\langle \langle f_2, h_{I^1} \rangle \big\rangle_{J^2} 
- \big\langle \langle f_1 \rangle_{I^1} \big\rangle_{I^2}
\big\langle \langle f_2, h_{I^1} \rangle \big\rangle_{I^2} \big)\big]
\\
&= \sum_{\substack{K^1 \in \D_{\rm{good}}^1 \\ K^2 \in \D_{\rm{good}}^2}}
\sum_{\substack{J^1 \in \D^1 \\ J^1 \supset K^1}}
\sum_{\substack{J^2 \in \D^2 \\ J^2 \supset K^2}}
\langle T(1, 1), h_{K^1} \otimes h_{K^2} \rangle 
\langle f_3, h_{K^1} \otimes h_{K^2} \rangle
\\
&\quad \big[ \langle h_{I^1} \rangle_{J^1} 
\big(\big\langle \langle f_1 \rangle_{J^2}, h_{I^1} \big\rangle
\big\langle \langle f_2 \rangle_{J^2} \big\rangle_{J^1} 
+ \big\langle \langle f_1 \rangle_{J^2} \big\rangle_{I^1}
\big\langle \langle f_2 \rangle_{J^2}, h_{I^1} \big\rangle \big)
\\
&\quad- \langle h_{I^1} \rangle_{J^1} 
\big(\big\langle \langle f_1 \rangle_{I^2}, , h_{I^1} \big\rangle
\big\langle \langle f_2 \rangle_{I^2} \big\rangle_{J^1} 
+ \big\langle \langle f_1 \rangle_{I^2} \big\rangle_{I^1}
\big\langle \langle f_2 \rangle_{I^2}, h_{I^1} \big\rangle \big)\big] 
\\
&= \sum_{\substack{K^1 \in \D_{\rm{good}}^1 \\ K^2 \in \D_{\rm{good}}^2}}
\sum_{\substack{J^1 \in \D^1 \\ J^1 \supset K^1}}
\sum_{\substack{J^2 \in \D^2 \\ J^2 \supset K^2}}
\langle T(1, 1), h_{K^1} \otimes h_{K^2} \rangle 
\langle f_3, h_{K^1} \otimes h_{K^2} \rangle
\\
&\quad \big[ \big(\big\langle \langle f_1 \rangle_{J^2} \big\rangle_{J^1}
\big\langle \langle f_2 \rangle_{J^2} \big\rangle_{J^1} 
- \big\langle \langle f_1 \rangle_{J^2} \big\rangle_{I^1}
\big\langle \langle f_2 \rangle_{J^2} \big\rangle_{I^1} \big)
\\
&\quad- \big(\big\langle \langle f_1 \rangle_{I^2} \big\rangle_{J^1}
\big\langle \langle f_2 \rangle_{I^2} \big\rangle_{J^1} 
- \big\langle \langle f_1 \rangle_{I^2} \big\rangle_{I^1}
\big\langle \langle f_2 \rangle_{I^2} \big\rangle_{I^1} \big)\big] 
\\
&= \sum_{\substack{K^1 \in \D_{\rm{good}}^1 \\ K^2 \in \D_{\rm{good}}^2}}
\sum_{\substack{J^2 \in \D^2 \\ J^2 \supset K^2}}
\langle T(1, 1), h_{K^1} \otimes h_{K^2} \rangle 
\langle f_3, h_{K^1} \otimes h_{K^2} \rangle
\\
&\quad \big[\big\langle \langle f_1 \rangle_{J^2} \big\rangle_{K^1}
\big\langle \langle f_2 \rangle_{J^2} \big\rangle_{K^1} 
- \big\langle \langle f_1 \rangle_{I^2} \big\rangle_{K^1}
\big\langle \langle f_2 \rangle_{I^2} \big\rangle_{K^1} \big]
\\
&= \sum_{\substack{K^1 \in \D_{\rm{good}}^1 \\ K^2 \in \D_{\rm{good}}^2}}
\langle T(1, 1), h_{K^1} \otimes h_{K^2} \rangle 
\langle f_1 \rangle_{K^1 \times K^2} \langle f_2 \rangle_{K^1 \times K^2} 
\langle f_3, h_{K^1} \otimes h_{K^2} \rangle. 
\end{align*}
Set 
\begin{align*}
a_K = C_0^{-1} \langle T(1, 1), h_{K^1} \otimes h_{K^2} \rangle, 
\quad \text{if } K^1 \in \D_{\rm{good}}^1 \text{ and } K^2 \in \D_{\rm{good}}^2, 
\end{align*}
otherwise set $a_K=0$. The hypothesis \eqref{H5} says that $T(1, 1) \in \CMO(\Rnn)$, which implies 
\begin{align*}
\sup_{\D} \sup_U \frac{1}{|U|} 
\sum_{I \in \D: \, I \subset U} |a_I|^2 
\le 1  
\quad \text{ and } \quad 
\lim_{N \to \infty} \sup_{\D} \sup_U \frac{1}{|U|} 
\sum_{I \notin \D(N): \, I \subset U} |a_I|^2 
= 0, 
\end{align*}
provided $C_0 > 0$ large enough. Hence, we obtain a full paraproduct: 
\begin{align*}
\sum_{k=1}^4 \mathscr{S}_{6, 1}^k 
= C_0 \langle \mathbf{F}_{\mathbf{a}}(f_1, f_2), f_3 \rangle. 
\end{align*}

\noindent{\bf $\bullet$ Partial paraproducts.}
Since $\mathscr{S}_{6, 3}^1$ and $\mathscr{S}_{6, 5}^1$ are symmetrical to $\mathscr{S}_{6, 2}^1$ and $\mathscr{S}_{6, 4}^1$ respectively, we only deal with $\mathscr{S}_{6, 2}^1$ and $\mathscr{S}_{6, 4}^1$: 
\begin{align*}
\mathscr{S}_{6, 2}^1 
&= - \sum_{\substack{K^1 \in \D_{\rm{good}}^1 \\ K^2 \in \D_{\rm{good}}^2}}
\sum_{\substack{J^1 \in \D^1 \\ J^1 \supset K^1}}
\sum_{\substack{J^2 \in \D^2 \\ J^2 \supset K^2}}
\langle T(1, 1 \otimes \mathbf{1}_{(J^2)^c}), h_{K^1} \otimes h_{K^2} \rangle
\\
&\quad\times 
\langle h_{I^1} \rangle_{J^1}  \langle h_{I^2} \rangle_{J^2}  
\big\langle \langle f_1, h_{I^2} \rangle, h_{I^1} \big\rangle 
\big\langle \langle f_2 \rangle_{J^2} \big\rangle_{J^1}
\langle f_3, h_{K^1} \otimes h_{K^2} \rangle, 
\\ 
\mathscr{S}_{6, 4}^1 
&= \sum_{\substack{K^1 \in \D_{\rm{good}}^1 \\ K^2 \in \D_{\rm{good}}^2}}
\sum_{\substack{J^1 \in \D^1 \\ J^1 \supset K^1}}
\sum_{\substack{J^2 \in \D^2 \\ J^2 \supset K^2}}
\langle T(1 \otimes \phi_{J^2}, 1 \otimes \mathbf{1}_{J^2}), h_{K^1} \otimes h_{K^2} \rangle
\\
&\quad\times 
\langle h_{I^1} \rangle_{J^1}  
\big\langle \langle f_1, h_{I^2} \rangle, h_{I^1} \big\rangle 
\big\langle \langle f_2 \rangle_{J^2} \big\rangle_{J^1}
\langle f_3, h_{K^1} \otimes h_{K^2} \rangle. 
\end{align*}
The summation $\mathscr{S}^3$ gives the corresponding terms 
\begin{align*}
\mathscr{S}_{6, 2}^3 
&= - \sum_{\substack{K^1 \in \D_{\rm{good}}^1 \\ K^2 \in \D_{\rm{good}}^2}}
\sum_{\substack{J^1 \in \D^1 \\ J^1 \supset K^1}}
\sum_{\substack{J^2 \in \D^2 \\ J^2 \supset K^2}}
\langle T(1, 1 \otimes \mathbf{1}_{(J^2)^c}), h_{K^1} \otimes h_{K^2} \rangle
\\
&\quad\times 
\langle h_{I^1} \rangle_{J^1}  \langle h_{I^2} \rangle_{J^2}  
\big\langle \langle f_1, h_{I^2} \rangle \big\rangle_{I^1}
\big\langle \langle f_2 \rangle_{J^2}, h_{I^1} \big\rangle
\langle f_3, h_{K^1} \otimes h_{K^2} \rangle, 
\\ 
\mathscr{S}_{6, 4}^3 
&= \sum_{\substack{K^1 \in \D_{\rm{good}}^1 \\ K^2 \in \D_{\rm{good}}^2}}
\sum_{\substack{J^1 \in \D^1 \\ J^1 \supset K^1}}
\sum_{\substack{J^2 \in \D^2 \\ J^2 \supset K^2}}
T(1 \otimes \phi_{J^2}, 1 \otimes \mathbf{1}_{J^2}), h_{K^1} \otimes h_{K^2} \rangle
\\
&\quad\times 
\langle h_{I^1} \rangle_{J^1}  
\big\langle \langle f_1, h_{I^2} \rangle \big\rangle_{I^1}
\big\langle \langle f_2 \rangle_{J^2}, h_{I^1} \big\rangle
\langle f_3, h_{K^1} \otimes h_{K^2} \rangle. 
\end{align*}
Then by \eqref{gghh}, 
\begin{align*}
\mathscr{S}_{6, 2}^1 + \mathscr{S}_{6, 2}^3 
&= - \sum_{\substack{K^1 \in \D_{\rm{good}}^1 \\ K^2 \in \D_{\rm{good}}^2}}
\sum_{\substack{J^2 \in \D^2 \\ J^2 \supset K^2}}
\langle T(1, 1 \otimes \mathbf{1}_{(J^2)^c}), h_{K^1} \otimes h_{K^2} \rangle
\\
&\quad\times 
\langle h_{I^2} \rangle_{J^2}  
\big\langle \langle f_1, h_{I^2} \rangle \big\rangle_{K^1}
\big\langle \langle f_2 \rangle_{J^2} \big\rangle_{K^1}
\langle f_3, h_{K^1} \otimes h_{K^2} \rangle, 
\end{align*}
and 
\begin{align*}
\mathscr{S}_{6, 4}^1 + \mathscr{S}_{6, 4}^3 
&= \sum_{\substack{K^1 \in \D_{\rm{good}}^1 \\ K^2 \in \D_{\rm{good}}^2}}
\sum_{\substack{J^2 \in \D^2 \\ J^2 \supset K^2}}
T(1 \otimes \phi_{J^2}, 1 \otimes \mathbf{1}_{J^2}), h_{K^1} \otimes h_{K^2} \rangle
\\
&\quad\times 
\big\langle \langle f_1, h_{I^2} \rangle \big\rangle_{K^1}
\big\langle \langle f_2 \rangle_{J^2} \big\rangle_{K^1}
\langle f_3, h_{K^1} \otimes h_{K^2} \rangle.  
\end{align*}

Recall that $I^i := (J^i)^{(1)}$, $i=1, 2$. Let $b_{K^2, Q} := - \langle h_{(Q^2)^{(1)}} \rangle_{Q^2} \langle T(1, 1 \otimes \mathbf{1}_{(Q^2)^c}), h_{Q^1} \otimes h_{K^2} \rangle$ and then set 
\begin{align*}
a_{K^2, Q} 
= \frac{b_{K^2, Q}}{C_0 [\ell(K^2)/\ell(Q^2)]^{\delta_2/2}} 
\end{align*}
if $K^1 \in \D_{\rm{good}}^1$, $Q^2 \in \D^2$ and $K^2 \in \D_{\rm{good}}^2$ satisfy $K^2 \subset Q^2$, and otherwise set $a_{K^2, Q} = 0$. Thus, it follows from Lemma \ref{lem:PPNN-1} that 
\begin{align*}
\mathscr{S}_{6, 2}^1 + \mathscr{S}_{6, 2}^3 
&= C_0 \sum_{k_2=0}^{\infty} 2^{-k_2 \frac{\delta_2}{2}}
\sum_{\substack{Q^1 \in \D^1 \\ Q^2 \in \D^2}} 
\sum_{K^2 \in \D_{k_2}(Q^2)} a_{K^2, Q} 
\\
&\quad\times \big\langle f_1, \overline{h}_{Q^1} \otimes h_{(Q^2)^{(1)}} \big\rangle 
\langle f_2, \overline{h}_{Q^1} \otimes \overline{h}_{Q^2} \rangle 
\langle f_3, h_{Q^1} \otimes h_{K^2} \rangle
\\
&= C_0 \sum_{k_2=1}^{\infty} 2^{-k_2 \frac{\delta_2}{2}}
\big\langle \mathbf{P}_{\D}^{2, 0, 1, k_2}(f_1, f_2), f_3 \big\rangle,  
\end{align*}
where 
\begin{align*}
\mathcal{F}^2(Q^2) := F_1^2(\ell(Q^2)) \widetilde{F}_2^2(2^{-k_2} \ell(Q^2)) \widetilde{F}_3^2(Q^2). 
\end{align*}
Analogously, replacing $b_{K^2, Q}$ above by $b_{K^2, Q} := T(1 \otimes \phi_{Q^2}, 1 \otimes \mathbf{1}_{Q^2}), h_{Q^1} \otimes h_{K^2} \rangle$, we invoke Lemma \ref{lem:PPNN-2} to obtain the same partial paraproduct:   
\begin{align*}
\mathscr{S}_{6, 4}^1 + \mathscr{S}_{6, 4}^3 
= C_0 \sum_{k_2=1}^{\infty} 2^{-k_2 \frac{\delta_2}{2}}
\big\langle \mathbf{P}_{\D}^{2, 0, 1, k_2}(f_1, f_2), f_3 \big\rangle. 
\end{align*}

\noindent{\bf $\bullet$ Bilinear dyadic shifts.} 
We present the estimates below in order to analyze the remaining terms $\mathscr{S}_j^1$, $j=6, 7, 8, 9$.

\begin{lemma}\label{lem:NN}
For every $j=6, 7, 8, 9$, there holds 
\begin{align*}
|\G_{I, J, K}^j| 
\lesssim \prod_{i=1}^2 \bigg[ \frac{\ell(K^i)}{\ell(J^i)} \bigg]^{\frac{\delta_i}{2}} 
F^i(K^i, J^i) \frac{|K^i|^{\frac12}}{|J^i|}, 
\end{align*}
where $F^i$ is defined in \eqref{def:FKQ}.
\end{lemma}

\begin{proof} 
We only consider the most difficult one $\G_{I, J, K}^6$. First, treat the case $\ell(J^1) > 2^{\vartheta} \ell(K^1)$ and $\ell(J^2) > 2^{\vartheta} \ell(K^2)$. In view of that $1 = \mathbf{1}_{3K^i} + \mathbf{1}_{(3K^i)^c}$, $i=1, 2$, we split 
\begin{align}\label{G6}
\G_{I, J, K}^6 
= \sum_{j=1}^4 \G_{I, J, K}^{6, j} 
\end{align}
where 
\begin{align*}
\G_{I, J, K}^{6, 1}  
&:= |J^1|^{-\frac12} |J^2|^{-\frac12} 
\langle h_{I^1} \rangle_{J^1}  \langle h_{I^2} \rangle_{J^2}  
\langle T(\mathbf{1}_{3K^1} \otimes \mathbf{1}_{3K^2}, 
\mathbf{1}_{(J^1)^c} \otimes \mathbf{1}_{(J^2)^c}), 
h_{K^1} \otimes h_{K^2} \rangle, 
\\
\G_{I, J, K}^{6, 2}  
&:= |J^1|^{-\frac12} |J^2|^{-\frac12} 
\langle h_{I^1} \rangle_{J^1}  \langle h_{I^2} \rangle_{J^2}  
\langle T(\mathbf{1}_{3K^1} \otimes \mathbf{1}_{(3K^2)^c}, 
\mathbf{1}_{(J^1)^c} \otimes \mathbf{1}_{(J^2)^c}), 
h_{K^1} \otimes h_{K^2} \rangle, 
\\
\G_{I, J, K}^{6, 3}  
&:= |J^1|^{-\frac12} |J^2|^{-\frac12} 
\langle h_{I^1} \rangle_{J^1}  \langle h_{I^2} \rangle_{J^2}  
\langle T(\mathbf{1}_{(3K^1)^c} \otimes \mathbf{1}_{3K^2}, 
\mathbf{1}_{(J^1)^c} \otimes \mathbf{1}_{(J^2)^c}), 
h_{K^1} \otimes h_{K^2} \rangle, 
\\
\G_{I, J, K}^{6, 4}  
&:= |J^1|^{-\frac12} |J^2|^{-\frac12} 
\langle h_{I^1} \rangle_{J^1}  \langle h_{I^2} \rangle_{J^2}  
\langle T(\mathbf{1}_{(3K^1)^c} \otimes \mathbf{1}_{(3K^2)^c}, 
\mathbf{1}_{(J^1)^c} \otimes \mathbf{1}_{(J^2)^c}), 
h_{K^1} \otimes h_{K^2} \rangle. 
\end{align*}
Recall the goodness of $K^i$ and the estimate \eqref{dxz}. Then by the compact full kernel representation (cf. \eqref{H1}), the H\"{o}lder condition, and Lemma \ref{lem:RR} applied to $\ell_i = \ell(K^i)^{\frac12} \ell(J^i)^{\frac12}$, there holds 
\begin{align}\label{G61}
\G_{I, J, K}^{6, 1} 
&\le \prod_{i=1}^2 \mathscr{R}_i^1(J^i, K^i) 
|I^i|^{-\frac12} |J^i|^{-\frac12} |K^i|^{-\frac12} 
\\ \nonumber 
&\lesssim \prod_{i=1}^2 \bigg[ \frac{\ell(K^i)}{\ell(J^i)} \bigg]^{\frac{\delta_i}{2}} 
F^i(K^i, J^i) \frac{|K^i|^{\frac12}}{|J^i|}. 
\end{align}
Replacing $\mathscr{R}_i^1$ in \eqref{G61} by $\mathscr{R}_i^i$ and $\mathscr{R}_i^2$ respectively, we obtain 
\begin{align}\label{G62}
|\G_{I, J, K}^{6, 2}| + |\G_{I, J, K}^{6, 4}|
\lesssim \prod_{i=1}^2 \bigg[ \frac{\ell(K^i)}{\ell(J^i)} \bigg]^{\frac{\delta_i}{2}} 
F^i(K^i, J^i) \frac{|K^i|^{\frac12}}{|J^i|}. 
\end{align}
Symmetrically to $\G_{6, 2}$, one has 
\begin{align}\label{G63}
|\G_{I, J, K}^{6, 3}|
\lesssim \prod_{i=1}^2 \bigg[ \frac{\ell(K^i)}{\ell(J^i)} \bigg]^{\frac{\delta_i}{2}} 
F^i(K^i, J^i) \frac{|K^i|^{\frac12}}{|J^i|}. 
\end{align}

Next, in the case $\ell(J^1) > 2^{\vartheta} \ell(K^1)$ and $\ell(J^2) \le 2^{\vartheta} \ell(K^2)$, as seen above, it suffices to consider $\G_{I, J, K}^{6, 1}$. The fact $\mathbf{1}_{(J^2)^c} = \mathbf{1}_{3J^2 \setminus J^2} + \mathbf{1}_{(3J^2)^c}$ allows us to write  
\begin{align}\label{G66112}
\G_{I, J, K}^{6, 1} = \G_{I, J, K}^{6, 1, 1} + \G_{I, J, K}^{6, 1, 2},  
\end{align}
where 
\begin{align*}
\G_{I, J, K}^{6, 1, 1}  
&:= |J^1|^{-\frac12} |J^2|^{-\frac12} 
\langle h_{I^1} \rangle_{J^1}  \langle h_{I^2} \rangle_{J^2}  
\langle T(\mathbf{1}_{3K^1} \otimes \mathbf{1}_{3K^2}, 
\mathbf{1}_{(J^1)^c} \otimes \mathbf{1}_{3J^2 \setminus J^2}), 
h_{K^1} \otimes h_{K^2} \rangle, 
\\
\G_{I, J, K}^{6, 1, 2}  
&:= |J^1|^{-\frac12} |J^2|^{-\frac12} 
\langle h_{I^1} \rangle_{J^1}  \langle h_{I^2} \rangle_{J^2}  
\langle T(\mathbf{1}_{3K^1} \otimes \mathbf{1}_{3K^2}, 
\mathbf{1}_{(J^1)^c} \otimes \mathbf{1}_{(3J^2)^c}), 
h_{K^1} \otimes h_{K^2} \rangle. 
\end{align*}
From the compact full kernel representation, the mixed size-H\"{o}lder condition, Lemmas \ref{lem:QQ} and  \ref{lem:RR} applied to $\ell_1 = \ell(K^1)^{\frac12} \ell(J^1)^{\frac12}$, it follows that 
\begin{align}\label{G611}
|\G_{I, J, K}^{6, 1, 1}| 
&\le \mathscr{R}_1^1(J^1, K^1)  \sum_{j=0}^{3^{n_2}-1} 
\sum_{k=1}^{3^{n_2}-1} \mathscr{Q}_2(K_j^2, J_k^2, K^2) 
\prod_{i=1}^2 |I^i|^{-\frac12} |J^i|^{-\frac12} |K^i|^{-\frac12}
\\ \nonumber 
&\lesssim \prod_{i=1}^2 \bigg[ \frac{\ell(K^i)}{\ell(J^i)} \bigg]^{\frac{\delta_i}{2}} 
F^i(K^i, J^i) \frac{|K^i|^{\frac12}}{|J^i|}. 
\end{align}
Moreover, using the compact full kernel representation, the H\"{o}lder condition, and Lemma \ref{lem:RR} applied to $\ell_1 = \ell(K^1)^{\frac12} \ell(J^1)^{\frac12}$ and $\ell_2 = \ell(K^2)$, we have 
\begin{align}\label{G612}
|\G_{I, J, K}^{6, 1, 2}|
&\le \mathscr{R}_1^1(J^1, K^1) \mathscr{R}_2^1(3J^2, K^2)
\prod_{i=1}^2 |I^i|^{-\frac12} |J^i|^{-\frac12} |K^i|^{-\frac12}
\\ \nonumber 
&\lesssim \prod_{i=1}^2 \bigg[ \frac{\ell(K^i)}{\ell(J^i)} \bigg]^{\frac{\delta_i}{2}} 
F^i(K^i, J^i) \frac{|K^i|^{\frac12}}{|J^i|}. 
\end{align}

Symmetrically, one can handle the case $\ell(J^1) \le 2^{\vartheta} \ell(K^1)$ and $\ell(J^2) > 2^{\vartheta} \ell(K^2)$. 

Finally, to deal with the case $\ell(J^1) \le 2^{\vartheta} \ell(K^1)$ and $\ell(J^2) \le 2^{\vartheta} \ell(K^2)$, by similarly, it is enough to treat $\G_{I, J, K}^{6, 1}$. In light of \eqref{G66112} and $\mathbf{1}_{(J^1)^c} = \mathbf{1}_{3J^1 \setminus J^1} + \mathbf{1}_{(3J^1)^c}$, we further split $\G_{I, J, K}^{6, 1, 1}$ and $\G_{I, J, K}^{6, 1, 2}$ as 
\begin{align}\label{G61212}
\G_{I, J, K}^{6, 1, 1} = \G_{I, J, K}^{6, 1, 1, 1} + \G_{I, J, K}^{6, 1, 1, 2} 
\quad \text{ and } \quad 
\G_{I, J, K}^{6, 1, 2} = \G_{I, J, K}^{6, 1, 2, 1} + \G_{I, J, K}^{6, 1, 2, 2},  
\end{align}
where  
\begin{align*}
\G_{I, J, K}^{6, 1, 1, 1}  
&:= |J^1|^{-\frac12} |J^2|^{-\frac12} 
\langle h_{I^1} \rangle_{J^1}  \langle h_{I^2} \rangle_{J^2}  
\langle T(\mathbf{1}_{3K^1} \otimes \mathbf{1}_{3K^2}, 
\mathbf{1}_{3J^1 \setminus J^1} \otimes \mathbf{1}_{3J^2 \setminus J^2}), 
h_{K^1} \otimes h_{K^2} \rangle, 
\\
\G_{I, J, K}^{6, 1, 1, 2}  
&:= |J^1|^{-\frac12} |J^2|^{-\frac12} 
\langle h_{I^1} \rangle_{J^1}  \langle h_{I^2} \rangle_{J^2}  
\langle T(\mathbf{1}_{3K^1} \otimes \mathbf{1}_{3K^2}, 
\mathbf{1}_{(3J^1)^c} \otimes \mathbf{1}_{3J^2 \setminus J^2}), 
h_{K^1} \otimes h_{K^2} \rangle, 
\\
\G_{I, J, K}^{6, 1, 2, 1}  
&:= |J^1|^{-\frac12} |J^2|^{-\frac12} 
\langle h_{I^1} \rangle_{J^1}  \langle h_{I^2} \rangle_{J^2}  
\langle T(\mathbf{1}_{3K^1} \otimes \mathbf{1}_{3K^2}, 
\mathbf{1}_{3J^1 \setminus J^1} \otimes \mathbf{1}_{(3J^2)^c}), 
h_{K^1} \otimes h_{K^2} \rangle, 
\\
\G_{I, J, K}^{6, 1, 2, 2}  
&:= |J^1|^{-\frac12} |J^2|^{-\frac12} 
\langle h_{I^1} \rangle_{J^1}  \langle h_{I^2} \rangle_{J^2}  
\langle T(\mathbf{1}_{3K^1} \otimes \mathbf{1}_{3K^2}, 
\mathbf{1}_{(3J^1)^c} \otimes \mathbf{1}_{(3J^2)^c}), 
h_{K^1} \otimes h_{K^2} \rangle. 
\end{align*}
Lemma \ref{lem:QQ}, along with the compact full kernel representation and the size condition, gives  
\begin{align}\label{G6111}
|\G_{I, J, K}^{6, 1, 1, 1}|
&\le \prod_{i=1}^2 \sum_{j=0}^{3^{n_i}-1} 
\sum_{k=1}^{3^{n_i}-1} \mathscr{Q}_i(K_j^i, J_k^i, K^i)  
|I^i|^{-\frac12} |J^i|^{-\frac12} |K^i|^{-\frac12}
\\ \nonumber 
&\lesssim \prod_{i=1}^2 \bigg[ \frac{\ell(K^i)}{\ell(J^i)} \bigg]^{\frac{\delta_i}{2}} 
F^i(K^i, J^i) \frac{|K^i|^{\frac12}}{|J^i|}. 
\end{align}
In light of the compact full kernel representation and the mixed size-H\"{o}lder condition, Lemmas \ref{lem:QQ} and \ref{lem:RR} applied to $\ell_1 = \ell(K^1)$ imply   
\begin{align}\label{G6112}
|\G_{I, J, K}^{6, 1, 1, 2}|
&\le \mathscr{R}_1^1(3J^1, K^1) \sum_{j=0}^{3^{n_2}-1} 
\sum_{k=1}^{3^{n_2}-1} \mathscr{Q}_2(K_j^2, J_k^2, K^2)
\prod_{i=1}^2 |I^i|^{-\frac12} |J^i|^{-\frac12} |K^i|^{-\frac12}
\\ \nonumber 
&\lesssim F^1(K^1, J^1) |K^1|^{-\frac12}  
\bigg[ \frac{\ell(K^2)}{\ell(J^2)} \bigg]^{\frac{\delta_2}{2}} 
F^2(K^2, J^2) \frac{|K^2|^{\frac12}}{|J^2|}
\\ \nonumber 
&\lesssim \prod_{i=1}^2 \bigg[ \frac{\ell(K^i)}{\ell(J^i)} \bigg]^{\frac{\delta_i}{2}} 
F^i(K^i, J^i) \frac{|K^i|^{\frac12}}{|J^i|}. 
\end{align}
Symmetrically,  
\begin{align}\label{G6121}
|\G_{I, J, K}^{6, 1, 2, 1}|
\lesssim \prod_{i=1}^2 \bigg[ \frac{\ell(K^i)}{\ell(J^i)} \bigg]^{\frac{\delta_i}{2}} 
F^i(K^i, J^i) \frac{|K^i|^{\frac12}}{|J^i|}. 
\end{align}
Additionally, invoking the compact full kernel representation, the H\"{o}lder condition, and Lemma \ref{lem:RR} applied to $\ell_1 = \ell(K^1)$ and $\ell_2 = \ell(K^2)$, we arrive at 
\begin{align}\label{G6122}
|\G_{I, J, K}^{6, 1, 2, 2}|
&\le \prod_{i=1}^2 \mathscr{R}_i^1(3J^i, K^i) 
|I^i|^{-\frac12} |J^i|^{-\frac12} |K^i|^{-\frac12}
\\ \nonumber 
&\lesssim \prod_{i=1}^2 F^i(K^i, J^i) |K^i|^{-\frac12}
\lesssim \prod_{i=1}^2 \bigg[ \frac{\ell(K^i)}{\ell(J^i)} \bigg]^{\frac{\delta_i}{2}} 
F^i(K^i, J^i) \frac{|K^i|^{\frac12}}{|J^i|}. 
\end{align}
The desired estimate for $\G_{I, J, K}^6$ is a consequence of \eqref{G6}--\eqref{G6122}.
\end{proof}

Now set 
\begin{align*}
a_{K, J} 
= \frac{\sum_{j=6}^9 \G_{I, J, K}^j}{C_0 \prod_{i=1}^2 [\ell(K^i)/\ell(J^i)]^{\delta_i/2}} 
\end{align*}
if $J^i \in \D^i$ and $K^i \in \D_{\rm{good}}^i$ satisfy $K^i \subset J^i$ for each $i=1, 2$, and otherwise set $a_{K, J} = 0$. If we define $\mathcal{F}(Q)$ as in \eqref{SS-FQ}, then Lemma \ref{lem:NN} implies  \eqref{FQ-1} and \eqref{FQ-2} hold whenever $C_0$ is large. Consequently, we conclude  
\begin{align*}
\sum_{j=6}^9 \mathscr{S}_{6, j}^1
&= C_0 \sum_{k_1=0}^{\infty} \sum_{k_2=0}^{\infty} 
2^{-k_1 \frac{\delta_1}{2}} 2^{-k_2 \frac{\delta_2}{2}}
\sum_{\substack{Q^1 \in \D^1 \\ Q^2 \in \D^2}} 
\sum_{\substack{K^1 \in \D_{k_1}(Q^1) \\ K^2 \in \D_{k_2}(Q^2)}} 
a_{K, Q} 
\\
&\quad\times \big\langle f_1, h_{(Q^1){(1)}} \otimes h_{(Q^2)^{(1)}} \big\rangle 
\langle f_2, h^0_{Q^1} \otimes h^0_{Q^2} \rangle 
\langle f_3, h_{K^1} \otimes h_{K^2} \rangle
\\
&= C_0 \sum_{k_1=1}^{\infty} \sum_{k_2=1}^{\infty} 
2^{-k_1 \frac{\delta_1}{2}} 2^{-k_2 \frac{\delta_2}{2}}
\big\langle \mathbf{S}_{\D}^{\sigma(k)}(f_1, f_2), f_3 \big\rangle, 
\end{align*}
where $\sigma(k) = \big((0, 0), (1, 1), (k_1, k_2) \big)$.

\section{Weighted compactness of dyadic operators}\label{sec:wcpt}
The goal of this section is to show Theorem \ref{thm:dyadic-cpt}, for which we have to first give  characterizations of compactness in the bi-parameter setting.

\subsection{Compactness criterions}\label{sec:cc}
Let $\tau_v$ be the translation operator, i.e., $\tau_v f(x) := f(x-v)$ for all $x, v \in \Rnn$. 

\begin{theorem}\label{thm:KRWA}
Let $w \in A_{p_0}(\Rnn)$ for some $p_0 \in (1, \infty)$. Let $0<p<\infty$ and $\K \subseteq L^p(w)$. Let $0<a<   \min\{p/p_0, 1\}$. Then $\K$ is precompact in $L^p(w)$ if and only if the following are satisfied:
\begin{list}{\rm (\theenumi)}{\usecounter{enumi}\leftmargin=1.2cm \labelwidth=1cm \itemsep=0.2cm \topsep=0.2cm \renewcommand{\theenumi}{\alph{enumi}}}
 
\item ${\displaystyle \sup_{f \in \K} \|f\|_{L^p(w)} < \infty}$, 

\item ${\displaystyle \lim_{A \to \infty} \sup_{f \in \K} 
\|f \mathbf{1}_{B_{\vec{n}}(0, A)^c}\|_{L^p(w)}=0}$, 

\item ${\displaystyle \lim_{r \to 0} \sup_{f \in \K} 
\bigg\| \bigg( \fint_{B_{\vec{n}}(0, r)} |\tau_y f - f|^a \, dy \bigg)^{\frac1a} \bigg\|_{L^p(w)} = 0}$. 
\end{list}
\end{theorem}

\begin{proof}
The proof is similar to that of \cite[Theorem 1.5]{CLSY}. As argued there, by means of a rescaling argument, the core of the proof is Theorem \ref{thm:RKB} below.
\end{proof}

\begin{theorem}\label{thm:RKB}
Let $p \in (1, \infty)$, $w \in A_p(\Rnn)$,  and $\K \subset L^p(w)$. Then $\K$ is precompact in $L^p(w)$ if and only if the following are satisfied:
\begin{list}{\rm (\theenumi)}{\usecounter{enumi}\leftmargin=1.2cm \labelwidth=1cm \itemsep=0.2cm \topsep=0.2cm \renewcommand{\theenumi}{\alph{enumi}}}
			
\item[\textup{(a)\phantom{$'$}}] ${\displaystyle \sup_{f \in \K} \|f\|_{L^p(w)} < \infty}$, 

\item[\textup{(b)\phantom{$'$}}] ${\displaystyle \lim_{A \to \infty} \sup_{f \in \K} 
\|f \mathbf{1}_{B_{\vec{n}}(0, A)^c}\|_{L^p(w)}=0}$, 

\item[\textup{(c)$'$}] ${\displaystyle \lim_{r \to 0} \sup_{f \in \K} 
\|f - \langle f \rangle_{B_{\vec{n}}(\cdot, r)}\|_{L^p(w)}=0}$. 
\end{list}
\end{theorem}

\begin{proof}
This can be shown as \cite[Theorem 4.1]{CLSY} together with the weighted boundedness of strong maximal operators (cf. \cite[p. 453]{GR}).
\end{proof}

Besides, in the unweighted case, the following gives a slightly different characterization of precompactness, which contains quasi-Banach exponents.

\begin{theorem}\label{thm:KRLp}
Let $p \in (0, \infty)$ and $\K \subset L^p(\Rnn)$. Then $\K$ is precompact in $L^p(\Rnn)$ if and only if the following are satisfied:
\begin{list}{\rm (\theenumi)}{\usecounter{enumi}\leftmargin=1.2cm \labelwidth=1cm \itemsep=0.2cm \topsep=0.2cm \renewcommand{\theenumi}{\alph{enumi}}}
 
\item[\textup{(a)\phantom{$''$}}] ${\displaystyle \sup_{f \in \K} \|f\|_{L^p} < \infty}$, 

\item[\textup{(b)\phantom{$''$}}] ${\displaystyle \lim_{A \to \infty} \sup_{f \in \K} 
\|f \mathbf{1}_{B_{\vec{n}}(0, A)^c}\|_{L^p}=0}$, 

\item[\textup{(c)$''$}] $\displaystyle \lim_{|v| \to 0} \sup_{f \in \K} 
\|\tau_v f - f \|_{L^p}=0$. 
\end{list}  
\end{theorem}

\begin{proof}
The proof is almost the same as that of \cite[Theorem 1.3]{CLSY}.
\end{proof}

In terms of quasi-Banach exponents, we present a variation of Theorem \ref{thm:KRWA}.

\begin{theorem}\label{thm:KRttt}
Let $w \in A_{p_0}(\Rnn)$ for some $p_0 \in (1, \infty)$. Let $p \in (0, \infty)$, $\K \subseteq L^p(w)$, and $0 < a \le p/p_0$. If $\K$ is precompact in $L^p(w)$, then the following hold:
\begin{list}{\rm (\theenumi)}{\usecounter{enumi}\leftmargin=1cm \labelwidth=1cm \itemsep=0.2cm \topsep=0.2cm \renewcommand{\theenumi}{\alph{enumi}}}
 
\item[\textup{(a)\phantom{$'''$}}] ${\displaystyle \sup_{f \in \K} \|f\|_{L^p(w)} < \infty}$, 

\item[\textup{(b)\phantom{$'''$}}] ${\displaystyle \lim_{A \to \infty} \sup_{f \in \K} 
\|f \mathbf{1}_{B_{\vec{n}}(0, A)^c}\|_{L^p(w)}=0}$, 

\item[\textup{(c)$'''$}] ${\displaystyle \lim_{r \to 0} \sup_{f \in \K} 
\bigg\| \bigg[ \fint_{B_{\vec{n}}(0, r)} |(\tau_y - \tau_{y_1} - \tau_{y_2} +I) f|^a \, dy \bigg]^{\frac1a} \bigg\|_{L^p(w)}= 0}$. 
\end{list}
\end{theorem}

\begin{proof}
Suppose that $\K$ is precompact in $L^p(w)$. Then $\K$ is totally bounded. Given $\varepsilon>0$, there exists a finite number of functions $\{f_k\}_{k=1}^N \subset \K$ such that
\begin{align*}
\K \subseteq \bigcup_{k=1}^N 
\big\{f \in L^p(w): \|f - f_k\|_{L^p(w)} < \varepsilon \big\}.  
\end{align*}
Let $f \in \K$ be an arbitrary function. Then there exists some $k \in \{1, \ldots, N\}$ so that 
\begin{align}\label{eq:fk-f}
\|f - f_k\|_{L^p(w)} \le \varepsilon.
\end{align}
which gives 
\begin{align*}
\|f\|_{L^p(w)} 
\le \|f - f_k\|_{L^p(w)} + \|f_k\|_{L^p(w)} 
\le \varepsilon + \max_{1 \leq k \leq N} \|f_k\|_{L^p(w)}. 
\end{align*}
This shows the condition (a). Since $\mathscr{C}_c^{\infty}(\Rnn)$ is dense in $L^p(w)$, there exists $g_k \in \mathscr{C}_c^{\infty}(\Rnn)$ such that 
\begin{align*}
\|f_k - g_k\|_{L^p(w)} \le \varepsilon, 
\end{align*}
which together with \eqref{eq:fk-f} implies  
\begin{align}\label{wlp-1}
\|f - g_k\|_{L^p(w)} 
\le \|f - f_k\|_{L^p(w)} + \|f_k - g_k\|_{L^p(w)} 
\lesssim \varepsilon. 
\end{align}
Let $\supp g_k \subset B_{\vec{n}}(0, A_k)$ for each $k=1, \ldots, N$, and $A_0 := \max\{A_1, \ldots, A_N\}$. Then for any $A \ge A_0$, by \eqref{eq:fk-f} and \eqref{wlp-1},  
\begin{align*}
\|f \mathbf{1}_{B_{\vec{n}}(0, A)^c}\|_{L^p(w)} 
\le \|f - g_k\|_{L^p(w)} + \|g_k \mathbf{1}_{B_{\vec{n}}(0, A)^c}\|_{L^p(w)} 
\lesssim \varepsilon. 
\end{align*}
which justifies the condition (b). To proceed, we split 
\begin{align}\label{wlp-2}
\bigg\| \bigg[ \fint_{B_{\vec{n}}(0, r)} |(\tau_y - \tau_{y_1} - \tau_{y_2} +I) f|^a \, dy \bigg]^{\frac1a} \bigg\|_{L^p(w)}
\lesssim \sum_{j=1}^5 \mathscr{I}_j,
\end{align}
where
\begin{align*}
\mathscr{I}_1 
& := \bigg\| \bigg[ \fint_{B_{\vec{n}}(0, r)} |(\tau_y - \tau_{y_1} - \tau_{y_2} +I) g_k|^a \, dy \bigg]^{\frac1a} \bigg\|_{L^p(w)}, 
\\
\mathscr{I}_2 
& := \bigg\| \bigg[ \fint_{B_{\vec{n}}(0, r)} |\tau_y f - \tau_y g_k|^a \, dy \bigg]^{\frac1a} \bigg\|_{L^p(w)},
\\
\mathscr{I}_3 
& := \bigg\| \bigg[ \fint_{B_{\vec{n}}(0, r)} |\tau_{y_1} f - \tau_{y_1} g_k|^a \, dy \bigg]^{\frac1a} \bigg\|_{L^p(w)},
\\
\mathscr{I}_4 
& := \bigg\| \bigg[ \fint_{B_{\vec{n}}(0, r)} |\tau_{y_2} f - \tau_{y_2} g_k|^a \, dy \bigg]^{\frac1a} \bigg\|_{L^p(w)},
\\
\mathscr{I}_5 
& := \bigg\| \bigg[ \fint_{B_{\vec{n}}(0, r)} |f - g_k|^a \, dy \bigg]^{\frac1a} \bigg\|_{L^p(w)}.
\end{align*}
The estimate \eqref{wlp-1} gives 
\begin{align}\label{wlp-11}
\mathscr{I}_5 
= \|f - g_k\|_{L^p(w)} 
\lesssim \varepsilon. 
\end{align}
Since $g_k \in \mathscr{C}_c^{\infty}(\Rnn)$, there exists some $r_0>0$ so that 
\begin{align*}
\Xi_0 := \sup_{\substack{|x_1 - y_1| < r_0 \\ |x_2 - y_2| < r_0}} 
|g_k(x_1, x_2) - g_k(x_1, y_2) - g_k(x_2, y_1) + g_k(y_1, y_2)| 
\le \varepsilon \, w(B_{\vec{n}}(0, 2A_0))^{-\frac1p}.
\end{align*}
Hence, for any $0<r_1, r_2<\min\{r_0, A_0\}$, this in turn leads to 
\begin{align}\label{wlp-3}
\mathscr{I}_1 
\le \Xi_0 \, \|\mathbf{1}_{B_{\vec{n}}(0, 2A_0)}\|_{L^p(w)}
\le  \varepsilon.
\end{align}
By the fact $w \in A_{p_0}(\Rnn) \subset A_{p/a}(\Rnn)$ and the weighted boundedness of strong maximal operators (cf. \cite[p. 453]{GR}), we have
\begin{align}\label{wlp-4}
\mathscr{I}_2
\le \big\|M_{\mathcal{R}}(|f - g_k|^a)^{\frac1a}\big\|_{L^p(w)}
\lesssim \|f - g_k\|_{L^p(w)}
\lesssim \varepsilon.
\end{align}
Note that for any $q \in (1, \infty)$,
\begin{align}\label{vvqq}
[v]_{A_q(\Rnn)} \simeq 
\max\Big\{\esssup_{x_1 \in \R^{n_1}} [v(x_1, \cdot)]_{A_q(\R^{n_2})}, \, 
\esssup_{x_2 \in \R^{n_2}} [v(\cdot, x_2)]_{A_q(\R^{n_1})} \Big\}.
\end{align} 
For any $x = (x_1, x_2) \in \Rnn$, simply denote $f_{x_1}(x_2) := f(x_1, x_2) =: f_{x_2}(x_1)$. Let $M_{n_i}$ be the Hardy--Littlewood maximal operator on $\R^{n_i}$, $i=1, 2$. Then \eqref{vvqq}, along with the weighted boundedness of maximal operator, yields 
\begin{align}\label{wlp-5}
\mathscr{I}_3^p
& \leq \int_{\R^{n_2}} \bigg[\int_{\R^{n_1}} 
M_{n_1}(|f_{x_2} - g_{k, x_2}|^a)(x_1)^{\frac{p}{a}} \, 
w(x_1, x_2) \, dx_1 \bigg] dx_2
\\ \nonumber
& \lesssim \int_{\R^{n_2}} \bigg[\int_{\R^{n_1}} 
(|f_{x_2} - g_{k, x_2}|^a)(x_1)^{\frac{p}{a}} \, 
w(x_1, x_2) \, dx_1 \bigg] dx_2
\\ \nonumber
& = \|f - g_k\|_{L^p(w)}^p 
\lesssim \varepsilon^p. 
\end{align}
Symmetrically, 
\begin{align}\label{wlp-6}
\mathscr{I}_4 \lesssim \varepsilon. 
\end{align}
Thus, the condition $\rm{(c)}'''$ follows from \eqref{wlp-2}--\eqref{wlp-6}. 
\end{proof}

\subsection{An uniform estimate for shifts}\label{sec:un}
For any $\D_0 := \D^1_0 \times \D^2_0 \subset \D^1 \times \D^2 =: \D$ and $k=(k_1, \ldots, k_{m+1})$ with $k_j = (k_j^1, k_j^2) \in \N^2$, we define
\begin{align*}
\mathbb{S}_{\D, \D_0}^k (\vec{f})
:= \sum_{Q \in \D_0}A_Q^k (\vec{f})
:= \sum_{Q = Q^1 \times Q^2 \in \D_0} 
\sum_{\substack{I_j \in \D_{k_j}(Q) \\ j=1, \ldots, m+1}}  a_{(I_j), Q} 
\prod_{j=1}^m \langle f_j, \widetilde{h}_{I_j^1} \otimes \widetilde{h}_{I_j^2} \rangle \, 
\widetilde{h}_{I_{m+1}^1} \otimes \widetilde{h}_{I_{m+1}^2},
\end{align*}
where for each $i=1, 2$, there exist two different indices $j_0^i, j_1^i \in \{1, \ldots, m+1\}$ so that $\widetilde{h}_{I_{j_0^i}^i} = h_{I_{j_0^i}^i}$, $\widetilde{h}_{I_{j_1^i}^i} = h_{I_{j_1^i}^i}$, and  $\widetilde{h}_{I_j^i} \in \{h_{I_j^i}^0, h_{I_j^i}\}$ for every $j \neq j_0^i, j_1^i$. Moreover, the coefficients $a_{(I_j), Q}$ satisfy  
\begin{align*}
|a_{(I_j), Q}| 
\le \mathcal{F}_0
\frac{\prod_{j=1}^{m+1} |I_j|^{\frac12}}{|Q|^m}, \quad \forall Q \in \D_0. 
\end{align*}

\begin{lemma}\label{lem:SDD}
There holds
\begin{align*}
\sup_k \sup_{\D} \sup_{\D_0 \subset \D} 
\|\mathbb{S}_{\D, \D_0}^k (\vec{f})\|_{L^r}
\lesssim \mathcal{F}_0 \prod_{j=1}^m \|f_j\|_{L^{r_j}},  
\end{align*}
for all $\frac1r = \sum_{j=1}^m \frac{1}{r_j}$ with $r, r_1, \ldots, r_m \in (1, \infty)$.

\end{lemma}

\begin{proof}
Let $r, r_1, \ldots, r_m \in (1, \infty)$ such that $\frac1r = \sum_{j=1}^m \frac{1}{r_j}$. Since our argument depends on whether $\widetilde{h}_{I_{m+1}^1}$ and $\widetilde{h}_{I_{m+1}^2}$ are cancellative or not, it suffices to treat two cases: 
\begin{align*}
& \textbf{Case 1: } 
\widetilde{h}_{I_{m+1}^1} = h_{I_{m+1}^1}
\quad \text{ or } \quad 
\widetilde{h}_{I_{m+1}^2} = h_{I_{m+1}^2}; 
\\
& \textbf{Case 2: } 
\widetilde{h}_{I_{m+1}^1} = h_{I_{m+1}^1}^0
\quad \text{ and } \quad 
\widetilde{h}_{I_{m+1}^2} = h_{I_{m+1}^2}^0. 
\end{align*}
First, in \textbf{Case 1}, by symmetry and similarly, we only focus on the case $\widetilde{h}_{I_{m+1}^1} = h_{I_{m+1}^1}$ and $\widetilde{h}_{I_1^1} = h_{I_1^1}$. By the orthogonality of Haar functions, it is easy to check that 
\begin{align}\label{DD}
\Delta_{Q^i}^{k^i} h_{I^i} 
= h_{I^i} \mathbf{1}_{\{I^i \in \D_{k^i}^i(Q^i)\}}
= h_{I^i} \mathbf{1}_{\{Q^i = K^i\}}, 
\quad \forall I^i \in \D_{k^i}^i(K^i), \quad i=1, 2, 
\end{align}
and then 
\begin{align}\label{D-1}
\Delta_{Q^1}^{k_{m+1}^1} 
\big(A_K^k(\vec{f}) \big) 
= A_K^k(\vec{f}) \mathbf{1}_{\{Q^1 = K^1\}} 
= A_K^k \big(\Delta_{Q^1}^{k_1^1} f_1, f_2, \ldots, f_m \big), 
\end{align}
which gives 
\begin{align}\label{D-2}
& \Delta_{Q^1}^{k_{m+1}^1} 
\big(\mathbb{S}_{\D, \D_0}^k (\vec{f}) \big) 
= A_Q^k(\vec{f}) \mathbf{1}_{\{Q \in \D_0\}}. 
\end{align}
Additionally, there holds 
\begin{align}\label{D-3}
& |A_Q^k (\vec{f})| 
\le \mathcal{F}_0 \prod_{j=1}^m \langle |f_j| \rangle_Q \, \mathbf{1}_Q, 
\quad\text{for all }Q \in \D_0.  
\end{align}
Thus, it follows from \eqref{ssf-2} and \eqref{D-1}--\eqref{D-3} that 
\begin{align}\label{sddk}
\|\mathbb{S}_{\D, \D'}^k (\vec{f})\|_{L^r}
&\simeq \bigg\|\bigg( \sum_{Q \in \D_0} 
\big|\Delta_{Q^1}^{k_{m+1}^1} 
 \big(\mathbb{S}_{\D, \D_0}^k (\vec{f}) \big) \big|^2 \bigg)^{\frac12} \bigg\|_{L^r} 
\\ \nonumber
&= \bigg\|\bigg( \sum_{Q \in \D_0} 
|A_Q^k (\vec{f})|^2 \bigg)^{\frac12}\bigg\|_{L^r}
\\ \nonumber
&= \bigg\|\bigg( \sum_{Q \in \D_0} 
|A_Q^k \big(\Delta_{Q^1}^{k_1^1} f_1, f_2, \ldots, f_m \big)|^2 
\bigg)^{\frac12}\bigg\|_{L^r}
\\ \nonumber
&\le \mathcal{F}_0 \bigg\|\bigg( \sum_{Q \in \D_0} 
\big\langle |\Delta_{Q^1}^{k_1^1} f_1| \big\rangle_Q^2 
\prod_{j=2}^m  \langle |f_j| \rangle_Q^2
\mathbf{1}_Q \bigg)^{\frac12}\bigg\|_{L^r}
\\ \nonumber
&\le \mathcal{F}_0 \bigg\|\bigg( \sum_{Q \in \D_0} 
|\Delta_{Q^1}^{k_1^1} f_1|^2 \bigg)^{\frac12} 
\prod_{j=2}^m M_{\mathcal{R}} f_j\bigg\|_{L^r}
\\ \nonumber
&\le \mathcal{F}_0 \bigg\|\bigg( \sum_{Q \in \D_0} 
|\Delta_{Q^1}^{k_1^1} f_1|^2 \bigg)^{\frac12} \bigg\|_{L^{r_1}}
\prod_{j=2}^m \|M_{\mathcal{R}} f_j\|_{L^{r_j}}
\\ \nonumber
& \lesssim \mathcal{F}_0 \prod_{j=1}^m \|f_j\|_{L^{r_j}},
\end{align}
where the implicit constants are independent of $k$, $\D$, $\D_0$, and $\vec{f}$. 

Next, let us handle \textbf{Case 2}. By the assumptions on $\widetilde{h}_{I_j^i}$, in this case, there must exist $j_0 \in \{1, \ldots, m\}$ such that $\widetilde{h}_{I_{j_0}^1} = h_{I_{j_0}^1}$. Set 
\begin{align}\label{qmm}
q = r'_{j_0}, \quad 
q_{j_0} = 2, \quad\text{ and }\quad 
q_j = r_j, \quad j = \{1, \ldots, m\} \setminus\{j_0\}.
\end{align} 
Then $\frac1q = \sum_{j=1}^m \frac{1}{q_j}$. Recall that $T^{j_0*}$ denotes the full adjoint of $T$ with respect to the $j_0$-th slot. Note that the shift $\mathbb{S}_{\D, \D_0}^{k,  j_0*}$ contains the cancellative Haar function $\widetilde{h}_{I_{m+1}^1} = h_{I_{m+1}^1}$, which is similar to \textbf{Case 1}. Since \eqref{ssf-2} holds for all exponents $p \in (1, \infty)$, as shown in \eqref{sddk}, one can obtain 
\begin{align}\label{SDFN}
\|\mathbb{S}_{\D, \D_0}^{k,  j_0*} (\vec{f})\|_{L^q}
\lesssim \mathcal{F}_0 \prod_{j=1}^m \|f_j\|_{L^{q_j}},  
\end{align}
where the implicit constant is independent of $k$, $\D$, $\D_0$, and $\vec{f}$. By duality, \eqref{qmm} and \eqref{SDFN} imply 
\begin{align*}
\|\mathbb{S}_{\D, \D_0}^k (\vec{f})\|_{L^r}
\lesssim \mathcal{F}_0 \prod_{j=1}^m \|f_j\|_{L^{r_j}}.
\end{align*}
This completes the proof. 
\end{proof}

\subsection{Compact shifts}
Let us start the proof of Theorem \ref{thm:dyadic-cpt} for $\mathbf{T}_{\w} = \S_{\D_{\w}}^k$. Recall the weighted boundedness of shifts (cf. \cite[Theorem 6.2]{LMV21}): 
\begin{align}\label{SDww}
\sup_{\w} \|\mathbb{S}_{\D_{\w}}^k (\vec{f})\|_{L^r(v^r)} 
\lesssim \prod_{j=1}^m \|f_j\|_{L^{r_j}(v_j^{r_j})}, 
\end{align}
for all $\vec{r} = (r_1, \ldots, r_m) \in (1, \infty]^m$ and $\vec{v} = (v_1, \ldots, v_m) \in A_{\vec{r}}(\Rnn)$, where $\frac1r = \sum_{j=1}^m \frac{1}{r_j} > 0$ and $v = \prod_{j=1}^m v_j$.
In light of Theorem \ref{thm:RdF-cpt} and \eqref{SDww}, it suffices to show that $\mathbb{E}_{\w} \mathbf{S}_{\D_{\omega}}^k$ is compact from $L^{2m}  \times \cdots \times L^{2m}$ to $L^2$. Then by Theorem \ref{thm:KRLp} and Minkowski's inequality, it is enough to prove the following 
\begin{align}
\label{SDKR-1}
&\sup_{\substack{\|f_j\|_{L^{2m}} \le 1 \\ j=1, \ldots, m}} 
\sup_{\w} \| \S_{\D_{\w}}^k(\vec{f}) \|_{L^2}
\lesssim 1, 
\\ 
\label{SDKR-2}
\lim_{A \to \infty} &\sup_{\substack{\|f_j\|_{L^{2m}} \le 1 \\ j=1, \ldots, m}} 
\sup_{\w} \|\S_{\D_{\w}}^k (\vec{f}) \, \mathbf{1}_{B_{\vec{n}}(0, A)^c}\|_{L^2} 
= 0, 
\\ 
\label{SDKR-3}
\lim_{|v| \to 0} &\sup_{\substack{\|f_j\|_{L^{2m}} \le 1 \\ j=1, \ldots, m}} 
\sup_{\w} \|\tau_v \,  \S_{\D_{\w}}^k (\vec{f}) 
- \S_{\D_{\w}}^k (\vec{f})\|_{L^2} 
=0. 
\end{align}
The inequality \eqref{SDKR-1} follows from \eqref{SDww}.

Recall the operator $\mathbb{S}_{\D, \D_0}^k$ defined in Section \ref{sec:un}. To demonstrate  \eqref{SDKR-2}, we define  
\begin{align*}
\S_{\D_{\w}}^{k, N} (\vec{f}) 
:= \sum_{Q \notin \D_{\w}(N)} A_Q^k (\vec{f}), \quad N \in \N.  
\end{align*}
Lemma \ref{lem:SDD} applied to $\D_0 = \D \setminus \D(N)$ gives 
\begin{align}\label{SNL2}
\sup_{\w} \|\S_{\D_{\w}}^{k, N} (\vec{f})\|_{L^2}
\lesssim \mathcal{F}_N \prod_{j=1}^m \|f_j\|_{L^{2m}},  
\end{align}
where the implicit constants are independent of $k$, $N$, and $\vec{f}$. Let $A \ge 2^4$ and $N := [\frac12 \log_2 A] \ge 2$. Note that for any $\w = (\w_1, \w_2) \in \Omega_1 \times \Omega_2$ and $I^1 \times I^2 \in \D_{\w}(N)$, 
\begin{align}\label{KN}
I^i \subset \{x_i \in \R^{n_i}: |x_i| \le (N+2) 2^N\} 
\subset \{|x_i| \le 2^{2N}\} 
\subset \{|x_i| \le A\}, \quad i=1, 2. 
\end{align}
Hence, applying \eqref{SNL2} and \eqref{KN}, we obtain 
\begin{align*}
\|\S_{\D_{\w}}^k (\vec{f}) \mathbf{1}_{B_{\vec{n}}(0, 2A)^c}\|_{L^2} 
= \|\S_{\D_{\w}}^{k, N} (\vec{f}) \mathbf{1}_{B_{\vec{n}}(0, 2A)^c}\|_{L^2} 
\lesssim \mathcal{F}_N \prod_{j=1}^m \|f_j\|_{L^{2m}}, 
\end{align*}
where the implicit constant is independent of $\w$, $A$, and $\vec{f}$. This shows \eqref{SDKR-2}.

It remains to verify \eqref{SDKR-3}. Let $0 < |v| = |v_1| + |v_2| \ll 2^{-10}$ and $\rho \ge 2$ be an integer chosen later. There exists an integer $N=N(v) \ge 2$ such that $2^{-\rho (N+1)} \le |v| < 2^{- \rho N}$. We then split  
\begin{align}\label{HFX} 
\|(\tau_v \S_{\D_{\w}}^k - \S_{\D_{\w}}^k) (\vec{f})\|_{L^2} 
\le \mathbf{X}_{\w}^1(v; \vec{f})  + \mathbf{X}_{\w}^2(v; \vec{f}), 
\end{align}
where 
\begin{align*}
\mathbf{X}_{\w}^1(v; \vec{f}) 
&:= \bigg\|\sum_{Q \in \D_{\w}(N)} 
\sum_{\substack{I_j \in \D_{\w, k_j}(Q) \\ j=1, \ldots, m+1}}   
a_{(I_j), Q} \langle f_1, h_{I_1^1} \otimes \widetilde{h}_{I_1^2} \rangle
\prod_{j=2}^m \langle f_j, \widetilde{h}_{I_j^1} \otimes \widetilde{h}_{I_j^2} \rangle \, 
\phi_{I_{m+1}}^v\bigg\|_{L^2}, 
\\ 
\mathbf{X}_{\w}^2(v; \vec{f}) 
&:= \bigg\|\sum_{Q \notin \D_{\w}(N)} 
\sum_{\substack{I_j \in \D_{\w, k_j}(Q) \\ j=1, \ldots, m+1}}   
a_{(I_j), Q} \langle f_1, h_{I_1^1} \otimes \widetilde{h}_{I_1^2} \rangle
\prod_{j=2}^m \langle f_j, \widetilde{h}_{I_j^1} \otimes \widetilde{h}_{I_j^2} \rangle \, 
\phi_{I_{m+1}}^v\bigg\|_{L^2}, 
\end{align*}
with $\phi_{I_{m+1}}^v := \tau_v (h_{I_{m+1}^1} \otimes \widetilde{h}_{I_{m+1}^2}) - h_{I_{m+1}^1} \otimes \widetilde{h}_{I_{m+1}^2}$. The inequality \eqref{SNL2} immediately gives 
\begin{align}\label{XVF}
\mathbf{X}_{\w}^2(v; \vec{f}) 
\le \|\tau_v \S_{\D_{\w}}^{k, N}(\vec{f})\|_{L^2} 
+ \|\S_{\D_{\w}}^{k, N}(\vec{f})\|_{L^2}
= 2 \|\S_{\D_{\w}}^{k, N}(\vec{f})\|_{L^2} 
\lesssim \mathcal{F}_N \prod_{j=1}^m \|f_j\|_{L^{2m}},  
\end{align}
where the implicit constant is independent of $v$, $\w$, and $\vec{f}$. Since $|v| \to 0$ implies $N(v) \to \infty$, the estimate \eqref{XVF}, along with the fact $\lim_{N \to \infty} \mathcal{F}_N = 0$, implies 
\begin{align}\label{HFX-1}
\lim_{|v| \to 0} \sup_{\w} 
\sup_{\substack{\|f_j\|_{L^{2m}} \le 1 \\ j=1, \ldots, m}} 
\mathbf{X}_{\w}^2(v; \vec{f})  = 0. 
\end{align}

To analyze $\mathbf{X}_{\w}^1(v; \vec{f})$, note that for any $N \ge 2$ and $i=1, 2$, 
\begin{align}\label{car-DN}
&\# \D_{\w_i}^i(N) 
\le \# \big\{J^i \in \D_{\w_i}^i: 2^{-N} \le \ell(J^i) \le 2^N, J^i \subset 2^{2N+1} \I^i \big\} 
\\ \nonumber 
&\le \sum_{-N \le k \le N} 2^{(2N+1)n_i}  2^{-kn_i} 
\le 2^{(2N+1)n_i} 2^{n_i N+1} 
= 2^{3n_iN+n_i+1}, 
\end{align}
and 
\begin{align}\label{PLP}
\|\phi_J^v\|_{L^2} 
& \le \|\tau_{v_1} h_{J^1} - h_{J^1}\|_{L^2} \|\tau_{v_2} \widetilde{h}_{J^2}\|_{L^2}
+ \|h_{J^1}\|_{L^2} \|\tau_{v_2} \widetilde{h}_{J^2} - \widetilde{h}_{J^2}\|_{L^2}
\\ \nonumber 
&\lesssim |v_1|^{\frac12} \ell(J^1)^{-\frac12} 
+ |v_2|^{\frac12} \ell(J^2)^{-\frac12}.  
\end{align}
Accordingly, by \eqref{car-DN} and \eqref{PLP}, 
\begin{align*}
\mathbf{X}_{\w}^1(v; \vec{f}) 
&\leq \sum_{Q \in \D_{\w}(N)}  
\prod_{j=1}^m \langle |f_j| \rangle_Q 
\sum_{I_{m+1} \in \D_{\w, k_{m+1}}(Q)}  
|I_{m+1}|^{\frac12} \|\phi_{I_{m+1}}^v\|_{L^2} 
\\
&\lesssim \sum_{Q \in \D_{\w}(N)}  
\prod_{j=1}^m |Q|^{-\frac{1}{2m}} \|f_j\|_{L^{2m}} 
\sum_{I_{m+1} \in \D_{\w, k_{m+1}}(Q)} |I_{m+1}|^{\frac12}
\\
&\quad\times \big[|v_1|^{\frac12} \ell(I_{m+1}^1)^{-\frac12}  
+ |v_2|^{\frac12} \ell(I_{m+1}^2)^{-\frac12} \big]
\\
&\lesssim \sum_{Q \in \D_{\w}(N)} 
2^{k_{m+1}^1 n_1/2} 2^{k_{m+1}^2 n_2/2} 
\prod_{j=1}^m \|f_j\|_{L^{2m}} 
\\
&\quad\times \big[ \big(|v_1|^{-1} 2^{-k_{m+1}^1} \ell(Q^1) \big)^{-\frac12}  
+ \big(|v_2|^{-1} 2^{-k_{m+1}^2} \ell(Q^2) \big)^{-\frac12} \big]
\\
&\lesssim 2^{(3n_1 + 3n_2) N} 
2^{k_{m+1}^1 n_1 + k_{m+1}^2 n_2} 
(|v| 2^{k_{m+1}^1 + k_{m+1}^2} 2^N)^{\frac12} 
\prod_{j=1}^m \|f_j\|_{L^{2m}} 
\\
&\le 2^{(k_{m+1}^1 + k_{m+1}^2) (n_1 + n_2 + 1)/2} 
2^{(3n_1 + 3n_2 +1 - \rho/2)N} 
\prod_{j=1}^m \|f_j\|_{L^{2m}} .
\end{align*}
Choosing $\rho > 2(3n_1 + 3n_2 +1)$, we conclude  
\begin{align}\label{HFX-2}
\lim_{|v| \to 0} \sup_{\w} 
\sup_{\substack{\|f_j\|_{L^{2m}} \le 1 \\ j=1, \ldots, m}} 
\mathbf{X}_{\w}^1(v; \vec{f})  =0. 
\end{align} 
Consequently, \eqref{SDKR-3} follows from \eqref{HFX}, \eqref{HFX-1}, and \eqref{HFX-2}.  
\qed

\subsection{Compact partial paraproducts}
Next, we would like to prove Theorem \ref{thm:dyadic-cpt} for $\mathbf{T}_{\w} = \mathbf{P}_{\D_{\w}}^{1, k}$, and the proof for $\mathbf{T}_{\w} = \mathbf{P}_{\D_{\w}}^{2, k}$ is symmetric. It was shown in \cite[Theorem 6.7]{LMV21} that 
\begin{align}\label{PDww}
\sup_{\w} \|\mathbb{P}_{\D_{\w}}^{1, k} (\vec{f})\|_{L^r(v^r)} 
\lesssim \prod_{j=1}^m \|f_j\|_{L^{r_j}(v_j^{r_j})}, 
\end{align}
for all $\vec{r} = (r_1, \ldots, r_m) \in (1, \infty]^m$ and $\vec{v} = (v_1, \ldots, v_m) \in A_{\vec{r}}(\Rnn)$, where $\frac1r = \sum_{j=1}^m \frac{1}{r_j} > 0$ and $v = \prod_{j=1}^m v_j$. By Theorem \ref{thm:RdF-cpt} and \eqref{PDww}, it is enough to show that $\mathbb{E}_{\w} \mathbf{P}_{\D_{\w}}^{1, k}$ is compact from $L^{2m}  \times \cdots \times  L^{2m}$ to $L^2$. With Theorem \ref{thm:KRLp} in hand, this is reduced to proving that 
\begin{align}
\label{PFK-1}
&\sup_{\substack{\|f_j\|_{L^{2m}} \le 1 \\ j=1, \ldots, m}} 
\sup_{\w} \|\mathbf{P}_{\D_{\w}}^{1, k}(\vec{f}) \|_{L^2}
\lesssim 1, 
\\ 
\label{PFK-2}
\lim_{A \to \infty} &\sup_{\substack{\|f_j\|_{L^{2m}} \le 1 \\ j=1, \ldots, m}} 
\sup_{\w} \|\mathbf{P}_{\D_{\w}}^{1, k}(\vec{f}) \, \mathbf{1}_{B_{\vec{n}}(0, A)^c}\|_{L^2}
= 0, 
\\ 
\label{PFK-3}
\lim_{|v| \to 0} &\sup_{\substack{\|f_j\|_{L^{2m}} \le 1 \\ j=1, \ldots, m}} 
\sup_{\w} \|\tau_v \, \mathbf{P}_{\D_{\w}}^{1, k} (\vec{f)} 
- \mathbf{P}_{\D_{\w}}^{1, k} (\vec{f})\|_{L^2} 
=0. 
\end{align}
Note that \eqref{PDww} implies \eqref{PFK-1}.

Without loss of generality, we may assume that $\widetilde{h}_{I_1^1} = h_{I_1^1}$ and $\widetilde{h}_{I_2^1} = h_{I_2^1}$. The estimate \eqref{KN} gives 
\begin{align}\label{EP}
\|\mathbf{P}_{\D_{\w}}^{1, k}(\vec{f}) \, \mathbf{1}_{B_{\vec{n}}(0, A)^c}\|_{L^2}
\le \|\Xi_{\w}^1(\vec{f})\|_{L^2} + \|\Xi_{\w}^2(\vec{f})\|_{L^2}, 
\end{align}
where 
\begin{align*}
\Xi_{\w}^1(\vec{f}) 
&:= \sum_{\substack{Q^1 \notin \D_{\w_1}^1(N) \\ Q^2 \in \D_{\w_2}^2}} 
\sum_{\substack{I_j^1 \in \D_{k_j}^1(Q^1) \\ j=1, \ldots, m+1}}
a_{(I_j^1), Q} 
\prod_{j=1}^m \langle f_j, \widetilde{h}_{I_j^1} \otimes \overline{h}_{j, Q^2} \rangle \, 
\widetilde{h}_{I_{m+1}^1} \otimes \overline{h}_{m+1, Q^2}, 
\\
\Xi_{\w}^2(\vec{f}) 
&:= \sum_{\substack{Q^1 \in \D_{\w_1}^1(N) \\ Q^2 \notin \D_{\w_2}^2(N)}} 
\sum_{\substack{I_j^1 \in \D_{k_j}^1(Q^1) \\ j=1, \ldots, m+1}}
a_{(I_j^1), Q} 
\prod_{j=1}^m \langle f_j, \widetilde{h}_{I_j^1} \otimes \overline{h}_{j, Q^2} \rangle \, 
\widetilde{h}_{I_{m+1}^1} \otimes \overline{h}_{m+1, Q^2}. 
\end{align*}
Let $f_{m+1} \in L^2(\Rnn)$ with $\|f_{m+1}\|_{L^2} \le 1$. Denote 
\begin{align*}
\Lambda(\vec{f}, f_{m+1}) 
&:= \sum_{Q^1 \in \D_{\w_1}^1} 
\sum_{\substack{I_j^1 \in \D_{k_j}^1(Q^1) \\ j=1, \ldots, m+1}}
\frac{\prod_{j=1}^{m+1} |I_j^1|^{\frac12}}{|Q^1|^m} 
\bigg\|\bigg[\sum_{Q^2 \in \D_{\w_2}^2} \prod_{j=1}^{m+1} \big| \langle f_j, 
\widetilde{h}_{I_j^1} \otimes \overline{h}_{j, Q^2} \rangle\big|^2 
\frac{\mathbf{1}_{Q^2}}{|Q^2|} \bigg]^{\frac12} \bigg\|_{L^1(\R^{n_2})}. 
\end{align*}
It was shown in \cite[p. 25]{LMV21} that 
\begin{align}\label{LAF}
\Lambda(\vec{f}, f_{m+1}) 
\lesssim \prod_{j=1}^m \|f_j\|_{L^{2m}},  
\end{align} 
which, along with \eqref{H1BMO} and the assumption on partial paraproducts, leads to 
\begin{align}\label{PG-1}
|\langle \Xi_{\w}^1(\vec{f}), f_{m+1} \rangle| 
&\lesssim \sum_{Q^1 \notin \D_{\w_1}^1(N)} 
\sum_{\substack{I_j^1 \in \D_{k_j}^1(Q^1) \\ j=1, \ldots, m+1}}
\bigg[\sup_{Q_0^2 \in \D_{\w_2}^2} \frac{1}{|Q_0^2|} 
\sum_{\substack{Q^2 \in \D_{\w_2}^2 \\ Q^2 \subset Q_0^2}} 
|a_{(I_j^1), Q}|^2 \bigg]^{\frac12} 
\\ \nonumber 
&\quad\times \int_{\R^{n_2}} \bigg[\sum_{Q^2 \in \D_{\w_2}^2} 
\prod_{j=1}^{m+1} \big| \langle f_j, \widetilde{h}_{I_j^1} \otimes \overline{h}_{j, Q^2} \rangle\big|^2 
\frac{\mathbf{1}_{Q^2}}{|Q^2|} \bigg]^{\frac12} \, dx_2 
\\ \nonumber 
&\lesssim \sup_{Q^1 \notin \D_{\w_1}^1(N)} \mathcal{F}^1(Q^1) \, 
\Lambda(\vec{f}, f_{m+1}) 
\lesssim \sup_{Q^1 \notin \D_{\w_1}^1(N)} \mathcal{F}^1(Q^1) \, 
\prod_{j=1}^m \|f_j\|_{L^{2m}},  
\end{align}
where the implicit constants are independent of $\w$. Analogously, 
\begin{align}\label{PG-2}
|\langle \Xi_{\w}^2(\vec{f}), f_{m+1} \rangle| 
&\lesssim \sum_{Q^1 \in \D_{\w_1}^1(N)} 
\sum_{\substack{I_j^1 \in \D_{k_j}^1(Q^1) \\ j=1, \ldots, m+1}}
\bigg[\sup_{Q_0^2 \in \D_{\w_2}^2} \frac{1}{|Q_0^2|} 
\sum_{\substack{Q^2 \notin \D_{\w_2}^2(N) \\ Q^2 \subset Q_0^2}} 
|a_{(I_j^1), Q}|^2 \bigg]^{\frac12} 
\\ \nonumber 
&\quad\times \int_{\R^{n_2}} \bigg[\sum_{Q^2 \in \D_{\w_2}^2} 
\prod_{j=1}^{m+1} \big| \langle f_j, \widetilde{h}_{I_j^1} \otimes \overline{h}_{j, Q^2} \rangle\big|^2 
\frac{\mathbf{1}_{Q^2}}{|Q^2|} \bigg]^{\frac12} \, dx_2 
\\ \nonumber 
&\lesssim \mathcal{F}^1_N \, \Lambda(\vec{f}, f_{m+1}) 
\lesssim  \mathcal{F}^1_N \, \prod_{j=1}^m \|f_j\|_{L^{2m}},  
\end{align}
where the implicit constants are independent of $\w$. Since 
\begin{align}\label{NFN}
\lim_{N \to \infty} \widetilde{\mathcal{F}}^1_N
:= \lim_{N \to \infty} \Big[ \sup_{\D^1} \sup_{Q^1 \notin \D^1(N)} \mathcal{F}^1(Q^1)
+ \mathcal{F}^1_N \Big]
= 0,
\end{align}
we deduce \eqref{PFK-2} from \eqref{EP}, \eqref{PG-1}, and \eqref{PG-2}. 

To justify \eqref{PFK-3}, let $0<|v| = |v_1| + |v_2| \ll 2^{-8}$ and $\rho \ge 2$ be an integer chosen later. Then there exists an integer $N=N(v) \ge 2$ so that $2^{-\rho (N+1)} < |v| \le 2^{-\rho N}$. We have 
\begin{align}\label{TVE-1}
\|\tau_v \, \mathbf{P}_{\D_{\w}}^{1, k} (\vec{f)} 
- \mathbf{P}_{\D_{\w}}^{1, k} (\vec{f})\|_{L^2} 
\le \|\Gamma_{\w}^{v, 1}(\vec{f})\|_{L^2} 
+ \|\Gamma_{\w}^{v, 2}(\vec{f})\|_{L^2}, 
\end{align}
where 
\begin{align*}
\Gamma_{\w}^{v, 1}(\vec{f}) 
&:= \sum_{Q \in \D_{\w}(N)} 
\sum_{\substack{I_j^1 \in \D_{k_j}^1(Q^1) \\ j=1, \ldots, m+1}} a_{(I_j^1), Q} 
\prod_{j=1}^m \langle f_j, \widetilde{h}_{I_j^1} \otimes \overline{h}_{j, Q^2} \rangle \, \psi_Q^v, 
\\
\Gamma_{\w}^{v, 2}(\vec{f}) 
&:= \sum_{Q \notin \D_{\w}(N)} 
\sum_{\substack{I_j^1 \in \D_{k_j}^1(Q^1) \\ j=1, \ldots, m+1}} a_{(I_j^1), Q} 
\prod_{j=1}^m \langle f_j, \widetilde{h}_{I_j^1} \otimes \overline{h}_{j, Q^2} \rangle \, \psi_Q^v, 
\end{align*}
where $\psi_Q^v := \tau_v (\widetilde{h}_{I_{m+1}^1} \otimes \overline{h}_{m+1, Q^2}) - \widetilde{h}_{I_{m+1}^1} \otimes \overline{h}_{m+1, Q^2}$. In light of \eqref{PG-1} and \eqref{PG-2}, we obtain 
\begin{align*}
\|\Gamma_{\w}^{v, 2}(\vec{f})\|_{L^2}
\le 2 \|\Xi_{\w}^1(\vec{f})\|_{L^2} + 2 \|\Xi_{\w}^2(\vec{f})\|_{L^2}
\lesssim \widetilde{\mathcal{F}}^1_N \prod_{j=1}^m \|f_j\|_{L^{2m}}, 
\end{align*}
which together with \eqref{NFN} implies  
\begin{align}\label{TVE-2}
\lim_{|v| \to 0} \sup_{\w} 
\sup_{\substack{\|f_j\|_{L^{2m}} \le 1 \\ j=1, \ldots, m}} 
\|\Gamma_{\w}^{v, 2}(\vec{f})\|_{L^2} = 0.
\end{align}

As argued in \eqref{PG-1}, there holds 
\begin{align}\label{TVE-3}
\|\Gamma_{\w}^{v, 1}(\vec{f})\|_{L^2}
\lesssim \sup_{\|f_{m+1}\|_{L^2} \le 1} 
\Lambda_{\w}^v (\vec{f}, f_{m+1}),  
\end{align}
where 
\begin{align}\label{TVE-4}
\Lambda_{\w}^v (\vec{f}, f_{m+1}) 
&:= \sum_{Q^1 \in \D_{\w_1}^1(N)} 
\sum_{\substack{I_j^1 \in \D_{k_j}^1(Q^1) \\ j=1, \ldots, m+1}}
\frac{\prod_{j=1}^{m+1} |I_j^1|^{\frac12}}{|Q^1|^m} 
\\ \nonumber 
&\qquad\times \bigg\|\bigg[\sum_{Q^2 \in \D_{\w_2}^2(N)} \prod_{j=1}^m \big| \langle f_j, 
\widetilde{h}_{I_j^1} \otimes \overline{h}_{j, Q^2} \rangle\big|^2
\big|\langle f_{m+1}, \psi_Q^v \rangle\big|^2 
\frac{\mathbf{1}_{Q^2}}{|Q^2|} \bigg]^{\frac12} \bigg\|_{L^1(\R^{n_2})}
\\ \nonumber
&\le \sum_{Q \in \D_{\w}(N)} 
\sum_{\substack{I_j^1 \in \D_{k_j}^1(Q^1) \\ j=1, \ldots, m+1}}
\frac{\prod_{j=1}^{m+1} |I_j^1|^{\frac12}}{|Q^1|^m}  
|Q^2|^{\frac12} \prod_{j=1}^m \big| \langle f_j, 
\widetilde{h}_{I_j^1} \otimes \overline{h}_{j, Q^2} \rangle\big|
\big|\langle f_{m+1}, \psi_Q^v \rangle\big|. 
\end{align}
To estimate the last term, one has to know the exact form of $\overline{h}_{m+1, Q^2}$. If $\overline{h}_{m+1, Q^2} = h_{Q^2}$, then 
\begin{align}\label{TVE-5}
\|\psi_Q^v\|_{L^2}
&\le \|\tau_{v_1} \widetilde{h}_{I_{m+1}^1} - \widetilde{h}_{I_{m+1}^1}\|_{L^2} 
\|\tau_{v_2} \overline{h}_{m+1, Q^2}\|_{L^2} 
\\ \nonumber
&\qquad + \|\widetilde{h}_{I_{m+1}^1}\|_{L^2} 
\|\tau_{v_2} \overline{h}_{m+1, Q^2} - \overline{h}_{m+1, Q^2}\|_{L^2} 
\\ \nonumber 
&\lesssim |v_1|^{\frac12} \ell(I_{m+1}^1)^{-\frac12} 
+ |v_2|^{\frac12} \ell(Q^2)^{-\frac12}. 
\end{align}
If $\overline{h}_{m+1, Q^2} = \frac{\mathbf{1}_{Q^2}}{|Q^2|}$, then 
\begin{align}\label{TVE-6}
\|\psi_Q^v\|_{L^2}
\lesssim |v_1|^{\frac12} \ell(I_{m+1}^1)^{-\frac12} |Q^2|^{-\frac12}
+ |v_2|^{\frac12} \ell(Q^2)^{-\frac12} |Q^2|^{-\frac12}. 
\end{align}
By the assumption, there exists only one function in $\{\overline{h}_{j, Q^2}\}_{j=1}^{m+1}$ being cancellative Haar function. This, along with \eqref{TVE-4}--\eqref{TVE-6}, gives 
\begin{align}\label{TVE-7}
\Lambda_{\w}^v (\vec{f}, f_{m+1}) 
&\lesssim \sum_{Q \in \D_{\w}(N)}  
\sum_{I_{m+1}^1 \in \D_{k_{m+1}}^1(Q^1)} 
|Q^2|^{\frac12} \prod_{j=1}^m \langle |f_j| \rangle_Q 
|I_{m+1}^1|^{\frac12} \|f_{m+1}\|_{L^2} 
\\ \nonumber 
&\qquad\times \big[|v_1|^{\frac12} \ell(I_{m+1}^1)^{-\frac12} 
+ |v_2|^{\frac12} \ell(Q^2)^{-\frac12} \big] 
\\ \nonumber 
&\lesssim |v|^{\frac12} \# \D_{\w}(N) \, 2^{k_{m+1} 	\frac{n_1}{2}} 
2^{N \frac{n_1}{2}} (2^{k_{m+1}} 2^N)^{\frac12} 
\prod_{j=1}^m \|f_j\|_{L^{2m}}
\\ \nonumber 
&\lesssim 2^{k_{m+1} (n_1 + 1)/2} 2^{(4n_1 + 4n_2 - \rho/2) N} 
\prod_{j=1}^m \|f_j\|_{L^{2m}}, 
\end{align}
where the implicit constants are independent of $v$ and $\w$. Then taking $\rho > 8n_1 + 8n_2$ and using \eqref{TVE-3} and \eqref{TVE-7}, we deduce 
\begin{align}\label{TVE-8}
\lim_{|v| \to 0} \sup_{\w} 
\sup_{\substack{\|f_j\|_{L^{2m}} \le 1 \\ j=1, \ldots, m}} 
\|\Gamma_{\w}^{v, 1}(\vec{f})\|_{L^2} = 0.
\end{align}
Hence, \eqref{PFK-3} follows from \eqref{TVE-1}, \eqref{TVE-2}, and \eqref{TVE-8}. 
\qed

\subsection{Compact full paraproducts}
Finally, let us turn to the proof of Theorem \ref{thm:dyadic-cpt} for $\mathbf{T}_{\w} = \mathbf{F}_{\mathbf{a}_{\w}}$. A careful checking of the proof of \cite[Theorem 6.21]{LMV21} yields 
\begin{align}\label{Fbww}
\sup_{\w} \|\mathbb{F}_{\mathbf{a}_{\w}}\|_{L^{r_1}(v_1^{r_1}) \times \cdots \times L^{r_m}(v_m^{r_m}) \to L^r(v^r)} 
\lesssim \sup_{\w} \|\mathbf{a}_{\w}\|_{\BMO(\D_{\w})}
\le 1, 
\end{align}
for all $\vec{r} = (r_1, \ldots, r_m) \in (1, \infty]^m$ and $\vec{v} = (v_1, \ldots, v_m) \in A_{\vec{r}}(\Rnn)$, where $\frac1r = \sum_{j=1}^m \frac{1}{r_j} > 0$ and $v = \prod_{j=1}^m v_j$. By Theorem \ref{thm:RdF-cpt} and \eqref{Fbww}, it is enough to show that $\mathbb{E}_{\w} \mathbf{F}_{\mathbf{a}_{\w}}$ is compact from $L^{2m}  \times \cdots \times  L^{2m}$ to $L^2$. With Theorem \ref{thm:KRLp} in hand, this is reduced to proving that 
\begin{align}
\label{PibFK-1}
&\sup_{\substack{\|f_j\|_{L^{2m}} \le 1 \\ j=1, \ldots, m}} 
\sup_{\w} \|\mathbf{F}_{\mathbf{a}_{\w}}(\vec{f}) \|_{L^2}
\lesssim 1, 
\\ 
\label{PibFK-2}
\lim_{A \to \infty} &\sup_{\substack{\|f_j\|_{L^{2m}} \le 1 \\ j=1, \ldots, m}} 
\sup_{\w} \|\mathbf{F}_{\mathbf{a}_{\w}}(\vec{f}) \, \mathbf{1}_{B_{\vec{n}}(0, A)^c}\|_{L^2}
= 0, 
\\ 
\label{PibFK-3}
\lim_{|v| \to 0} &\sup_{\substack{\|f_j\|_{L^{2m}} \le 1 \\ j=1, \ldots, m}} 
\sup_{\w} \|\tau_v \, \mathbf{F}_{\mathbf{a}_{\w}} (\vec{f)} 
- \mathbf{F}_{\mathbf{a}_{\w}} (\vec{f})\|_{L^2} 
=0. 
\end{align}
Since it no longer needs the cancellation of Haar functions, it suffices to consider the case 
\begin{align*}
\mathbf{F}_{\mathbf{a}_{\w}} (\vec{f}) 
= \sum_{I = I^1 \times I^2 \in \D_{\w}} a_I 
\langle f_1, h_{I^1} \otimes \overline{h}_{I^2} \rangle 
\langle f_2, \overline{h}_{I^1} \otimes h_{I^2} \rangle 
\prod_{j=3}^m \langle f_j \rangle_{I^1 \times I^2} \, 
\overline{h}_{I^1} \otimes \overline{h}_{I^2},  
\end{align*}
where $\overline{h}_{I^i} = \frac{\mathbf{1}_{I^i}}{|I^i|}$, $i=1, 2$.

The estimate \eqref{PibFK-1} follows from \eqref{Fbww}. To show \eqref{PibFK-2}, let $\mathbf{a}_{\w}^N := \big\{a_I \mathbf{1}_{\{I \notin \D_{\w}(N) \}} \big\}_{I \in \D_{\w}}$. The assumption gives 
\begin{align}\label{ANB}
\lim_{N \to \infty} \sup_{\w} \|\mathbf{a}_{\w}^N\|_{\BMO(\D_{\w})} = 0. 
\end{align}
By \eqref{KN} and \eqref{Fbww}, there holds
\begin{align*}
\|\mathbf{F}_{\mathbf{a}_{\w}} (\vec{f}) \mathbf{1}_{B_{\vec{n}}(0, A)^c}\|_{L^2} 
& = \|\mathbf{F}_{\mathbf{a}_{\w}^N} (\vec{f}) \mathbf{1}_{B_{\vec{n}}(0, A)^c}\|_{L^2} 
\le \|\mathbf{F}_{\mathbf{a}_{\w}^N} (\vec{f})\|_{L^2}  
\\
& \lesssim \sup_{\w} \|\mathbf{a}_{\w}^N\|_{\BMO(\D_{\w})} 
\prod_{j=1}^m \|f_j\|_{L^{2m}},  
\end{align*}
which along with \eqref{ANB} yields \eqref{PibFK-2}. 

Let $0<|v| = |v_1| + |v_2| \ll 2^{-4}$ and $\rho \ge 2$ be an integer chosen later. Then there exists an integer $N=N(v) \ge 2$ so that $2^{-\rho (N+1)} < |v| \le 2^{-\rho N}$. We split 
\begin{align}\label{HPB-2}
\|(\tau_v \mathbf{F}_{\mathbf{a}_{\w}} 
- \mathbf{F}_{\mathbf{a}_{\w}}) (\vec{f})\|_{L^2} 
\le \Upsilon_{\w}^1(v; \vec{f}) + \Upsilon_{\w}^2(v; \vec{f}), 
\end{align}
where 
\begin{align*}
\Upsilon_{\w}^1(v; \vec{f}) 
&:= \bigg\|\sum_{I \in \D_{\w}(N)} a_I  
\langle f_1, h_{I^1} \otimes \overline{h}_{I^2} \rangle 
\langle f_2, \overline{h}_{I^1} \otimes h_{I^2} \rangle 
\prod_{j=3}^m \langle f_j \rangle_{I^1 \times I^2} \, \phi_I^v \bigg\|_{L^2},
\\ 
\Upsilon_{\w}^2(v; \vec{f})
&:= \bigg\|\sum_{I \notin \D_{\w}(N)}  a_I
\langle f_1, h_{I^1} \otimes \overline{h}_{I^2} \rangle 
\langle f_2, \overline{h}_{I^1} \otimes h_{I^2} \rangle 
\prod_{j=3}^m \langle f_j \rangle_{I^1 \times I^2} \, \phi_I^v \bigg\|_{L^2}, 
\end{align*}
with $\phi_I^v := \tau_v (\overline{h}_{I^1} \otimes \overline{h}_{I^2}) - \overline{h}_{I^1} \otimes \overline{h}_{I^2}$. Invoking \eqref{Fbww}, we arrive at   
\begin{align}\label{UP2}
\Upsilon_{\w}^2(v; \vec{f}) 
&\le 2 \bigg\|\sum_{I \notin \D_{\w}(N)} a_I  
\langle f_1, h_{I^1} \otimes \overline{h}_{I^2} \rangle 
\langle f_2, \overline{h}_{I^1} \otimes h_{I^2} \rangle 
\prod_{j=3}^m \langle f_j \rangle_{I^1 \times I^2} 
\overline{h}_{I^1} \otimes \overline{h}_{I^2} \bigg\|_{L^2}
\\ \nonumber 
&= 2 \| \mathbf{F}_{\mathbf{a}_{\w}^N} (\vec{f})\|_{L^2}
\lesssim \sup_{\w \in \Omega} \|\mathbf{a}_{\w}^N\|_{\BMO(\D_{\w})}  
\prod_{j=1}^m \|f_j\|_{L^{2m}}, 
\end{align}
where the implicit constant does not depend on $\omega$. Thus, \eqref{ANB} and \eqref{UP2} imply 
\begin{align}\label{HO-1}
\lim_{|v| \to 0} \sup_{\w} 
\sup_{\substack{\|f_j\|_{L^{2m}} \le 1 \\ j=1, \ldots, m}} 
\Upsilon_{\w}^2(v; \vec{f}) = 0.
\end{align}
To proceed, note that 
\begin{align}\label{bcar}
\sum_{I \in \D_{\w}:\, I \subset U} |a_I|^2 
\le \|\mathbf{a}_{\w}\|_{\BMO(\D_{\w})}^2 \, |U|
\le |U|,   
\end{align}
for all open sets $U \subset \Rnn$ with $0<|U|<\infty$, and 
\begin{align}
\|\phi_I^v\|_{L^2}
&\le \|\tau_{v_1} \overline{h}_{I^1} - \overline{h}_{I^1}\|_{L^2} 
\|\tau_{v_2} \overline{h}_{I^2}\|_{L^2} 
+ \|\overline{h}_{I^1}\|_{L^2} 
\|\tau_{v_2} \overline{h}_{I^2} - \overline{h}_{I^2}\|_{L^2} 
\\ \nonumber 
&\lesssim |v_1|^{\frac12} \ell(I^1)^{-\frac{n_1}{2} - \frac12} \ell(I^1)^{\frac{n_1}{2}} 
+ |v_2|^{\frac12} \ell(I^2)^{-\frac{n_2}{2} - \frac12} \ell(I^2)^{\frac{n_2}{2}}. 
\end{align}
Hence, it follows from the Cauchy-Schwarz inequality, \eqref{car-DN}, \eqref{bcar}, and Lemma \ref{lem:Car} that 
\begin{align*}
\Upsilon_{\w}^1(v; \vec{f})
&\le \sum_{I \in \D_{\w}(N)} |a_I| |I|^{\frac12} 
\prod_{j=1}^m \langle |f_j| \rangle_{I} \|\phi_I^v\|_{L^2}
\\
&\lesssim |v|^{\frac12} \, 2^{\frac{N}{2}(n_1 + n_2 + 1)} 
\sum_{I \in \D_{\w}(N)} |a_I|  
\prod_{j=1}^m \langle |f_j| \rangle_I
\\
&\lesssim |v|^{\frac12} \, 2^{\frac{N}{2}(n_1 + n_2 + 1)} 
\big[\# \D_{\w}(N) \big]^{\frac12} 
\bigg[\sum_{I \in \D} \prod_{j=1}^m |a_I|^{\frac2m} 
\langle |f_j|^m \rangle_I^{\frac2m} \bigg]^{\frac12}
\\
&\lesssim |v|^{\frac12} \, 2^{\frac{N}{2}(n_1 + n_2 + 1)} 
2^{\frac{N}{2} (3n_1 + 3n_2)} 
\prod_{j=1}^m \bigg[\sum_{I \in \D} |a_I|^2 
\langle |f_j|^m \rangle_I^2 \bigg]^{\frac{1}{2m}}
\\
&\lesssim 2^{\frac{N}{2}(4n_1 + 4n_2 + 1 - \rho)} 
\prod_{j=1}^m \|M_{\mathcal{R}} (|f_j|^m)\|_{L^2}^{\frac1m} 
\\
&\lesssim 2^{\frac{N}{2}(4n_1 + 4n_2 + 1 - \rho)} 
\prod_{j=1}^m \|f_j\|_{L^{2m}}, 
\end{align*}
where we have used that $M_{\mathcal{R}}$ is bounded on $L^p(\Rnn)$ for any $p \in (1, \infty)$. Picking $\rho > 4n_1 + 4n_2 + 1$, the above leads to  
\begin{align}\label{HO-2}
\lim_{|v| \to 0} \sup_{\w} 
\sup_{\substack{\|f_j\|_{L^{2m}} \le 1 \\ j=1, \ldots, m}} 
\Upsilon_{\w}^1(v; \vec{f}) = 0.
\end{align}
Therefore, \eqref{PibFK-3} is a consequence of \eqref{HPB-2}--\eqref{HO-2}. This completes the proof.  
\qed

\section{Mean continuity of commutators}\label{sec:wcc} 
In this section we aim to demonstrate Theorem \ref{thm:dyadic-bT}. Our analysis is based on the size of dyadic cubes and the cancellation of Haar functions.

\subsection{Commutators of shifts}\label{sec:cs}
First, let us show the mean continuity of $[\b, \mathbb{E}_{\w} \mathbb{S}_{\D_{\w}}^k]_{\a}$. Let $\a \in \N^m \setminus \{0\}^m$ and $\b = (b_1, \ldots, b_m) \in \cmo(\Rnn)^m$. It follows from \eqref{SDww} and Theorem \ref{thm:TTb} that 
\begin{align}\label{SDW-2}
\sup_{\w} \|[\b, \mathbb{S}_{\D_{\w}}^k]_{\a} (\vec{f})\|_{L^r(v^r)} 
\lesssim \prod_{j=1}^m \|b_j\|_{\bmo}^{\alpha_j} \|f_j\|_{L^{r_j}(v_j^{r_j})}, 
\end{align}
for all $\vec{r} = (r_1, \ldots, r_m) \in (1, \infty]^m$ and $\vec{v} = (v_1, \ldots, v_m) \in A_{\vec{r}}(\Rnn)$, where $\frac1r = \sum_{j=1}^m \frac{1}{r_j} > 0$ and $v = \prod_{j=1}^m v_j$. 
In view of \eqref{SDww} and Theorem \ref{thm:RdF-bT}, it suffices to prove
\begin{equation}\label{SDW-3}
\begin{array}{c}
\text{$[\b, \E_{\w} \mathbb{S}_{\D_{\w}}^k]_{\a}$ is mean continuous from} 
\\[4pt]
L^{p_1}(\Rnn) \times \cdots \times L^{p_m}(\Rnn) \text{ to } L^p(\Rnn)
\\[4pt]
\text{for some $\vec{p} = (p_1, \ldots, p_m) \in (1, \infty)^m$ with $\frac1p = \sum_{j=1}^m \frac{1}{p_j} \in (0, 1)$}.
\end{array}
\end{equation} 
In light of \eqref{SDW-2} and that $\cmo(\Rnn) = \overline{\mathscr{C}_c^{\infty}(\Rnn)}^{\bmo}$, we may assume that $\b \in \mathscr{C}_c^{\infty}(\Rnn)^m$. 

We claim that for each $i \in \{1, \ldots, m\}$, 
\begin{align}\label{SDC-3}
\lim_{|v| \to 0} \sup_{\substack{\|f_j\|_{L^{p_j}} \le 1 \\ j=1, \ldots, m}} 
\sup_{\w} \|(\tau_v - \tau_{v_1} - \tau_{v_2} + I) 
[\b, \mathbb{S}_{\D_{\w}}^k]_{e_i} (\vec{f})\|_{L^p} 
= 0,
\end{align}
for some $p, p_1, \ldots, p_m \in (1, \infty)$ satisfying $\frac1p = \sum_{j=1}^m \frac{1}{p_j}$. Assuming \eqref{SDC-3} holds momentarily, let us conclude \eqref{SDW-3} as follows. For any $i_1, i_2 \in \{1, \ldots, m\}$, there holds
\begin{align*}
\big[\b, [\b, \mathbb{S}_{\D_{\w}}^k]_{e_{i_1}} \big]_{e_{i_2}}(\vec{f})
= b_{i_2} [\b, \mathbb{S}_{\D_{\w}}^k]_{e_{i_1}}(\vec{f}) 
- [\b, \mathbb{S}_{\D_{\w}}^k]_{e_{i_1}}(f_1, \ldots, b_{i_2} f_{i_2}, \ldots, f_m), 
\end{align*}
which gives
\begin{align}\label{LLL}
(\tau_v - \tau_{v_1} - \tau_{v_2} + I) \big[\b, [\b, \mathbb{S}_{\D_{\w}}^k]_{e_{i_1}} \big]_{e_{i_2}}(\vec{f})
= \mathscr{L}_1 (\vec{f}) + \mathscr{L}_2 (\vec{f}) - \mathscr{L}_3 (\vec{f}),
\end{align}
where
\begin{align*}
\mathscr{L}_1 (\vec{f})
& := \tau_v b_{i_2} \, (\tau_v - \tau_{v_1} - \tau_{v_2} + I) 
[\b, \mathbb{S}_{\D_{\w}}^k]_{e_{i_1}} (\vec{f}),
\\
\mathscr{L}_2 (\vec{f})
& := (\tau_v b_{i_2} - \tau_{v_1} b_{i_2}) \tau_{v_1} 
[\b, \mathbb{S}_{\D_{\w}}^k]_{e_{i_1}} (\vec{f})
\\
&\quad \, \, + (\tau_v b_{i_2} - \tau_{v_2} b_{i_2}) \tau_{v_2} 
[\b, \mathbb{S}_{\D_{\w}}^k]_{e_{i_1}} (\vec{f})
- (\tau_v b_{i_2} - b_{i_2}) [\b, \mathbb{S}_{\D_{\w}}^k]_{e_{i_1}} (\vec{f}),
\\
\mathscr{L}_3 (\vec{f})
& := (\tau_v - \tau_{v_1} - \tau_{v_2} + I) [\b, \mathbb{S}_{\D_{\w}}^k]_{e_{i_1}} 
(f_1, \ldots, b_{i_2} f_{i_2}, \ldots, f_m).
\end{align*}
Then given $\varepsilon>0$, by \eqref{SDW-2} and \eqref{SDC-3}, there exists some $\delta_0 = \delta_0(\varepsilon) > 0$ such that for any $0 < |v| < \delta_0$, 
\begin{align}\label{LLL-1}
\|\mathscr{L}_1 (\vec{f})\|_{L^p}
& \le \|b_{i_2}\|_{L^{\infty}} \|(\tau_v - \tau_{v_1} - \tau_{v_2} + I) 
[\b, \mathbb{S}_{\D_{\w}}^k]_{e_{i_1}} (\vec{f})\|_{L^p}
\lesssim \varepsilon \prod_{j=1}^m \|f_j\|_{L^{p_j}},
\end{align}
\begin{align}\label{LLL-2}
\|\mathscr{L}_2 (\vec{f})\|_{L^p}
\le 3 |v| \|\nabla b_{i_2}\|_{L^{\infty}} 
\|[\b, \mathbb{S}_{\D_{\w}}^k]_{e_{i_1}} (\vec{f})\|_{L^p}
\lesssim \varepsilon \prod_{j=1}^m \|f_j\|_{L^{p_j}},
\end{align}
and
\begin{align}\label{LLL-3}
\|\mathscr{L}_3 (\vec{f})\|_{L^p}
\lesssim \varepsilon \|b_{i_2} f_{i_2}\|_{L^{p_{i_2}}} \prod_{j \ne i_2} \|f_j\|_{L^{p_j}}
\lesssim \varepsilon \prod_{j=1}^m \|f_j\|_{L^{p_j}},
\end{align}
where the implicit constants are independent of $\varepsilon$, $v$, $\w$, and $\vec{f}$. Hence, by \eqref{LLL}--\eqref{LLL-3}, 
\begin{align*}
\lim_{|v| \to 0} \sup_{\substack{\|f_j\|_{L^{p_j}} \le 1 \\ j=1, \ldots, m}} 
\sup_{\w} \|(\tau_v - \tau_{v_1} - \tau_{v_2} + I) 
\big[\b, [\b, \mathbb{S}_{\D_{\w}}^k]_{e_{i_1}}]_{e_{i_2}} (\vec{f})\|_{L^p} 
= 0. 
\end{align*}
By induction and the same technique as above, we eventually obtain that for any $\a \in \N^m \setminus \{0\}^m$,
\begin{align*}
\lim_{|v| \to 0} \sup_{\substack{\|f_j\|_{L^{p_j}} \le 1 \\ j=1, \ldots, m}} 
\sup_{\w} \|(\tau_v - \tau_{v_1} - \tau_{v_2} + I) 
[\b, \mathbb{S}_{\D_{\w}}^k]_{\a}] (\vec{f})\|_{L^p} 
= 0,
\end{align*}
which along with \eqref{SDW-2} and Minkowski's inequality implies \eqref{SDW-3} as desired.

It remains to verify \eqref{SDC-3}. By symmetry, it is enough to show the case $i=1$. Let $\supp (b_1) \subset B_{\vec{n}}(0, 2^{N_0})$ for some $N_0 \ge 4$. Choose 
\begin{equation}\label{ppp}
\begin{aligned}
& 1 < p < p_1 < \min \bigg\{ \frac{n_1}{n_1-1}, \frac{n_2}{n_2-1} \bigg\}
\quad \text{and} \quad 
\\
& 0 < \gamma < \min\big\{1, \, p(1-n_1/p'_1), \, p(1-n_2/p'_2) \big\}.
\end{aligned}
\end{equation}
Then pick $p_2, \ldots, p_m \in (1, \infty)$ such that $\frac1p = \sum_{j=1}^m \frac{1}{p_j}$. Observe that there exists a finite collection of dyadic rectangles $\big\{P_{\kappa} = P_{\kappa}^1 \times P_{\kappa}^2 \big\}_{\kappa=1}^{\kappa_0} \subset \D_{\w}$ such that $\ell(P_{\kappa}^1) = \ell(P_{\kappa}^2) = 2^{N_0}$, $P_{\kappa} \cap B_{\vec{n}}(0, 2^{N_0}) \neq \emptyset$, $\kappa = 1, \ldots, \kappa_0$, and  
\begin{align}\label{bbn}
\supp (b_1) \subset B_{\vec{n}}(0, 2^{N_0}) 
\subset \bigcup_{\kappa=1}^{\kappa_0} P_{\kappa} \subset B_{\vec{n}}(0, A/2),  
\end{align}
where $\kappa_0$ is a universal integer.

Let $v = (v_1, v_2) \in \Rnn$ with $0<|v| = |v_1| + |v_2| < 1$. Note that
\begin{align*}
& (\tau_v - \tau_{v_1} - \tau_{v_2} + I) (b_1 \cdot \widetilde{h}_{I^1} 
\otimes \widetilde{h}_{I^2}) 
\\
& = b_1 \cdot (\tau_{v_1} \widetilde{h}_{I^1} - \widetilde{h}_{I^1}) 
\otimes (\tau_{v_2} \widetilde{h}_{I^2} - \widetilde{h}_{I^2}) 
- (b_1 - \tau_v b_1) \cdot \tau_v (\widetilde{h}_{I^1} \otimes \widetilde{h}_{I^2})
\\
& \quad + (b_1 - \tau_{v_1} b_1) \cdot \tau_{v_1} \widetilde{h}_{I^1} \otimes \widetilde{h}_{I^2}
 + (b_1 - \tau_{v_2} b_1) \cdot \tau_{v_2} \widetilde{h}_{I^1} \otimes \widetilde{h}_{I^2},
\end{align*}
and 
\begin{align*}
(\tau_v - \tau_{v_1} - \tau_{v_2} + I) (\widetilde{h}_{I^1} \otimes \widetilde{h}_{I^2})
= (\tau_{v_1} \widetilde{h}_{I^1} - \widetilde{h}_{I^1}) 
\otimes (\tau_{v_2} \widetilde{h}_{I^2} - \widetilde{h}_{I^2}).
\end{align*}
For convenience, simply denote
\begin{align*}
\widetilde{a}_{(I_j), Q} := a_{(I_j), Q} \prod_{j=2}^m 
\langle f_j, \widetilde{h}_{I_j^1} \otimes \widetilde{h}_{I_j^2} \rangle
\quad \text{ and } \quad
\sum_Q := \sum_{Q \in \D_{\w}} 
\sum_{\substack{I_j \in \D_{\w, k_j}(Q) \\ j=1, \ldots, m+1}}.
\end{align*}
Then we arrive at
\begin{align}\label{ttt-1}
(\tau_v - \tau_{v_1} - \tau_{v_2} + I) [\b, \mathbb{S}_{\D_{\w}}^k] (\vec{f}) 
= \Gamma^0 (\vec{f}) + \Gamma^1 (\vec{f}) - \Gamma^2 (\vec{f}),
\end{align}
where
\begin{align*}
\Gamma^0 (\vec{f}) 
& := - (b_1 - \tau_v b_1) \tau_v \mathbb{S}_{\D_{\w}}^k (\vec{f})
+ (b_1 - \tau_{v_1} b_1) \tau_{v_1} \mathbb{S}_{\D_{\w}}^k (\vec{f})
+ (b_1 - \tau_{v_2} b_1) \tau_{v_2} \mathbb{S}_{\D_{\w}}^k (\vec{f}), 
\\
\Gamma^1 (\vec{f})
& := \sum_Q \widetilde{a}_{(I_j), Q} \, b_1 \,  
\langle f_1, \widetilde{h}_{I_1^1} \otimes \widetilde{h}_{I_1^2} \rangle
(\tau_{v_1} \widetilde{h}_{I_{m+1}^1} - \widetilde{h}_{I_{m+1}^1}) 
\otimes (\tau_{v_2} \widetilde{h}_{I_{m+1}^2} - \widetilde{h}_{I_{m+1}^2}),
\\
\Gamma^2 (\vec{f})
& := \sum_Q \widetilde{a}_{(I_j), Q} \,
\langle b_1 f_1, \widetilde{h}_{I_1^1} \otimes \widetilde{h}_{I_1^2} \rangle
(\tau_{v_1} \widetilde{h}_{I_{m+1}^1} - \widetilde{h}_{I_{m+1}^1}) 
\otimes (\tau_{v_2} \widetilde{h}_{I_{m+1}^2} - \widetilde{h}_{I_{m+1}^2}).
\end{align*}
The inequality \eqref{SDww} implies
\begin{align}\label{GA0}
\|\Gamma^0 (\vec{f})\|_{L^p} 
\le (|v| + |v_1| + |v_2|) \|\nabla b_1\|_{L^{\infty}} \|\mathbb{S}_{\D_{\w}}^k (\vec{f})\|_{L^p}
\lesssim |v| \prod_{j=1}^m \|f_j\|_{L^{p_j}},
\end{align}
where the implicit constant is independent of $\w$, $v$, and $\vec{f}$.

To bound $\Gamma^1$ and $\Gamma^2$, for any $x = (x_1, x_2) \in \Rnn$, define 
\begin{align*}
b_Q^1(x_1) & := b_1(x_1, c_{Q^2}) - b_1(c_{Q^1}, c_{Q^2}), 
\\
b_Q^2(x_2) & := b_1(c_{Q^1}, x_2) - b_1(c_{Q^1}, c_{Q^2}), 
\\
b_Q(x_1, x_2) & := b_1(x_1, x_2) - b_1(x_1, c_{Q^2}) - b_1(c_{Q^1}, x_2) + b_1(c_{Q^1}, c_{Q^2}).
\end{align*}
Perform the decomposition 
\begin{align*}
\Gamma^1 - \Gamma^2
= \widetilde{\Gamma}^1 - \widetilde{\Gamma}^2 
+ \widetilde{\Gamma}^3 - \widetilde{\Gamma}^4
+ \widetilde{\Gamma}^5 - \widetilde{\Gamma}^6, 
\end{align*}
where 
\begin{align*}
\widetilde{\Gamma}^1 (\vec{f})
& := \sum_Q \widetilde{a}_{(I_j), Q} \, b_Q \, 
\langle f_1, \widetilde{h}_{I_1^1} \otimes \widetilde{h}_{I_1^2} \rangle
(\tau_{v_1} \widetilde{h}_{I_{m+1}^1} - \widetilde{h}_{I_{m+1}^1}) 
\otimes (\tau_{v_2} \widetilde{h}_{I_{m+1}^2} - \widetilde{h}_{I_{m+1}^2}),
\\
\widetilde{\Gamma}^2 (\vec{f})
& := \sum_Q \widetilde{a}_{(I_j), Q} \,
\langle b_Q f_1, \widetilde{h}_{I_1^1} \otimes \widetilde{h}_{I_1^2} \rangle
(\tau_{v_1} \widetilde{h}_{I_{m+1}^1} - \widetilde{h}_{I_{m+1}^1}) 
\otimes (\tau_{v_2} \widetilde{h}_{I_{m+1}^2} - \widetilde{h}_{I_{m+1}^2}),
\\
\widetilde{\Gamma}^3 (\vec{f})
& := \sum_Q \widetilde{a}_{(I_j), Q} \, b_Q^1 \, 
\langle f_1, \widetilde{h}_{I_1^1} \otimes \widetilde{h}_{I_1^2} \rangle
(\tau_{v_1} \widetilde{h}_{I_{m+1}^1} - \widetilde{h}_{I_{m+1}^1}) 
\otimes (\tau_{v_2} \widetilde{h}_{I_{m+1}^2} - \widetilde{h}_{I_{m+1}^2}),
\\
\widetilde{\Gamma}^4 (\vec{f})
& := \sum_Q \widetilde{a}_{(I_j), Q} \,
\langle b_Q^1 f_1, \widetilde{h}_{I_1^1} \otimes \widetilde{h}_{I_1^2} \rangle
(\tau_{v_1} \widetilde{h}_{I_{m+1}^1} - \widetilde{h}_{I_{m+1}^1}) 
\otimes (\tau_{v_2} \widetilde{h}_{I_{m+1}^2} - \widetilde{h}_{I_{m+1}^2}),
\\
\widetilde{\Gamma}^5 (\vec{f})
& := \sum_Q \widetilde{a}_{(I_j), Q} \, b_Q^2 \, 
\langle f_1, \widetilde{h}_{I_1^1} \otimes \widetilde{h}_{I_1^2} \rangle
(\tau_{v_1} \widetilde{h}_{I_{m+1}^1} - \widetilde{h}_{I_{m+1}^1}) 
\otimes (\tau_{v_2} \widetilde{h}_{I_{m+1}^2} - \widetilde{h}_{I_{m+1}^2}),
\\
\widetilde{\Gamma}^6 (\vec{f})
& := \sum_Q \widetilde{a}_{(I_j), Q} \,
\langle b_Q^2 f_1, \widetilde{h}_{I_1^1} \otimes \widetilde{h}_{I_1^2} \rangle
(\tau_{v_1} \widetilde{h}_{I_{m+1}^1} - \widetilde{h}_{I_{m+1}^1}) 
\otimes (\tau_{v_2} \widetilde{h}_{I_{m+1}^2} - \widetilde{h}_{I_{m+1}^2}).
\end{align*}
In what follows, we will demonstrate 
\begin{align}\label{GAj}
\|\widetilde{\Gamma}^i (\vec{f})\|_{L^p} 
\lesssim |v|^{\frac{\gamma}{p}} \prod_{j=1}^m \|f_j\|_{L^{p_j}}, 
\quad i = 1, \ldots, 6,
\end{align}
where the implicit constant is independent of $\w$, $v$, and $\vec{f}$. Thus, \eqref{SDC-3} is a consequence of \eqref{ttt-1}--\eqref{GAj}.

\noindent{\bf $\bullet$ Estimates for $\widetilde{\Gamma}^2$.} 
We split
\begin{align*}
\widetilde{\Gamma}^2
= \widetilde{\Gamma}^2_1 + \widetilde{\Gamma}^2_2 
+ \widetilde{\Gamma}^2_3 + \widetilde{\Gamma}^2_4,
\end{align*}
where
\begin{align*}
\widetilde{\Gamma}^2_1 (\vec{f})
& := \sum_{\substack{Q \cap \supp(b_1) \neq \emptyset \\ 
\ell(Q^1) \le 2^{N_0} \\ \ell(Q^2) \le 2^{N_0}}} 
\widetilde{a}_{(I_j), Q} \,
\langle b_Q f_1, \widetilde{h}_{I_1^1} \otimes \widetilde{h}_{I_1^2} \rangle
(\tau_{v_1} \widetilde{h}_{I_{m+1}^1} - \widetilde{h}_{I_{m+1}^1}) 
\otimes (\tau_{v_2} \widetilde{h}_{I_{m+1}^2} - \widetilde{h}_{I_{m+1}^2}),
\\
\widetilde{\Gamma}^2_2 (\vec{f})
& := \sum_{\substack{Q \cap \supp(b_1) \neq \emptyset \\ 
\ell(Q^1) \le 2^{N_0} \\ \ell(Q^2) > 2^{N_0}}} 
\widetilde{a}_{(I_j), Q} \,
\langle b_Q f_1, \widetilde{h}_{I_1^1} \otimes \widetilde{h}_{I_1^2} \rangle
(\tau_{v_1} \widetilde{h}_{I_{m+1}^1} - \widetilde{h}_{I_{m+1}^1}) 
\otimes (\tau_{v_2} \widetilde{h}_{I_{m+1}^2} - \widetilde{h}_{I_{m+1}^2}),
\\
\widetilde{\Gamma}^2_3 (\vec{f})
& := \sum_{\substack{Q \cap \supp(b_1) \neq \emptyset \\ 
\ell(Q^1) > 2^{N_0} \\ \ell(Q^2) \le 2^{N_0}}} 
\widetilde{a}_{(I_j), Q} \,
\langle b_Q f_1, \widetilde{h}_{I_1^1} \otimes \widetilde{h}_{I_1^2} \rangle
(\tau_{v_1} \widetilde{h}_{I_{m+1}^1} - \widetilde{h}_{I_{m+1}^1}) 
\otimes (\tau_{v_2} \widetilde{h}_{I_{m+1}^2} - \widetilde{h}_{I_{m+1}^2}),
\\
\widetilde{\Gamma}^2_4 (\vec{f})
& := \sum_{\substack{Q \cap \supp(b_1) \neq \emptyset \\ 
\ell(Q^1) > 2^{N_0} \\ \ell(Q^2) > 2^{N_0}}} 
\widetilde{a}_{(I_j), Q} \,
\langle b_Q f_1, \widetilde{h}_{I_1^1} \otimes \widetilde{h}_{I_1^2} \rangle
(\tau_{v_1} \widetilde{h}_{I_{m+1}^1} - \widetilde{h}_{I_{m+1}^1}) 
\otimes (\tau_{v_2} \widetilde{h}_{I_{m+1}^2} - \widetilde{h}_{I_{m+1}^2}).
\end{align*}
An elementary estimate gives 
\begin{align}\label{AIJQ}
|\widetilde{a}_{(I_j), Q}|
\lesssim |Q|^{\frac{1}{p_1} - \frac1p} \prod_{j=2}^m \|f_j\|_{L^{p_j}},
\end{align}
and
\begin{align}\label{tauvh}
\|\tau_{v_i} \widetilde{h}_{I^i} - \widetilde{h}_{I^i}\|_{L^s(\R^{n_i})} 
\lesssim |I^i|^{\frac1s - \frac12} \big[ |v_i|/\ell(I^i) \big]^{\frac{\gamma}{s}}, \qquad 
0 < \gamma \le 1 < s < \infty.
\end{align}
Then by \eqref{ppp}--\eqref{bbn} and \eqref{AIJQ}--\eqref{tauvh}, 
\begin{align*}
\|\widetilde{\Gamma}^2_1 (\vec{f}) \|_{L^p}
& \lesssim \sum_{\kappa=1}^{\kappa_0} \sum_{Q \subset P_{\kappa}} 
|Q|^{-\frac1p + \frac{1}{p_1}} 
\|\nabla b_1\|_{L^{\infty}} \ell(Q^1) \ell(Q^2) 
\|f_1 \mathbf{1}_Q\|_{L^{p_1}} 
|Q|^{\frac12 - \frac{1}{p_1}}
\\
&\quad \times |Q|^{\frac1p - \frac12}
\big[ |v_1|/\ell(Q^1) \big]^{\frac{\gamma}{p}} 
\big[ |v_2|/\ell(Q^2) \big]^{\frac{\gamma}{p}}
\prod_{j=2}^m \|f_j\|_{L^{p_j}}
\\
& \lesssim |v|^{\frac{\gamma}{p}} \sum_{\kappa=1}^{\kappa_0} \sum_{s_1, s_2 \ge 0} 
\sum_{\substack{(Q^1)^{(s_1)} = P_{\kappa}^1 \\ (Q^2)^{(s_2)} = P_{\kappa}^2}} 
\ell(Q^1)^{1-\frac{\gamma}{p}} \ell(Q^2)^{1-\frac{\gamma}{p}} 
\|f_1 \mathbf{1}_Q\|_{L^{p_1}} \prod_{j=2}^m \|f_j\|_{L^{p_j}}
\\
& \lesssim |v|^{\frac{\gamma}{p}} \sum_{s_1, s_2 \ge 0}
2^{-s_1(1-\frac{\gamma}{p}-\frac{n_1}{p'_1})} 
2^{-s_2(1-\frac{\gamma}{p}-\frac{n_2}{p'_1})} 
\prod_{j=1}^m \|f_j\|_{L^{p_j}}
\\
&\lesssim |v|^{\frac{\gamma}{p}} \prod_{j=1}^m \|f_j\|_{L^{p_j}}.
\end{align*}
Similarly, 
\begin{align*}
\|\widetilde{\Gamma}^2_2 (\vec{f}) \|_{L^p}
& \lesssim \sum_{\kappa=1}^{\kappa_0} 
\sum_{\substack{Q^1 \subset P_{\kappa}^1 \\ Q^2 \supset P_{\kappa}^2}} 
\|\nabla b_1\|_{L^{\infty}} \ell(Q^1) \|f_1 \mathbf{1}_{Q^1}\|_{L^{p_1}} 
\big[ |v_1|/\ell(Q^1) \big]^{\frac{\gamma}{p}} 
\big[ |v_2|/\ell(Q^2) \big]^{\frac{\gamma}{p}}
\prod_{j=2}^m \|f_j\|_{L^{p_j}}
\\
& \lesssim |v|^{\frac{\gamma}{p}} \sum_{\kappa=1}^{\kappa_0} \sum_{s_1, s_2 \ge 0} 
\sum_{\substack{(Q^1)^{(s_1)} = P_{\kappa}^1 \\ Q^2 = (P_{\kappa}^2)^{(s_2)}}} 
\ell(Q^1)^{1-\frac{\gamma}{p}} \ell(Q^2)^{-\frac{\gamma}{p}} 
\|f_1 \mathbf{1}_{Q^1}\|_{L^{p_1}} 
\prod_{j=2}^m \|f_j\|_{L^{p_j}}
\\
& \lesssim |v|^{\frac{\gamma}{p}} \sum_{s_1, s_2 \ge 0}
2^{-s_1(1-\frac{\gamma}{p}-\frac{n_1}{p'_1})} 
2^{-s_2 \frac{\gamma}{p}} 
\prod_{j=1}^m \|f_j\|_{L^{p_j}}
\\
&\lesssim |v|^{\frac{\gamma}{p}} \prod_{j=1}^m \|f_j\|_{L^{p_j}}.
\end{align*}
Symmetrically, one has
\begin{align*}
\|\widetilde{\Gamma}^2_3 (\vec{f}) \|_{L^p}
\lesssim |v|^{\frac{\gamma}{p}} \prod_{j=1}^m \|f_j\|_{L^{p_j}}.
\end{align*}
Moreover,
\begin{align*}
\|\widetilde{\Gamma}^2_4 (\vec{f}) \|_{L^p}
& \lesssim \sum_{\kappa=1}^{\kappa_0} 
\sum_{\substack{Q^1 \supset P_{\kappa}^1 \\ Q^2 \supset P_{\kappa}^2}} 
\|b_1\|_{L^{\infty}}
\big[ |v_1|/\ell(Q^1) \big]^{\frac{\gamma}{p}} 
\big[ |v_2|/\ell(Q^2) \big]^{\frac{\gamma}{p}} 
\prod_{j=1}^m \|f_j\|_{L^{p_j}}
\\
& \lesssim |v|^{\frac{\gamma}{p}} \sum_{s_1, s_2 \ge 0} 
2^{-s_1 \frac{\gamma}{p}} 2^{-s_2 \frac{\gamma}{p}} 
\prod_{j=1}^m \|f_j\|_{L^{p_j}} 
\lesssim |v|^{\frac{\gamma}{p}} \prod_{j=1}^m \|f_j\|_{L^{p_j}}.
\end{align*}
Then gathering these estimates above, we obtain 
\begin{align*}
\|\widetilde{\Gamma}^2 (\vec{f}) \|_{L^p}
\lesssim |v|^{\frac{\gamma}{p}} \prod_{j=1}^m \|f_j\|_{L^{p_j}}.
\end{align*}

\noindent{\bf $\bullet$ Estimates for $\widetilde{\Gamma}^1$.} 
Rewrite
\begin{align}\label{GA1123}
\widetilde{\Gamma}^1
= \widetilde{\Gamma}^1_1 + \widetilde{\Gamma}^1_2 + \widetilde{\Gamma}^1_3,
\end{align}
where 
\begin{align*}
\widetilde{\Gamma}^1_1 (\vec{f})
& := \sum_{\substack{\ell(Q^1) > |v_1| \\ \ell(Q^2) > |v_2|}} 
\widetilde{a}_{(I_j), Q} \, b_Q \,
\langle f_1, \widetilde{h}_{I_1^1} \otimes \widetilde{h}_{I_1^2} \rangle
(\tau_{v_1} \widetilde{h}_{I_{m+1}^1} - \widetilde{h}_{I_{m+1}^1}) 
\otimes (\tau_{v_2} \widetilde{h}_{I_{m+1}^2} - \widetilde{h}_{I_{m+1}^2}),
\\
\widetilde{\Gamma}^1_2 (\vec{f})
& := \sum_{\substack{Q^1 \in \D^1 \\ \ell(Q^2) \le |v_2|}} 
\widetilde{a}_{(I_j), Q} \, b_Q \,
\langle f_1, \widetilde{h}_{I_1^1} \otimes \widetilde{h}_{I_1^2} \rangle
(\tau_{v_1} \widetilde{h}_{I_{m+1}^1} - \widetilde{h}_{I_{m+1}^1}) 
\otimes (\tau_{v_2} \widetilde{h}_{I_{m+1}^2} - \widetilde{h}_{I_{m+1}^2}),
\\
\widetilde{\Gamma}^1_3 (\vec{f})
& := \sum_{\substack{\ell(Q^1) \le |v_1| \\ \ell(Q^2) > |v_2|}} 
\widetilde{a}_{(I_j), Q} \, b_Q \,
\langle f_1, \widetilde{h}_{I_1^1} \otimes \widetilde{h}_{I_1^2} \rangle
(\tau_{v_1} \widetilde{h}_{I_{m+1}^1} - \widetilde{h}_{I_{m+1}^1}) 
\otimes (\tau_{v_2} \widetilde{h}_{I_{m+1}^2} - \widetilde{h}_{I_{m+1}^2}).
\end{align*}
Observe that $\supp \big( (\tau_{v_1} \widetilde{h}_{I_{m+1}^1} - \widetilde{h}_{I_{m+1}^1}) 
\otimes (\tau_{v_2} \widetilde{h}_{I_{m+1}^2} - \widetilde{h}_{I_{m+1}^2}) \big) \subset 3Q$ whenever $I_{m+1}^1 \times I_{m+1}^2 \subset Q^1 \times Q^2 = Q$, $\ell(Q^1) > |v_1|$, and $\ell(Q^2) > |v_2|$. Thus, as shown for $\widetilde{\Gamma}^2$ above,
\begin{align}\label{GA11}
\|\widetilde{\Gamma}^1_1 (\vec{f}) \|_{L^p}
\lesssim |v|^{\frac{\gamma}{p}} \prod_{j=1}^m \|f_j\|_{L^{p_j}}.
\end{align}
To treat $\widetilde{\Gamma}^1_2$, we split 
\begin{align}\label{GA121234}
\widetilde{\Gamma}^1_2 
= \widetilde{\Gamma}^1_{2, 1} - \widetilde{\Gamma}^1_{2, 2} 
- \widetilde{\Gamma}^1_{2, 3} + \widetilde{\Gamma}^1_{2, 4}, 
\end{align}
where
\begin{align*}
\widetilde{\Gamma}^1_{2, 1} (\vec{f})
& := \sum_{\substack{Q^1 \in \D^1 \\ \ell(Q^2) \le |v_2|}} 
\widetilde{a}_{(I_j), Q} \, b_Q \,
\langle f_1, \widetilde{h}_{I_1^1} \otimes \widetilde{h}_{I_1^2} \rangle \,
\tau_v \big(\widetilde{h}_{I_{m+1}^1} \otimes \widetilde{h}_{I_{m+1}^2}\big),
\\
\widetilde{\Gamma}^1_{2, 2} (\vec{f})
& := \sum_{\substack{Q^1 \in \D^1 \\ \ell(Q^2) \le |v_2|}} 
\widetilde{a}_{(I_j), Q} \, b_Q \,
\langle f_1, \widetilde{h}_{I_1^1} \otimes \widetilde{h}_{I_1^2} \rangle \,
\tau_{v_1} \widetilde{h}_{I_{m+1}^1} \otimes \widetilde{h}_{I_{m+1}^2},
\\
\widetilde{\Gamma}^1_{2, 3} (\vec{f})
& := \sum_{\substack{Q^1 \in \D^1 \\ \ell(Q^2) \le |v_2|}} 
\widetilde{a}_{(I_j), Q} \, b_Q \,
\langle f_1, \widetilde{h}_{I_1^1} \otimes \widetilde{h}_{I_1^2} \rangle \,
\widetilde{h}_{I_{m+1}^1} \otimes \tau_{v_2} \widetilde{h}_{I_{m+1}^2},
\\
\widetilde{\Gamma}^1_{2, 4} (\vec{f})
& := \sum_{\substack{Q^1 \in \D^1 \\ \ell(Q^2) \le |v_2|}} 
\widetilde{a}_{(I_j), Q} \, b_Q \,
\langle f_1, \widetilde{h}_{I_1^1} \otimes \widetilde{h}_{I_1^2} \rangle \, 
\widetilde{h}_{I_{m+1}^1} \otimes \widetilde{h}_{I_{m+1}^2}.
\end{align*}
First, consider $\widetilde{\Gamma}^1_{2, 4}$:
\begin{align*}
\widetilde{\Gamma}^1_{2, 4} 
= \widetilde{\Gamma}^{1, 1}_{2, 4} - \widetilde{\Gamma}^{1, 2}_{2, 4}, 
\end{align*}
where
\begin{align*}
\widetilde{\Gamma}^{1, 1}_{2, 4} (\vec{f})
& := \sum_{\substack{Q^1 \in \D^1 \\ \ell(Q^2) \le |v_2| \\ Q \cap \supp(b_1) \neq \emptyset}} 
\widetilde{a}_{(I_j), Q} [b_1(x_1, x_2) - b_1(x_1, c_{Q^2})] 
\langle f_1, \widetilde{h}_{I_1^1} \otimes \widetilde{h}_{I_1^2} \rangle \, 
\widetilde{h}_{I_{m+1}^1} \otimes \widetilde{h}_{I_{m+1}^2}, 
\\
\widetilde{\Gamma}^{1, 2}_{2, 4} (\vec{f})
& := \sum_{\substack{Q^1 \in \D^1 \\ \ell(Q^2) \le |v_2| \\ Q \cap \supp(b_1) \neq \emptyset}} 
\widetilde{a}_{(I_j), Q} [b_1(c_{Q^1}, x_2) - b_1(c_{Q^1}, c_{Q^2})] 
\langle f_1, \widetilde{h}_{I_1^1} \otimes \widetilde{h}_{I_1^2} \rangle \, 
\widetilde{h}_{I_{m+1}^1} \otimes \widetilde{h}_{I_{m+1}^2}.
\end{align*}
Note that
\begin{align}\label{bbL}
|b'_1(x)|
:= |b_1(x_1, x_2) - b_1(x_1, c_{Q^2})|
\le \|\nabla b_1\|_{L^{\infty}} \, \ell(Q^2), 
\end{align}
for all $x_1 \in \R^{n_1}$ and $x_2 \in Q^2$. By \eqref{bbL} and the boundedness of $m$-linear one-parameter shift (cf. \cite[Theorem 5.3]{CLSY} in the bilinear case), there holds
\begin{align*}
\|\widetilde{\Gamma}^{1, 1}_{2, 4} (\vec{f})\|_{L^p}
& \le \sum_{\kappa=1}^{\kappa_0} 
\sum_{\substack{\ell(Q^2) \le |v_2| \\ Q^2 \subset P_{\kappa}^2}} 
\Bigg\| \bigg\|b'_1 \sum_{Q^1 \in \D^1}
a_{(I_j), Q} \prod_{j=1}^m 
\big\langle \langle f_j, \widetilde{h}_{I_j^2} \rangle, \widetilde{h}_{I_j^1} \big\rangle
\widetilde{h}_{I_{m+1}^1} \bigg\|_{L^p(\R^{n_1})} 
\widetilde{h}_{I_{m+1}^2} \Bigg\|_{L^p(\R^{n_2})}
\\
& \lesssim \sum_{\kappa=1}^{\kappa_0} 
\sum_{\substack{\ell(Q^2) \le |v_2| \\ Q^2 \subset P_{\kappa}^2}} 
\ell(Q^2) \Big[ |Q^2|^{-m} \prod_{j=1}^{m+1} |I_j^2|^{\frac12} \Big]
\prod_{j=1}^m \|\langle f_j, \widetilde{h}_{I_j^2} \rangle\|_{L^p(\R^{n_1})}
\|\widetilde{h}_{I_{m+1}^2}\|_{L^p(\R^{n_2})}
\\
& \lesssim |v|^{\frac{\gamma}{p}} \sum_{\kappa=1}^{\kappa_0} 
\sum_{Q^2 \subset P_{\kappa}^2} 
\ell(Q^2)^{1-\frac{\gamma}{p}} \|f_1 \mathbf{1}_{Q^2}\|_{L^{p_1}} 
\prod_{j=2}^m \|f_j\|_{L^{p_j}}
\\
& \lesssim |v|^{\frac{\gamma}{p}} \sum_{\kappa=1}^{\kappa_0} 
\sum_{s_2 \ge 0} \sum_{(Q^2)^{(s_2)} = P_{\kappa}^2}
\ell(Q^2)^{1-\frac{\gamma}{p}} 
\|f_1 \mathbf{1}_{Q^2}\|_{L^{p_1}} 
\prod_{j=2}^m \|f_j\|_{L^{p_j}}
\\
& \lesssim |v|^{\frac{\gamma}{p}} \sum_{s_2 \ge 0}
2^{-s_2(1-\frac{\gamma}{p}-\frac{n_2}{p'_1})} 
\prod_{j=1}^m \|f_j\|_{L^{p_j}}
\\
& \lesssim |v|^{\frac{\gamma}{p}} \prod_{j=1}^m \|f_j\|_{L^{p_j}},
\end{align*} 
provided \eqref{ppp}. If we denote 
\begin{align*}
a_{(I_j), Q}^{x_2} 
:= a_{(I_j), Q} [b_1(c_{Q^1}, x_2) - b_1(c_{Q^1}, c_{Q^2})],
\end{align*} 
then 
\begin{align}\label{aaijk}
\sup_{x_2 \in Q^2} |a_{(I_j), Q}^{x_2}| 
\le \|\nabla b_1\|_{L^{\infty}} \ell(Q^2) 
\frac{\prod_{j=1}^{m+1} |I_j|^{\frac12}}{|Q|^m}. 
\end{align}
Using \eqref{aaijk} and the same argument as above, we obtain 
\begin{align*}
\|\widetilde{\Gamma}^{1, 2}_{2, 4} (\vec{f})\|_{L^p}
&\lesssim |v|^{\frac{\gamma}{p}} \prod_{j=1}^m \|f_j\|_{L^{p_j}}.
\end{align*}
Hence, 
\begin{align}\label{GA124}
\|\widetilde{\Gamma}^1_{2, 4} (\vec{f})\|_{L^p}
\le \|\widetilde{\Gamma}^{1, 1}_{2, 4} (\vec{f})\|_{L^p}
+ \|\widetilde{\Gamma}^{1, 2}_{2, 4} (\vec{f})\|_{L^p}
&\lesssim |v|^{\frac{\gamma}{p}} \prod_{j=1}^m \|f_j\|_{L^{p_j}}.
\end{align}

Next, to treat $\widetilde{\Gamma}^1_{2, 1}$, rewrite 
\begin{align}\label{GA12112}
\widetilde{\Gamma}^1_{2, 1} 
= \widetilde{\Gamma}^{1, 1}_{2, 1} + \widetilde{\Gamma}^{1, 2}_{2, 1},
\end{align}
where
\begin{align*}
\widetilde{\Gamma}^{1, 1}_{2, 1} (\vec{f})
& := \sum_{\substack{Q^1 \in \D^1 \\ \ell(Q^2) \le |v_2|}} 
\widetilde{a}_{(I_j), Q} \, \tau_v b_Q \,
\langle f_1, \widetilde{h}_{I_1^1} \otimes \widetilde{h}_{I_1^2} \rangle \,
\tau_v \big(\widetilde{h}_{I_{m+1}^1} \otimes \widetilde{h}_{I_{m+1}^2}\big),
\\
\widetilde{\Gamma}^{1, 2}_{2, 1} (\vec{f})
& := \sum_{\substack{Q^1 \in \D^1 \\ \ell(Q^2) \le |v_2|}}  
\widetilde{a}_{(I_j), Q} \, (b_Q - \tau_v b_Q) \,
\langle f_1, \widetilde{h}_{I_1^1} \otimes \widetilde{h}_{I_1^2} \rangle \,
\tau_v \big(\widetilde{h}_{I_{m+1}^1} \otimes \widetilde{h}_{I_{m+1}^2} \big).
\end{align*}
The inequality \eqref{GA124} gives 
\begin{align}\label{GA1121}
\|\widetilde{\Gamma}^{1, 1}_{2, 1} (\vec{f})\|_{L^p}
= \|\tau_v(\widetilde{\Gamma}^1_{2, 4} (\vec{f}))\|_{L^p}
= \|\widetilde{\Gamma}^1_{2, 4} (\vec{f})\|_{L^p}
\lesssim |v|^{\frac{\gamma}{p}} \prod_{j=1}^m \|f_j\|_{L^{p_j}}.
\end{align}
For $\widetilde{\Gamma}^{1, 2}_{2, 1}$, there holds 
\begin{align}\label{GA1221123}
\widetilde{\Gamma}^{1, 2}_{2, 1} 
= \widetilde{\Gamma}^{1, 2, 1}_{2, 1} 
- \widetilde{\Gamma}^{1, 2, 2}_{2, 1} 
- \widetilde{\Gamma}^{1, 2, 3}_{2, 1},
\end{align}
where 
\begin{align*}
\widetilde{\Gamma}^{1, 2, 1}_{2, 1} (\vec{f})
& := [b_1(x) - \tau_v b_1(x)] 
\sum_{\substack{Q^1 \in \D^1 \\ \ell(Q^2) \le |v_2|}} 
\widetilde{a}_{(I_j), Q} \,
\langle f_1, \widetilde{h}_{I_1^1} \otimes \widetilde{h}_{I_1^2} \rangle \,
\tau_v \big(\widetilde{h}_{I_{m+1}^1} \otimes \widetilde{h}_{I_{m+1}^2}\big),
\\
\widetilde{\Gamma}^{1, 2, 2}_{2, 1} (\vec{f})
& := \sum_{\substack{Q^1 \in \D^1 \\ \ell(Q^2) \le |v_2|}} 
\widetilde{a}_{(I_j), Q} \, [b_1(x_1, c_{Q^2}) - b_1(x_1 - v_1, c_{Q^2})] 
\langle f_1, \widetilde{h}_{I_1^1} \otimes \widetilde{h}_{I_1^2} \rangle \,
\tau_v \big(\widetilde{h}_{I_{m+1}^1} \otimes \widetilde{h}_{I_{m+1}^2}\big),
\\
\widetilde{\Gamma}^{1, 2, 3}_{2, 1} (\vec{f})
& := \sum_{\substack{Q^1 \in \D^1 \\ \ell(Q^2) \le |v_2|}} 
\widetilde{a}_{(I_j), Q} \, [b_1(c_{Q^1}, x_2) - b_1(c_{Q^1}, x_2 - v_2)]  
\langle f_1, \widetilde{h}_{I_1^1} \otimes \widetilde{h}_{I_1^2} \rangle \,
\tau_v \big(\widetilde{h}_{I_{m+1}^1} \otimes \widetilde{h}_{I_{m+1}^2}\big).
\end{align*}
Lemma \ref{lem:SDD} gives 
\begin{align}\label{GA12121}
\|\widetilde{\Gamma}^{1, 2, 1}_{2, 1} (\vec{f})\|_{L^p}
\le |v| \|\nabla b_1\|_{L^{\infty}} 
\bigg\|\sum_{\substack{Q^1 \in \D^1 \\ \ell(Q^2) \le |v_2|}} a_{(I_j), Q} \,
\prod_{j=1}^m \langle f_j, \widetilde{h}_{I_j^1} \otimes \widetilde{h}_{I_j^2} \rangle \,
\widetilde{h}_{I_{m+1}^1} \otimes \widetilde{h}_{I_{m+1}^2}\bigg\|_{L^p}
\lesssim |v| \prod_{j=1}^m \|f_j\|_{L^{p_j}}.
\end{align}
If we let 
\begin{align*}
a_{(I_j), Q}^{x_2} 
:= a_{(I_j), Q} [b_1(c_{Q^1}, x_2 + v_2) - b_1(c_{Q^1}, x_2)],
\end{align*} 
then 
\begin{align}\label{ax2}
\sup_{x_2 \in \R^{n_2}} |a_{(I_j), Q}^{x_2}| 
\le |v| \|\nabla b_1\|_{L^{\infty}}
\frac{\prod_{j=1}^{m+1} |I_j|^{\frac12}}{|Q|^m}. 
\end{align}
Invoking \eqref{ax2} and following the proof of Lemma \ref{lem:SDD}, we have 
\begin{align}\label{GA12321}
\|\widetilde{\Gamma}^{1, 2, 3}_{2, 1} (\vec{f})\|_{L^p}
= \bigg\|\sum_{\substack{Q^1 \in \D^1 \\ \ell(Q^2) \le |v_2|}} a_{(I_j), Q}^{x_2} \,
\prod_{j=1}^m \langle f_j, \widetilde{h}_{I_j^1} \otimes \widetilde{h}_{I_j^2} \rangle \,
\widetilde{h}_{I_{m+1}^1} \otimes \widetilde{h}_{I_{m+1}^2}\bigg\|_{L^p}
\lesssim |v| \prod_{j=1}^m \|f_j\|_{L^{p_j}}. 
\end{align}
Symmetrically, there holds
\begin{align*}
\|\widetilde{\Gamma}^{1, 2, 2}_{2, 1} (\vec{f})\|_{L^p}
\lesssim |v| \prod_{j=1}^m \|f_j\|_{L^{p_j}},  
\end{align*}
which along with \eqref{GA1221123}, \eqref{GA12121}, and \eqref{GA12321} gives 
\begin{align}\label{GA1221}
\|\widetilde{\Gamma}^{1, 2}_{2, 1} (\vec{f})\|_{L^p}
\lesssim |v| \prod_{j=1}^m \|f_j\|_{L^{p_j}}.
\end{align}
Then it follows from \eqref{GA12112}, \eqref{GA1121}, and \eqref{GA1221} that
\begin{align}\label{GA121}
\|\widetilde{\Gamma}^1_{2, 1} (\vec{f})\|_{L^p} 
\lesssim |v|^{\frac{\gamma}{p}} \prod_{j=1}^m \|f_j\|_{L^{p_j}}. 
\end{align}
Similarly, one can show 
\begin{align}\label{GA1223}
\|\widetilde{\Gamma}^1_{2, 2} (\vec{f})\|_{L^p} 
+ \|\widetilde{\Gamma}^1_{2, 3} (\vec{f})\|_{L^p} 
\lesssim |v|^{\frac{\gamma}{p}} \prod_{j=1}^m \|f_j\|_{L^{p_j}}. 
\end{align}
Consequently, by \eqref{GA121234}, \eqref{GA124}, \eqref{GA121}, and \eqref{GA1223},
\begin{align}\label{GA12}
\|\widetilde{\Gamma}^1_2 (\vec{f})\|_{L^p} 
\lesssim |v|^{\frac{\gamma}{p}} \prod_{j=1}^m \|f_j\|_{L^{p_j}}. 
\end{align}

Note that $\widetilde{\Gamma}^1_3$ is essentially symmetrical to $\widetilde{\Gamma}^1_2$ because the restriction $\ell(Q^2) > |v_2|$ contributes nothing. Hence, a similar argument as above yields 
\begin{align}\label{GA13}
\|\widetilde{\Gamma}^1_3 (\vec{f})\|_{L^p} 
\lesssim |v|^{\frac{\gamma}{p}} \prod_{j=1}^m \|f_j\|_{L^{p_j}}. 
\end{align}
Collecting \eqref{GA1123}, \eqref{GA11}, \eqref{GA12}, and \eqref{GA13}, we conclude 
\begin{align*}
\|\widetilde{\Gamma}^1 (\vec{f})\|_{L^p} 
\lesssim |v|^{\frac{\gamma}{p}} \prod_{j=1}^m \|f_j\|_{L^{p_j}}. 
\end{align*}

\noindent{\bf $\bullet$ Estimates for $\widetilde{\Gamma}^3$ and $\widetilde{\Gamma}^4$.} 
Observe that 
\begin{align}
\label{KBL-1}
|\langle b_Q^1 f_1, \widetilde{h}_{I_1^1}\rangle|
& \le 2 \|b_1\|_{L^{\infty}} \|f_1\|_{L^{p_1}(\R^{n_1})} 
\|\widetilde{h}_{I_1^1}\|_{L^{p'_1}(\R^{n_1})},
\\
\label{KBL-2}
|\langle b_Q^1 f_1, \widetilde{h}_{I_1^1}\rangle|
& \le \ell(Q^1) \|\nabla b_1\|_{L^{\infty}} 
\|f_1 \mathbf{1}_{Q^1}\|_{L^{p_1}(\R^{n_1})} \
\|\widetilde{h}_{I_1^1}\|_{L^{p'_1}(\R^{n_1})}.
\end{align}
Set
\begin{align*}
H_{Q^1} := 
\sum_{Q^2: \, Q \cap \supp(b_1) \neq \emptyset} a_{(I_j), Q} 
\big\langle \langle b_Q^1 f_1, \widetilde{h}_{I_1^1} \rangle, \widetilde{h}_{I_1^2} \big\rangle
\prod_{j=2}^m \big\langle \langle f_j, \widetilde{h}_{I_j^1} \rangle, \widetilde{h}_{I_j^2} \big\rangle
\widetilde{h}_{I_{m+1}^2}.
\end{align*}
In view of \eqref{ppp}, \eqref{tauvh}, and \eqref{KBL-1}--\eqref{KBL-2}, we follow the proof of  \cite[Theorem 5.3]{CLSY} to arrive at
\begin{align*}
\|\widetilde{\Gamma}^4 (\vec{f})\|_{L^p}
& \lesssim \bigg\|\sum_Q \widetilde{a}_{(I_j), Q} \,
\langle b_Q^1 f_1, \widetilde{h}_{I_1^1} \otimes \widetilde{h}_{I_1^2} \rangle
(\tau_{v_1} \widetilde{h}_{I_{m+1}^1} - \widetilde{h}_{I_{m+1}^1}) 
\otimes \widetilde{h}_{I_{m+1}^2} \bigg\|_{L^p}
\\
& \le \sum_{Q^1} \| H_{Q^1} \|_{L^p(\R^{n_2})} 
\|\tau_{v_1} \widetilde{h}_{I_{m+1}^1} - \widetilde{h}_{I_{m+1}^1}\|_{L^p(\R^{n_1})}
\\
& \lesssim \sum_{\kappa=1}^{\kappa_0} 
\sum_{\substack{\ell(Q^1) > 2^{N_0} \\ Q^1 \cap P_{\kappa}^1 \neq \emptyset}} 
|Q^1|^{\frac{1-m}{2}} \big\| \|f_1\|_{L^{p_1}(\R^{n_1})} \big\|_{L^{p_1}(\R^{n_2})}
\|\widetilde{h}_{I_1^1}\|_{L^{p'_1}(\R^{n_1})} 
\\
&\qquad\qquad\qquad \times 
\prod_{j=2}^m \|\langle f_j, \widetilde{h}_{I_j^1} \rangle \|_{L^{p_j}(\R^{n_2})}
\|\tau_{v_1} \widetilde{h}_{I_{m+1}^1} - \widetilde{h}_{I_{m+1}^1}\|_{L^p(\R^{n_1})}
\\
& \quad + \sum_{\kappa=1}^{\kappa_0} 
\sum_{\substack{\ell(Q^1) \le 2^{N_0} \\ Q^1 \cap P_{\kappa}^1 \neq \emptyset}} 
|Q^1|^{\frac{1-m}{2}} \ell(Q^1) 
\big\| \|f_1 \mathbf{1}_{Q^1}\|_{L^{p_1}(\R^{n_1})} \big\|_{L^{p_1}(\R^{n_2})}
\|\widetilde{h}_{I_1^1}\|_{L^{p'_1}(\R^{n_1})} 
\\
&\qquad\qquad\qquad \times 
\prod_{j=2}^m \|\langle f_j, \widetilde{h}_{I_j^1} \rangle \|_{L^{p_j}(\R^{n_2})}
\|\tau_{v_1} \widetilde{h}_{I_{m+1}^1} - \widetilde{h}_{I_{m+1}^1}\|_{L^p(\R^{n_1})}
\\
& \lesssim |v_1|^{\frac{\gamma}{p}} \sum_{\kappa=1}^{\kappa_0} 
\sum_{Q^1 \supset P_{\kappa}^1} \ell(Q^1)^{-\frac{\gamma}{p}} 
\prod_{j=1}^m \|f_j\|_{L^{p_j}} 
\\
& \quad + |v_1|^{\frac{\gamma}{p}} \sum_{\kappa=1}^{\kappa_0} 
\sum_{Q^1 \subset P_{\kappa}^1} 
\ell(Q^1)^{1-\frac{\gamma}{p}} \|f_1 \mathbf{1}_{Q^1}\|_{L^{p_1}}
 \prod_{j=2}^m \|f_j\|_{L^{p_j}} 
\\
& \lesssim |v|^{\frac{\gamma}{p}}
\bigg[\sum_{s_1 \ge 0} 2^{-s_1 \frac{\gamma}{p}}  
+  \sum_{s_1 \ge 0} 2^{-s_1(1-\frac{\gamma}{p}- \frac{n_1}{p'_1})} \bigg]
\prod_{j=1}^m \|f_j\|_{L^{p_j}}
\\
& \lesssim |v|^{\frac{\gamma}{p}} \prod_{j=1}^m \|f_j\|_{L^{p_j}}, 
\end{align*}
where the implicit constants are independent of $\w$, $v$, and $\vec{f}$.

In order to dominate $\widetilde{\Gamma}^3$, it suffices to treat 
\begin{align*}
\widetilde{\Gamma}^3_0 (\vec{f})
:= \sum_Q \widetilde{a}_{(I_j), Q} \, b_Q^1
\langle f_1, \widetilde{h}_{I_1^1} \otimes \widetilde{h}_{I_1^2} \rangle
(\tau_{v_1} \widetilde{h}_{I_{m+1}^1} - \widetilde{h}_{I_{m+1}^1}) 
\otimes \widetilde{h}_{I_{m+1}^2},
\end{align*}
which can be split into 
\begin{align}\label{GA30}
\widetilde{\Gamma}^3_0 
= \widetilde{\Gamma}^3_1 + \widetilde{\Gamma}^3_2 
+ \widetilde{\Gamma}^3_3 - \widetilde{\Gamma}^3_4,
\end{align}
where
\begin{align*}
\widetilde{\Gamma}^3_1 (\vec{f})
& := \sum_{\substack{\ell(Q^1) > |v_1| \\ Q^2 \in \D^2}} 
\widetilde{a}_{(I_j), Q} \, b_Q^1 \,
\langle f_1, \widetilde{h}_{I_1^1} \otimes \widetilde{h}_{I_1^2} \rangle
(\tau_{v_1} \widetilde{h}_{I_{m+1}^1} - \widetilde{h}_{I_{m+1}^1}) 
\otimes \widetilde{h}_{I_{m+1}^2},
\\
\widetilde{\Gamma}^3_2 (\vec{f})
& := \sum_{\substack{\ell(Q^1) \le |v_1| \\ Q^2 \in \D^2}} 
\widetilde{a}_{(I_j), Q} \, \tau_{v_1} b_Q^1 \,
\langle f_1, \widetilde{h}_{I_1^1} \otimes \widetilde{h}_{I_1^2} \rangle \,
\tau_{v_1} \widetilde{h}_{I_{m+1}^1} \otimes \widetilde{h}_{I_{m+1}^2},
\\
\widetilde{\Gamma}^3_3 (\vec{f})
& := \sum_{\substack{\ell(Q^1) \le |v_1| \\ Q^2 \in \D^2}} 
\widetilde{a}_{(I_j), Q} \, [b_1(x_1, c_{Q^2}) - b_1(x_1 - v_1, c_{Q^2})]
\langle f_1, \widetilde{h}_{I_1^1} \otimes \widetilde{h}_{I_1^2} \rangle \,
\tau_{v_1} \widetilde{h}_{I_{m+1}^1} \otimes \widetilde{h}_{I_{m+1}^2},
\\
\widetilde{\Gamma}^3_4 (\vec{f})
& := \sum_{\substack{\ell(Q^1) \le |v_1| \\ Q^2 \in \D^2}} 
\widetilde{a}_{(I_j), Q} \, b_Q^1 \,
\langle f_1, \widetilde{h}_{I_1^1} \otimes \widetilde{h}_{I_1^2} \rangle \,
\widetilde{h}_{I_{m+1}^1} \otimes \widetilde{h}_{I_{m+1}^2}.
\end{align*}
Since $\supp(\tau_{v_1} \widetilde{h}_{I_{m+1}^1} - \widetilde{h}_{I_{m+1}^1}) \subset 3Q^1$ for any $I_{m+1}^1 \subset Q^1$ and $\ell(Q^1) > |v_1|$, a similar argument as for $\widetilde{\Gamma}^4$ gives
\begin{align}\label{GA31}
\|\widetilde{\Gamma}^3_1 (\vec{f})\|_{L^p}
\lesssim |v|^{\frac{\gamma}{p}} \prod_{j=1}^m \|f_j\|_{L^{p_j}}. 
\end{align}
Applying the same strategy as for $\widetilde{\Gamma}^{1, 2, 2}_{2, 1}$ in \eqref{GA12321}, we have
\begin{align}\label{GA3234}
\|\widetilde{\Gamma}^3_2 (\vec{f})\|_{L^p}
= \|\widetilde{\Gamma}^3_4 (\vec{f})\|_{L^p}
\lesssim |v| \prod_{j=1}^m \|f_j\|_{L^{p_j}}
\quad \text{and} \quad 
\|\widetilde{\Gamma}^3_3 (\vec{f})\|_{L^p}
\lesssim |v| \prod_{j=1}^m \|f_j\|_{L^{p_j}}. 
\end{align}
Therefore, it follows from \eqref{GA30}--\eqref{GA3234} that
\begin{align*}
\|\widetilde{\Gamma}^3 (\vec{f})\|_{L^p}
\le 2 \|\widetilde{\Gamma}^3_0 (\vec{f})\|_{L^p}
\lesssim |v|^{\frac{\gamma}{p}} \prod_{j=1}^m \|f_j\|_{L^{p_j}}. 
\end{align*}

\noindent{\bf $\bullet$ Estimates for $\widetilde{\Gamma}^5$ and $\widetilde{\Gamma}^6$.} 
Since $\widetilde{\Gamma}^5$ and $\widetilde{\Gamma}^6$ are essentially symmetrical to $\widetilde{\Gamma}^3$ and $\widetilde{\Gamma}^4$ respectively, one can follow the argument above to deduce 
\begin{align*}
\|\widetilde{\Gamma}^5 (\vec{f})\|_{L^p}
+ \|\widetilde{\Gamma}^6 (\vec{f})\|_{L^p}
\lesssim |v|^{\frac{\gamma}{p}} \prod_{j=1}^m \|f_j\|_{L^{p_j}}. 
\end{align*}
So far, we have completed the proof of \eqref{GAj}.
\qed

\subsection{Commutators of partial paraproducts}\label{sec:cp}
Next, we prove the mean continuity of $[\b, \mathbb{E}_{\w} \mathbb{P}_{\D_{\w}}^{1, k}]_{\a}$, and a symmetrical proof holds for $[\b, \mathbb{E}_{\w} \mathbb{P}_{\D_{\w}}^{2, k}]_{\a}$. We will adopt the same strategy as in Section \ref{sec:cs}, and also clarify the similarities and differences in the current scenario. 

For any $k=(k_1, \ldots, k_{m+1}) \in \N^{m+1}$, a partial paraproduct $\mathbb{P}_{\D}^{1, k}$ can be viewed as a shift $\mathbb{S}_{\D}^{\mathbf{k}}$ of complexity $\mathbf{k} = \big((k_1, 0), \ldots, (k_{m+1}, 0) \big)$. Indeed, by definition of partial paraproducts, there holds
\begin{align}\label{cpaa-1}
|a_{(I_j), Q}| := \big| a_{(I_j^1), Q} |Q^2|^{-\frac{m}{2}} \big|
\le \frac{\prod_{j=1}^{m+1} |I_j^1|^{\frac12}}{|Q^1|^m} 
\frac{\prod_{j=1}^{m+1} |Q^2|^{\frac12}}{|Q^2|^m}, 
\end{align}
and
\begin{align}\label{cpaa-2}
\mathbb{P}_{\D}^{1, k}(\vec{f})
= \sum_{Q = Q^1 \times Q^2 \in \D} 
\sum_{\substack{I_j^1 \in \D_{k_j}^1(Q^1) \\ j=1, \ldots, m+1}} 
a_{(I_j), Q} 
\prod_{j=1}^m \langle f_j, \widetilde{h}_{I_j^1} \otimes \widetilde{h}_{j, Q^2} \rangle \, 
\widetilde{h}_{I_{m+1}^1} \otimes \widetilde{h}_{m+1, Q^2},
\end{align}
where the functions $\widetilde{h}_{I_j^1}$ and $\widetilde{h}_{j, Q^2}$ satisfy the following: there exist two different indices $j_0, j_1 \in \{1, \ldots, m+1\}$ so that $\widetilde{h}_{I_{j_0}^1} = h_{I_{j_0}^1}$, $\widetilde{h}_{I_{j_1}^1} = h_{I_{j_1}^1}$, and  $\widetilde{h}_{I_j^1} \in \{h_{I_j^1}^0, h_{I_j^1}\}$ for every $j \neq j_0, j_1$; there exists $j_2 \in \{1, \ldots, m+1\}$ so that $\widetilde{h}_{j_2, Q^2} = h_{Q^2}$ and $\widetilde{h}_{j, Q^2} = h_{Q^2}^0$ for every $j \ne j_2$. Thus, $\mathbb{P}_{\D}^{1, k}$ coincides with a shift $\mathbb{S}_{\D}^{\mathbf{k}}$ with $\mathbf{k} = \big((k_1, 0), \ldots, (k_{m+1}, 0) \big)$.

In view of \eqref{cpaa-1} and \eqref{cpaa-2}, we utilize \eqref{PDww} instead of \eqref{SDww} to reduce the mean continuity to the estimate \eqref{SDC-3}. To this end, we perform the decomposition \eqref{ttt-1}. The term $\Gamma_0$ can be handled by means of \eqref{PDww}. Due to \eqref{cpaa-1}, the estimates for $\widetilde{\Gamma}_2$ and $\widetilde{\Gamma}_4$ are the same as in Section \ref{sec:cs}. Note that $\widetilde{\Gamma}^5$ and $\widetilde{\Gamma}^6$ are essentially symmetrical to $\widetilde{\Gamma}^3$ and $\widetilde{\Gamma}^4$ respectively. Thus, it is enough to analyze $\widetilde{\Gamma}^1$ and $\widetilde{\Gamma}^3$. 

Observe that $\|\widetilde{\Gamma}^3 f\|_{L^p} \le 2 \|\widetilde{\Gamma}^3_0 f\|_{L^p}$ and $\widetilde{\Gamma}^3_0 = \widetilde{\Gamma}^3_1 + \widetilde{\Gamma}^3_2 + \widetilde{\Gamma}^3_3 - \widetilde{\Gamma}^3_4$, where $\widetilde{\Gamma}^3_1$ is similar to $\widetilde{\Gamma}^4$, and $\widetilde{\Gamma}^3_j$ is similar to $\widetilde{\Gamma}^{1, 2, 3}_{2, 1}$ for each $j=2, 3, 4$.

Since $\widetilde{\Gamma}^1
= \widetilde{\Gamma}^1_1 + \widetilde{\Gamma}^1_2 + \widetilde{\Gamma}^1_3$,where $\widetilde{\Gamma}^1_1$ is similar to $\widetilde{\Gamma}^2$, and $\widetilde{\Gamma}^1_3$ is symmetrical to $\widetilde{\Gamma}^1_2$, in order to bound $\widetilde{\Gamma}^1$, it suffices to treat $\widetilde{\Gamma}^1_2 = \widetilde{\Gamma}^1_{2, 1} - \widetilde{\Gamma}^1_{2, 2} - \widetilde{\Gamma}^1_{2, 3} + \widetilde{\Gamma}^1_{2, 4}$. By \eqref{cpaa-1}, the estimate for $\widetilde{\Gamma}^1_{2, 4} = \widetilde{\Gamma}^{1, 1}_{2, 4} - \widetilde{\Gamma}^{1, 2}_{2, 4}$  remain unchanged. Since $\widetilde{\Gamma}^1_{2, 2}$ and $\widetilde{\Gamma}^1_{2, 3}$ are similar to $\widetilde{\Gamma}^1_{2, 1}$, the estimate for $\widetilde{\Gamma}^1_2$ is reduced to that for $\widetilde{\Gamma}^1_{2, 1} = \widetilde{\Gamma}^{1, 1}_{2, 1} + \widetilde{\Gamma}^{1, 2}_{2, 1}$. Note that $\widetilde{\Gamma}^{1, 1}_{2, 1} = \tau_v \widetilde{\Gamma}^1_{2, 4}$ and $\widetilde{\Gamma}^{1, 2}_{2, 1} = \widetilde{\Gamma}^{1, 2, 1}_{2, 1} - \widetilde{\Gamma}^{1, 2, 2}_{2, 1} - \widetilde{\Gamma}^{1, 2, 3}_{2, 1}$, where $\widetilde{\Gamma}^{1, 2, 2}_{2, 1}$ is identical to that in Section \ref{sec:cs} because of the size estimate for $a_{(I_j), Q}$ \eqref{cpaa-1}, and $\widetilde{\Gamma}^{1, 2, 2}_{2, 1}$ is symmetrical to $\widetilde{\Gamma}^{1, 2, 3}_{2, 1}$. Consequently, it remains to dominate $\widetilde{\Gamma}^{1, 2, 3}_{2, 1}$.

By changing variables, it just needs to consider 
\begin{align*}
\Xi (\vec{f})(x)
& := \sum_{\substack{Q^1 \in \D^1 \\ \ell(Q^2) \le |v_2|}} 
a_{(I_j^1), Q} \, b_{Q^1}(x_2)  \prod_{j=1}^m 
\langle f_j, \widetilde{h}_{I_j^1} \otimes \overline{h}_{j, Q^2} \rangle \,
\widetilde{h}_{I_{m+1}^1} \otimes \overline{h}_{m+1, Q^2}(x),
\end{align*}
where $b_{Q^1}(x_2) := b(c_{Q^1}, x_2+v_2) - b(c_{Q^1}, x_2)$.

\noindent{\bf$\bullet$ Case 1: $\overline{h}_{m+1, Q^2} = \frac{\mathbf{1}_{Q^2}}{|Q^2|}$.} Let $f_{m+1} \in L^{p'}(\Rnn)$ with $\|f_{m+1}\|_{L^p} \le 1$. Following the proof of \cite[Theorem 6.7]{LMV21}, we treat the quantity $|\langle \Xi(\vec{f}), f_{m+1} \rangle|$, which contains the term 
\begin{align}\label{bQM}
\bigg| \bigg\langle b_{Q^1} f_{m+1}, h_{Q^1}^0 
\otimes \frac{\mathbf{1}_{Q^2}}{|Q^2|} \bigg\rangle \bigg| 
\le |v| \|\nabla b_1\|_{L^{\infty}} |Q^1|^{\frac12} M_{\D} f_{m+1}(x), \quad\forall x \in Q. 
\end{align}
In light of \eqref{bQM}, the rest of the proof is the same as there. Thus, 
\begin{align*}
\|\Xi (\vec{f})\|_{L^p} 
\lesssim |v| \prod_{j=1}^m \|f_j\|_{L^{p_j}}, 
\end{align*}
where the implicit constant is independent of $\D$, $v$, and $\vec{f}$.

\noindent{\bf$\bullet$ Case 2: $\overline{h}_{m+1, Q^2} = h_{Q^2}$.}  By the assumption on partial paraproducts, we may assume $\widetilde{h}_{I_1^1} = h_{I_1^1}$. Then we see that 
\begin{align*}
\Xi^{1*} (\vec{f})
= \sum_{\substack{Q^1 \in \D^1 \\ I_j^1 \in \D^1_{k_j}(Q^1) \\ j=1, \ldots, m+1}} 
h_{I_1^1} \otimes \bigg[ \sum_{\ell(Q^2) \le |v_2|} 
a_{(I_j^1), Q} \, \big\langle b_{Q^1} \langle f_1, \widetilde{h}_{I_{m+1}^1} \rangle, h_{Q^2} \big \rangle  \prod_{j=2}^m \big\langle \langle f_j, \widetilde{h}_{I_j^1} \rangle, \overline{h}_{j, Q^2} \big\rangle \,
\overline{h}_{1, Q^2} \bigg].
\end{align*}
By duality, it suffices to show 
\begin{align}\label{XFL}
\|\Xi^{1*} (\vec{f})\|_{L^r} 
\lesssim |v| \prod_{j=1}^m \|f_j\|_{L^{r_j}}, 
\end{align}
for all $r, r_1, \ldots, r_m \in (1, \infty)$ satisfying $\frac1r = \sum_{j=1}^m \frac{1}{r_j}$, where the implicit constant is independent of $\D$, $v$, and $\vec{f}$.

To proceed, we define an $m$-linear one-parameter paraproduct on $\R^{n_2}$: 
\begin{align*}
\Pi_{a} (\vec{g}) 
:= \sum_{Q^2 \in \D^2} a_{Q^2} \, 
\prod_{j=1}^m \langle g_j, \overline{h}_{j, Q^2} \rangle_I \, \overline{h}_{m+1, Q^2}, 
\end{align*}
where there exists a unique $j_0 \in \{1, \ldots, m+1\}$ so that $\overline{h}_{j_0, Q^2} = h_{Q^2}$ and $\overline{h}_{j, Q^2} = \frac{\mathbf{1}_{Q^2}}{|Q^2|}$ for each $j \neq j_0$, and $\sup_{\D^2} \|a\|_{\BMO(\D^2)} \lesssim 1$. As shown in \cite[Theorem 5.24]{CLSY}, 
\begin{align*}
\|\Pi_a (\vec{g})\|_{L^p}
\lesssim \sup_{\D^2} \|a\|_{\BMO(\D^2)} \prod_{j=1}^m \|f_j\|_{L^{p_j}}.
\end{align*}
which along with \cite[eq. (6.10)]{LMV20} gives
\begin{align}\label{PFG}
\bigg\| \bigg( \sum_{j \in \Z} \bigg|\sum_{k \in \Z} 
\Pi_{a_{k, j}} (\vec{g}_{k, j}) \bigg|^q \bigg)^{\frac1q}\bigg\|_{L^p(\R^{n_2})}
\lesssim \sup_{k, j \in \Z} \|a_{k, j}\|_{\BMO}
\bigg\| \bigg( \sum_{j \in \Z} \bigg|\sum_{k \in \Z} 
\prod_{i=1}^m (M_{\D} g_{k, j}^i) \bigg|^q \bigg)^{\frac1q}\bigg\|_{L^p(\R^{n_2})},
\end{align}
for all $p, q \in (0, \infty)$. Recall that 
\begin{align}\label{ABM}
\|(a_{(I_j^1), Q})\|_{\BMO(\D^2)} 
\lesssim \frac{\prod_{j=1}^{m+1} |I_j^1|^{\frac12}}{|Q^1|}
=: a_{(I_j^1), Q^1}.
\end{align}
Denote 
\begin{align*}
\varphi (f) 
:= \sum_{I^1 \in \D^1} h_{I^1} \otimes M_{\D^2} (\langle f, h_{I^1} \rangle).
\end{align*}
Then 
\begin{align}\label{TDM}
M_{\D^2} (\langle f, h_{I^1} \rangle)
= \langle \varphi (f), h_{I^1} \rangle 
\quad \text{and} \quad
M_{\D^2} (\langle f, h_{I^1}^0 \rangle) 
\le \langle M_{\D} f, h_{I^1}^0 \rangle, 
\end{align}
For convenience of notation, suppose that $\widetilde{h}_{I_{m+1}^1} = h_{I_{m+1}^1}$ and $\widetilde{h}_{I_j^1} = h_{I_j^1}^0$ for every $j=2, \ldots, m$. Then by \eqref{ssf-2} and \eqref{PFG}--\eqref{TDM}, we deduce
\begin{align*}
& \|\Xi^{1*} (\vec{f})\|_{L^r}
\lesssim \bigg\| \bigg(\sum_{Q^1 \in \D^1} 
\bigg| \sum_{(I_j^1)} 
h_{I_1^1} \otimes \bigg( \sum_{\ell(Q^2) \le |v_2|} a_{(I_j), Q} \, 
\cdots \cdots \overline{h}_{1, Q^2} 
\bigg) \bigg|^2 \bigg)^{\frac12} \bigg\|_{L^r}
\\
& \lesssim \sup_{Q^1 \in \D^1} \|b_{Q^1}\|_{L^{\infty}} 
\bigg\| \bigg(\sum_{Q^1 \in \D^1} 
\bigg| \sum_{(I_j^1)} a_{(I_j^1), Q^1} \, 
h_{I_1^1} \otimes \Big[M_{\D^2} (\langle f_1, \widetilde{h}_{I_{m+1}^1} \rangle)
\prod_{j=2}^m M_{\D^2} (\langle f_j, \widetilde{h}_{I_j^1} \rangle) \Big]
\bigg|^2 \bigg)^{\frac12} \bigg\|_{L^r}
\\
& \lesssim |v| \bigg\|\bigg(\sum_{Q^1 \in \D^1} 
\bigg| \sum_{(I_j^1)} a_{(I_j^1), Q^1} \, 
h_{I_1^1} \otimes \Big[ \langle \varphi(f_1), h_{I_{m+1}^1} \rangle 
\prod_{j=2}^m \langle M_{\D} f_j, h_{I_j^1}^0 \rangle \Big] \bigg|^2 \bigg)^{\frac12} \bigg\|_{L^r}
\\
& = |v| \Bigg\| \bigg\|\bigg(\sum_{Q^1 \in \D^1} 
\bigg| \sum_{(I_j^1)} a_{(I_j^1), Q^1} \, 
\Big[ \langle \varphi(f_1), h_{I_{m+1}^1} \rangle 
\prod_{j=2}^m \langle M_{\D} f_j, h_{I_j^1}^0 \rangle \Big] 
\bigg|^2 \frac{\mathbf{1}_{I_1^1}}{|I_1^1|} 
\bigg)^{\frac12} \bigg\|_{L^r(\R^{n_1})} \Bigg\|_{L^r(\R^{n_2})}
\\
& \simeq |v| \Bigg\| \bigg\| \sum_{Q^1 \in \D^1} 
\sum_{(I_j^1)} a_{(I_j^1), Q^1} \, 
\langle \varphi(f_1), h_{I_{m+1}^1} \rangle 
\prod_{j=2}^m \langle M_{\D} f_j, h_{I_j^1}^0 \rangle \, h_{I_1^1}
\bigg\|_{L^r(\R^{n_1})} \Bigg\|_{L^r(\R^{n_2})}
\\
& \lesssim |v| \bigg\| \|\varphi(f_1)\|_{L^{r_1}(\R^{n_1})}
\prod_{j=2}^m \|M_{\D}(f_j)\|_{L^{r_j}(\R^{n_1})} \bigg\|_{L^r(\R^{n_2})}
\\
& = |v| \|\varphi(f_1)\|_{L^{r_1}} 
\prod_{j=2}^m \|M_{\D}(f_j)\|_{L^{r_j}}
\lesssim |v| \prod_{j=1}^m \|f_j\|_{L^{r_j}},
\end{align*}
where ``$\simeq$" is due to \eqref{sdff}, and after that we used the boundedness of $m$-linear one-parameter shifts. This completes the proof of \eqref{XFL}.
\qed

\subsection{Commutators of full paraproducts}\label{sec:cf}
Finally, let us demonstrate the mean continuity of $[\b, \mathbb{E}_{\w} \mathbb{F}_{\mathbf{a}_{\w}}]_{\a}$. We will modify the proof in Section \ref{sec:cs} to the current situation. 

A full paraproduct $\mathbb{F}_{\mathbf{a}}$ can be viewed as a shift $\mathbb{S}_{\D}^k$ of complexity $k = \big((0, 0), \ldots, (0, 0) \big)$. Indeed, by definition of full paraproducts, there holds
\begin{align}\label{cfaa-1}
|a'_I| := \big| a_I |I|^{-\frac{m}{2}} \big|
\lesssim |I|^{- \frac{m-1}{2}}, \quad \forall I \in \D,
\end{align}
and
\begin{align}\label{cfaa-2}
\mathbb{F}_{\mathbf{a}}(\vec{f})
= \sum_{I = I^1 \times I^2 \in \D} a'_I 
\prod_{j=1}^m \langle f_j, \widetilde{h}_{j, I^1} \otimes \widetilde{h}_{j, I^2} \rangle \, 
\widetilde{h}_{m+1, I^1} \otimes \widetilde{h}_{m+1, I^2},
\end{align}
where there exist $j_0^1, j_0^2 \in \{1, \ldots, m+1\}$ so that $\widetilde{h}_{j_0^1, I^1} = h_{I^1}$, $\widetilde{h}_{j_0^2, I^2} = h_{I^2}$, $\widetilde{h}_{j, I^1} = h_{I^1}^0$ for every $j \ne j_0^1$, and $\widetilde{h}_{j, I^2} = h_{I^2}^0$ for every $j \ne j_0^2$. Then it is clear that $\mathbb{F}_{\mathbf{a}}$ coincides with a shift $\mathbb{S}_{\D}^k$ with $k = \big((0, 0), \ldots, (0, 0) \big)$.

As argued in Section \ref{sec:cp}, with \eqref{Fbww}, \eqref{cfaa-1}, and \eqref{cfaa-2} in hand, we are eventually reduced to handling 
\begin{align*}
\Xi (\vec{f})(x) := 
\sum_{\substack{I^1 \in \D^1 \\ \ell(I^2) \le |v_2|}} a_I \, b_{I_1}(x_2) \, 
\prod_{j=1}^m \langle f_j, \overline{h}_{j, I^1} \otimes \overline{h}_{j, I^2} \rangle \, 
\overline{h}_{m+1, I^1} \otimes \overline{h}_{m+1, I^2}(x),
\end{align*}
where $b_{I^1}(x_2) := b(c_{I^1}, x_2+v_2) - b(c_{I^1}, x_2)$, and $\overline{h}_{j, I^1}$ and $\overline{h}_{j, I^2}$ are defined in Definition \ref{def:paraproduct}. Let $f_{m+1} \in L^{p'}(\Rnn)$ with $\|f_{m+1}\|_{L^{p'}} \le 1$. Then
\begin{align}\label{XFF}
&|\langle \Xi (\vec{f}), f_{m+1} \rangle|
\le \sum_{I = I^1 \times I^2 \in \D} |a_I|
\prod_{j=1}^m |\langle f_j, \overline{h}_{j, I^1} \otimes \overline{h}_{j, I^2} \rangle| \, 
|\langle b_{I^1} f_{m+1}, \overline{h}_{m+1, I^1} \otimes \overline{h}_{m+1, I^2}|
\\ \nonumber
&\lesssim \|\mathbf{a}\|_{\BMO(\D)} \bigg\| \bigg[\sum_{I \in \D} 
\prod_{j=1}^m |\langle f_j, \overline{h}_{j, I^1} \otimes \overline{h}_{j, I^2} \rangle|^2 
|\langle b_{I^1} f_{m+1}, \overline{h}_{m+1, I^1} \otimes \overline{h}_{m+1, I^2}|^2 
\frac{\mathbf{1}_I}{|I|} \bigg]^{\frac12} \bigg\|_{L^1}.
\end{align}
\noindent{$\bullet$ \bf Case 1: $\overline{h}_{m+1, I^2} = \frac{\mathbf{1}_{I^2}}{|I^2|}$ and $\overline{h}_{m+1, I^1} = \frac{\mathbf{1}_{I^1}}{|I^1|}$.} By the assumption on full paraproducts, we may assume that $\overline{h}_{1, I^1} = h_{I^1}$ and $\overline{h}_{2, I^2} = h_{I^2}$. It is not hard to verify that
\begin{align}
\label{bhh-11}
& |\langle f_1, h_{I^1} \otimes \overline{h}_{I^2} \rangle|
= |\langle \Delta_{I^1} f_1, h_{I^1} \otimes \overline{h}_{I^2} \rangle|
\le \langle |\Delta_{I^1} f_1| \rangle_I \, |I^1|^{\frac12},
\\
\label{bhh-12}
& |\langle f_2, \overline{h}_{I^1} \otimes h_{I^2} \rangle|
= |\langle \Delta_{I^2} f_2, \overline{h}_{I^1} \otimes h_{I^2} \rangle|
\le \langle |\Delta_{I^2} f_2| \rangle_I \, |I^2|^{\frac12},
\\
\label{bhh-13}
& |\langle f_j, \overline{h}_{I^1} \otimes \overline{h}_{I^2} \rangle|
\le \langle |f_j| \rangle_I, \quad j=3, \ldots, m, 
\\
\label{bhh-14}
& |\langle b_{I^1} f_{m+1}, \overline{h}_{I^1} \otimes \overline{h}_{I^2} \rangle|
\le |v| \|\nabla b\|_{L^{\infty}} \langle |f_{m+1}| \rangle_I.
\end{align}
Hence, it follows from \eqref{XFF}--\eqref{bhh-14} that
\begin{align}\label{XFF-1}
|\langle \Xi (\vec{f}), f_{m+1} \rangle|
\lesssim |v| \bigg\| \bigg[\sum_{I \in \D} \langle |\Delta_{I^1} f_1| \rangle_I^2
\langle |\Delta_{I^2} f_2| \rangle_I^2 \prod_{j=3}^m \langle |f_j| \rangle|^2 
\, \mathbf{1}_I \bigg]^{\frac12} \bigg\|_{L^1}.
\end{align}
\noindent{$\bullet$ \bf Case 2: $\overline{h}_{m+1, I^2} = \frac{\mathbf{1}_{I^2}}{|I^2|}$ and $\overline{h}_{m+1, I^1} = h_{I^1}$.} By symmetry, we may assume that $\overline{h}_{2, I^2} = h_{I^2}$. Then
\begin{align}
\label{bhh-21}
& |\langle f_1, \overline{h}_{I^1} \otimes \overline{h}_{I^2} \rangle|
\le \langle |f_1| \rangle_I,
\\
\label{bhh-22}
& |\langle f_2, \overline{h}_{I^1} \otimes h_{I^2} \rangle|
= |\langle \Delta_{I^2} f_2, \overline{h}_{I^1} \otimes h_{I^2} \rangle|
\le \langle |\Delta_{I^2} f_2| \rangle_I \, |I^2|^{\frac12},
\\
\label{bhh-23}
& |\langle f_j, \overline{h}_{I^1} \otimes \overline{h}_{I^2} \rangle|
\le \langle |f_j| \rangle_I, \quad j=3, \ldots, m, 
\\
\label{bhh-24}
& |\langle b_{I_1} f_{m+1}, h_{I^1} \otimes \overline{h}_{I^2} \rangle|
= |\langle b_{I_1} \Delta_{I^1} f_{m+1}, h_{I^1} \otimes \overline{h}_{I^2} \rangle|
\lesssim |v| \langle |\Delta_{I^1} f_{m+1}| \rangle_I \, |I^1|^{\frac12}.
\end{align}
By \eqref{XFF} and \eqref{bhh-21}--\eqref{bhh-24}, there holds
\begin{align}\label{XFF-2}
|\langle \Xi (\vec{f}), f_{m+1} \rangle|
\lesssim |v| \bigg\| \bigg[\sum_{I \in \D} \langle |\Delta_{I^2} f_2| \rangle_I^2
\langle |\Delta_{I^1} f_{m+1}| \rangle_I^2 \prod_{j=1, 3, \ldots, m} \langle |f_j| \rangle|^2 
\, \mathbf{1}_I \bigg]^{\frac12} \bigg\|_{L^1}.
\end{align}
\noindent{$\bullet$ \bf Case 3: $\overline{h}_{m+1, I^2} = h_{I^2}$ and $\overline{h}_{m+1, I^1} = \frac{\mathbf{1}_{I^1}}{|I^1|}$.} By symmetry, we may assume that $\overline{h}_{1, I^1} = h_{I^1}$. Then
\begin{align}
\label{bhh-31}
& |\langle f_1, h_{I^1} \otimes \overline{h}_{I^2} \rangle|
= |\langle \Delta_{I^1} f_1, h_{I^1} \otimes \overline{h}_{I^2} \rangle|
\le \langle |\Delta_{I^1} f_1| \rangle_I \, |I^1|^{\frac12},
\\
\label{bhh-32}
& |\langle f_j, \overline{h}_{I^1} \otimes \overline{h}_{I^2} \rangle|
\le \langle |f_j| \rangle_I, \quad j=2, \ldots, m, 
\\
\label{bhh-33}
& |\langle b_{I_1} f_{m+1}, \overline{h}_{I^1} \otimes h_{I^2} \rangle|
\lesssim |v| \langle |f_{m+1}| \rangle_I \, |I^2|^{\frac12}.
\end{align}
As a consequence of \eqref{XFF} and \eqref{bhh-31}--\eqref{bhh-33}, we have
\begin{align}\label{XFF-3}
|\langle \Xi (\vec{f}), f_{m+1} \rangle|
\lesssim |v| \bigg\| \bigg[\sum_{I \in \D} 
\langle |\Delta_{I^1} f_1| \rangle_I^2 \prod_{j=2}^{m+1} \langle |f_j| \rangle|^2 
\, \mathbf{1}_I \bigg]^{\frac12} \bigg\|_{L^1}.
\end{align}
\noindent{$\bullet$ \bf Case 4: $\overline{h}_{m+1, I^2} = h_{I^2}$ and $\overline{h}_{m+1, I^1} = h_{I^1}$.} Then one has
\begin{align}
\label{bhh-41}
& |\langle f_j, \overline{h}_{I^1} \otimes \overline{h}_{I^2} \rangle|
\le \langle |f_j| \rangle_I, \quad j=1, \ldots, m, 
\\
\label{bhh-42}
& |\langle b_{I_1} f_{m+1}, h_{I^1} \otimes h_{I^2} \rangle|
= |\langle b_{I_1} \Delta_{I^1} f_{m+1}, h_{I^1} \otimes h_{I^2} \rangle|
\lesssim |v| \langle |\Delta_{I^1} f_{m+1}| \rangle_I \, |I|^{\frac12}.
\end{align}
In view of \eqref{XFF} and \eqref{bhh-31}--\eqref{bhh-33}, we arrive at
\begin{align}\label{XFF-4}
|\langle \Xi (\vec{f}), f_{m+1} \rangle|
\lesssim |v| \bigg\| \bigg[\sum_{I \in \D}
\prod_{j=1}^m \langle |f_j| \rangle|^2 \, \langle |\Delta_{I^1} f_{m+1}| \rangle_I^2
\, \mathbf{1}_I \bigg]^{\frac12} \bigg\|_{L^1}.
\end{align}
Having obtained \eqref{XFF-1}, \eqref{XFF-2}, \eqref{XFF-3}, and \eqref{XFF-4}, we apply \cite[Theorem 5.5]{LMV21} to conlude 
\begin{align*}
|\langle \Xi (\vec{f}), f_{m+1} \rangle|
\lesssim |v| \prod_{j=1}^{m+1} \|f_j\|_{L^{p_j}}.
\end{align*}
The proof is complete. 
\qed

\section{Extrapolation of compactness}\label{sec:RdF}
This section is devoted to proving Theorems \ref{thm:RdF-cpt} and \ref{thm:RdF-bT}, which have been used in Sections \ref{sec:wcpt} and \ref{sec:wcc} respectively. After that, we will conclude  Theorems \ref{thm:cpt} and \ref{thm:bT} from Theorems \ref{thm:repre}--\ref{thm:RdF-bT}.

\subsection{Interpolation of compactness}
Let us first show Theorem \ref{thm:RdF-cpt} in the bi-parameter setting. Although the proof follows  the scheme in \cite{CLSY}, we present some details to clarify the main ingredients. 

First, we give an interpolation of multiple weights below, which allows us to use interpolation of compactness to show extrapolation of compactness.   

\begin{lemma}\label{lem:AA}
Let $\frac1p = \sum_{j=1}^m \frac{1}{p_j}>0$ with $p_1, \ldots, p_m \in (1, \infty]$ and $\frac1s = \sum_{j=1}^m \frac{1}{s_j}$ with $s_1, \ldots, s_m \in [1, \infty]$. Assume that $\vec{w}\in A_{\vec{p}}(\Rnn)$ and $\vec{v} \in A_{\vec{s}}(\Rnn)$. Then there exists $\theta \in (0, 1)$ such that $\vec{u} \in A_{\vec{r}}(\Rnn)$, 
where 
\begin{align*}
\frac1r = \sum_{j=1}^m \frac{1}{r_j}, \quad 
u = \prod_{j=1}^m u_j, \quad 
w_j = u_j^{1-\theta} v_j^{\theta},\quad 
\frac{1}{p_j} = \frac{1-\theta}{r_j} + \frac{\theta}{s_j}, \quad j=1, \ldots, m. 
\end{align*}
\end{lemma}

\begin{proof}
The proof heavily relies on the reverse H\"{o}lder inequality for $A_p(\Rnn)$, which was proved by \cite[p. 458]{GR} in the multi-parameter case. With this in hand, the proof is similar to that of \cite[Lemma C.1]{CLSY}.
\end{proof}

Next, let us recall extrapolation of boundedness in the bi-parameter case \cite[Theorem 3.12]{LMV21}, which extends the one-parameter extrapolation result in \cite{Nie}. 

\begin{theorem}\label{thm:RdF-bdd}
Let $\mathcal{F}$ be a family of $(m+1)$-tuples of non-negative functions. Assume that there exists some $\vec{q} = (q_1, \ldots, q_m) \in [1, \infty]^m$ such that for all $\vec{v} \in A_{\vec{q}}(\Rnn)$, 
\begin{align*}
\|f\|_{L^q(v^q)} 
\lesssim \prod_{j=1}^m \|f_j\|_{L^{q_j}(v_j^{q_j})}, 
\quad \forall (f, f_1, \ldots, f_m) \in \mathcal{F}, 
\end{align*}
where $\frac1q = \sum_{j=1}^m \frac{1}{q_j}$ and $v = \prod_{j=1}^m v_j$. Then for all $\vec{p} = (p_1, \ldots, p_m) \in (1, \infty]^m$ and for all $\vec{w} = (w_1, \ldots, w_m) \in A_{\vec{p}}(\Rnn)$, 
\begin{align*}
\|f\|_{L^p(w^p)} 
\lesssim \prod_{j=1}^m \|f_j\|_{L^{p_j}(w_j^{p_j})}, 
\quad \forall (f, f_1, \ldots, f_m) \in \mathcal{F}, 
\end{align*}
where $\frac1p = \sum_{j=1}^m \frac{1}{p_j} > 0$ and $w = \prod_{j=1}^m w_j$. 
\end{theorem}

Beyond that, the core of the proof of Theorem \ref{thm:RdF-cpt} is an interpolation of multilinear compact operators below, which determines the range of exponents. The greatest difficulty of Theorem \ref{thm:cpt-inter} is to obtain the quasi-Banach range of exponents, for which we use compactness criterion Theorem \ref{thm:KRWA}.

\begin{theorem}\label{thm:cpt-inter}
Let $0< p_0, q_0 <\infty$ and $1 \leq p_j, q_j \leq \infty$, $j=1, \ldots, m$. Let $u_0^{p_0}, v_0^{q_0} \in A_{\infty}(\Rnn)$ and $u_j, v_j$ be weights on $\Rnn$, $j=1, \ldots, m$. Assume that 
\begin{align}
\label{cptinter-1} 
&T: L^{p_1}(u_1^{p_1}) \times \cdots \times L^{p_m}(u_m^{p_m}) \to L^{p_0}(u_0^{p_0}) \text{ boundedly}, 
\\
\label{cptinter-2} 
&T: L^{q_1}(v_1^{q_1}) \times \cdots \times L^{q_m}(v_m^{q_m}) \to L^{q_0}(v_0^{q_0}) \text{ compactly}. 
\end{align}
Then, 
\begin{align}\label{cptinter-3} 
T: L^{r_1}(w_1^{r_1}) \times \cdots \times L^{r_m}(w_m^{r_m}) \to L^{r_0}(w_0^{q_0}) \text{ compactly}, 
\end{align}
where 
\begin{align*}
\theta \in (0, 1), \quad 
w_j = u_j^{1-\theta} v_j^{\theta}, \quad 
\frac{1}{r_j} = \frac{1-\theta}{p_j} + \frac{\theta}{q_j}, \quad j=0, 1, \ldots, m. 
\end{align*}
\end{theorem}

\begin{proof}
Let $u_0^{p_0}, v_0^{q_0} \in A_{\infty}(\Rnn)$. Then $u_0^{p_0} \in A_s(\Rnn)$ for some $s \in (1, \infty)$. Picking $0 < a < \min\{p_0/s, 1\}$, we see that $u_0^{p_0} \in A_s(\Rnn) \subset A_{p_0/a}(\Rnn)$. 
Note that for any $\rho>0$, 
\begin{align*}
\bigg(\fint_{B_{\vec{n}}(0, \rho)} |\tau_y f(x) - f(x)|^a \, dy \bigg)^{\frac1a}
\lesssim M_{\mathcal{R}}(|f|^a)(x)^{\frac1a} + |f(x)|, 
\end{align*}
which, along with the weighted boundedness of $M_{\mathcal{R}}$ (cf. \cite[p. 453]{GR}) and \eqref{cptinter-1}, implies 
\begin{align}\label{tauy}
&\bigg\|\bigg(\fint_{B_{\vec{n}}(0, \rho)} |(\tau_y T - T)(\vec{f})|^a dy \bigg)^{\frac1a}\bigg\|_{L^{p_0}(u_0^{p_0})} 
\\ \nonumber
&\lesssim \|M_{\mathcal{R}}(|T(\vec{f})|^a)^{\frac1a}\|_{L^{p_0}(u_0^{p_0})} 
+ \|T(\vec{f})\|_{L^{p_0}(u_0^{p_0})} 
\\ \nonumber
&\lesssim \|T(\vec{f})\|_{L^{p_0}(u_0^{p_0})} 
\lesssim \prod_{j=1}^m \|f_j\|_{L^{p_j}(u_j^{p_j})}. 
\end{align}

Let $\varepsilon>0$. By \eqref{cptinter-2} and Theorem \ref{thm:KRWA}, there exist $A_0 = A_0(\varepsilon) > 0$ and $\rho_0 = \rho_0(\varepsilon) > 0$ such that for all $A \ge A_0$ and $0<\rho<\rho_0$, 
\begin{align}
\label{KAM-1}
\sup_{\substack{\|f_j\|_{L^{q_j}(v_j^{q_j})} \le 1 \\ j=1, \ldots, m}} 
& \|T(\vec{f})\|_{L^{q_0}(v_0^{q_0})} \le C_1, 
\\
\label{KAM-2}
\sup_{\substack{\|f_j\|_{L^{q_j}(v_j^{q_j})} \le 1 \\ j=1, \ldots, m}} 
& \|T(\vec{f}) \mathbf{1}_{B_{\vec{n}}(0, A)^c}\|_{L^{q_0}(v_0^{q_0})} \le \varepsilon, 
\\
\label{KAM-3}
\sup_{\substack{\|f_j\|_{L^{q_j}(v_j^{q_j})} \le 1 \\ j=1, \ldots, m}} 
& \bigg\|\bigg(\fint_{B_{\vec{n}}(0, \rho)} |(\tau_y T - T)(\vec{f})|^a dy 
\bigg)^{\frac1a}\bigg\|_{L^{q_0}(v_0^{q_0})} 
\le \varepsilon,  
\end{align}
where $C_1>0$ is an absolute constant. Moreover, by \eqref{cptinter-1} and \eqref{tauy}, for any $A>0$ and $\rho>0$, 
\begin{align}
\label{KAM-11}
\sup_{\substack{\|f_j\|_{L^{p_j}(u_j^{p_j})} \le 1 \\ j=1, \ldots, m}} 
& \|T(\vec{f})\|_{L^{p_0}(u_0^{p_0})} 
\le C_0, 
\\
\label{KAM-22}
\sup_{\substack{\|f_j\|_{L^{p_j}(u_j^{p_j})} \le 1 \\ j=1, \ldots, m}} 
& \|T(\vec{f}) \mathbf{1}_{B_{\vec{n}}(0, A)^c}\|_{L^{p_0}(u_0^{p_0})} 
\le C_0, 
\\
\label{KAM-33}
\sup_{\substack{\|f_j\|_{L^{p_j}(u_j^{p_j})} \le 1 \\ j=1, \ldots, m}} 
& \bigg\|\bigg(\fint_{B_{\vec{n}}(0, \rho)} |(\tau_y T - T)(\vec{f})|^a dy 
\bigg)^{\frac1a}\bigg\|_{L^{p_0}(u_0^{p_0})} 
\le C_0,  
\end{align}
where $C_0>0$ is an absolute constant. Thus, interpolating between \eqref{KAM-1} and \eqref{KAM-11} gives 
\begin{align}\label{KAM-111}
\sup_{\substack{\|f_j\|_{L^{r_j}(w_j^{r_j})} \le 1 \\ j=1, \ldots, m}} 
\|T(\vec{f})\|_{L^{r_0}(w_0^{r_0})} 
\le C_0^{1-\theta} C_1^{\theta}.  
\end{align}
Similarly, by \eqref{KAM-2} and \eqref{KAM-22}, 
\begin{align}\label{KAM-222}
\sup_{\substack{\|f_j\|_{L^{r_j}(w_j^{r_j})} \le 1 \\ j=1, \ldots, m}} 
& \|T(\vec{f}) \mathbf{1}_{B_{\vec{n}}(0, A)^c}\|_{L^{r_0}(w_0^{r_0})} 
\le C_0^{1-\theta} \varepsilon^{\theta}, 
\end{align}
for any $A \ge A_0$. Moreover, by interpolation on mixed-norm Lebesgue spaces (cf. \cite[Theorem 3.5]{COY}), \eqref{KAM-3} and \eqref{KAM-33} imply 
\begin{align}\label{KAM-333}
\sup_{\substack{\|f_j\|_{L^{r_j}(w_j^{r_j})} \le 1 \\ j=1, \ldots, m}} 
& \bigg\|\bigg(\fint_{B_{\vec{n}}(0, \rho)} |(\tau_y T - T)(\vec{f})|^a dy 
\bigg)^{\frac1a}\bigg\|_{L^{r_0}(w_0^{r_0})} 
\le C_0^{1-\theta} \varepsilon^{\theta}, 
\end{align}
for any $0<\rho<\rho_0$. Therefore, \eqref{cptinter-3} follows from \eqref{KAM-111}--\eqref{KAM-333} and Theorem \ref{thm:KRWA}. 
\end{proof}

\subsection{Proof of Theorem \ref{thm:RdF-cpt}}
Fix $\vec{r} = (r_1, \ldots, r_m) \in (1, \infty]^m$ with $\frac1r = \sum_{j=1}^m \frac{1}{r_j} > 0$. Let $\vec{w} \in A_{\vec{r}}(\Rnn)$ and $w = \prod_{j=1}^m w_j$. Recall that the assumption \eqref{RdFcpt-1}: 
\begin{align}\label{RdFcpt-3}
T: L^{p_1}(u_1^{p_1}) \times \cdots \times L^{p_m}(u_m^{p_m}) \to L^p(u^p) \text{ compactly}, 
\end{align} 
where $\vec{u} \in A_{\vec{p}}(\Rnn)$. Then by Lemma \ref{lem:AA} applied to $\vec{w} \in A_{\vec{r}}(\Rnn)$ and $\vec{u} \in A_{\vec{p}}(\Rnn)$, there exists $\theta \in (0, 1)$ such that $\vec{v} \in A_{\vec{s}}(\Rnn)$, where 
\begin{align}\label{eq:AA}
\frac1s = \sum_{j=1}^m \frac{1}{s_j}, \quad 
v = \prod_{j=1}^m v_j, \quad 
w_j = v_j^{1-\theta} u_j^{\theta},\quad 
\frac{1}{r_j} = \frac{1-\theta}{s_j} + \frac{\theta}{p_j}, \quad j=1, \ldots, m. 
\end{align}
In view of Theorem \ref{thm:RdF-bdd}, the assumption \eqref{RdFcpt-2} implies 
\begin{align}\label{TLT}
T: L^{t_1}(\sigma_1^{t_1}) \times \cdots \times L^{\sigma_m}(\sigma_m^{t_m}) \to L^t(\sigma^t)
\text{ boundedly}, 
\end{align}
for all $\vec{t} = (t_1, \ldots, t_m) \in (1, \infty]^m$ and for all $\vec{\sigma} = (\sigma_1, \ldots, \sigma_m) \in A_{\vec{t}}(\Rnn)$, where $\frac1t = \sum_{j=1}^m \frac{1}{t_j} > 0$ and $\sigma = \prod_{j=1}^m \sigma_j$. Then \eqref{TLT} applied to $\vec{v} \in A_{\vec{s}}(\Rnn)$ yields 
\begin{align}\label{tss}
T: L^{s_1}(v_1^{s_1}) \times \cdots \times L^{v_m}(v_m^{s_m}) \to L^s(v^s)
\text{ boundedly}. 
\end{align}
Hence, it follows from \eqref{RdFcpt-3}, \eqref{eq:AA}, \eqref{tss}, and Theorem \ref{thm:cpt-inter} that 
\begin{align*}
T: L^{r_1}(w_1^{r_1}) \times \cdots \times L^{r_m}(w_m^{r_m}) \to L^r(w^r) \text{ compactly}. 
\end{align*}
This completes the proof of Theorem \ref{thm:RdF-cpt}. 
\qed

\subsection{Extrapolation for commutators}
Before proving Theorem \ref{thm:RdF-bT}, let us present an extrapolation for commutators in the spirit of \cite{BMMST}.
 
\begin{theorem}\label{thm:TTb}
Let $T$ be an $m$-linear operator. Assume that there exists some $\vec{q} = (q_1, \ldots, q_m) \in [1, \infty]^m$ such that for all $\vec{v} = (v_1, \ldots, v_m) \in A_{\vec{q}}(\Rnn)$, 
\begin{align}\label{TTb-1}
\|T(\vec{f})\|_{L^q(v^q)} 
\lesssim \prod_{j=1}^m \|f_j\|_{L^{q_j}(v_j^{q_j})}, 
\end{align}
where $\frac1q = \sum_{j=1}^m \frac{1}{q_j}$ and $v = \prod_{j=1}^m v_j$. Then for any $\alpha \in \N^m$ and $\b = (b_1, \ldots, b_m) \in \bmo(\Rnn)^m$, 
\begin{align}\label{TTb-2}
\|[\b, T]_{\a}(\vec{f})\|_{L^p(w^p)} 
\lesssim \prod_{j=1}^m \|b_j\|_{\bmo}^{\alpha_j} \|f_j\|_{L^{p_j}(w_j^{p_j})},  
\end{align}
for all $\vec{p} = (p_1, \ldots, p_m) \in (1, \infty]^m$ and for all $\vec{w} = (w_1, \ldots, w_m) \in A_{\vec{p}}(\Rnn)$, where $\frac1p = \sum_{j=1}^m \frac{1}{p_j} > 0$ and $w = \prod_{j=1}^m w_j$.  
\end{theorem}

\begin{proof}
Fix $\alpha \in \N^m$ and $\b = (b_1, \ldots, b_m) \in \bmo(\Rnn)^m$. Choose $\vec{r} = (r_1, \ldots, r_m) \in (1, \infty)^m$ so that $\frac1r = \sum_{j=1}^m \frac{1}{r_j} < 1$. In view of Theorem \ref{thm:RdF-bdd}, the hypothesis \eqref{TTb-1} implies that 
\begin{align}\label{TTb-3}
\|T(\vec{f})\|_{L^r(u^r)} 
\lesssim \prod_{j=1}^m \|f_j\|_{L^{r_j}(u_j^{r_j})}, 
\end{align}
for all $\vec{u} = (u_1, \ldots, u_m) \in A_{\vec{r}}(\Rnn)$, where $u = \prod_{j=1}^m u_j$.

In the bi-parameter setting, recall the reverse H\"{o}lder inequality \cite[p. 458]{GR} and the John--Nirenberg inequality \cite[Section 5.6]{HT}. Using these two inequalities, as argued in \cite[p. 84]{BMMST}, one can show that for any $\vec{v}= (v_1, \ldots, v_m) \in A_{\vec{p}}(\Rnn)$, there holds
\begin{align}\label{TTb-4}
\vec{u} = (u_1, \ldots, u_m) 
:= \big(v_1 e^{-\mathrm{Re}(z_1) b_1}, \ldots, v_m e^{-\mathrm{Re}(z_m) b_m} \big) 
\in A_{\vec{p}}(\Rnn), 
\end{align}
whenever $|z_j| \le \eta_j$ for some $\eta_j>0$, $j=1, \ldots, m$. Then following the proof of {\cite[Theorem 4.13]{BMMST}}, we apply the Cauchy integral trick, \eqref{TTb-3}, and \eqref{TTb-4} to deduce that 
\begin{align}\label{TTb-5}
\|[\b, T]_{\a}(\vec{f})\|_{L^r(v^r)}
\lesssim \prod_{j=1}^m \|b_j\|_{\bmo}^{\alpha_j} \|f_j\|_{L^{r_j}(v_j^{r_j})}, 
\end{align}
for all $\vec{v} = (v_1, \ldots, v_m) \in A_{\vec{r}}(\Rnn)$, where $v = \prod_{j=1}^m v_j$. Therefore, it follows from \eqref{TTb-5} and Theorem \ref{thm:RdF-bdd} that 
\begin{align*}
\|[\b, T]_{\a}(\vec{f})\|_{L^p(w^p)}
\lesssim \prod_{j=1}^m \|b_j\|_{\bmo}^{\alpha_j} \|f_j\|_{L^{p_j}(w_j^{p_j})}, 
\end{align*}
for all $\vec{p} = (p_1, \ldots, p_m) \in (1, \infty]^m$ and for all $\vec{w} = (w_1, \ldots, w_m) \in A_{\vec{p}}(\Rnn)$, where $\frac1p = \sum_{j=1}^m \frac{1}{p_j} > 0$ and $w = \prod_{j=1}^m w_j$. This justifies \eqref{TTb-2}.  
\end{proof}

\subsection{Proof of Theorem \ref{thm:RdF-bT}}
Let $\vec{r} = (r_1, \ldots, r_m) \in (1, \infty]^m$ with $\frac1r = \sum_{j=1}^m \frac{1}{r_j} > 0$. Let $\vec{w} \in A_{\vec{r}}(\Rnn)$ and $w = \prod_{j=1}^m w_j$. Let $\varepsilon>0$ be an arbitrary number. By the assumption \eqref{RdFbT-1}, there exists some $\rho_0 = \rho_0(\varepsilon) > 0$ such that 
\begin{align}\label{RbT-1}
\bigg\|\bigg[\fint_{B_{\vec{n}}(0, \rho)}
\big|(\tau_y - \tau_{y_1} - \tau_{y_2} + I)[\b, T]_{\a} (\vec{f}) \big|^{a_0} \, dy 
\bigg]^{\frac{1}{a_0}}\bigg\|_{L^p(u^p)} 
\le \varepsilon \prod_{j=1}^m \|f_j\|_{L^{p_j}(u_j^{p_j})}, 
\end{align} 
for all $0 < \rho < \rho_0$ and for some $a_0 \in (0, 1]$. Recall that $\vec{w} \in A_{\vec{r}}(\Rnn)$ and $\vec{u} \in A_{\vec{p}}(\Rnn)$. Then invoking Lemma \ref{lem:AA}, one can find some $\theta \in (0, 1)$ such that $\vec{v} \in A_{\vec{s}}(\Rnn)$, where 
\begin{align}\label{RbT-2}
\frac1s = \sum_{j=1}^m \frac{1}{s_j}, \quad 
v = \prod_{j=1}^m v_j, \quad 
w_j = v_j^{1-\theta} u_j^{\theta},\quad 
\frac{1}{r_j} = \frac{1-\theta}{s_j} + \frac{\theta}{p_j}, \quad j=1, \ldots, m. 
\end{align}

On the other hand, by the assumption \eqref{RdFbT-2} and Theorem \ref{thm:TTb},
\begin{align}\label{RbT-3}
[\b, T]_{\a}: L^{t_1}(\sigma_1^{t_1}) \times \cdots \times L^{\sigma_m}(\sigma_m^{t_m}) \to L^t(\sigma^t)
\text{ boundedly}, 
\end{align}
for all $\vec{t} = (t_1, \ldots, t_m) \in (1, \infty]^m$ and for all $\vec{\sigma} = (\sigma_1, \ldots, \sigma_m) \in A_{\vec{t}}(\Rnn)$, where $\frac1t = \sum_{j=1}^m \frac{1}{t_j} > 0$ and $\sigma = \prod_{j=1}^m \sigma_j$. Hence, applying \eqref{RbT-3} to $\vec{v} \in A_{\vec{s}}(\Rnn)$, we obtain 
\begin{align}\label{RbT-4}
\|[\b, T]_{\a} (\vec{f})\|_{L^s(v^s)} 
\lesssim \prod_{j=1}^m \|f_j\|_{L^{s_j}(v_j^{s_j})}, 
\end{align}
where the implicit constant is independent of $\vec{f}$. In light of Theorem \ref{thm:ww}, there holds $v^s \in A_{ms}(\Rnn) \subset A_{s/a}(\Rnn)$, where $0 < a \le \min\{a_0, \frac{1}{m}\}$. Observe that for any $\rho > 0$ and $x \in \Rnn$,
\begin{align*} 
& \bigg[\fint_{B_{\vec{n}}(0, \rho)} 
\big|(\tau_y - \tau_{y_1} - \tau_{y_2} + I) f(x) \big|^a \, dy \bigg]^{\frac1a}
\\
& \lesssim M_{\mathcal{R}}(|f|^a)(x)^{\frac1a} 
+ M_{n_1}(|f_{x_2}|^a)(x_1)^{\frac1a} 
+ M_{n_2}(|f_{x_1}|^a)(x_2)^{\frac1a} 
+ |f(x)|, 
\end{align*} 
which, along with the weighted boundedness of $M_{\mathcal{R}}$ (cf. \cite[p. 453]{GR}), \eqref{vvqq}, and \eqref{RbT-4}, implies 
\begin{align}\label{RbT-5}
&\bigg\|\bigg[\fint_{B_{\vec{n}}(0, \rho)} 
\big|(\tau_y - \tau_{y_1} - \tau_{y_2} + I) [\b, T]_{\a}(\vec{f}) \big|^a dy \bigg]^{\frac1a}\bigg\|_{L^s(v^s)} 
\\ \nonumber
&\lesssim \big\|M_{\mathcal{R}} \big(|[\b, T]_{\a}(\vec{f})|^a \big)^{\frac1a}\big\|_{L^s(v^s)} 
+ \big\|M_{n_1} \big( \big| \big([\b, T]_{\a}(\vec{f})\big)_{x_2} \big|^a \big)(x_1)^{\frac1a}\big\|_{L^s(v^s)}
\\ \nonumber
&\quad + \big\|M_{n_2} \big( \big| \big([\b, T]_{\a}(\vec{f})\big)_{x_1} \big|^a \big)(x_2)^{\frac1a}\big\|_{L^s(v^s)}
+ \|[\b, T]_{\a}(\vec{f})\|_{L^s(v^s)} 
\\ \nonumber
&\lesssim \|[\b, T]_{\a}(\vec{f})\|_{L^s(v^s)} 
\lesssim \prod_{j=1}^m \|f_j\|_{L^{s_j}(v_j^{s_j})}, 
\end{align}
for all $\rho > 0$. As a consequence of \eqref{RbT-1} (which holds for $a$ in place of $a_0$), \eqref{RbT-2}, \eqref{RbT-5}, and \cite[Theorem 3.5]{COY}, we conclude that 
\begin{align*}
\sup_{\substack{\|f_j\|_{L^{r_j}(w_j^{r_j})} \le 1 \\ j=1, \ldots, m}} 
\bigg\|\bigg[\fint_{B_{\vec{n}}(0, \rho)}
\big|(\tau_y - \tau_{y_1} - \tau_{y_2} + I)[\b, T]_{\a} (\vec{f}) \big|^a \, dy 
\bigg]^{\frac1a}\bigg\|_{L^r(w^r)} 
\le C_0^{1-\theta} \varepsilon^{\theta}, 
\end{align*}
for all $0 < \rho < \rho_0$. This and \eqref{RbT-3} show Theorem \ref{thm:RdF-bT}. 
\qed

\subsection{Proof of Theorem \ref{thm:cpt}}
Having established fundamental estimates above, let us demonstrate Theorem \ref{thm:cpt}.

Pick $\vec{p} = (p_1, \ldots, p_m) \in (1, \infty)^m$ such that $\frac1p = \sum_{j=1}^m \frac{1}{p_j} \in (0, 1)$. Let $\vec{u} = (u_1, \ldots, u_m) \in A_{\vec{p}}(\Rnn)$. Theorem \ref{thm:repre} asserts that $T$ admits a compact $m$-linear bi-parameter dyadic representation. This, along with Theorem \ref{thm:dyadic-cpt}, gives  
\begin{align}\label{tth-1}
T: L^{p_1}(u_1^{p_1}) \times \cdots \times L^{p_m}(u_m^{p_m}) \to L^p(u^p) 
\, \text{ compactly}, 
\end{align}
where we have used an $m$-linear bi-parameter version of \cite[Lemma 4.52]{CLSY}. Moreover, the hypotheses \eqref{H1}--\eqref{H5} imply that $T$ is an $m$-linear bi-parameter Calder\'{o}n--Zygmund operator, which together with \cite[Theorem 1.2]{LMV21} yields 
\begin{align}\label{tth-2}
T: L^{q_1}(v_1^{q_1}) \times \cdots \times L^{q_m}(v_m^{q_m}) \to L^q(v^q) 
\, \text{ boundedly}, 
\end{align}
for all $\vec{q} = (q_1, \ldots, q_m) \in (1, \infty]^m$ and for all $\vec{v} = (v_1, \ldots, v_m) \in A_{\vec{q}}(\Rnn)$, where $\frac1q = \sum_{j=1}^m \frac{1}{q_j} > 0$ and $v = \prod_{j=1}^m v_j$. Accordingly, invoking Theorem \ref{thm:RdF-cpt}, \eqref{tth-1}, and \eqref{tth-2}, we conclude that 
\begin{align*}
T: L^{r_1}(w_1^{r_1}) \times \cdots \times L^{r_m}(w_m^{r_m}) \to L^r(w^r) 
\, \text{ compactly}, 
\end{align*} 
for all $\vec{r} = (r_1, \ldots, r_m) \in (1, \infty]^m$ and $\vec{w} = (w_1, \ldots, w_m) \in A_{\vec{r}}(\Rnn)$, where $\frac1r = \sum_{j=1}^m \frac{1}{r_j} > 0$ and $w = \prod_{j=1}^m w_j$. The proof is complete.  
\qed

\subsection{Proof of Theorem \ref{thm:bT}}\label{sec:bTc}
Let $\b = (b_1, \ldots, b_m) \in \cmo(\Rnn)^m$ and $T$ be an $m$-linear bi-parameter Calder\'{o}n--Zygmund operator. Then in view of Theorem \ref{thm:RdF-bT} and the weighted boundedness \eqref{tth-2}, to obtain the mean continuity of $[\b, T]_{\a}$ on weighted Lebesgue spaces, it suffices to show that 
\begin{equation}\label{eq:bT}
\begin{array}{c}
\text{$[\b, T]_{\a}$ is mean continuous from} 
\\[4pt]
L^{p_1}(\Rnn) \times \cdots \times L^{p_m}(\Rnn) \text{ to } L^p(\Rnn)
\\[4pt]
\text{for some $\vec{p} = (p_1, \ldots, p_m) \in (1, \infty)^m$ with $\frac1p = \sum_{j=1}^m \frac{1}{p_j} \in (0, 1)$}.
\end{array}
\end{equation} 
By \cite[Proposition 7.7]{LMV21}, $T$ admits an $m$-linear bi-parameter dyadic representation (cf. Definition \ref{def:repre}). This and Theorem \ref{thm:dyadic-bT} imply \eqref{eq:bT} as desired. 
\qed 
\black

\appendix
\section{Uchiyama's characterization of $\CMO$ space}\label{sec:Uch}
For the convenience of the reader, we present all details in the proof of the characterization of $\CMO$ although the central case $p=1$ was given in \cite[Lemma 3]{Uch}. It is helpful to understand the precise structure of $\CMO(\Rn)$.

\begin{theorem}\label{thm:Uch}
Let $p \in [1, \infty)$ and $f \in L^1_{\loc}(\Rn)$. Then $f \in \CMO(\Rn)$ if and only if $f \in \BMO(\Rn)$ and the following properties hold:
\begin{align}
\label{rr-1}
\gamma_1(f) & :=
\lim_{r \to 0} \sup_{\substack{Q \in \mathcal{Q} \\ \ell(Q) \le r}} 
\mathcal{O}_p(f; Q) = 0,
\\
\label{rr-2}
\gamma_2(f) & :=
\lim_{r \to \infty} \sup_{\substack{Q \in \mathcal{Q} \\ \ell(Q) \ge r}}
\mathcal{O}_p(f; Q) = 0,
\\
\label{rr-3}
\gamma_3(f) & :=
\lim_{r \to \infty} \sup_{\substack{Q \in \mathcal{Q} \\ \d(Q, 0) \ge r}}
\mathcal{O}_p(f; Q) = 0,
\end{align}
where $\displaystyle{\mathcal{O}_p(f; Q) := \bigg[\fint_Q |f(x) - \langle f \rangle_Q|^p \, dx \bigg]^{\frac1p}}$. 
\end{theorem}

\begin{proof}
Let $f \in \CMO(\Rn)$ and $\varepsilon>0$. Then there exists $f_{\varepsilon} \in \mathscr{C}_c^{\infty}(\Rn)$ such that $\|f - f_{\varepsilon}\|_{\BMO} \le \varepsilon$, which gives $\|f\|_{\BMO} \le \|f - f_{\varepsilon}\|_{\BMO} + \|f_{\varepsilon}\|_{\BMO} \le \varepsilon + 2 \|f_{\varepsilon}\|_{L^{\infty}} < \infty$. Since $f_{\varepsilon}$ is uniformly continuous on $\Rn$, there exists some $r_{\varepsilon} > 0$ so that
\begin{align}
\sup_{|x-y| \le r_{\varepsilon}} |f_{\varepsilon}(x) - f_{\varepsilon}(y)| \le \varepsilon.
\end{align}
By definition, we have
\begin{align}\label{OF}
\mathcal{O}_p(f; Q) 
\le \mathcal{O}_p(f - f_{\varepsilon}; Q) 
+ \mathcal{O}_p(f_{\varepsilon}; Q),
\end{align}
which implies 
\begin{align}\label{OF-1}
\mathcal{O}_p(f; Q) 
\le \|f - f_{\varepsilon}\|_{\BMO} 
+ \sup_{x, y \in Q} |f_{\varepsilon}(x) - f_{\varepsilon}(y)|
\le 2 \varepsilon, 
\quad \forall Q: \ell(Q) \le r_{\varepsilon},
\end{align}
and
\begin{align}\label{OF-2}
\mathcal{O}_p(f; Q) 
\le \|f - f_{\varepsilon}\|_{\BMO} + 2 \|f_{\varepsilon}\|_{L^p} |Q|^{-\frac1p}
\le 2 \varepsilon, 
\quad \forall Q: |Q| \ge 2^p \varepsilon^{-p} \|f_{\varepsilon}\|_{L^p}^p.
\end{align}
Let $\supp(f_{\varepsilon}) \subset B(0, A_{\varepsilon}/2)$ for some $A_{\varepsilon} > 0$. If $\d(Q, 0) \ge A_{\varepsilon}$, then $Q \cap \supp(f_{\varepsilon}) = \emptyset$ and $\mathcal{O}_p(f_{\varepsilon}; Q) = 0$. This along with \eqref{OF} yields 
\begin{align}\label{OF-3}
\mathcal{O}_p(f; Q) 
\le \|f-f_{\varepsilon}\|_{\BMO} + 0
\le \varepsilon, \quad \forall Q: \d(Q, 0) \ge A_{\varepsilon}.
\end{align}
Hence, \eqref{OF-1}--\eqref{OF-3} imply $\gamma_1(f) = \gamma_2(f) = \gamma_3(f) = 0$.

To show the converse, it suffices to treat the case $p=1$ because $\mathcal{O}(f; Q) := \mathcal{O}_1(f; Q) \le \mathcal{O}_p(f; Q)$ for any $p \in [1, \infty)$. Let $f \in \BMO(\Rn)$ with $\gamma_i(f) = 0$, $i=1, 2, 3$. Let $\varepsilon > 0$ be an arbitrary number. It is enough to show that there exist $f_{\varepsilon} \in \mathscr{C}_c^{\infty}(\Rn)$ and a simple function $g_{\varepsilon}$ such that 
\begin{align*}
\|f_{\varepsilon} - g_{\varepsilon}\|_{\BMO} \lesssim \varepsilon 
\quad \text{ and } \quad 
\|f - g_{\varepsilon}\|_{\BMO} \lesssim \varepsilon.
\end{align*}
For any $r>0$, let $Q(r)$ denote the cube centered at the origin with sidelength $2r$. By the condition $\gamma_1(f) = \gamma_3(f) = 0$, there exists some $j_0 = j_0(\varepsilon) \ge 2$ such that 
\begin{align}\label{OFj0}
\sup_{\ell(Q) \le 2^{-j_0}} \mathcal{O}(f; Q) 
+ \sup_{\d(Q, 0) \ge 2^{j_0}} \mathcal{O}(f; Q) 
\le \varepsilon.
\end{align}
Moreover, the fact $\gamma_2(f) = 0$ implies that for some $j_1 = j_1(\varepsilon) > 3j_0$,  
\begin{align}\label{OFk0}
\sup_{\ell(Q) \ge 2^{j_1}} \mathcal{O}(f; Q) 
\le 2^{-2nj_0} \varepsilon.
\end{align}
Let $\D$ be the standard dyadic grid on $\Rn$. Define $\mathcal{F} = \bigcup_{j \ge j_0} \mathcal{F}_j$, where 
\begin{align*}
& \mathcal{F}_{j_0} 
:= \big\{Q \in \D: Q \subset \overline{Q(2^{j_0})}, \ell(Q) = 2^{-j_0}\big\},
\\
& \mathcal{F}_j 
:= \big\{Q \in \D: Q \subset \overline{Q(2^j) \setminus Q(2^{j-1})}, \ell(Q) = 2^{-j_0+j-j_0}\big\}, \quad \forall j > j_0.
\end{align*}
Observe that $\mathcal{F}$ is a collection of disjoint dyadic cubes, and for any $x \in \Rn$, there exists a unique cube $Q_x \in \mathcal{F}$ such that $x \in Q_x$. Then by \eqref{OFj0},
\begin{align}\label{Oscx}
\sup_{x \in \Rn} \mathcal{O}(f; Q_x) 
\le \varepsilon.
\end{align}

Now choose a non-negative function $\varphi \in \mathscr{C}_c^{\infty}(\Rn)$ with $\supp(\varphi) \subset B(0, 1/2)$ and $\int_{\Rn} \varphi \, dx =1$. Define $\varphi_{j_0}(x) := 2^{n j_0} \varphi(2^{j_0} x)$,
\begin{equation*}
g_{\varepsilon}(x) 
= \begin{cases}
\langle f \rangle_{Q_x}, & x \in Q(2^{j_1}),
\\
\langle f \rangle_{Q(2^{j_1}) \setminus Q(2^{j_1-1})}, & x \notin Q(2^{j_1}),
\end{cases}
\end{equation*}
and 
\begin{equation*}
f_{\varepsilon} 
= \varphi_{j_0} * \big(g_{\varepsilon} - \langle f \rangle_{Q(2^{j_1}) \setminus Q(2^{j_1-1})}\big).
\end{equation*}
Note that $g_{\varepsilon}$ is a simple function and $f_{\varepsilon} \in \mathscr{C}_c^{\infty}(\Rn)$. In addition, for any $x, y \in \Rn$ with $|x-y| \le 2^{-j_0-1}$, it must hold $\overline{Q}_x \cap \overline{Q}_y \neq \emptyset$. Then by Lemma \ref{lem:gege}, 
\begin{align*}
\|f_{\varepsilon} - g_{\varepsilon}\|_{\BMO}
& = \|\varphi_{j_0}* g_{\varepsilon} - g_{\varepsilon}\|_{\BMO}
\le 2 \|\varphi_{j_0}* g_{\varepsilon} - g_{\varepsilon}\|_{L^{\infty}}
\\
&\le 2 \|\varphi_{j_0}\|_{L^1}
\sup_{|x-y| \le 2^{-j_0-1}} |g_{\varepsilon}(x) - g_{\varepsilon}(y)|
\lesssim  \varepsilon. 
\end{align*}

Let us next verify $\|f - g_{\varepsilon}\|_{\BMO} \lesssim \varepsilon$. Let $Q$ be an arbitrary cube in $\Rn$. We are going to demonstrate
\begin{align}\label{OFV}
\mathcal{O}(f - g_{\varepsilon}; Q) 
\lesssim \varepsilon,
\end{align}
where the implicit constant depends only on the dimension $n$. 
First, in the case $Q \cap Q(2^{j_1-1}) = \emptyset$, we invoke \eqref{OFj0} and Lemma \ref{lem:gege} to obtain
\begin{align*}
\mathcal{O}(f - g_{\varepsilon}; Q) 
\le \mathcal{O}(f; Q) + \sup_{x, y \notin Q(2^{j_1-1})} |g_{\varepsilon}(x) - g_{\varepsilon}(y)|
\lesssim \varepsilon.
\end{align*}
Second, to consider the case $Q \subset Q(2^{j_1})$, set $\ell(Q_{x_0}) := \max \big\{\ell(Q_x): Q_x \cap Q \ne \emptyset \big\}$. In the scenario $\ell(Q_{x_0}) \le 4 \ell(Q)$, observe that 
\begin{align}\label{Qxfge}
\fint_{Q_x} |f - g_{\varepsilon}| \, dy 
\lesssim  \varepsilon, \quad \forall x \in Q(2^{j_1}). 
\end{align}
Indeed, for all $x \in Q(2^{j_1})$ and $y \in Q_x$, we have $y \in Q_y = Q_x \subset Q(2^{j_1})$ and $g_{\varepsilon}(y) = \langle f \rangle_{Q_y} = \langle f \rangle_{Q_x}$, which imply $\fint_{Q_x} |f - g_{\varepsilon}| \, dy = \mathcal{O}(f; Q_x) \le \varepsilon$, provided \eqref{Oscx}. The condition $Q_x \cap Q \ne \emptyset$ and $\ell(Q_x) \le 4 \ell(Q)$ implies $Q_x \subset 9Q$. Thus, by \eqref{Qxfge} and the disjointness of $\{Q_x\}$,
\begin{align*}
\mathcal{O}(f-g_{\varepsilon}; Q) 
\le 2 \fint_Q |f - g_{\varepsilon}| \, dy 
\le \sum_{Q_x \cap Q \ne \emptyset} \frac{2}{|Q|} 
\int_{Q_x} |f - g_{\varepsilon}| \, dy 
\lesssim \frac{|9 Q|}{|Q|} \varepsilon
\lesssim \varepsilon.
\end{align*}
In the setting $\ell(Q_{x_0}) > 4 \ell(Q)$, note that 
\begin{align}\label{QQM}
Q \cap Q(2^{j_0}) \ne \emptyset 
\quad \Longrightarrow \quad 
\ell(Q) < 2^{-j_0}.
\end{align}
Indeed, if $x_0 \in Q(2^{j_0+1})$, then $\ell(Q) < \ell(Q_{x_0})/4 \le 2^{-j_0+1}/4 < 2^{-j_0}$; if $x_0 \in Q(2^j) \setminus Q(2^{j-1})$ for some $j \ge j_0+2$, then $Q_{x_0} \subset Q(2^j) \setminus Q(2^{j-1})$ and $\ell(Q) < \ell(Q_{x_0})/4 = 2^{j-2j_0-2}$, which contradicts the fact $\ell(Q) \ge 2^{j-1} - 2^{j_0}$ because $Q \cap Q_{x_0} \ne \emptyset$, $Q \cap Q(2^{j_0}) \ne \emptyset$, and $2^{j-1} - 2^{j_0} - 2^{j-2j_0-2} = 2^{j-1}[1 - 2^{-(j-j_0-2)-1} - 2^{-2j_0-1}] > 0$. Consequently, it follows from \eqref{OFj0} and \eqref{QQM} that 
\begin{align}\label{Ofeq}
\mathcal{O}(f; Q) \le \varepsilon.
\end{align}
In addition, for any $x \in Q(2^{j_1})$ with $\overline{Q}_x \cap \overline{Q}_{x_0} \ne \emptyset$, there holds $2 \ell(Q) \le \ell(Q_{x_0})/2 \le \ell(Q_x) \le 2 \ell(Q_{x_0})$. This indicates 
\begin{align*}
Q \subset \bigcup_{x \in \Lambda} Q_x 
\quad\text{with} \quad Q_x \cap Q \ne \emptyset 
\quad \text{and} \quad 
\# \Lambda \lesssim 1.
\end{align*}
Then for all $x, y \in Q$, one can a sequence $\{x_i\}_{i=1}^{N_0} \subset \Lambda$ such that $\overline{Q}_{x_i} \cap \overline{Q}_{x_{i+1}} \neq \emptyset$, $i=0, 1, \ldots, N_0 = N_0(x, y)$, where $x_0 := x$ and $x_{N_0+1} := y$. Hence, by Lemma \ref{lem:gege},
\begin{align*}
\mathcal{O}(g_{\varepsilon}; Q) 
\le \sup_{x, y \in Q} |g_{\varepsilon}(x) - g_{\varepsilon}(y)|
\le \sup_{x, y \in Q} \sum_{i=0}^{N_0} |g_{\varepsilon}(x_i) - g_{\varepsilon}(x_{i+1})|
\lesssim \varepsilon,
\end{align*}
which along with \eqref{Ofeq} yields 
\begin{align*}
\mathcal{O}(f-g_{\varepsilon}; Q) 
\le \mathcal{O}(f; Q) + \mathcal{O}(g_{\varepsilon}; Q) 
\lesssim \varepsilon.
\end{align*}

Finally, to treat the case $Q \cap Q(2^{j_1-1}) \ne \emptyset$ and $Q \cap Q(2^{j_1})^c \ne \emptyset$, let $j_2 > j_1$ be the smallest integer satisfying $Q \subset Q(2^{j_2})$. Then, 
\begin{align}\label{Ofg12}
\mathcal{O}(f - g_{\varepsilon}; Q) 
\lesssim \mathcal{O}(f - g_{\varepsilon}; Q(2^{j_2}))
\le \mathscr{I}_1 + \mathscr{I}_2,
\end{align}
where
\begin{align*}
\mathscr{I}_1
& :=  \frac{1}{|Q(2^{j_2})|} \int_{Q(2^{j_2}) \setminus Q(2^{j_1})} 
|(f-g_{\varepsilon}) - \langle f-g_{\varepsilon} \rangle_{Q(2^{j_2})}| \, dy,
\\
\mathscr{I}_2
& := \frac{1}{|Q(2^{j_2})|} \int_{Q(2^{j_1})} 
|(f-g_{\varepsilon}) - \langle f-g_{\varepsilon} \rangle_{Q(2^{j_2})}| \, dy.
\end{align*}
We split
\begin{align*}
\langle g_{\varepsilon} \rangle_{Q(2^{j_2})}
= \bigg[1 - \frac{|Q(2^{j_1})|}{|Q(2^{j_2})|}\bigg]
\langle f \rangle_{Q(2^{j_1}) \setminus Q(2^{j_1-1})}
+ \frac{|Q(2^{j_1})|}{|Q(2^{j_2})|} 
\langle g_{\varepsilon} \rangle_{Q(2^{j_1})},
\end{align*}
which together with \eqref{OFj0} implies
\begin{align}\label{Ofg121}
\mathscr{I}_1
& \le \mathcal{O}(f; Q(2^{j_2}))
+ \big| \langle f \rangle_{Q(2^{j_1}) \setminus Q(2^{j_1-1})} 
- \langle g_{\varepsilon} \rangle_{Q(2^{j_2})} \big|
\\ \nonumber
& \le \varepsilon + \big|\langle f \rangle_{Q(2^{j_1}) \setminus Q(2^{j_1-1})} 
- \langle g_{\varepsilon} \rangle_{Q(2^{j_1})} \big|
\\ \nonumber
& \le \varepsilon 
+ \big|\langle f \rangle_{Q(2^{j_1}) \setminus Q(2^{j_1-1})} 
- \langle f \rangle_{Q(2^{j_1})} \big|
+ \langle |f - g_{\varepsilon}| \rangle_{Q(2^{j_1})}
\\ \nonumber
& \le \varepsilon 
+ 2^n \mathcal{O}(f; Q(2^{j_1})) 
+ \langle |f - g_{\varepsilon}| \rangle_{Q(2^{j_1})}
\\ \nonumber
& \le (1+2^n) \varepsilon 
+ \langle |f - g_{\varepsilon}| \rangle_{Q(2^{j_1})}. 
\end{align}
Similarly, 
\begin{align}\label{Ofg122}
\mathscr{I}_2
& \le \frac{1}{|Q(2^{j_2})|} \int_{Q(2^{j_1})} 
|(f-g_{\varepsilon}) - \langle f-g_{\varepsilon} \rangle_{Q(2^{j_2})}| \, dy
\\ \nonumber
& \le \langle |f - g_{\varepsilon}| \rangle_{Q(2^{j_1})} 
+ 2^{-(j_2 - j_1) n} 
|\langle f \rangle_{Q(2^{j_2})} - \langle g_{\varepsilon} \rangle_{Q(2^{j_2})}|
\\ \nonumber
& \le \langle |f - g_{\varepsilon}| \rangle_{Q(2^{j_1})} 
+ |\langle f \rangle_{Q(2^{j_1}) \setminus Q(2^{j_1-1})} 
- \langle g_{\varepsilon} \rangle_{Q(2^{j_2})}|
\\ \nonumber
&\quad + 2^{-(j_2 - j_1) n} 
|\langle f \rangle_{Q(2^{j_2})} - \langle f \rangle_{Q(2^{j_1}) \setminus Q(2^{j_1-1})}|
\\ \nonumber
& \lesssim \varepsilon + \langle |f - g_{\varepsilon}| \rangle_{Q(2^{j_1})} 
+ 2^{-(j_2 - j_1) n} 
|\langle f \rangle_{Q(2^{j_2})} - \langle f \rangle_{Q(2^{j_1}) \setminus Q(2^{j_1-1})}|.
\end{align}
The inequality \eqref{Oscx} gives
\begin{align}\label{Ofg123}
\langle |f - g_{\varepsilon}| \rangle_{Q(2^{j_1})} 
& \le \sum_{Q_x \subset Q(2^{j_1})} \frac{1}{|Q(2^{j_1})|} 
\int_{Q_x} |f - g_{\varepsilon}| \, dy
\\ \nonumber
& = \sum_{Q_x \subset Q(2^{j_1})} \frac{|Q_x|}{|Q(2^{j_1})|} \mathcal{O}(f; Q_x)
\le \varepsilon,
\end{align}
and \eqref{OFk0} leads to
\begin{align}\label{Ofg124}
& |\langle f \rangle_{Q(2^{j_2})} - \langle f \rangle_{Q(2^{j_1}) \setminus Q(2^{j_1-1})}|
\\ \nonumber
& \le \sum_{j_1 < j \le j_2} |\langle f \rangle_{Q(2^j)} - \langle f \rangle_{Q(2^{j-1})}| 
+ |\langle f \rangle_{Q(2^{j_1})} - \langle f \rangle_{Q(2^{j_1}) \setminus Q(2^{j_1-1})}|
\\ \nonumber
& \le \sum_{j_1 \le j \le j_2} 2^n \mathcal{O}(f; Q(2^j)) 
\lesssim (j_2 - j_1) \varepsilon.
\end{align}
Now gathering \eqref{Ofg12}--\eqref{Ofg124}, we conclude 
\begin{align*}
\mathcal{O}(f - g_{\varepsilon}; Q) 
\lesssim \varepsilon + (j_2 - j_1) 2^{-(j_2 - j_1)n} \varepsilon 
\le 2 \varepsilon.
\end{align*}
This completes the proof of \eqref{OFV}.
\end{proof}

\begin{lemma}\label{lem:gege}
If $x, y \notin Q(2^{j_1-1})$ or $\overline{Q}_x \cap \overline{Q}_y \neq \emptyset$, then 
\begin{align*}
|g_{\varepsilon}(x) - g_{\varepsilon}(y)|
\lesssim  \varepsilon. 
\end{align*}
\end{lemma}

\begin{proof}
We first deal with the case $x, y \notin Q(2^{j_1-1})$. If $x, y \notin Q(2^{j_1})$, then $g_{\varepsilon}(x) = g_{\varepsilon}(y)$. If $x, y \in Q(2^{j_1}) \setminus Q(2^{j_1-1})$, we have $Q_x \cup Q_y \subset  Q(2^{j_1})$, $\ell(Q_x) = \ell(Q_y) = 2^{j_1-2j_0}$, and 
\begin{align*}
& |g_{\varepsilon}(x) - g_{\varepsilon}(y)|
= |\langle f \rangle_{Q_x} - \langle f \rangle_{Q_y}| 
\\ 
& \le |\langle f \rangle_{Q_x} - \langle f \rangle_{Q(2^{j_1})}|
+ |\langle f \rangle_{Q_y} - \langle f \rangle_{Q(2^{j_1})}| 
\\
& \le \fint_{Q_x} |f - \langle f \rangle_{Q(2^{j_1})}| 
+ \fint_{Q_y} |f - \langle f \rangle_{Q(2^{j_1})}|
\\
& \le \bigg[\frac{|Q(2^{j_1})|}{|Q_x|} + \frac{|Q(2^{j_1})|}{|Q_y|} \bigg]
\fint_{Q(2^{j_1})} |f - \langle f \rangle_{Q(2^{j_1})}|
\\
& = 2^{2nj_0+1} \mathcal{O}(f; Q(2^{j_1})) 
\lesssim \varepsilon,
\end{align*}
where \eqref{OFk0} was used in the last step. If $x \notin Q(2^{j_1})$ and $y \in Q(2^{j_1}) \setminus Q(2^{j_1-1})$, then the inequality \eqref{OFk0} gives
\begin{align*}
& |g_{\varepsilon}(x) - g_{\varepsilon}(y)|
= |\langle f \rangle_{Q(2^{j_1}) \setminus Q(2^{j_1-1})} - \langle f \rangle_{Q_y}| 
\\
& \le \fint_{Q_y} |f- \langle f \rangle_{Q(2^{j_1})}| 
+ \fint_{Q(2^{j_1}) \setminus Q(2^{j_1-1})} |f- \langle f \rangle_{Q(2^{j_1})}|
\\
& \le \big( 2^{(2j_0+1)n} + 2^n \big) \mathcal{O}(f; Q(2^{j_1}))
\lesssim  \varepsilon. 
\end{align*}
A similar argument holds in the case $x \in Q(2^{j_1}) \setminus Q(2^{j_1-1})$ and $y \notin Q(2^{j_1})$. 

Let us turn to the situation $\overline{Q}_x \cap \overline{Q}_y \neq \emptyset$. The case $x, y \notin Q(2^{j_1-1})$ has been shown above. If $x \in Q(2^{j_1-1})$ and $y \notin Q(2^{j_1-1})$, then the condition $\overline{Q}_x \cap \overline{Q}_y \neq \emptyset$ implies 
\begin{align*}
& y \in Q(2^{j_1}) \setminus Q(2^{j_1-1}), \quad 
2\ell(Q_x) = \ell(Q_y) = 2^{j_1-2j_0}, \quad 
Q_x \subset 2Q_y, \quad \text{and}
\\
& \d(2Q_y, 0) = 2^{j_1 - 1} - \ell(Q_y)/2 
= 2^{j_0} \cdot 2^{j_1 - 3j_0 - 1} (2^{2j_0} - 1) \ge 2^{j_0},
\end{align*}
which together with \eqref{OFj0} leads to
\begin{align*}
|g_{\varepsilon}(x) - g_{\varepsilon}(y)|
& = |\langle f \rangle_{Q_x} - \langle f \rangle_{Q_y}|
\le \fint_{Q_x} |f - \langle f \rangle_{2Q_y}|
+ \fint_{Q_y} |f - \langle f \rangle_{2Q_y}|
\\
& \le \bigg[\frac{|2Q_y|}{|Q_x|} + \frac{|2Q_y|}{|Q_y|}\bigg] 
\fint_{2Q_y} |f - \langle f \rangle_{2Q_y}|
= (4^n + 2^n) \mathcal{O}(f, 2Q_y) 
\lesssim \varepsilon.
\end{align*}
Symmetrically, one can handle the case $x \notin Q(2^{j_1-1})$ and $y \in Q(2^{j_1-1})$.
\end{proof}

\end{document}